\documentclass[11pt]{amsart}
    \usepackage{amsmath,amssymb,amsthm,amsfonts,amsrefs}
    \usepackage{geometry}
    \usepackage{graphics,graphicx}
    \usepackage[all]{xy}
    \usepackage[colorlinks, linkcolor=blue, anchorcolor=blue,
            citecolor=blue]{hyperref}

    \usepackage[dvipsnames]{xcolor}
    \usepackage{enumerate}
    \usepackage{tikz}
    \usepackage{tikz-cd}
    \usepackage{xcolor,float}
    \usepackage{mathrsfs}
    \usepackage{caption}
    \usepackage{comment}

    \usepackage{fancyhdr}
    \usepackage{mathrsfs}

\newtheorem{thm}{Theorem}[section]
\newtheorem{lemma}[thm]{Lemma}
\newtheorem{example}[thm]{Example}
\newtheorem{prop}[thm]{Proposition}

\newtheorem{cor}[thm]{Corollary}
\newtheorem{question}[thm]{Question}

\newtheorem{definition}[thm]{Definition}
\newtheorem{remark}[thm]{Remark}

\newtheorem{ex}[thm]{Example}
\theoremstyle{remark}

\numberwithin{equation}{section}

\newcommand{\D}{\mathbb{D}}
\newcommand{\N}{\mathbb{N}}
\newcommand{\Z}{\mathbb{Z}}

\newcommand{\R}{\mathbb{R}}
\newcommand{\C}{\mathbb{C}}

\newcommand{\SL}{\mathcal{L}}

\newcommand{\A}{\mathcal{A}}
\newcommand{\La}{\Lambda}
\newcommand{\la}{\lambda}

\renewcommand{\a}{\alpha}
\newcommand{\w}{\mathtt{w}}

\newcommand{\dd}{\partial}

\newcommand{\pfo}{\noindent Proof of }

\newcommand{\sse}{\subseteq}
\newcommand{\lr}{\longrightarrow}

\newcommand{\GL}{\operatorname{GL}}
\newcommand{\PGL}{\operatorname{PGL}}

\newcommand{\sh}{\operatorname{sh}}

\newcommand{\Mod}{\operatorname{mod}}
\newcommand{\Fun}{\operatorname{Fun}}
\newcommand{\Perf}{\operatorname{Perf}}
\newcommand{\Loc}{\operatorname{Loc}}
\newcommand{\Sh}{\operatorname{Sh}}

\newcommand{\ol}{\overline}

\newcommand{\bC}{\mathbb{C}}
\newcommand{\bD}{\mathbb{D}}

\newcommand{\bG}{\mathbb{G}}

\newcommand{\bR}{\mathbb{R}}
\newcommand{\bZ}{\mathbb{Z}}

\newcommand{\cA}{\mathcal{A}}
\newcommand{\cB}{\mathcal{B}}
\newcommand{\cC}{\mathcal{C}}
\newcommand{\cD}{\mathcal{D}}

\newcommand{\cF}{\mathcal{F}}
\newcommand{\cG}{\mathcal{G}}

\newcommand{\cI}{\mathcal{I}}

\newcommand{\cL}{\mathcal{L}}
\newcommand{\cM}{\mathcal{M}}

\newcommand{\cR}{\mathcal{R}}

\newcommand{\cX}{\mathcal{X}}

\newenvironment{WL}{\color{blue}{\bf WL:} \footnotesize}{}

\begin{document}

	\title[Conjugate Fillings and Legendrian Weaves]{Conjugate Fillings and Legendrian Weaves\\\vspace{0.5cm} {\bf\footnotesize{-- A comparative study on Lagrangian fillings --}}}
\date{}
\author{Roger Casals}
\address{Univ.~ of California Davis, Dept.~of Mathematics, Shields Avenue, Davis, CA 95616, USA}
\email{casals@math.ucdavis.edu}
\author{Wenyuan Li}
\address{Northwestern University, Dept.~of Mathematics, Sheridan Road, Evanston, IL 60202, USA}
\email{wenyuanli2023@u.northwestern.edu}
\maketitle

\begin{abstract}
First, we show that conjugate Lagrangian fillings, associated to plabic graphs, and Lagrangian fillings obtained as Reeb pinching sequences are both Hamiltonian isotopic to Lagrangian projections of Legendrian weaves. In general, we establish a series of new Reidemeister moves for hybrid Lagrangian surfaces. These allow for explicit combinatorial isotopies between the different types of Lagrangian fillings and we use them to show that Legendrian weaves indeed generalize these previously known combinatorial methods to construct Lagrangian fillings. This generalization is strict, as weaves are typically able to produce infinitely many distinct Hamiltonian isotopy classes of Lagrangian fillings, whereas conjugate surfaces and Reeb pinching sequences produce finitely many fillings.\\

Second, we compare the sheaf quantizations associated to each such types of Lagrangian fillings and show that the cluster structures in the corresponding moduli of pseudo-perfect objects coincide. In particular, this shows that the cluster variables in Bott-Samelson cells, given as generalized minors, are geometric microlocal holonomies associated to sheaf quantizations. Similar results are presented for the Fock-Goncharov cluster variables in the moduli spaces of framed local systems. In the course of the article and its appendices, we also establish several technical results needed for a rigorous comparison between the different Lagrangian fillings and their microlocal sheaf invariants.


\end{abstract}
\setcounter{tocdepth}{1}
\tableofcontents

\section{Introduction}\label{sec:intro}

The object of this article will be to show that the Lagrangian fillings obtained via the combinatorics of plabic graphs are Hamiltonian isotopic to the Lagrangian fillings obtained from Legendrian weaves. The main geometric contribution of the article is the comparison of such Hamiltonian isotopy classes through the study of hybrid Lagrangian surfaces, including new results and applications from a series of Reidemeister moves that we establish for such Lagrangian surfaces. The main algebraic contribution is the comparison of Legendrian invariants from the microlocal theory of sheaves where we show that, under the above Hamiltonian isotopies, the independently defined cluster coordinates, through conjugate surfaces and through weaves, match. These cluster coordinates are associated to the ring of functions of the moduli derived stack of pseudo-perfect objects in the dg-subcategory of compact objects within the dg-category of constructibles sheaves with singular support along the corresponding Legendrians.\\

\scriptsize
\begin{table}[h]
\centering
\begin{tikzpicture}
\node (a) at (0,0){

\begin{tabular}{ |c| } 
\hline
{\bf Plabic combinatorics}\\
\hline
Alternating strand diagrams \cite{Pos06AltGraph},\\
Goncharov-Kenyon conjugate surfaces \cite{GonKen13,STWZ19}, \\
$n$-triangulations \cite{FockGon06a} and ideal webs \cite{Gon17Web}, \\ 
plabic fences in $\R$-morsifications \cite{Casals20Sk,FPST17MorseMut},\\ generalized minors and Pl\"ucker coordinates \cite{FZ02ClusterI,BFZ05}.\\
\hline
\end{tabular}

};
\node[xshift=5cm] (b) at (a.east) 
{
\begin{tabular}{ |c| } 
\hline
{\bf Legendrian weaves}\\
\hline
Legendrian (-1)-closures of positive braids \cite{Ng01Satellite,CasalsNg21},\\
Exact Lagrangian spectral curves \cite{CasalsZas20,GMN13,TWZ19Kasteleyn},\\
Legendrian weaves for triangulations \cite{CasalsZas20,CGGS20AlgWeave,CGGLSS22}, \\

Lagrangian A'Campo-Gusein-Zade skeleta \cite{Casals20Sk},\\
weaves for GP-graphs and cluster coordinates \cite{CasalsWeng22,GaoShenWeng20}.\\
\hline
\end{tabular}
};
\draw[->](a)-- node [text width=2.5cm,midway,above,align=center] {This} node [text width=2.5cm,midway,below,align=center] {paper} (b);
\end{tikzpicture}
\caption{Schematics of contributions in the present manuscript, connecting the study of plabic combinatorics (left) with the symplectic topology of Legendrian weaves (right). The former was initiated by A.~Postnikov \cite{Pos06AltGraph} and is rooted in the study of cluster algebras \cite{FZ02ClusterI,BFZ05,FPST17MorseMut} and their topological incarnations by A.~Goncharov and V.~Fock \cite{FockGon06a,FockGon06b,GonKen13}. The latter has been a central ingredient in many of the proofs of recent results in the study of Lagrangian fillings, e.g.~ \cite{AnBaeLee21ADE,AnBaeLee21DEtilde,CasalsGao20,CGGS20AlgWeave,CGGLSS22,CasalsWeng22,Hughes21A,Hughes21D}.}\label{fig:TableIntro}
\end{table}
\normalsize

\noindent The manuscript also shows that Lagrangian fillings obtained via Reeb pinching sequences are compactly supported Hamiltonian isotopic to Lagrangian projections of Legendrian weaves. In general, a recurring theme of our results is that Legendrian weaves strictly generalize all standard constructions of Lagrangian fillings. Explicit diagrammatic methods to transform such fillings into Legendrian weaves are also provided. Schematics of a few implications from this paper are depicted in Table \ref{fig:TableIntro}. Finally, in addition to the above results, the article and its appendices establish rigorous comparisons between the different objects and techniques employed in different sources in the literature.

\subsection{Scientific context} The combinatorics of plabic graphs \cite{Pos06AltGraph} have found several applications to the study of cluster algebras \cite{FWZ,GLSS1,GLSS2,Scott1,Scott2,serhiyenko2019cluster} and the birational geometry of moduli of local systems \cite{FockGon06a,GMN13,GMN14,GonKen13,Gon17Web,GonKon21}. Independently, recent advances \cite{CasalsWeng22,CGGLSS22,GaoShenWeng20} have been able to establish new connections between contact and symplectic topology and the study of cluster algebras. For instance, Legendrian weaves \cite{CasalsZas20,CGGS20AlgWeave} have been used to construct cluster algebras on braid varieties, in particular resolving Leclerc's conjecture \cite{Leclerc}, and, conversely, cluster algebras have been used to provide the first instances of Legendrian links with infinitely many fillings \cite{CasalsGao20,GaoShenWeng20}, up to compactly supported Hamiltonian isotopy.\\

\noindent In Type A, the common denominator of these two areas of research is anchored in the study of Lagrangian fillings of Legendrian knots, in the language of the latter, and conjugate surfaces bounding alternating strand diagrams, in the language of the former. There are two layers of study: the geometry of the surfaces and their boundaries, as smooth or Lagrangian surfaces, and the algebraic invariants that can be associated to them. The first layer is entirely geometric, with a focus on the isotopy classes of the surfaces and links at hand, may they be smooth or Hamiltonian isotopies. The second layer leads to the construction of cluster algebras, cluster seeds and categorifications thereof. This manuscript is structured in the same manner, with Sections \ref{sec:prelim}, \ref{sec:localmove} and \ref{sec:main_proofs} focusing on the geometry, and Sections \ref{sec:quantize} and \ref{sec:cluster} studying the algebraic invariants.\\

From the perspective of low-dimensional symplectic topology, our results serve to both add and connect recent developments in the study of Lagrangian fillings, including the obstruction methods in \cites{EHK16,STWZ19,Pan17Toruslink,CasalsGao20,CasalsZas20,GaoShenWeng20Filling,CasalsNg21} and the constructions in  \cites{EHK16,STWZ19,TreuZas16,CasalsGao20,CasalsZas20,GaoShenWeng20Filling,CasalsNg21,AnBaeLee21ADE,AnBaeLee21DEtilde,Hughes21D}. In brief, there are currently three combinatorial methods to construct Lagrangian fillings for Legendrian $(-1)$-closures of positive braids:
\begin{enumerate}
  \item conjugate Lagrangian fillings \cite{STWZ19}.
  \item pinching sequences of Reeb chords \cite{EHK16},
    \item free Legendrian weaves \cite{CasalsZas20}.
    \end{enumerate}
In a nutshell, this article will establish that Legendrian weaves, Method (3), generalizes the prior two methods whenever they can be compared, i.e.~conjugate Lagrangian fillings and Lagrangian fillings obtained via pinching sequences are Lagrangian projections of Legendrian weaves. In addition, both Methods (1) and (2) can only yield finitely many Lagrangian fillings, whereas Method (3) is known to produce infinitely many Lagrangian fillings \cite{CasalsGao20,CasalsZas20} in many cases.\\

From the perspective of cluster algebras, our results show that the cluster variables constructed from plabic graphs, in the form of Pl\"ucker coordinates and their generalizations, actually coincide with the microlocal holonomies studied in \cite{CasalsWeng22}. Note that the former, in the shape of (factors of) generalized minors, are key in the constructions of \cite{GLSS1,GLSS2}, whereas weaves and microlocal holonomies (both monodromies and merodromies) are a core ingredient in \cite{CGGLSS22,CasalsWeng22}. These algebraic comparisons are guided by the Hamiltonian isotopies that we construct between conjugate Lagrangian surfaces and Lagrangian projections of Legendrian weaves.

\subsection{Main results} Let $\bG\sse\Sigma$ be a plabic graph in a smooth surface $\Sigma$. In \cite[Section 1.1]{GonKen13}, an embedded smooth surface $C(\bG)\sse T^*\Sigma$ is constructed; it is called the conjugate surface. The alternating strand diagram of $\bG$ \cite[Section 14]{Pos06AltGraph} is a cooriented immersed curve in $\Sigma$ and thus naturally lifts to a Legendrian link $\La(\bG)\sse(T^{*,\infty}\Sigma,\ker \la_{st})$ in the ideal contact boundary of the standard cotangent bundle $(T^*\Sigma,\la_{st})$. In \cite[Section 4.2]{STWZ19} it is shown that there exists an embedded exact Lagrangian $L(\bG)\sse(T^*\Sigma,\la_{st})$ which is smoothly isotopic to $C(\bG)$. We show in Proposition \ref{prop:conj-unique} that the Hamiltonian isotopy class of $L(\bG)$ is unique, in that it is independent of the choices used in its construction, and it is thus well-defined to speak of the conjugate Lagrangian filling of $\La(\bG)$ given by $L(\bG)$, up to Hamiltonian isotopy.\\

Let us consider the set $\mathcal{C}(\Sigma)$ of plabic graphs on $\Sigma$ associated to either of the following combinatorial objects: an $n$-triangulation, a grid plabic fence and the reduced plabic graphs for $\mbox{Gr}(2,m)$. These are three of the most common general constructions of plabic graphs. For the first type, $\Sigma$ is an arbitrary (marked) smooth surface, whereas $\Sigma=\D^2$ is the 2-disk, possibly with marked points on the boundary, for the second and third types. Now, the plabic graphs for each of these objects were respectively introduced in \cite{FockGon06a} (see also ideal webs \cite{Gon17Web}) in the study of higher Teichm\"uller theory and moduli spaces of local systems, in \cite{FPST17MorseMut} in the study of real Morsifications of isolated plane curve singularities (see also \cite{CasalsWeng22}), and in \cite{FZ02ClusterI,Scott2} in the study of cluster algebras of finite type $A_{m-3}$, which corresponds to the coordinate ring of functions on $\mbox{Gr}(2,m)$, $m\geq3$. Note that the latter type can be understood as a special case of an ideal triangulation of a disk with boundary marked points; it is emphasized as a separate case because all cluster seeds are actually described by plabic graphs. The Legendrian weaves in $(J^1\Sigma,\ker(dz-\la_{st}))$ associated to each of these combinatorial objects were respectively introduced in \cite[Section 3.1]{CasalsZas20} for $n$-triangulations, in \cite[Section 3]{CasalsWeng22} for (grid) plabic graphs, and the weave associated to a reduced plabic graph for $\mbox{Gr}(2,m)$ is defined to be the 2-weave dual to the triangulation, see \cite{CasalsZas20,TreuZas16}. We refer to them as standard weaves and denote them by $\mathfrak{w}(\bG)$; we denote by $\mathfrak{w}(\bG)$ both the planar weave itself and the Legendrian surface associate to it, as this distinction is always clear by context.\\

{\bf \textcolor{blue}{First}}, the main symplectic geometric result shows that conjugate Lagrangian surfaces are Hamiltonian isotopic to Legendrian weaves:

\begin{thm}\label{thm:main1}
Let $\Sigma$ be a smooth surface, $\bG\in\mathcal{C}(\Sigma)$ and $L(\bG)\sse(T^{*}\Sigma,\la_{st})$ its conjugate Lagrangian surface. Then there exists an embedded Lagrangian $w(\bG)\sse(T^{*}\Sigma,\la_{st})$ Hamiltonian isotopic to $w(\bG)$ which is the Lagrangian projection of the standard weave $\mathfrak{w}(\bG)$.
\end{thm}

\noindent The Hamiltonian isotopy in Theorem \ref{thm:main1} is {\it not} a compactly supported isotopy. In fact, a non-trivial Legendrian isotopy needs to be applied to even compare the Legendrian links of $\dd \mathfrak{w}(\bG)$ and the alternating strand diagram of $\bG$. Theorem \ref{thm:main1} is proven by first developing a series of new Reidemeister moves for hybrid Lagrangian surfaces, which allow us to interpolate between conjugate surfaces and weaves. These moves are shown in Table \ref{fig:TableMovesIntro} and proven in Theorem \ref{thm:hybridReidemeister}. These moves are of independent interest as well. In addition, they are likely to be also necessary when comparing the two recent resolutions \cite{CGGLSS22} and \cite{GLSS1,GLSS2} of Leclerc's conjecture on cluster algebras for Richardson varieties, as the former uses weaves and the latter use conjugate surfaces.

\begin{figure}[h!]
  \centering
  \includegraphics[width=1.0\textwidth]{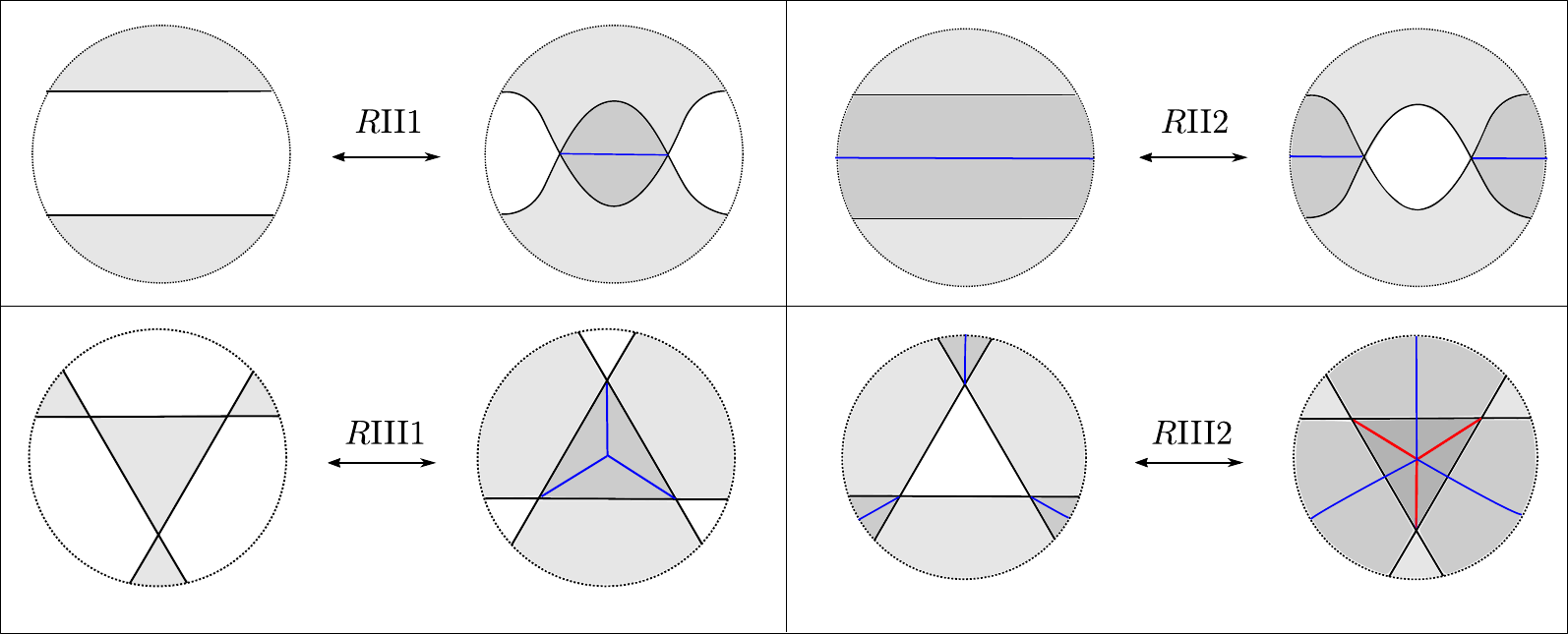}\\
  \caption{Reidemeister-type moves used in order to transition from a conjugate surface towards a weave. These are proven in Section \ref{sec:localmove}. In Section \ref{sec:main_proofs} these moves are used to prove Theorem \ref{thm:main1}.}\label{fig:TableMovesIntro}
\end{figure}

In the comparison in Theorem \ref{thm:main1} several subtleties appear. These include the behaviour at infinity of conjugate Lagrangian surfaces versus that of weaves, and the uniqueness of the Hamiltonian isotopy class of the Lagrangian conjugate surface $L(\bG)$. The necessary technical results to address these differences are established in Section \ref{sec:prelim}. In addition, both Section \ref{sec:prelim} and Section \ref{sec:main_proofs} discuss Lagrangian fillings obtained via Reeb pinching sequences. In particular, it is established in Section \ref{ssec:proof_pinching_to_weave} that such Lagrangian fillings are also Lagrangian projections of Legendrian weaves.\\

{\bf \textcolor{blue}{Second}}, Section \ref{sec:quantize} studies the sheaf quantizations of these Lagrangian fillings. In particular, we review the alternating sheaf quantization of \cite{STWZ19} and set up the sheaf quantization for Legendrian weaves following \cite{CasalsZas20}. The technical details needed for the comparison of these sheaf quantizations are provided (see also Appendices \ref{appen:cat-sheaf} and \ref{appen:sheaf}). Sheaf quantizations at hand, we can then compare the a priori different cluster algebra structures on the moduli derived stack of pseudo-perfect objects of the dg-category of wrapped sheaves. Namely, the combinatorial cluster structures arising from the theory of plabic graphs (with Pl\"ucker coordinates and generalized minors, e.g.~see \cite{GaoShenWeng20}) and that recently constructed using microlocal sheaf theory \cite{CasalsWeng22}. Note that the latter enjoys Hamiltonian functoriality, whereas the former is only known to be invariant under plabic graph moves. Section \ref{sec:cluster} shows that the cluster structures coincide, as follows.\\

Let $\beta\in\mbox{Br}^+$ be a positive braid word, $\La_{\beta}\sse T^{*,\infty}\R^2$ its associated Legendrian cylindrical closure and $\Lambda_{\beta}^\prec\sse(\R^3,\xi_{st})$ its rainbow closure. Denote by $\bG_\beta \in\mathcal{C}(\D^2)$ the plabic fence associated to $\beta$, see \cite[Section 2.5]{CasalsWeng22} or \cite{FPST17MorseMut,STWZ19}. Section \ref{sec:cluster} compares different moduli spaces: the moduli spaces $\cM_1^\textit{fr}(\bD^2, \Lambda_{\beta\Delta^2})_0$, resp.~$\cM_1^\textit{fr}(\bD^2, \Lambda_{\beta}^\prec)_0$, of microlocal rank-one framed sheaves on $\La_{\beta\Delta^2}$, resp.~ $\Lambda_{\beta}^\prec$, as introduced in Appendix \ref{appen:sheaf-moduli}, and the flag configuration space $\mathrm{Conf}_{\beta}^e(\cC)$, as introduced in \cite[Definition 4.3]{GaoShenWeng20}, where $e=\mbox{id}$ is the empty braid. Following the geometry in Theorem \ref{thm:main1}, and once is established in Section \ref{sec:cluster} that these moduli spaces are all isomorphic, it is natural to study the two known cluster algebras:\\

\begin{itemize}
    \item[(i)] The {\it generalized minor cluster variables} on flag configuration space $\C[\mathrm{Conf}_{\beta}^e(\cC)]$ defined in \cite[Section 4]{GaoShenWeng20}. This is a cluster algebra whose cluster variables in the initial seed are given by generalized minors dictated by the plabic fence $\bG_\beta$. In particular, they generalize the initial seeds associated to double Bruhat cells and reduced plabic graphs for the open positroid cells.\\
    
    \item[(ii)] The {\it microlocal cluster variables} on $\C[\cM_1^\textit{fr}(\bD^2, \Lambda_{\beta}^\prec)_0]$ defined in \cite[Section 4]{CasalsWeng22}. This is a cluster structure whose cluster variables in the initial seed are given by microlocal merodromies of a sheaf quantization of the Legendrian weave $\mathfrak{w}(\bG)$.\footnote{The article \cite{CasalsWeng22} constructs cluster structures in the more general setting of grid plabic graphs, which include the cases of plabic fences. Nevertheless, we must restrict to plabic fences to compare to \cite{GaoShenWeng20} as the latter is only defined in that setting.} Since the microlocal cluster $\mathcal{A}$-variables are equally defined for $\cM_1^\textit{fr}(\bD^2, \Lambda_{\beta\Delta^2})_0$ and $\cM_1^\textit{fr}(\bD^2, \Lambda_{\beta}^\prec)_0$, and these moduli are isomorphic, we focus on the latter.\\
\end{itemize}

\noindent Note that both structures above give cluster algebras \cite{FZ02ClusterI,FWZ}, not just upper cluster algebras \cite{BFZ05}, or merely cluster $\mathcal{X}$-structures \cite{FockGon06a} (or the partial structures defined in \cite{STWZ19}). Namely, the comparison of cluster $\mathcal{A}$-variables for cluster algebras is the strongest setting possible. In the second part of the article, we show that these cluster algebras coincide:

\begin{thm}\label{thm:main_algebra1}
Let $\beta$ be a positive braid and $\bG(\beta)$ its associated plabic fence. Then the coordinate ring $\C[\mathrm{Conf}_{\beta}^e(\cC)]$, endowed with the minor cluster algebra structure, and the coordinate ring $\C[\cM_1^\textit{fr}(\bD^2, \Lambda_{\beta}^\prec)_0]$ endowed with the microlocal cluster algebra structure, are isomorphic cluster algebras.\\

Furthermore, there exists an isomorphism that sends the initial seed in $\C[\mathrm{Conf}_{\beta}^e(\cC)]$, given by the toric chart associated to the conjugate Lagrangian surface $L(\bG_\beta)$, to the initial seed in $\C[\cM_1^\textit{fr}(\bD^2, \Lambda_{\beta}^\prec)_0]$ given by the toric chart associated to $\mathfrak{w}(\bG_\beta)$.
\end{thm}

\noindent It follows from Theorem \ref{thm:main_algebra1} and the Hamiltonian functoriality in \cite{CasalsWeng22} that the minor cluster variables can, a posteriori, be defined intrinsically in terms of symplectic topological techniques and are invariant under Legendrian Reidemeister moves applied to the relative Lagrangian skeleton. Note also that the generalized minor cluster algebra is used crucially in the cluster algebra construction for braid varieties \cite{CGGLSS22} and the resolution of Leclerc's conjecture. Theorem \ref{thm:main1} and Theorem \ref{thm:main_algebra1} together imply that the cluster structure in \cite{CGGLSS22}, which uses a hybrid of minor coordinates and weaves, coincides with either \cite{CasalsWeng22} or \cite{GaoShenWeng20} in the case of double Bott-Samelson cells, even at the level of surfaces and cycles. Finally, we note that Theorem \ref{thm:main_algebra1} implies the analogous result for cluster $\mathcal{X}$-variables, which can defined in terms of $\mathcal{A}$-variables and symplectically correspond to microlocal monodromies, along absolute cycles, instead of microlocal merodromies, along relative cycles; see Section \ref{sec:cluster}.\\

Theorem \ref{thm:main_algebra1} covers the plabic graphs $\bG\in\mathcal{C}(\D^2)$ given by plabic fences and $\mbox{Gr}(2,m)$. In the remaining case of ideal $n$-triangulations on smooth surfaces $\Sigma$, the comparison must restrict to the cluster Poisson structure, as only cluster $\mathcal{X}$-variables are defined both in the ideal triangulation setting of conjugate Lagrangian surfaces and Legendrian weaves.     Consider the moduli space $\cX_{\GL_n}(\Sigma)$ of $\GL_n$-framed local systems on a punctured surface $\Sigma$, as defined in \cite{FockGon06a,Gon17Web}. Ibid shows that it is a cluster $\cX$-variety. The techniques developed in this article allow us to compare the Fock-Goncharov coordinates \cite{FockGon06a}, the Gaiotto-Moore-Neitzke non-abelianization coordinates \cite{GMN14} and the coordinates defined by microlocal monodromies from sheaf quantizations \cite{CasalsZas20,STWZ19,TreuZas16}. In particular, we provide a symplectic geometric explanation of the Fock-Goncharov coordinates by identifying them with microlocal monodromies. In precise terms, given an $n$-triangulation on $\Sigma$ and $\bG_n^*$ its associated bipartite graph, see Sections \ref{sec:prelim} and \ref{sec:cluster}, we have the following maps:\\

\begin{itemize}
    \item[(i)] The embedding $\iota_{\mu mon}:H^1(L(\bG_n^*); \Bbbk^\times) \hookrightarrow \cX_{\GL_n}(\Sigma)$ via microlocal monodromy on conjugate Lagrangians, as constructed in \cite{STWZ19};\\
    
    \item[(ii)] The embedding $\iota_{FG}:H^1(L(\bG_n^*); \Bbbk^\times) \hookrightarrow \cX_{\GL_n}(\Sigma)$ obtained via the Fock-Goncharov bipartite graph construction in \cite{FockGon06a};\\
    
    \item[(iii)] The embedding $\iota_{\mu\mathfrak{w}}:H^1(\mathfrak{w}(\bG_n^*); \Bbbk^\times) \hookrightarrow \cX_{\GL_n}(\Sigma)$ via microlocal monodromy on Legendrian weaves, as constructed in \cite{CasalsZas20};\\
    
    \item[(iv)] The embedding $\iota_{SN}:H^1(\mathfrak{w}(\bG_n^*); \Bbbk^\times) \hookrightarrow \cX_{\GL_n}(\Sigma)$ obtained via the Gaiotto-Moore-Neitzke non-abelianization maps on spectral networks, constructed in \cite{GMN13}.\\
\end{itemize}

Section \ref{sec:cluster} compares these maps and, among others, we prove the following result:

\begin{thm}\label{cor:FG=GMN}
    Let $\Sigma$ be a closed surface, $\bG_n^*$ the $\GL_n$-bipartite graph, equiv.~ the $A_n^*$-ideal web, associated to an ideal $n$-triangulation, and $\Bbbk$ a field. Then the following holds:\\
    
    \begin{enumerate}
        \item[(1)] The algebraic morphisms coincide $\iota_{\mu mon}=\iota_{FG}$ and $\iota_{\mu\mathfrak{w}}=\iota_{SN}$.\\
        
        \item[(2)] There exists a Hamiltonian isotopy that brings the conjugate Lagrangian surface $L(\bG)$ to the Lagrangian projection of the weave $\mathfrak{w}(\bG_n^*)$ and whose induced isomorphism $\phi:H^1(L(\bG_n^*); \Bbbk^\times)\lr H^1(\mathfrak{w}(\bG_n^*); \Bbbk^\times)$ is such that:\\
        \begin{itemize}
            \item[(i)] $\phi$ intertwines microlocal and Fock-Goncharov charts, i.e.~ $\iota_{FG}=\iota_{\mu\mathfrak{w}}\circ \phi$;\\
            \item[(ii)] $\phi$ identifies the cluster $\mathcal{X}$-variables in these two toric charts of $\cX_{\GL_n}(\Sigma)$.
        \end{itemize}
    \end{enumerate}
\end{thm}

\noindent Similar to Theorem \ref{cor:FG=GMN}, we also identify the corresponding cluster $\cX$-coordinates on the moduli of framed $\PGL_n$-local systems $\cX_{\PGL_n}(\Sigma)$. This follows by forgetting the frozen variables, as explained in Section \ref{sec:framed-locsys}.

\subsection*{Acknowledgements} We are grateful to Honghao Gao, Xin Jin, Linhui Shen, David Treumann, Daping Weng, and Eric Zaslow for useful discussions. R.~Casals is supported by the NSF CAREER DMS-1942363 and a Sloan Research Fellowship of the Alfred P. Sloan Foundation.

\section{Constructions of Lagrangian fillings}\label{sec:prelim}

This section studies three ways of constructing Lagrangian fillings for Legendrian links:

\begin{itemize}
    \item[-] Conjugate Lagrangian fillings of alternating Legendrians, in Subsection \ref{sec:conjugate},
    \item[-] Free Legendrian weaves, in Subsection \ref{sec:weave},
    \item[-] Pinching sequences of Reeb chords, in Section \ref{sec:pinching}.
\end{itemize}

\noindent These three constructions are respectively introduced in \cite[Section 4.2]{STWZ19}, \cite[Section 7.1]{CasalsZas20} and \cite[Section 6.5]{EHK16}. See also \cite[Section 3.3]{ArnoldSing} for the pinching saddle cobordism. In this section we present descriptions that allow us to compare these constructions. For instance, according to the original constructions, the Legendrian boundary of a conjugate Lagrangian and that of a free Legendrian weave live in different contact manifolds. This issue is resolved in Subsection \ref{sec:weave-at-infty}. This section also contains some new results and necessary details not formerly available in the literature, such as the uniqueness of the Hamiltonian isotopy class of the conjugate Lagrangian surface in \cite[Proposition 4.9]{STWZ19}, thus showing independence of the choices and local models that exist in the construction.\\

Let $\Sigma$ be a smooth surface. Consider the 4-dimensional exact symplectic manifold $(T^*\Sigma,d\lambda_\text{st})$, equipped with the exact symplectic form whose primitive is the Liouville 1-form $\lambda_\text{st}$, and its 3-dimensional ideal contact boundary at infinity $(T^{*,\infty}\Sigma,\xi_\text{st})$ with the contact structure $\xi_\text{st} := \ker(\lambda_\text{st}|_{T^{*,\infty}M})$. Following \cite{Giroux17Ideal} we denote by $\ol{T^*\Sigma} := T^*\Sigma \cup T^{*,\infty}\Sigma$ the associated ideal Liouville domain with boundary. In a local coordinate chart $(x_1, x_2, \xi_1, \xi_2)$ of $T^*\Sigma$, $(x_1,x_2)\in\Sigma$ and $(\xi_1,\xi_2)\in T^*_{(x_1,x_2)}\Sigma$, the Liouville form reads
    $$\lambda_\text{st} = \xi_1dx_1 + \xi_2dx_2.$$
We also consider the front projection $\pi: T^{*,\infty}\Sigma \rightarrow \Sigma$, whose fibers are Legendrian circles in $(T^{*,\infty}\Sigma,\xi_\text{st})$. For a Legendrian submanifold $\Lambda \subset T^{*,\infty}\Sigma$, its image $\pi(\Lambda)$ is generically a curve with only finitely many transverse double points and simple cusp singularities. By virtue of $\Lambda$ being Legendrian, it suffices to specify the curve $\pi(\Lambda)$ and a conormal direction of $\pi(\Lambda)$: a co-oriented front $\pi(\Lambda)$ recovers $\Lambda$. In Figure \ref{fig:Reidemeister}, we use the hair on the strands of $\pi(\Lambda)$ to represent the specified conormal direction. In the following sections, the co-orientation is omitted in the diagram if there is no ambiguity. A Legendrian isotopy induces a front homotopy; there are two front homotopies that we often use, referred to as Reidemeister moves II and III. They are illustrated in Figure \ref{fig:Reidemeister}.

\begin{figure}[h!]
  \centering
  \includegraphics[width=1.0\textwidth]{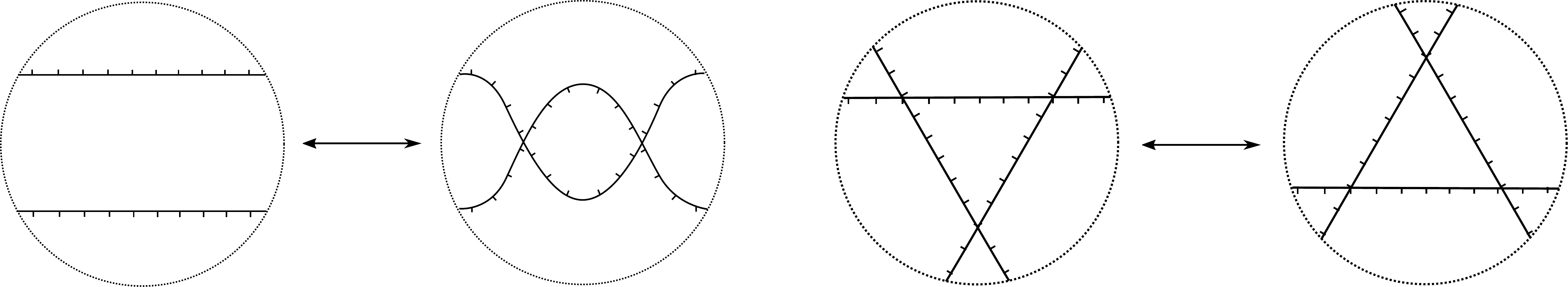}\\
  \caption{The two Legendrian tangles on the left, represented by the co-oriented fronts, are Legendrian isotopic via a Reidemeister~II move, while the Legendrian tangles on the right are Legendrian isotopic via a Reidemeister~III move. The hair on the strands indicates the conormal direction that determines the Legendrian tangles.}\label{fig:Reidemeister}
\end{figure}

\subsection{Conjugate Lagrangians}\label{sec:conjugate}
Let us review the construction of alternating Legendrian links and their conjugate Lagrangian fillings, following \cite{STWZ19}; similar constructions in the cotangent bundles of tori has also appeared in \cite{Hicks20,Matessi21}. We also introduce some new terminologies and technical lemmas not present in \cite{STWZ19} that are necessary to manipulate conjugate Lagrangian fillings and compare them to Legendrian weaves.

\begin{definition}[\cite{STWZ19}*{Definition 4.1}]\label{def:alternating}
    Let $\Lambda \subset T^{*,\infty}\Sigma$ be a Legendrian link whose front projection $\pi|_\Lambda: \Lambda \rightarrow \Sigma$ is an immersion with only transverse double points. By definition, an alternating coloring is a coloring on the connected components of $\Sigma \backslash \pi(\Lambda)$ by black, white or null, such that
\begin{enumerate}
  \item the conormal direction on the boundary of a black (resp.~white) component points inwards (resp.~outwards);
  \item the conormal direction on the boundary of a null component alternates between inward and outward at a crossing.
  \item no black and white components, nor two null components, share an edge.
\end{enumerate}
By definition, a Legendrian $\Lambda$ equipped with an alternating coloring of $\Sigma \backslash \pi(\Lambda)$ is said to be an alternating Legendrian.
\end{definition}
\noindent A null component in an alternating coloring is also called a face in the literature \cite{Gon17Web,CasalsWeng22}. From an alternating coloring we can construct a bi-colored graph, by connecting the black and white vertices (resp.~in the black and white regions) when they share a common crossing.  By Definition \ref{def:alternating}, no edge will connect two vertices of the same color and thus we indeed obtain a bipartite graph. Conversely, given a bipartite graph, we can consider its alternating strand diagram \cite{Pos06AltGraph,GonKen13}, which has a unique alternating Legendrian lift \cite{STWZ19}*{Proposition 4.6}. In that sense, alternating Legendrians are a contact geometric incarnation of the combinatorics of bipartite graphs and their alternating strand diagrams.

\begin{definition}[\cite{STWZ19}*{Definition 4.8}]\label{def:graphical-conj}
    Let $\Lambda \subset T^{*,\infty}\Sigma$ be an alternating Legendrian link. By definition, an exact Lagrangian $L\subset T^*\Sigma$ is said to be a conjugate Lagrangian filling of $\Lambda$ if it satisfies the following properties:
\begin{enumerate}
  \item $\ol{\pi(L)}$ is the closure of all white and black regions, so that $L$ is diffeomorphic to the real blowup of $\ol{\pi(L)}$ at all the crossings.
  \item The intersection of the closure $\ol{L} \subset \ol{T^*\Sigma}$ and the contact boundary is $\ol{L} \cap T^{*,\infty}\Sigma = \Lambda$.
  \item For any neighbourhood $U$ of $\Lambda \subset \ol{T^*\Sigma}$, there is a Hamiltonian isotopy $\varphi_t$ supported on $U$, so that for $r$ sufficiently large, $\varphi_1(L) \cap \{(x, \xi) | |\xi| > r\} = \Lambda \times (r,+\infty)$.
\end{enumerate}
\end{definition}

\noindent It is shown in \noindent \cite[Proposition 4.9]{STWZ19} that an alternating Legendrian link admits a conjugate Lagrangian filling. In order to show its uniqueness, we need to first review the construction:

\begin{prop}[\cite{STWZ19}]\label{prop:conj-exist}
Let $\Lambda \subset T^{*,\infty}\Sigma$ be an alternating Legendrian. Then there exists a conjugate Lagrangian filling for $\Lambda$.
\end{prop}
\begin{proof}
    For each black, resp.~white, region $B$, resp.~$W$, consider a function $m_{B}: \ol{B} \rightarrow [0,+\infty)$ such that $m_{B}^{-1}(0) = \partial B$, resp.~$m_{W}: \ol{W} \rightarrow [0,+\infty)$ such that $m_{W}^{-1}(0) = \partial W$, each $m_B$ and $m_W$ is smooth, and 0 is a regular value. Define the functions
    $$f_B(x, y) := \ln(m_B(x, y)), \,\,\, f_W(x, y) := -\ln(m_W(x, y)).$$
    Away from each crossing, the graphs $L_{df_{B}},L_{df_{W}}\sse T^*\Sigma$ of the differentials of $f_B,f_W$ give an exact embedded Lagrangian submanifolds of $(T^*\Sigma,d\la_{st})$. Near the crossing, suppose that the front $\pi(\Lambda)$ coincides with the coordinate axes, $B$ is the first quadrant and $W$ is the third quadrant. Consider the exact Lagrangian embedding
    $$i_\times: \mathbb{R}\times (0, 1) \rightarrow T^*\mathbb{R}^2; \,\,\,\, (s, t) \mapsto \left(st, s(1-t), \sqrt{\frac{1-t}{t}}, \sqrt{\frac{t}{1-t}}\right),$$
    whose the primitive is $f_\times(s, t) = 2s\sqrt{t(1 - t)}$. By changing variables, let
    $$f_\times(x, y) = 2\mathrm{sgn}(x)\sqrt{xy}.$$
    Then $i_\times$ can be viewed as a smooth realization of the Lagrangian defined by the graph of the differential $L_{df_\times}$. Now we can glue the functions $f_B, f_W$ and $f_\times$ together, by using a partition of unity on the closure of all black and white regions, so that we obtain a smooth function  $f$ such that the graph of the differential $L_{df}$ defines an exact Lagrangian embedding $L\sse(T^*\Sigma,\lambda_{st})$ satisfying conditions (1) \& (2).\\

    \noindent We still need to check the condition~(3). Because of the defining functions we have chosen, for sufficiently large $r_0\in\R_+$, the intersection $L \cap \{(x, \xi)\in T^*\Sigma: |\xi| > r_0\}$ is contained in a Weinstein tubular neighbourhood $U$ of the cylinder $\Lambda \times (r_0,\infty)$ as a the differential of a function $F: \Lambda \times (r_0, \infty) \rightarrow \bR$. Thus locally in $U$, the Lagrangian $L$ coincides with
    $$L_{dF} \subset U \cong \{(y, \eta) \in T^*(\Lambda \times (r_0,\infty)) \,:\, |\eta| < \epsilon\},$$
    for small enough $\epsilon\in\R_+$. Choose a cut-off function $\beta: \mathbb{R}_+ \rightarrow [0, 1]$ such that $\beta(r) = 1$ when $r \leq r_0$, $\beta(r) = 1$ when $r$ is sufficiently large, and $\beta'(r)$ is sufficiently small. Then $L_{dF}$ is Hamiltonian isotopic to $L_{d(\beta F)}$ and, for some $r \gg r_0$,
    $$L_{d(\beta F)} \cap \{(x, \xi)\in T^*\Sigma: |\xi| > r\} = \Lambda \times (r, +\infty).$$
    Hence the Lagrangian $L$ satisfies condition~(3) and the proof is completed.
\end{proof}

Let us now show that for a given alternating Legendrian link, the conjugate Lagrangian filling of the alternating Legendrian is unique up to Hamiltonian isotopy, which is implicitly assumed but not explained or proved in \cite{STWZ19}. In fact, the following stronger statement holds:
    
\begin{prop}\label{prop:conj-unique}
    Let $\Lambda \subset T^{*,\infty}\Sigma$ be an alternating Legendrian. Then the space of conjugate Lagrangian fillings of $\Lambda$ is weakly contractible.
\end{prop}
\begin{proof}
By conditions (1) \& (2) in Definition \ref{def:graphical-conj}, conjugate Lagrangian fillings $L \subset T^*\Sigma$ of $\La$ are graphs of exact 1-forms away from the crossings of the front $\pi(\Lambda) \subset \Sigma$. Fix a standard conjugate Lagrangian filling $L_\text{std} \subset T^*\Sigma$ as constructed in Proposition \ref{prop:conj-exist}, and let $U \subset \Sigma$ be a disjoint union of small open balls around the crossings in $\pi(\Lambda) \subset \Sigma$. Consider a smooth family of conjugate Lagrangians $L_t \subset T^*\Sigma,$ $t \in \partial \bD^k$. We need to build an extension of this family to a family of conjugate Lagrangians  conjugate Lagrangians $L_t \subset T^*\Sigma,$ $t \in \bD^k$. We first consider such contractibility of the space of the conjugate Lagrangian fillings away from crossings, and then consider contractibility near crossings.
    
\noindent First, we claim that the $\partial \bD^k$-family of conjugate Lagrangians $L_t \cap T^*(\Sigma \backslash U)$ extends to a smooth $\bD^k$-family of conjugate Lagrangians in $T^*(\Sigma \backslash U)$ such that for $0 \in \bD^k$
    $$L_0 \cap T^*(\Sigma \backslash U) = L_\text{std} \cap T^*(\Sigma \backslash U).$$
    Indeed, this is because $L_t \cap T^*(\Sigma \backslash U)$ are graphs of exact 1-forms in the black (resp.~white regions) and the space of smooth functions on the black (resp.~white) regions that converge to infinity near the boundary is contractible. In particular, this implies that the $\partial \bD^k$-family $L_t \cap T^*\Sigma$, $t\in\partial \bD^k$, extends to a family with parameter $t\in\partial \bD^k \times [0, 1]$ such that
    $$L_t \cap T^*(\Sigma \backslash U) = L_\text{std} \cap T^*(\Sigma \backslash U), \;\; \forall\, t \in \partial \bD^k \times \{0\}.$$
    \noindent This addresses the situation in $\Sigma\backslash U$, away from the crossings of $\pi(\La)$.\\
    
    Second, consider a $\partial \bD^k$-family of conjugate Lagrangians $L_t'$, $t\in\partial \bD^k$, and assume that $L_t' \cap T^*(\Sigma \backslash U) = L_\text{std} \cap T^*(\Sigma \backslash U)$. We now need to show that the $\partial \bD^k$-family of conjugate Lagrangians $L_t' \cap T^*U$ extends to a $\bD^k$-family of conjugate Lagrangians in $T^*U$ such that for $0 \in \bD^k$
    $$L_0' \cap T^*U = L_\text{std} \cap T^*U.$$
    For that, note that $T^*U$ is a Liouville domain with piecewise smooth contact boundary 
    $$\partial_\infty(T^*U) = T^*U|_{\partial U} \cup_{T^*(\partial U)} T^{*,\infty}U,$$
    and $L_t' \cap \partial_\infty(T^*U)$ is a piecewise Legendrian knot where
    \[\begin{split}
    L_t' \cap T^*\bD^2|_{\partial \bD^2} =& \,\{(\theta, \sqrt{\tan\theta}, 1/\sqrt{\tan\theta}) | \theta \in (0, \pi/2) \cup (\pi, 3\pi/2)\}, \\
    L_t' \cap T^{*,\infty}\bD^2 &= \{(x, 0, 0, \eta) | \eta > 0\} \cup \{(0, y, \xi, 0) | \xi > 0\}.
    \end{split}\]
    By smoothing the corners of $\partial_\infty(T^*U)$, we get a standard contact sphere $(S^3,\xi_\text{std})$ and the intersection $L_t' \cap S^3$ is the standard Legendrian unknot with maximal Thurston-Bennequin number, since it has a Lagrangian disk filling. By \cite[Theorem 1.1.B]{EliPol96}, the space of Lagrangian disks that fill the Legendrian unknot is contractible, which completes the proof.
\end{proof}


    

\subsubsection{Conjugate Fillings from $n$-triangulations and positive braids}\label{ssec:triangle-positive} By \cite{STWZ19}*{Theorem 5.4}, a large class of Legendrian links have alternating representatives. Among them, two specific classes of conjugate Lagrangians that have appeared in the literature are those associated to $n$-triangulations \cite{FockGon06a,Gon17Web} and those associated to positive braids \cite{Pos06AltGraph,FPST17MorseMut}. Let us discuss them in detail.\\

{\bf $n$-triangulations and $A_n^*$-graphs}. Consider a smooth surface $\Sigma$ equipped with marked points $\{x_1, ..., x_m\}$. An ideal triangulation of the marked surface $(\Sigma,\{x_1, ..., x_m\})$ is a triangulation of $\Sigma$ with all the vertices in $\{x_1, ..., x_m\}$. Fix such an ideal triangulation and, for each triangle, consider a refined $n$-triangulation that divides the triangle into $n^2$ smaller triangles, see \cite[Section 1.15]{FockGon06a}. We label the small triangles of the resulting $n$-triangulation as white or black so that every triangle on the boundary of the original triangle is white, and no triangles of the same color share an edge. Then we label the $n$ edges in each boundary of the triangle as black. The dual bipartite graph $\bG_n^*\subset\Sigma$ of the refined $n$-triangulation is constructed by connecting the black and white vertex when the edge is contained in the triangle. This graph is called the $A_n^*$-graph \cite{Gon17Web}*{Section 3.2.2} and it is depicted in Figure \ref{fig:triangle-conj} left.\\

\noindent The alternating strand diagram of the plabic graph $\bG_n^*$ defines an alternating Legendrian in the unit cotangent bundle of the punctured surface $T^{*,\infty}\Sigma_\text{op} = T^{*,\infty}(\Sigma \backslash \{x_1, ..., x_m\})$. Let $\bigcirc_n(x_i)$ be the Legendrian $n$-satellite of an inward unit conormal of a circle centered at $x_i$, whose front projection is $n$ concentric circles centered at $x_i$. Then the alternating Legendrian link $\Lambda(\bG_n^*)$ associated to $\bG_n^*$ is Legendrian isotopic to $\bigcirc_n = \bigcup_{i=1}^m\bigcirc_n(x_i)$. A piece of this is depicted in Figures \ref{fig:triangle-conj} center and right. The conjugate Lagrangian filling $L(\bG_n^*)$ associated to this plabic graph lives in $T^{*}\Sigma_\text{op}$.\\

\begin{figure}[h!]
  \centering
  \includegraphics[width=0.95\textwidth]{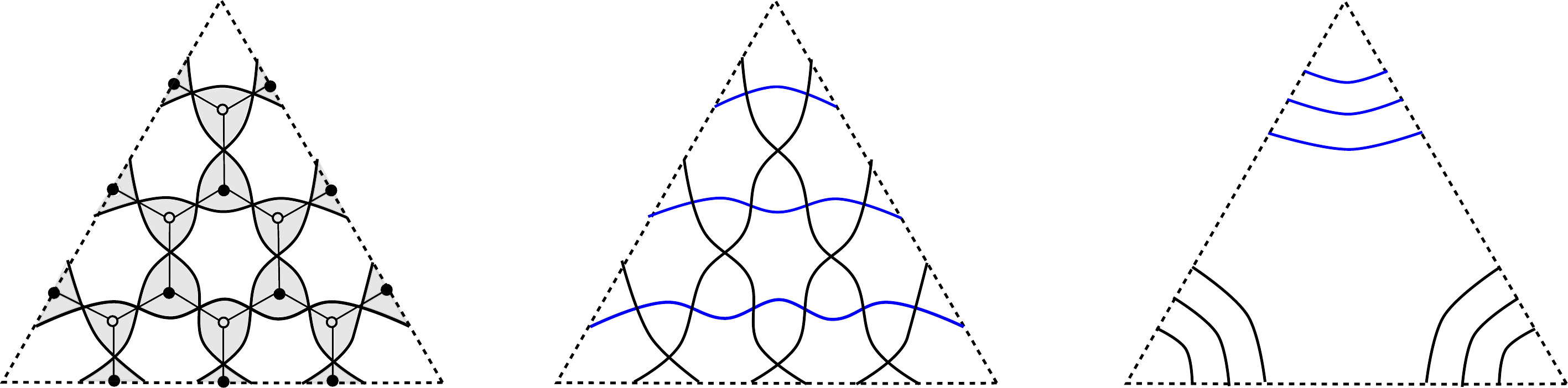}\\
  \caption{The $A_n^*$-bipartite graph $\bG_n^*$ in a single triangle and its corresponding alternating Legendrian and conjugate Lagrangian on the left. The alternating Legendrian $\Lambda(\bG_n^*)$ in the middle is Legendrian isotopic to the link $\bigcirc_n$ on the right, which is the Legendrian $n$-satellite of inward unit conormal of circles centered at the vertices $\{x_1, ..., x_m\}$. Here $n=3$.}\label{fig:triangle-conj}
\end{figure}

There is a variation of the above, in terms of $A_n$-graphs. In this case, we still consider a smooth surface $\Sigma$ equipped with marked points $\{x_1, ..., x_m\}$ and an associated ideal triangulation with the refined $n$-triangulation. Nevertheless, we now label the small triangles as black or null so that every triangle on the boundary of the original large one is black, and no triangles of the same color share an edge, and label the $(n+1)(n+2)/2$ vertices of the triangles as white. The plabic graph $\bG_n$ of the refined $n$-triangulation, also called the $A_n$-graph \cite{Gon17Web}*{Section 3.2.1}, is defined by connecting the black and white vertex when the edge is contained in the triangle. See Figure \ref{fig:triangle-conj-part} left.\\ 
    
\noindent In this case, the alternating strand diagram of $\bG_n$ defines an alternating Legendrian $\Lambda(\bG_n) \subset T^{*,\infty}\Sigma$, Legendrian isotopic to $\bigcirc_{n-1} = \bigcup_{i=1}^m\bigcirc_{n-1}(x_i)$. Pieces of these fronts are depicted in Figure \ref{fig:triangle-conj-part} center and right. Note that, in contrast with the case of the $A_n^*$-bipartite graph $\bG_n^*$, the conjugate Lagrangian filling $L(\bG_n)$ no longer lives in $T^*\Sigma_\text{op}$, but instead only lives in the cotangent bundle of the closed surface $T^*\Sigma$. Indeed, the vertices $\{x_1, ..., x_m\}$ of the ideal triangulations are in null regions for the conjugate surface associated to $\Gamma_n^*$, but the vertices $\{x_1, ..., x_m\}$ belong to black regions for the conjugate surface associated to $\bG_n$.

\begin{figure}[h!]
    \centering
    \includegraphics[width=0.9\textwidth]{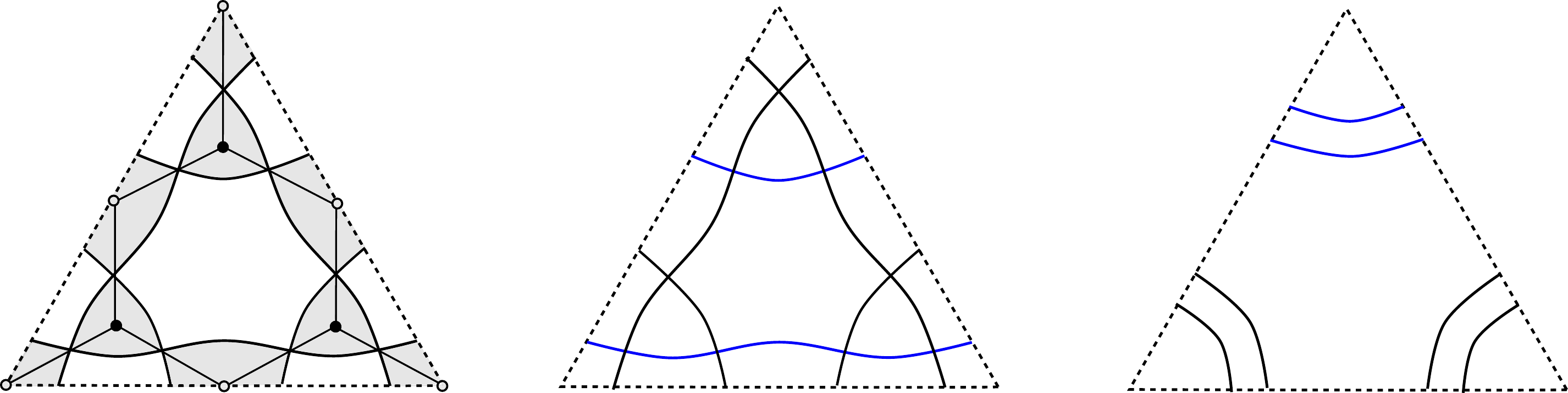}\\
    \caption{The $A_n$-bipartite graph $\bG_n$ in a single triangle and its corresponding alternating Legendrian and conjugate Lagrangian on the left. The alternating Legendrian $\Lambda(\bG_n)$ in the middle is Legendrian isotopic to the link $\bigcirc_{n-1}$ on the right. Here $n=3$.}
    \label{fig:triangle-conj-part}
\end{figure}

{\bf Positive braid closures}. Consider a positive braid of $n$-strands of the form $\beta \Delta^2$, where $\beta \in\mbox{Br}_n^+$ belongs to the positive braid monoid in $n$-strands and 
$\Delta := \mathtt{w}_{0,n} = (s_{1})(s_{2}s_{1})(s_{3}s_{2}s_{1})\dots \dots (s_{n-1}s_{n-2} \dots s_{1})$
is the positive half twist. This positive braid $\beta\Delta^2$ defines an annular front in $S^1\times\R$ and thus a Legendrian link $\La(\beta\Delta^2)\subset(J^1S^1,\xi_{st})$, see \cite[Section 2.2]{CasalsNg21}. By the Legendrian neighborhood theorem, $(J^1S^1,\xi_{st})$ can be identified with a neighborhood of any Legendrian knot in  $T^{*,\infty}\bR^2$ such that the zero-section in $(J^1S^1,\xi_{st})$ is identified with that Legendrian. In particular we can choose either of the following two Legendrian knots:

\begin{itemize}
    \item[(i)] The Legendrian cotangent fiber of $T^{*,\infty}\bR^2$. This Legendrian embedding is Legendrian isotopic to the embedding $S^1 \hookrightarrow T^{*,\infty}\bR^2$ as the outward unit conormal bundle of a disk, see Figure \ref{fig:braidclosure} (left). The Legendrian link $\La(\beta\Delta^2)\subset(J^1S^1,\xi_{st})$ is then satellited to a Legendrian link $\La_{\beta\Delta^2} \subset T^{*,\infty}\bR^2$. A front in $\R^2$ for $\Lambda_{\beta\Delta^2}$ is depicted in Figure \ref{fig:braidclosure} (right).\\
    
    \item[(ii)] The standard Legendrian unknot $S^1 \hookrightarrow (\bR^3,\xi_{st}) \hookrightarrow T^{*,\infty}\bR^2$.\footnote{Note that the outward unit conormal bundle of a disk and the standard unknot are Hamiltonian isotopic after compactifying the contact manifold $T^{*,\infty}\bR^2$ to $S^3$, and thus so are $\Lambda_{\beta\Delta^2}$ and $\Lambda_{\beta\Delta^2}'$.} Figure \ref{fig:NgBraidclosure} (left) depicts a front for the standard Legendrian unknot. The Legendrian link $\La(\beta\Delta^2)\subset(J^1S^1,\xi_{st})$ is satellited to a Legendrian link $\La'_{\beta\Delta^2} \subset T^{*,\infty}\bR^2$. A front for $\La'_{\beta\Delta^2}$ is depicted in Figure \ref{fig:NgBraidclosure} (center).
\end{itemize}

\noindent Note that $\La'_{\beta\Delta^2}$ is null-homologous in $T^{*,\infty}\bR^2$ whereas $\La_{\beta\Delta^2}$ is not. It follows from \cite{Ng01Satellite} that the Legendrian link $\Lambda_{\beta\Delta^2}' \subset T^{*,\infty}\bR^2$ is Legendrian isotopic to the rainbow braid closure $\Lambda_{\beta}^\prec$ of the positive braid $\beta$ in $\bR^3 \subset T^{*,\infty}\bR^2$; see Figure \ref{fig:NgBraidclosure} (right).\\

 In addition to the satellites $\La_{\beta\Delta^2}$ and $\La'_{\beta\Delta^2}$, there is a third construction of a Legendrian link from a positive braid $\beta$. Indeed, a braid word $\beta = s_{i_1}\dots s_{i_k}$ determines an alternating Legendrian representative of $\Lambda_\beta$ as follows. Begin with a bicolored graph $\bG_{\beta}$ in the plane consisting of $k$ horizontal line segments running from $(0, i) \in \mathbb{R}^2$ to $(k + 1, i) \in \mathbb{R}^2$, $1 \leq i \leq n$, with white vertices at both ends. For $1 \leq j \leq k$, insert a vertical segment along the line $y = j$ connecting the line $x = i_j$ to the line $x = 1 + i_j$, with a black vertex at its top and a white vertex at its bottom. We can add extra white and black vertices in each horizontal strand to obtain a bipartite graph $\bG_{\beta}$. These plabic graphs are called plabic fences \cite{FPST17MorseMut} and the associated alternating Legendrian $\Lambda(\bG_{\beta})$ is depicted in Figure \ref{fig:braid-conj} (left). From the resulting alternating Legendrian link $\Lambda(\bG_{\beta})$, one obtains the front projection of $\Lambda_{\beta\Delta^2}$ in $T^{*,\infty}\bR^2$ by sliding all upward co-oriented strands (the blue strands) to the top via Reidemeister moves (Figure \ref{fig:braid-conj} right).\footnote{In Figure \ref{fig:braid-conj}, for the plabic fence, we are labeling the strands from top to bottom ($s_1$ is the top crossing while $s_{n-1}$ is the bottom one), while for the Legendrian cylindrical or rainbow braid closure, we are labeling the strans from outside to inside ($s_1$ is the outermost crossing while $s_{n-1}$ is the innermost one).} The advantage of this construction is that the alternating strand diagram automatically comes equipped with the conjugate Lagrangian filling. Thus, presenting a Legendrian knot $\La_{\beta\Delta^2}$ through a plabic fence endows it with an embedded exact Lagrangian filling. In Section \ref{sec:main_proofs} we discuss how to obtain these same Lagrangian fillings through weaves and, independently, through pinching sequences.

\begin{figure}[h!]
  \centering
  \includegraphics[width=0.5\textwidth]{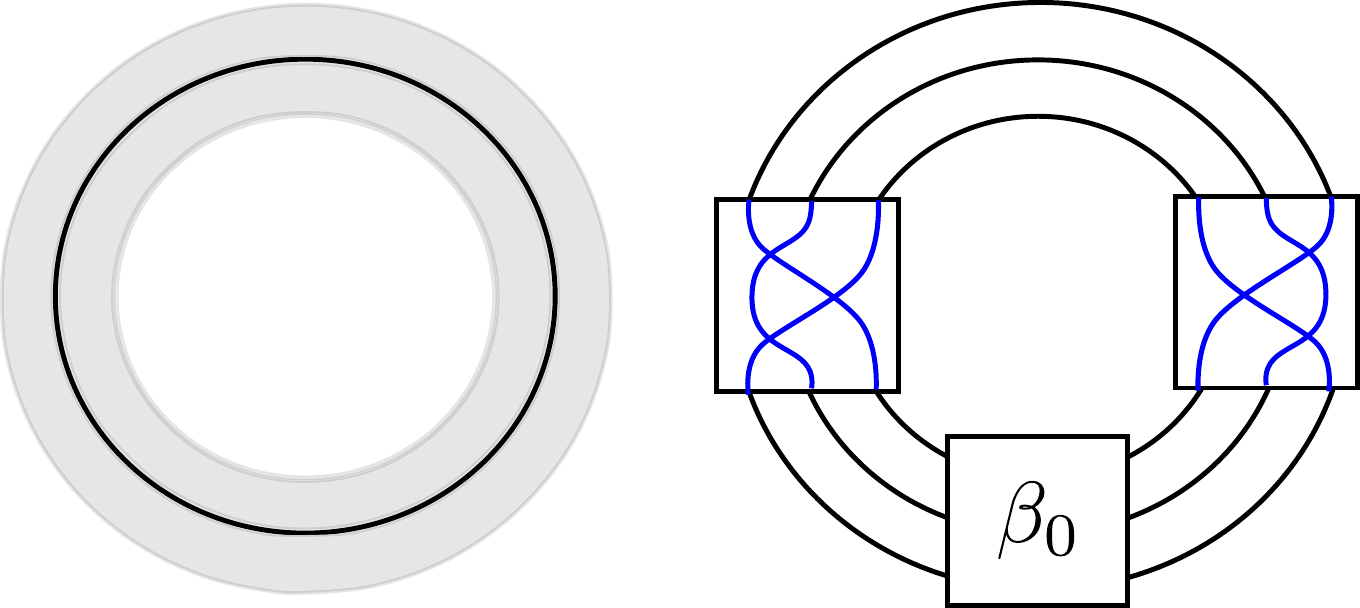}\\
  \caption{The grey region shown on the left is the image of $J^1(S^1)$ under the front projection, where $S^1$ is the outward unit conormal bundle of a disk. The Legendrian link on the right is $\Lambda_{\beta\Delta^2} \subset J^1(S^1) \subset T^{*,\infty}\bR^2$, where each collection of the blue strands is a copy of $\Delta$.}\label{fig:braidclosure}
\end{figure}
\begin{figure}[h!]
  \centering
  \includegraphics[width=0.85\textwidth]{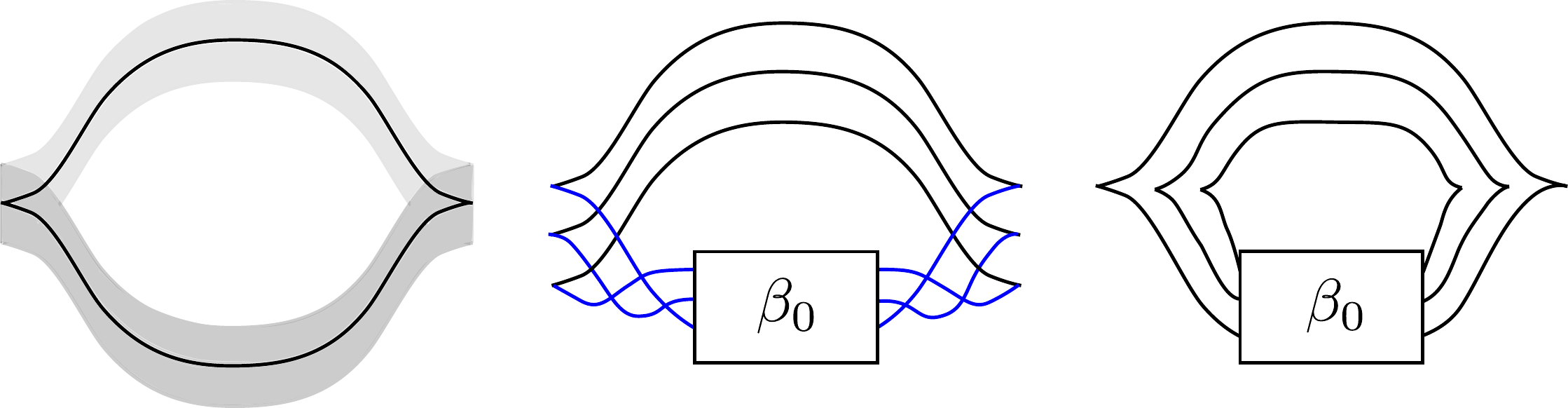}\\
  \caption{The grey region shown on the left is the image of $J^1(S^1)$ under the front projection, where $S^1$ is the Legendrian unknot $S^1 \hookrightarrow \bR^3 \hookrightarrow T^{*,\infty}\bR^2$. The picture in the middle is $\Lambda_{\beta\Delta^2}' \subset J^1(S^1) \subset T^{*,\infty}\bR^2$ in the neighbourhood of the unknot, where the blue strands on the left/right each gives a copy of $\Delta$. After doing Reidemeister~II moves, the Legendrian link in the middle is Legendrian isotopic to the link on the right, which is the Legendrian rainbow closure $\Lambda_{\beta}^\prec$ of $\beta$ in $\bR^3 \subset T^{*,\infty}\bR^2$.}\label{fig:NgBraidclosure}
\end{figure}
\begin{figure}[h!]
  \centering
  \includegraphics[width=0.95\textwidth]{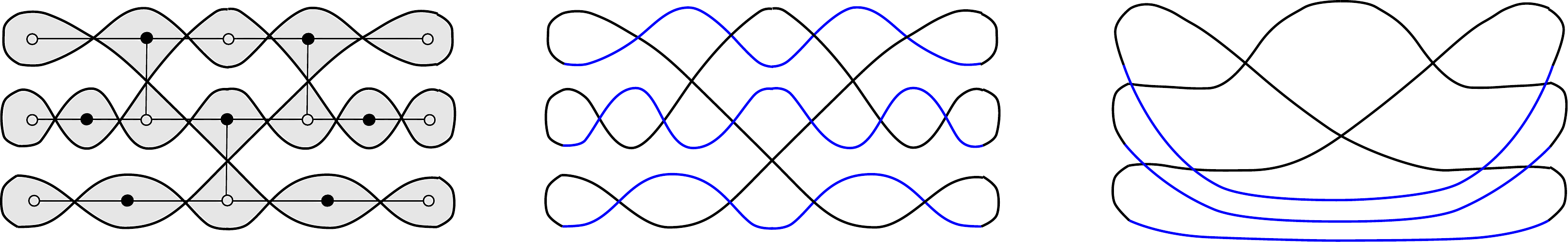}\\
  \caption{The bipartite graph $\bG_{\beta}$ associated to the braid $\beta$, its corresponding alternating Legendrian link $\Lambda(\bG_{\beta})$ and conjugate Lagrangian on the left. The alternating Legendrian link $\Lambda(\bG_{\beta})$ in the middle is Legendrian isotopic to $\Lambda_{\beta\Delta^2}$ on the right by pulling all blue strands to the bottom.}\label{fig:braid-conj}
\end{figure}

\begin{ex}[Grid plabic graphs]\label{ex:gridplabic}
    In \cite{CasalsWeng22}*{Section 2.1} plabic fences are generalized to grid plabic graphs by allowing internal lollipops (one-valent vertices). The alternating strand diagrams of grid plabic graphs similarly define Legendrian links equipped with conjugate Lagrangian fillings. Note that there are many Legendrian knots which are alternating strand diagrams of grid plabic graphs but not of plabic fences, see \cite{CasalsWeng22}*{Section 2.5}.\footnote{We remark that the convention of plabic fences in \cite{CasalsWeng22} is different from ours due to the black vertices on the right of each strand \cite{CasalsWeng22}*{Section 2.1}. The alternating Legendrians they get are Legendrian isotopic to $\Lambda_{\beta\Delta^2}' \subset T^{*,\infty}\bR^2$ instead of $\Lambda_{\beta\Delta^2} \subset T^{*,\infty}\bR^2$.}\hfill$\Box$
\end{ex}

\subsubsection{Basis of $H_1(L,\Z)$ for conjugate Lagrangians}\label{sec:H1cycle-conj}
    Given any alternating Legendrian representative and its conjugate Lagrangian filling, there is a combinatoric way to define 1-cycles in the conjugate Lagrangian, as follows.

\begin{definition}\label{def:H1cycle-conj}
    Let $\La\subset T^{*,\infty}\Sigma$ be an alternating Legendrian link, $\Sigma$ a closed oriented marked surface and $\Sigma_\text{op} = \Sigma \backslash \{x_1, ..., x_m\}$ the complement of its marked points. A null component $F$ of the alternating coloring of $\Sigma$ is called an internal null component if it is contained in $\Sigma_\text{op}$.
    
    For a null region $F$ of the alternating coloring of $\Sigma$, we denote by $\gamma_{F,\pi} \in H^1(\pi(L(\bG)); \bZ)$ the 1-cycle in the projection of $L(\bG)$ given by the oriented boundary of $F$. By definition, the 1-cycle $\gamma_F \in H_1(L(\bG); \bZ)$ is the lift of $\gamma_{F,\pi}$ to the conjugate Lagrangian filling $L(\bG)$.
\end{definition}

\begin{lemma}[\cite{Gon17Web}*{Section 2.5}]\label{lem:H1cycle-conj}
    Let $\Sigma_\text{op} = \Sigma \backslash \{x_1, ..., x_m\}$ be a surface with punctures where $\Sigma$ is a closed oriented surface. Then there are exact sequences
    \begin{gather*}
        0 \lr H_2(\Sigma; \bZ) \lr \bZ[\{\gamma_F |F\,\text{null components}\}] \lr H_1(L(\bG); \bZ) \lr H_1(\Sigma; \bZ) \lr 0.\\
        0 \lr \bZ[\{\gamma_F |F\,\text{internal null components}\}] \lr H_1(L(\bG); \bZ) \lr H_1(\Sigma_\text{op}; \bZ) \lr 0.
    \end{gather*}
    In particular, $\sum_{F\,\text{null components}}\gamma_F = 0.$
\end{lemma}

The lemma implies that $\{\gamma_F|F\,\text{null components}\}$ are not necessarily linearly independent. However, taking the quotient by the rank 1 relation in Lemma \ref{lem:H1cycle-conj}, we get a collection of linearly independent cycles. When $H_1(\Sigma; \bZ) = 0$, this collection is a basis of $H_1(L(\bG); \bZ)$.
\noindent In particular, when $\Sigma = S^2$ with a single marked point $x_1 = \infty$, i.e.~$\bG$ is a bipartite graph on a disk, $\{\gamma_F|F\,\text{internal null components}\}$ forms a basis of the conjugate Lagrangian filling $L(\bG)$.

\subsubsection{Vertex reduction}\label{sec:vertreduction}
A vertex reduction defines a Legendrian isotopy between 2 alternating Legendrians $\Lambda_0$ and $\Lambda_1$ that differ by a Reidemeister~II move, as shown in Figure \ref{fig:RIImove0}. In the corresponding bipartite graph, a degree 2 white (or black) vertex is removed. Let us show that the corresponding conjugate Lagrangians are Hamiltonian isotopic.

\begin{figure}[h!]
  \centering
  \includegraphics[width=0.7\textwidth]{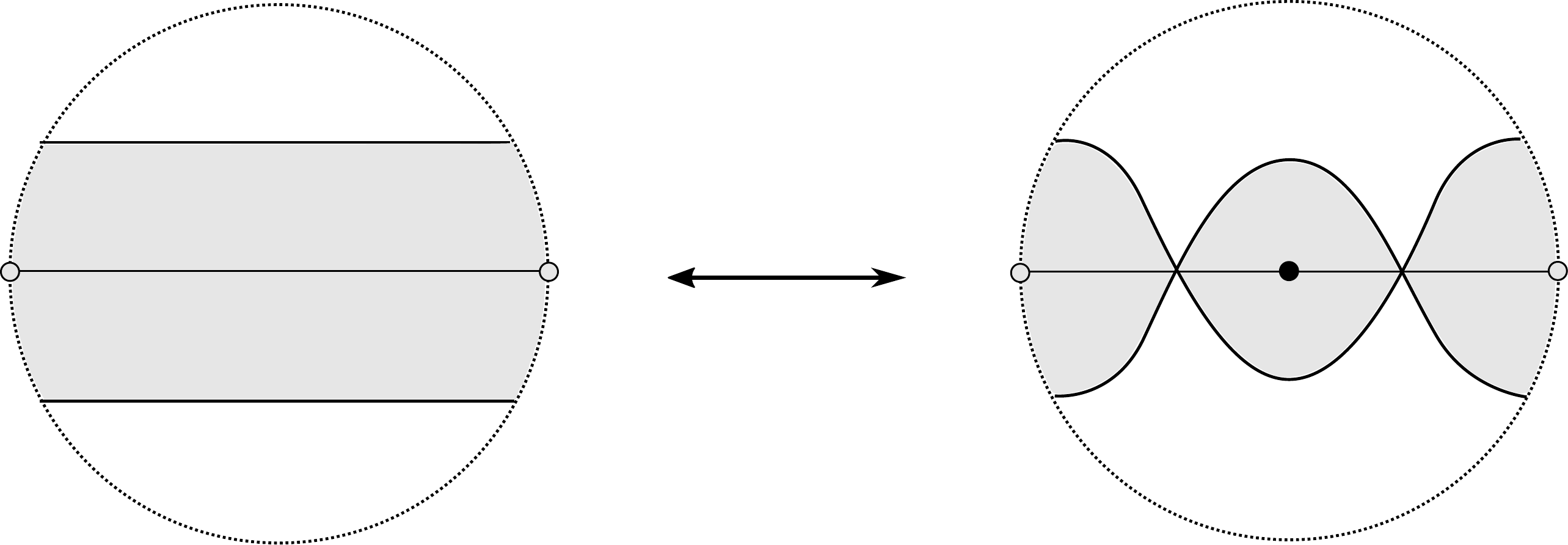}\\
  \caption{The conjugate Lagrangians $L_0, L_1$ that differ by a vertex reduction move.}\label{fig:RIImove0}
\end{figure}

\begin{prop}\label{prop:RII-move0}
    Let $L_0, L_1$ be conjugate Lagrangian fillings of the alternating Legendrians $\Lambda_0, \Lambda_1$ differing by a vertex reduction (Figure \ref{fig:RIImove0}). Then $L_0, L_1$ are Hamiltonian isotopic.
\end{prop}
\begin{proof}
    Let $\Lambda_1$ be the Legendrian whose front projection is the curve $\{(x, y)\in\R^2 | y = \pm(x^2 - 1)\}$, modeling Figure \ref{fig:RIImove0} (right), and $\Lambda_0$ be the Legendrian whose front projection is the curve $\{(x, y)\in\R^2 | y = \pm(x^2 + 1)\}$, modeling Figure \ref{fig:RIImove0} (left). We consider the local model of $L_1$ given by the equation
    $$L_1 = \left\{\Big(s, (2t-1)(s^2-1), -\frac{2s}{\sqrt{t(1-t)}}, \frac{2t-1}{\sqrt{t(1-t)}} \Big)\Big|s \in \bR, t\in (0, 1)\right\}.$$
    The primitive of the exact Lagrangian $L_1$ is $f_{L_1} = -2(s^2-1)\sqrt{t(1-t)}$.\\ 
    
\noindent Consider a Hamiltonian $H:T^*\R^2\lr\R$, $H(x, y, \xi, \eta) = H(\eta)$ such that $H'(\eta) > 0$ when $\eta > 0$, $H'(\eta) < 0$ when $\eta < 0$, and $\lim_{\eta \rightarrow \pm\infty}H'(\eta) = \pm 2$. Then the Hamiltonian flow is defined by
    $$\varphi_H^1(x, y, \xi, \eta) = (x, y + H'(\eta), \xi, \eta).$$
    On the level sets of $H(\eta)$ where $\eta$ is a constant, the parameter $t$ is a constant and hence the front projection $\pi(L_0 \cap H^{-1}(c))$ are two parabolas which are symmetric along the $x$-axis. The front projection $\pi(\varphi_H^1(L_0) \cap H^{-1}(c))$ consists of two parabolas and by construction the level curves do not intersect. Therefore, under the Hamiltonian isotopy, $\varphi_H^1(L_1) = L_0$, where $L_0$ is a graphical Lagrangian filling of the Legendrian $\Lambda_0$.
\end{proof}
\noindent 
We have presented the reduction of black vertices, an identical argument proves the reduction for white vertices.

\begin{figure}[h!]
  \centering
  \includegraphics[width=1.0\textwidth]{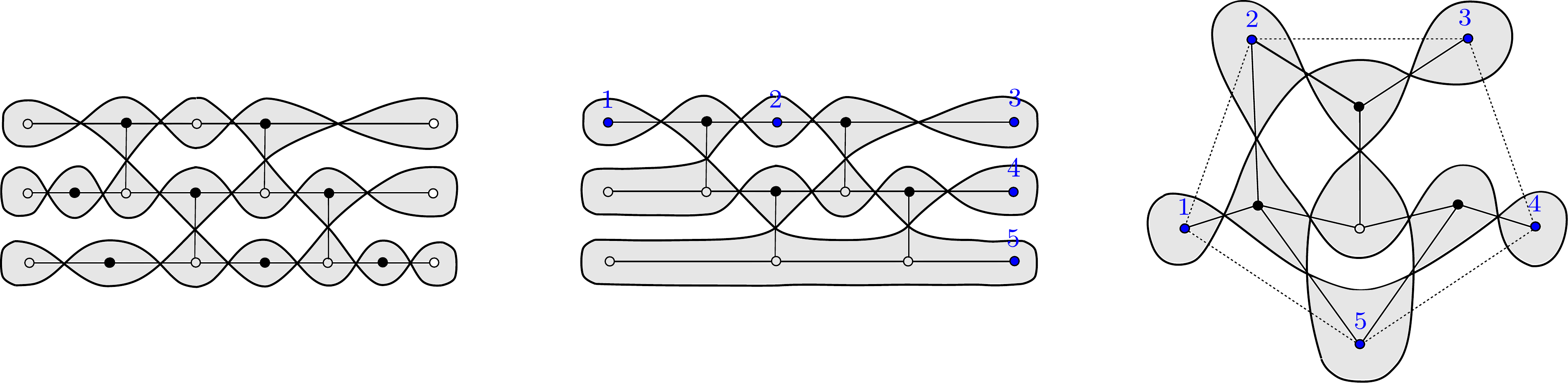}\\
  \caption{The alternating Legendrian corresponding to the alternating strand diagram for the Legendrian $(3,2)$-torus knot and its corresponding vertex reduction. The white vertices on the top row and the rightmost column form the vertices of an pentagon (the white vertices on the bottom row are identified as one single vertex on the bottom right). In general the plabic graph for the Legendrian $(n, k)$-torus link can be viewed as a plabic graph on a $(n+k)$-gon.}\label{fig:braidgraphpolygon}
\end{figure}

\begin{ex}[Legendrian torus links]\label{ex:braid-conj-red}
Consider the positive braid $\beta = (s_1s_2...s_{n-1})^k\in\mbox{Br}_n^+$ and its associated plabic fence, see Figure \ref{fig:braidgraphpolygon} (left). The Legendrian link $\Lambda_\beta$ represents the Legendrian positive $(n, k)$-torus link when considered in $(S^3,\xi_{st})$, once the Legendrian fiber of $T^{*,\infty}\R^2$ is satellited to the standard Legendrian unknot. Now, by removing the black vertices in the upmost strand and rightmost interval of all strands of the plabic fence, we are performing Reidemeister~II moves to the corresponding alternating Legendrians (Figure \ref{fig:braidgraphpolygon}), and the resulting graph is no longer bipartite. One can make it a bipartite graph by replacing all black and white vertices in a common region by a single black and white vertex. At this stage, we can view the white vertices on the bottom row, the $(n + 1)$ white vertices on the top row, and finally the remaining $(n - 2)$ white vertices on the rightmost columns as the vertices of an $(n + k)$-gon. This is illustrated at the center and right of Figure \ref{fig:braidgraphpolygon}.\\

\noindent    For the alternating Legendrian associated to the plabic graph, the strand starting from a white vertex $v$ on the boundary will go to a vertex $\pi(v)$, and $\pi$ determines a permutation of the set $\{1, 2, ..., k+n\}$ called the strand permutation \cite{Pos06AltGraph}*{Section 13}. For the plabic graph whose associated alternating Legendrian link is the $(n, k)$-torus link, one can show that the strand permutation is $\pi = (k+1,...,k+n, 1, 2, ..., k)$.\\

\begin{figure}[h!]
  \centering
  \includegraphics[width=0.95\textwidth]{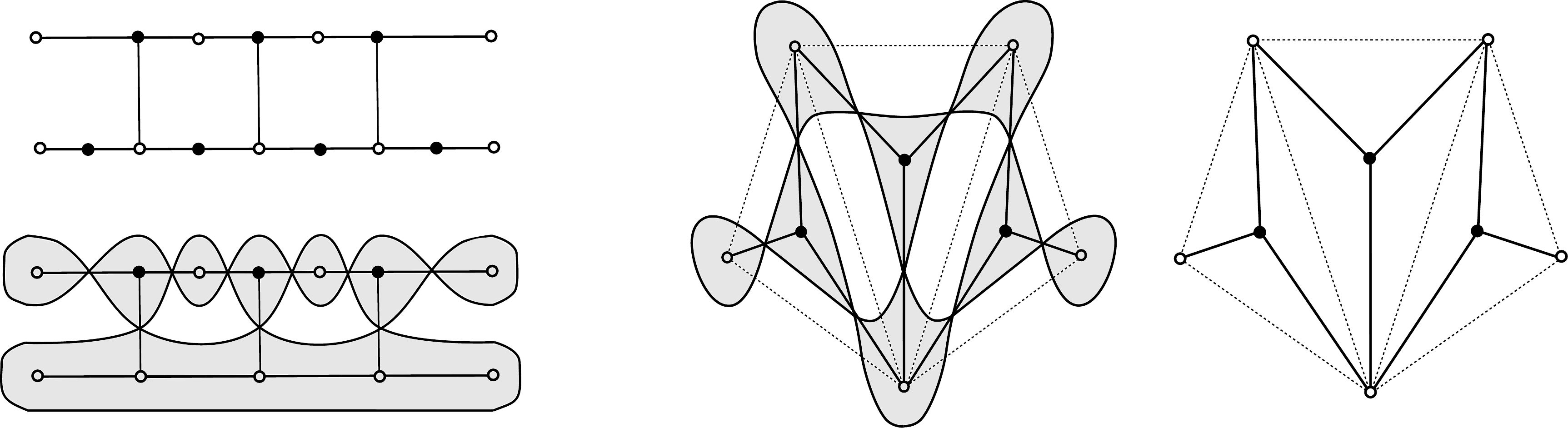}\\
  \caption{The alternating Legendrian corresponding to the alternating strand diagram for the Legendrian $(2, 3)$-torus knot (the trefoil knot) and its corresponding vertex reduction. The remaining 5 white vertices (the ones on the bottom row are viewed as a single white vertex) determines a pentagon, and the plabic graph determines a triangulation of the pentagon.}\label{fig:graphtriangulation}
\end{figure}

Finally, note that when $n = 2$ and $\beta = s_1^k$, $\Lambda_{\beta\Delta^2}$ represents the positive Legendrian $(2, k)$-torus link. By removing the black vertices in the upmost strand of the plabic graph, we are performing Reidemeister~II moves to the corresponding alternating Legendrians. Again viewing the white vertices as vertices of an $(k + 2)$-gon, we can connect the $(k + 2)$ white vertices and define a triangulation of an $(k + 2)$-gon, as shown in Figure \ref{fig:graphtriangulation}. Different triangulations of the $(k + 2)$-gon yield different alternating Legendrian representatives of $\Lambda_\beta$ and, when compared, different Lagrangian fillings.\hfill$\Box$

\end{ex}

\subsubsection{Square moves}\label{sec:squaremove}
    The square move is the Legendrian isotopy between alternating Legendrians $\Lambda_0$ and $\Lambda_1$ shown in Figure \ref{fig:SqMove}. Their associated bipartite graphs correspond to two different triangulations of a square, where white vertices correspond to vertices of the square and black vertices correspond to faces in the triangulation. The corresponding conjugate Lagrangians associated to $\Lambda_0$ and $\Lambda_1$ differ by a Lagrangian disk surgery \cite{Pol91Surgery}, see also \cite{LalondeSikorav91,Yau13Surg,Haug15AntiSurg}. In fact, \cite{STWZ19}*{Proposition 5.15} show that the two conjugate Lagrangian fillings of the alternating Legendrians $\Lambda_0, \Lambda_1$ that differ by a square move are related by a Lagrangian surgery whose vanishing cycle is $\gamma_F$ where $F$ is the null region at the center of Figure \ref{fig:SqMove}.

\begin{figure}[h!]
  \centering
  \includegraphics[width=0.7\textwidth]{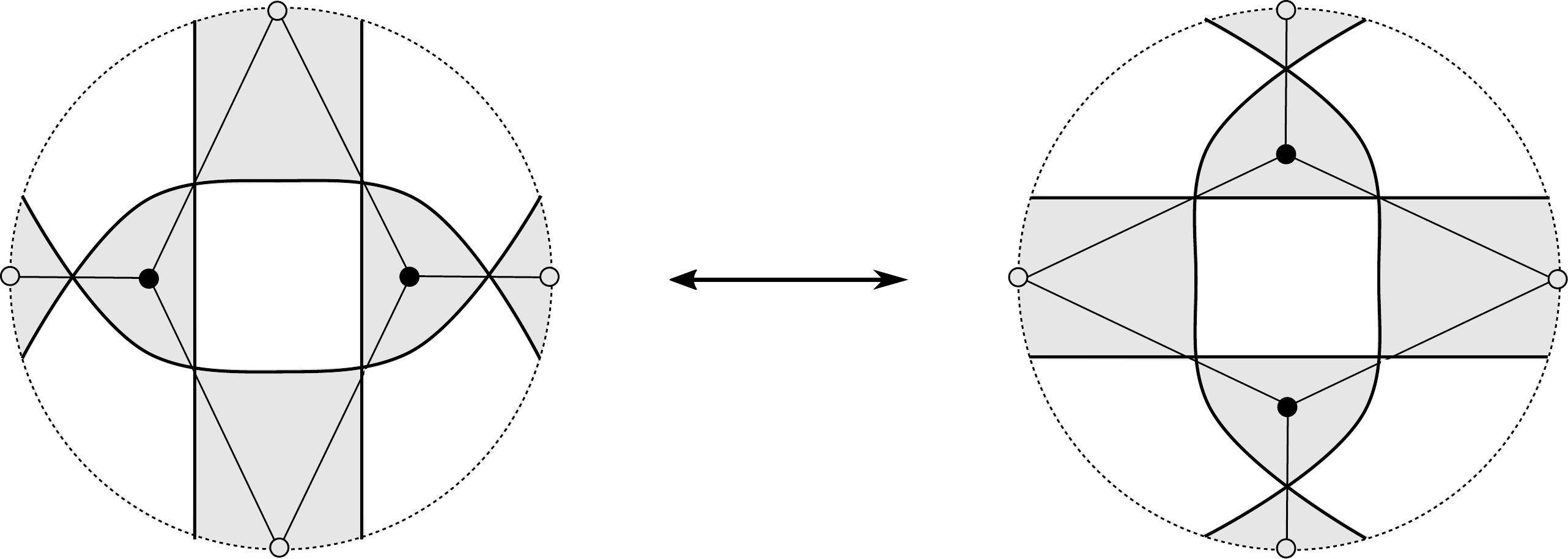}\\
  \caption{The square move of alternating Legendrian links and their conjugate Lagrangian fillings.}\label{fig:SqMove}
\end{figure}

Technically, the local model constructed in the proof of \cite{STWZ19}*{Proposition 5.15} is different from the local model considered in Proposition \ref{prop:conj-exist} or \cite{STWZ19}*{Proposition 4.9}. A priori it is not clear whether the two different models are Hamiltonian isotopic. However, Proposition \ref{prop:conj-unique} implies that this is the case, and hence makes the proof complete.



\begin{ex}[Legendrian torus links]\label{ex:2,nlink-conj}
    Following Example \ref{ex:braid-conj-red}, for a positive $n$-braid $\beta = (s_1s_2\dots s_{n-1})^k$, the corresponding Legendrian $(n, k)$-torus link $\Lambda_{\beta\Delta^2}$ has alternating Legendrian representatives and conjugate Lagrangian fillings coming from plabic graphs on an $(n+k)$-gon, having $(n+k)$ white vertices on the boundary, with strand permutatio $(n+1,...,n+k, 1, 2, ..., k)$. Such a plabic graph of the $(n+k)$-gon determines a weakly separable collection of $n$ elements in $\{1, 2..., n+k\}$ as introduced in \cite{OPS15WeakSep}. By considering different weakly separable collections of $\{1, 2..., n+k\}$, we obtain different alternating Legendrian representatives of $\Lambda_{\beta\Delta^2}$, and thus (potentially) different conjugate Lagrangian fillings of $\Lambda_{\beta\Delta^2}$. In fact, any different pair of plabic graphs of the same strand permutation type are related by a sequence of vertex reductions and square moves \cite{Pos06AltGraph}*{Theorem 13.4}.\\

    \noindent For a positive 2-braid $\beta = s_1^k$, which yields the Legendrian $(2, k)$-torus link, \cite{FZ02ClusterI} shows that the number of possible reduced plabic graphs, i.e.~the number of weakly separable collections, is the same as the number of triangulations of an $(k+2)$-gon, as shown in Figure \ref{fig:braidgraphpolygon}. Therefore, by performing square moves, we obtain $C_k = \frac{1}{k+1}\binom{2k}{k}$, the $k$-th Catalan number, different alternating Legendrian representatives; their conjugate Lagrangian fillings can be shown to yield different Hamiltonian isotopy classes of Lagrangian fillings, see \cite{STWZ19}*{Proposition 6.2} which is base on \cite{STWZ19}*{Theorem 5.17} or \cite{GaoShenWeng20,CasalsWeng22}.\footnote{The argument in \cite{STWZ19}*{Section 6} is a little ambigous in that it is unclear whether they used $\cX$-variables or $\cA$-variables, but  using $\cA$-variables, one can give a sheaf-theoretic proof that these conjugate Lagrangian fillings are distinct for all weakly separated collections.}
\end{ex}

\subsection{Legendrian weaves}\label{sec:weave} Legendrian weaves were introduced in \cite[Section 2]{CasalsZas20}, see also \cite{TreuZas16} and \cite[Section 3]{CasalsWeng22}. They have found several applications to contact topology \cite{CasalsZas20,Hughes21A,Hughes21D,AnBaeLee21ADE,AnBaeLee21DEtilde} and the study of cluster algebras \cite{CGGS20AlgWeave,CGGLSS22,CasalsWeng22}. Let $\Sigma$ be a smooth surface and $(J^1\Sigma,\xi_{st})$ its 1-jet space, with the standard contact form $\alpha_\text{st} = dz - \lambda_\text{st}$, where $z$ is the coordinate on $\bR$ and $\lambda_\text{st}$ is the standard Liouville form on $T^*\Sigma$.\\

\noindent According to \cite{CasalsZas20}*{Definition 2.2}, an $n$-graph $\mathfrak{w}$ on a smooth surface $\Sigma$ is a set of $n-1$ embedded graphs $\{G_i\}_{1\leq i\leq n-1}$ with only trivalent vertices (Figure \ref{fig:weave-vertex} left) such that $G_i$ and $G_{i+1}$ intersect only at hexagonal points (Figure \ref{fig:weave-vertex} right). The convention throughout the paper is that when the color \textcolor{blue}{blue} is associated to $G_i$, then the color \textcolor{red}{red} will be associated to $G_{i+1}$, and the color \textcolor{Green}{green} to $G_{i+2}$.

\begin{figure}[h!]
  \centering
  \includegraphics[width=0.4\textwidth]{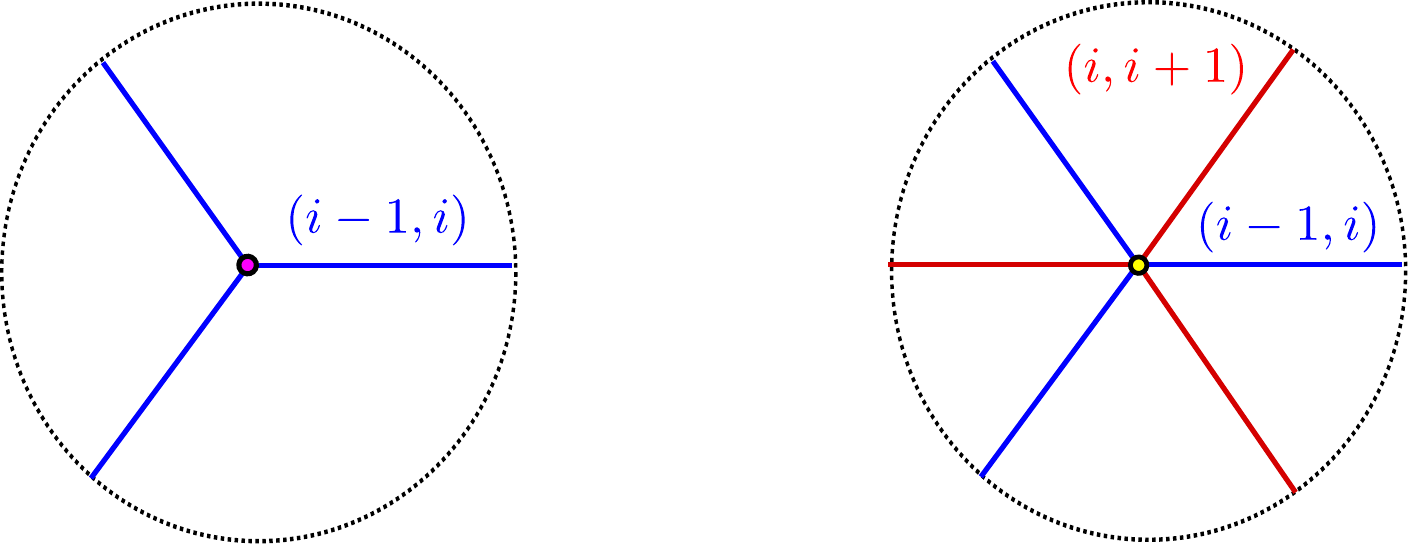}\\
  \caption{A trivalent vertex (left) and a hexagonal vertex (right).}\label{fig:weave-vertex}
\end{figure}

\noindent In order to study Legendrian surfaces in $(J^1\Sigma,\xi_{st})$, it suffices to study their front projections in $\Sigma \times \bR$. For a generic Legendrian surface, front projections are immersed surfaces in $\Sigma \times \bR$ with certain singularities \cite[Chapter 3]{ArnoldSing}. Given an $n$-graph $\mathfrak{w} \subset\Sigma$, we can consider a singular surface in $\Sigma \times \bR$ whose transverse self intersection set is $\mathfrak{w}$ and its projection onto $\Sigma$ is an $n$-fold cover branched over the trivalent vertices of $G$. In this case, only the three singularities of the front projection are depicted in Figure \ref{fig:weave-wavefront}.

\begin{figure}[h!]
  \centering
  \includegraphics[width=0.8\textwidth]{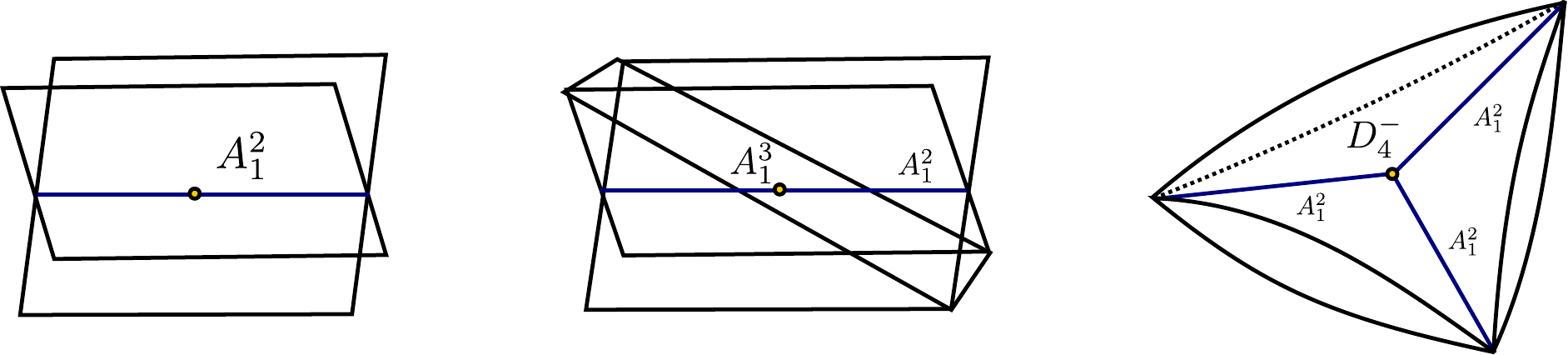}\\
  \caption{Three singularities of a Legendrian front projection: $A_1^2$ (left), $A_1^3$ (middle), and $D_4^-$ (right). Both $A_1^2$ and $A_1^3$ are generic but $D_4^-$ is not generic and it perturbs to a front with three swallowtail singularities.}\label{fig:weave-wavefront}
\end{figure}

\noindent Following \cite{CasalsZas20}*{Definition 2.7}, the precise definitions of this front and its associated Legendrian surface read as follows:

\begin{definition}\label{def:weave}
    Let $\mathfrak{w}$ be an $n$-graph on a surface $\Sigma$. By definition, the Legendrian weave $\widetilde{L}(\mathfrak{w}) \subset (J^1\Sigma,\xi_{st})$ is the Legendrian surface whose front projection onto $\Sigma \times \mathbb{R}$ is locally characterized by the following forms (Figure \ref{fig:weave-def}):
\begin{enumerate}
  \item In a chart $D\subset\Sigma$ such that $D \cap \mathfrak{w} = \emptyset$, the front projection is
  $$\bigcup_{1\leq j\leq n} D \times \{j\} \subset D \times \mathbb{R};$$
  \item in a chart $D\subset\Sigma$ such that $D \cap \mathfrak{w}$ is a single line segment in $G_i\,(1\leq i\leq n-1)$, the front projection is
  $$\bigcup_{1\leq j\leq n,\,j \neq i, i+1} D \times \{j\} \cup \Pi(A_1^2) \subset D \times \mathbb{R}$$
  where $\Pi(A_1^2)$ is isotopic to $\{(x_1, x_2, z) | x_1^2 - z^2 = 0\}$;\\
  \item in a chart $D\subset\Sigma$ such that $D \cap \mathfrak{w}$ is a neighbourhood of a hexagonal intersection point of $G_i$ and $G_{i+1}\,(1\leq i\leq n-2)$, the front projection is
  $$\bigcup_{1\leq j\leq n,\,j \neq i, i+1,i+2} D \times \{j\} \cup \Pi(A_1^3) \subset D \times \mathbb{R}$$
  where $\Pi(A_1^2)$ is isotopic to $\{(x_1, x_2, z) | (x_1 - z)(x_1 + z)(z - x_2) = 0\}$;\\
  \item in a chart $D\subset\Sigma$ such that $D \cap \mathfrak{w}$ is a neighbourhood of a trivalent vertex in $G_i\,(1\leq i\leq n-1)$, the front projection is
  $$\bigcup_{1\leq j\leq n,\,j \neq i, i+1} D \times \{j\} \cup \Pi(D_4^-) \subset D \times \mathbb{R}$$
  where $\Pi(D_4^-)$ is isotopic to $\{(x_1, x_2, z) | x_1+\sqrt{-1}x_2 = w^2, z = \mathrm{Re}(w^3), w \in \bC\}$.\footnote{We are label the graphs from bottom to top: $G_1$ is on the bottom while $G_{n-1}$ is on the top. This is compatible with the convention that the crossings of the braid closure are labelled from outside to inside.}
\end{enumerate}
\end{definition}

\begin{figure}[h!]
  \centering
  \includegraphics[width=0.8\textwidth]{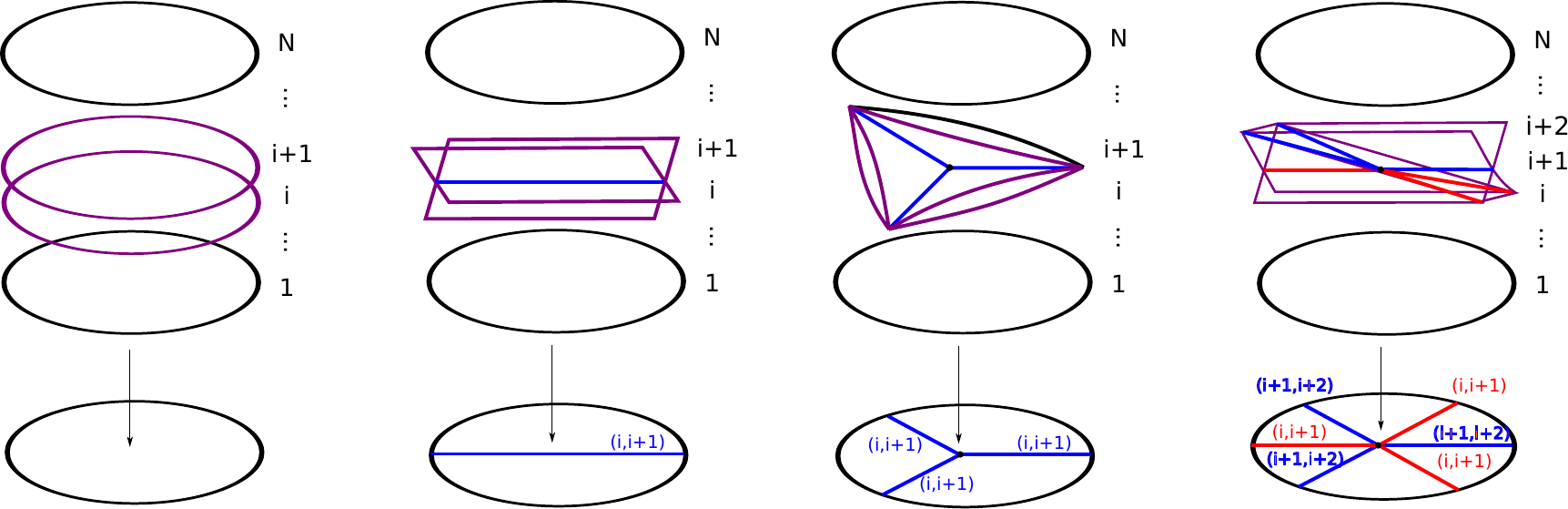}\\
  \caption{The front projections of Legendrian weaves in the 4 cases of Definition \ref{def:weave}.}\label{fig:weave-def}
\end{figure}

\subsubsection{Basis of $H_1(\widetilde{L}(\mathfrak{w});\Z)$ for Legendrian weaves}\label{sec:H1cycle-weave}
    For a Legendrian weave $\widetilde{L}(\mathfrak{w})$ associated to an $n$-graph $\mathfrak{w}$, there is currently no full description of all cycles in $H_1(\widetilde{L}(\mathfrak{w}); \bZ)$ in a general case. Nevertheless, we can always consider $\sf Y$-cycles and in many relevant cases, e.g.~Demazure weaves \cite{CGGS20AlgWeave,CGGLSS22} and weaves from grid plabic graphs \cite{CasalsWeng22}, these are known to span $H_1(\widetilde{L}(\mathfrak{w}); \bZ)$. For an introduction to $\sf Y$-cycles see \cite[Section 2.4]{CasalsZas20}, or also \cite[Section 3.2]{CasalsWeng22} and \cite[Section 4]{CGGLSS22}; the latter for a more combinatorial account. We use the standard notation of short and long $\sf I$-cycles from those articles as well.

\subsubsection{Free Legendrian weaves}\label{ssec:ex-free-weave}
    Let $\Sigma$ be a smooth surface with marked points $\{x_1, ..., x_m\}$, $\Sigma_\text{op} := \Sigma \backslash \{x_1, ..., x_m\}$, $\{B_\epsilon(x_1), ..., B_\epsilon(x_m)\}$ a disjoint collection of open balls around the marked points, and $\Sigma_\text{cl} := \Sigma \backslash (\bigcup_{1\leq i\leq m}B_\epsilon(x_i))$. Consider a Legendrian link in the 1-jet space $(J^1(\partial \Sigma_\text{cl}),\xi_{st})$. We can study Legendrian weaves in $(J^1\Sigma_\text{cl},\xi_{st})$ whose boundary is the Legendrian link and near the boundary the weave is conical; see \cite{CasalsZas20}*{Section 7.1} or also \cite{PanRuther19Funct}*{Section 4.3}.\\
    
    \noindent Projecting a Legendrian weave in $J^1\Sigma_\text{cl}$ onto $T^*\Sigma_\text{cl}$, by quotienting the Reeb direction, we obtain an immersed exact Lagrangian filling of the Legendrian link in $J^1(\partial \Sigma_\text{op})$.\footnote{Conversely, any immersed exact Lagrangian filling lifts to a Legendrian surface in $J^1\Sigma_\text{cl}$.} It is thus possible to study (a priori some) Lagrangian fillings of links in $T^*\Sigma_\text{op}$, by studying Legendrian weaves in $J^1\Sigma_\text{cl}$. Since immersed points of such exact Lagrangian fillings correspond to Reeb chords in its Legendrian weave lifting, Legendrian weaves without Reeb chords yield embedded Lagrangian fillings. According to \cite{CasalsZas20}*{Definition 7.2}], a Legendrian weave in $J^1\Sigma_\text{op}$ is said to be free if it can be isotoped to a weave with no Reeb chords.\\
    
    Parallel to Subsection \ref{ssec:triangle-positive} for conjugate Lagrangian fillings, we now present the constructions of Legendrian weaves from $n$-triangulations and positive braids.\\

{\bf $n$-triangulations and Legendrian weaves} Let $\Sigma$ be a smooth surface with marked points $\{x_1, ..., x_m\}$, and $\bigcirc_n \subset J^1\Sigma_\text{cl}$ be the union of trivial $n$-braids, which is the Legendrian $n$-satellite of the zero section in $J^1(\partial \Sigma_\text{cl})$. Given an ideal $n$-triangulation of $(\Sigma,\{x_1, ..., x_m\})$,  \cite{CasalsZas20}*{Section 3.1} constructs a free Legendrian weave $\widetilde{L}(\mathfrak{w}_n^*)$ associated to this ideal triangulation, as shown in Figure \ref{fig:triangle-weave}. The corresponding $n$-graph is defined as follows. Consider the subdivision of the triangle into $n^2$ small triangles. For each of the triangle, insert a trivalent vertex in $G_1$ dual to the triangle (i.e.~the vertex is the center and the edges orthogonally intersect the edges of the triangle). Then whenever three of these edges in $G_1$ collide, insert a $(G_1, G_2)$-hexagonal vertex, and inductively whenever three of the edges in $G_i$ collide, insert a $(G_i, G_{i+1})$-hexagonal vertex. The resulting $n$-graph defines a free Legendrian weave $\widetilde{L}(\mathfrak{w}_n^*)$ with boundary $\bigcirc_n$.\footnote{In this notation, the superscript is added to make the notation consistent with the ones in the conjugate Lagrangian filling part.}

\begin{figure}[h!]
  \centering
  \includegraphics[width=0.9\textwidth]{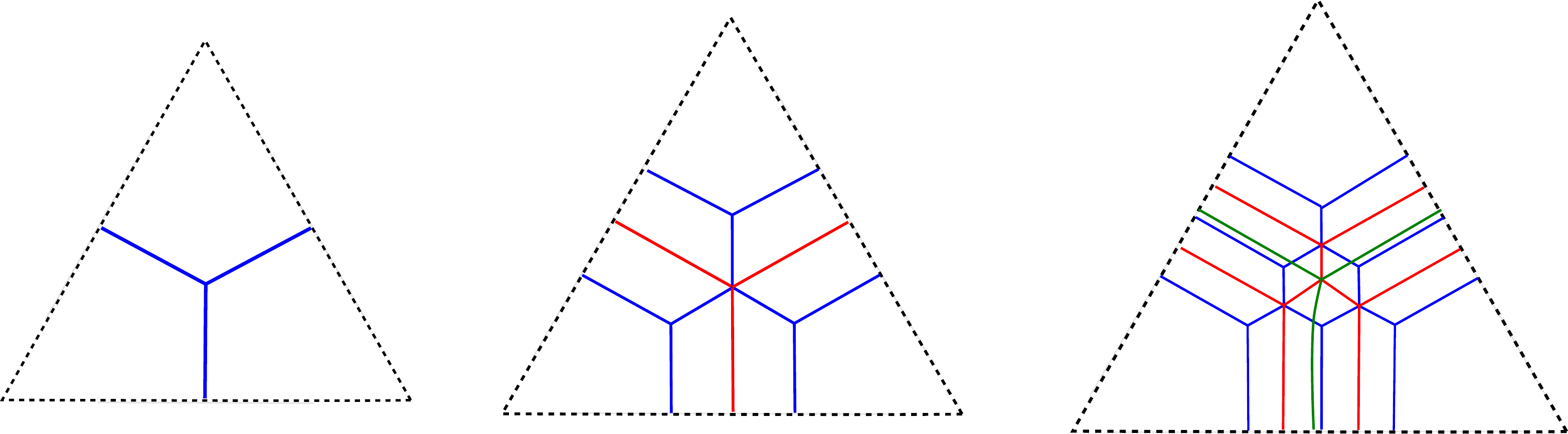}\\
  \caption{The free Legendrian weaves $\widetilde{L}(\mathfrak{w}_n^*)$ associated to an ideal $n$-triangulation on punctured surfaces for $n = 2, 3$ and $4$. The vertices are the marked points of the surface.}\label{fig:triangle-weave}
\end{figure}

{\bf Positive braids and Legendrian weaves}.
In line with Subsection \ref{ssec:triangle-positive}, consider the Legendrian positive braid closure $\Lambda(\beta\Delta^2) \subset J^1S^1$. Viewing $J^1S^1$ as the contact neighbourhood of the standard Legendrian unknot $S^1 \subset S^3$, $\Lambda_{\beta\Delta^2} \subset S^3$ is Legendrian isotopic to the (Legendrian lift of the) rainbow closure $\Lambda_{\beta}^\prec$ of the braid $\beta$ in $\bR^3 \subset T^{*,\infty}\bR^2$ after compactifying $T^{*,\infty}\bR^2$ to $S^3$ by sending the Legendrian fiber to the standard Legendrian unknot. In \cite{CasalsWeng22}*{Section 3.3.2}, we constructed a free Legendrian weave which is the lift of an embedded exact Lagrangian filling of $\La_{\beta\Delta^2}$, as follows.\\

\begin{figure}[h!]
    \centering
    \includegraphics[width=0.7\textwidth]{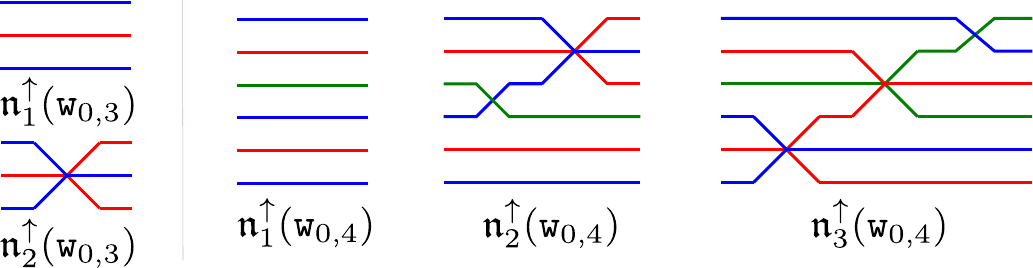}
    \caption{The Legendrian weave $\mathfrak{n}_i^\uparrow(\mathtt{w}_{0,n})$ for $n = 3$ (on the left) and $n = 4$ (on the right).}
    \label{fig:braid-weave-cross}
\end{figure}

\noindent Let $\mathfrak{n}_i^\uparrow(\mathtt{w}_{0,n})$ be the horizontal weave consisting of only hexagonal vertices, which on the left equals the half twist $\Delta = \mathtt{w}_{0,n}$, and brings the $i$-th $s_1$-strand to the top, and $\mathfrak{n}_i^\uparrow(\mathtt{w}_{0,n})^\text{op}$ be the weave defined by reflecting $\mathfrak{n}_i$ along the vertical axis (Figure \ref{fig:braid-weave-cross}). For the transposition $\sigma_i\,(1 \leq i \leq n-1)$, we consider the weave $\mathfrak{c}_i^\uparrow(\mathtt{w}_{0,n})$ defined by the horizontal concatenation of $\mathfrak{n}_i^\uparrow(\mathtt{w}_{0,n})$, $\mathfrak{c}_i^\uparrow$ and $\mathfrak{n}_i^\uparrow(\mathtt{w}_{0,n})^\text{op}$, where $\mathfrak{c}_i^\uparrow$ is given by the horizontal weave on the right of $\mathfrak{n}_i^\uparrow(\mathtt{w}_{0,n})$ and an extra trivalent vertex on the top strand with a vertical edge going to the bottom (Figure \ref{fig:braid-weave-cross2}). 
    For the positive braid $\beta = s_{i_1} s_{i_2} \dots s_{i_k}$, we define the weave $\mathfrak{w}_{\beta\Delta^2} = \mathfrak{w}(\bG_\beta)$ (the latter is the notation in \cite{CasalsWeng22}) to be the horizontal concatenation of $\mathfrak{c}_{i_1}^\uparrow(\mathtt{w}_{0,n}), \dots, \mathfrak{c}_{i_k}^\uparrow(\mathtt{w}_{0,n})$, with Legendrian boundary $\Lambda_{\beta\Delta^2}$.\footnote{We shall see that this Lagrangian filling is Hamiltonian isotopic to the one obtained by the sequence of pinching all the crossings of $\Lambda_{\beta}^\prec$ from right to left \cite{EHK16}.}

\begin{figure}[h!]
    \centering
    \includegraphics[width=1.0\textwidth]{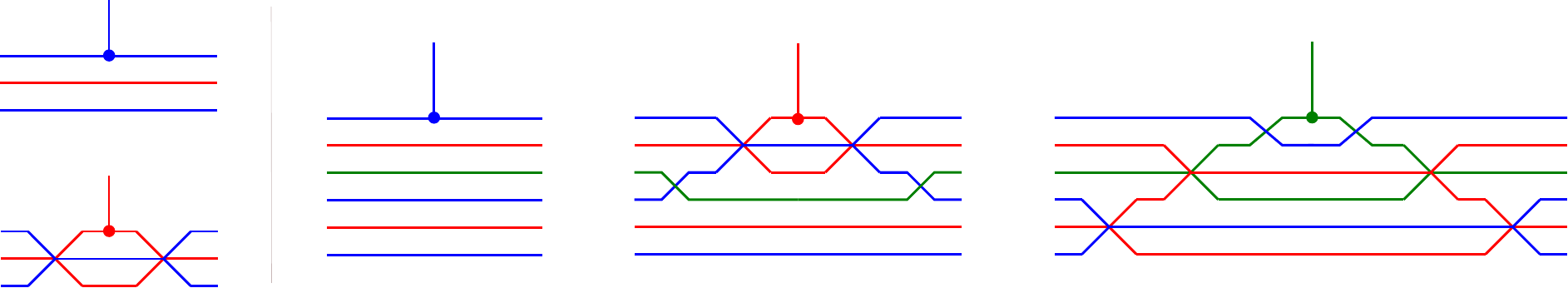}
    \caption{The Legendrian weave $\mathfrak{c}_i^\uparrow(\mathtt{w}_{0,n})$ which is the concatenation of $\mathfrak{n}_i^\uparrow(\mathtt{w}_{0,n})$, $\mathfrak{c}_i$ and $\mathfrak{n}_i^\uparrow(\mathtt{w}_{0,n})^\text{op}$ for $n = 3$ (on the left) and $n = 4$ (on the right).}
    \label{fig:braid-weave-cross2}
\end{figure}

\begin{ex}[Weaves for grid plabic graphs]\label{ex:lolipop-weave} In parallel with Example \ref{ex:gridplabic}, free Legendrian weave fillings associated to the alternating Legendrians of grid plabic graphs are constructed in \cite[Section 3]{CasalsWeng22}. Section \ref{sec:main_proofs} shows that the corresponding conjugate Lagrangian fillings are particular instances of the Lagrangian fillings associated to such free weaves.\hfill$\Box$



\end{ex}

\subsubsection{Weave Reidemeister moves} Figure \ref{fig:ReidemeisterWeave} illustrates the weave equivalences we will use. Exchanging any of these two local models, depicted left and right in each of the entries of , results in Legendrian isotopic Legendrian weave. Since the local models are free, implementing these equivalences for free Legendrian weaves yields (compactly supported) Hamiltonian isotopies between their associated embedded exact Lagrangian fillings. The necessary proofs and details are provided in \cite{CasalsZas20}*{Theorem 4.2}.
    
    \begin{figure}[h!]
  \centering
  \includegraphics[width=1.0\textwidth]{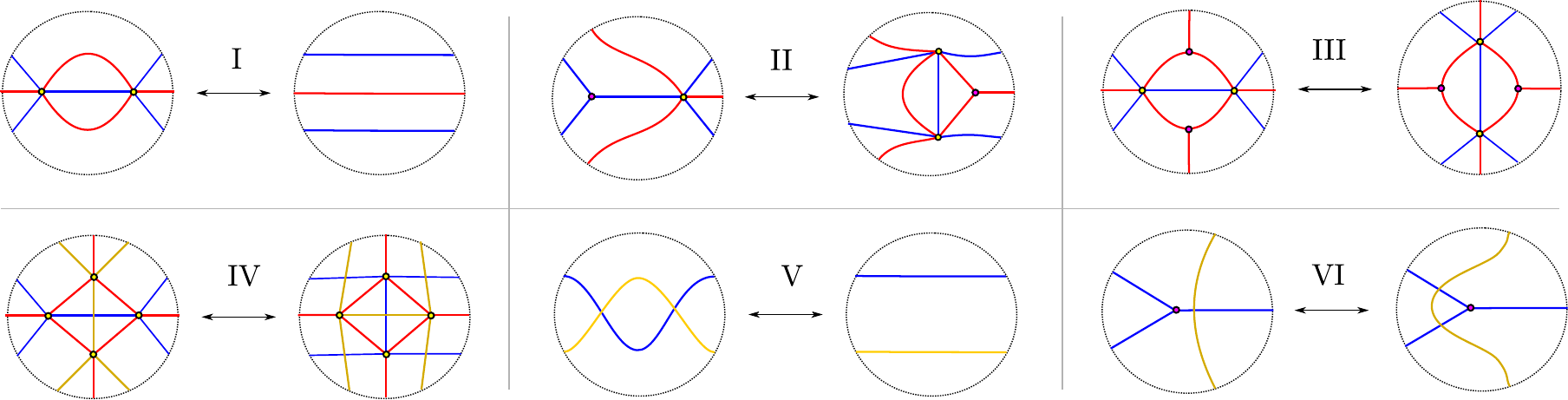}\\
  \caption{Table of some surface Legendrian Reidemeister moves for Legendrian weaves.}\label{fig:ReidemeisterWeave}
\end{figure}

\subsubsection{Lagrangian mutations in weaves} Figure \ref{fig:mutation-weave} depicts two operations on Legendrian weaves introduced in \cite[Section 4.8]{CasalsZas20}. The result of applying such operations yield a Legendrian surface which is smoothly isotopic and the associated Lagrangian projections are Lagrangian isotopic. Nevertheless, in contrast with the equivalences in Figure \ref{fig:ReidemeisterWeave}, the local operations in Figure \ref{fig:mutation-weave} typically change the Legendrian isotopic type of the Legendrian surface and also change the Hamiltonian isotopy class of the associated Lagrangian projections. It is proven in \cite{CasalsZas20}*{Theorem 4.21} that these induce Lagrangian disks surgeries.

\begin{figure}[h!]
  \centering
  \includegraphics[width=1.0\textwidth]{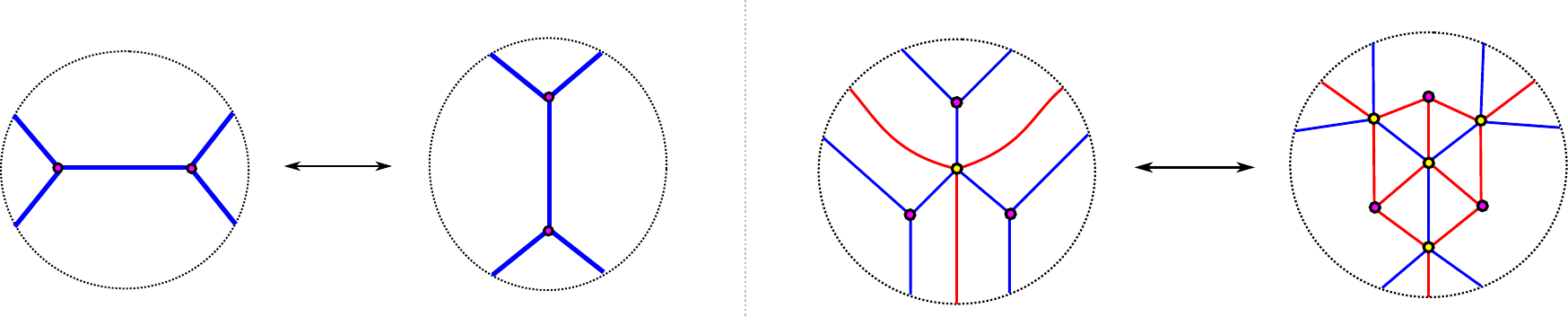}\\
  \caption{The Legendrian mutations along a type~$\sf I$ or type~$\sf Y$-cycle.}\label{fig:mutation-weave}
\end{figure}


\begin{ex}[Legendrian $(2, k)$-torus links]\label{ex:2,nlink-weave}
Consider the max-tb Legendrian $(2, k)$-torus links $\Lambda_{2, k}$. For the braid word $\beta\Delta^2 = s_1^{k+2}$ in $J^1S^1$, consider an $(k+2)$-gon whose vertices are the midpoints of the crossings in $\beta\Delta^2$. For each triangulation of the $(k+2)$-gon, we can define a graph $G$ on $\bD^2$ by taking the dual graph of the triangulation, that is, to associate each triangle with a vertex, and associate each common edge between 2 triangles an edge between the corresponding vertices; see Figure \ref{fig:2,nlink-weave}. This results in a free Legendrian weave in $J^1\bD^2$.\\

\noindent     For two adjacent triangles in a triangulation, we can perform a flip:  remove the common diagonal edge in the quadrilateral and connect the other two opposite vertices. This results in a new triangulation such that the corresponding two Legendrian weaves are related by a mutation.\hfill$\Box$
\end{ex}
\begin{figure}[h!]
  \centering
  \includegraphics[width=0.6\textwidth]{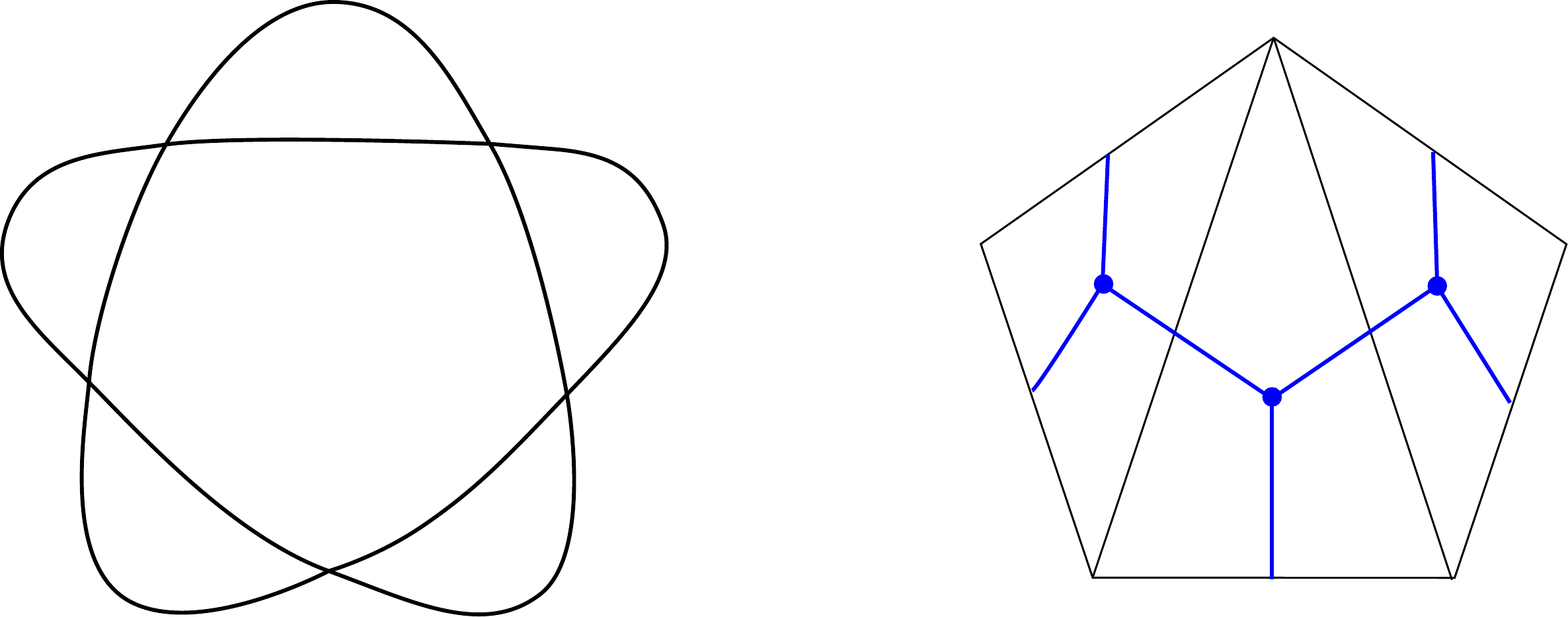}\\
  \caption{The Legendrian weave associated to a triangulation of the $(k+2)$-gon whose Lagrangian projection defines a Lagrangian filling a Legendrian $(2, k)$-torus link. Here $k=3$.}\label{fig:2,nlink-weave}
\end{figure}

\subsection{Weaves with boundary at infinity}\label{sec:weave-at-infty} An aim of this manuscript is to compare conjugate Lagrangian fillings with Legendrian weaves and show that the former are particular instances of the later. Now, free Legendrian weaves $\widetilde{L} \subset J^1\Sigma_\text{cl}$ with boundary $\Lambda \subset J^1(\partial \Sigma_\text{cl})$ define exact Lagrangian fillings of $\Lambda$ inside $T^*\Sigma_\text{cl}$, where $J^1(\partial \Sigma_\text{cl})$ is identified with a contact neighbourhood of $\partial \Sigma_\text{cl}$. In this section we now discuss how Legendrian weaves can be viewed as Lagrangian fillings of Legendrian links in $T^{*,\infty}\Sigma$, so that they can be compared with conjugate Lagrangian fillings.

    Following the notations of $\Sigma, \Sigma_\text{op}, \Sigma_\text{cl}$ in the previous sections, let $\bigcirc_1 \subset T^{*,\infty}\Sigma_\text{op}$ be the union of small inward unit conormal bundles around all the punctures $\{x_1, ..., x_m\} \subset \Sigma$, and $\Lambda \subset T^{*,\infty}\Sigma_\text{op}$ be a Legendrian satellite of $\bigcirc_1$, i.e.~we may assume that $\Lambda$ is contained in a contact neighbourhood of $\bigcirc_1 \subset T^{*,\infty}\Sigma_\text{op}$, which is contactomorphic to $J^1(\partial \Sigma_\text{cl})$; we thus fix a contact embedding $J^1(\partial \Sigma_\text{cl}) \hookrightarrow T^{*,\infty}\Sigma_\text{op}$.\\
    
    \noindent Compactifying $T^{*}\Sigma_\text{op}$ to $\ol{T^{*}\Sigma}_\text{op}$ by adding the ideal contact boundary, there is a graphical exact Lagrangian embedding $\ol{L}_{df_{\Sigma_\text{cl}}} \subset \ol{T^{*}\Sigma}_\text{op}$ of the Legendrian $\bigcirc \subset T^{*,\infty}\Sigma_\text{op}$ that is diffeomorphic to $\Sigma_\text{cl}$, where $f_{\Sigma_\text{cl}}(x) \rightarrow +\infty$ when $x \rightarrow \partial \Sigma_\text{cl}$. Hence we have an exact symplectic embedding 
    $\ol{T^*\Sigma}_\text{cl} \hookrightarrow \ol{T^*\Sigma}_\text{op}$, of Liouville domains with contact boundary, whose restriction to the boundary $T^{*,\infty}\Sigma_\text{op}$ is the fixed contact embedding $J^1(\partial \Sigma_\text{cl}) \hookrightarrow T^{*,\infty}\Sigma_\text{op}$. For any free Legendrian weave $\widetilde{L} \subset J^1\Sigma_\text{cl}$ with boundary $\Lambda \subset J^1(\partial \Sigma_\text{cl})$, the Lagrangian projection in $T^*\Sigma_\text{cl}$ can thus be viewed as a Lagrangian filling of $\Lambda$ via the above symplectic embedding.

\begin{figure}[h!]
  \centering
  \includegraphics[width=1.0\textwidth]{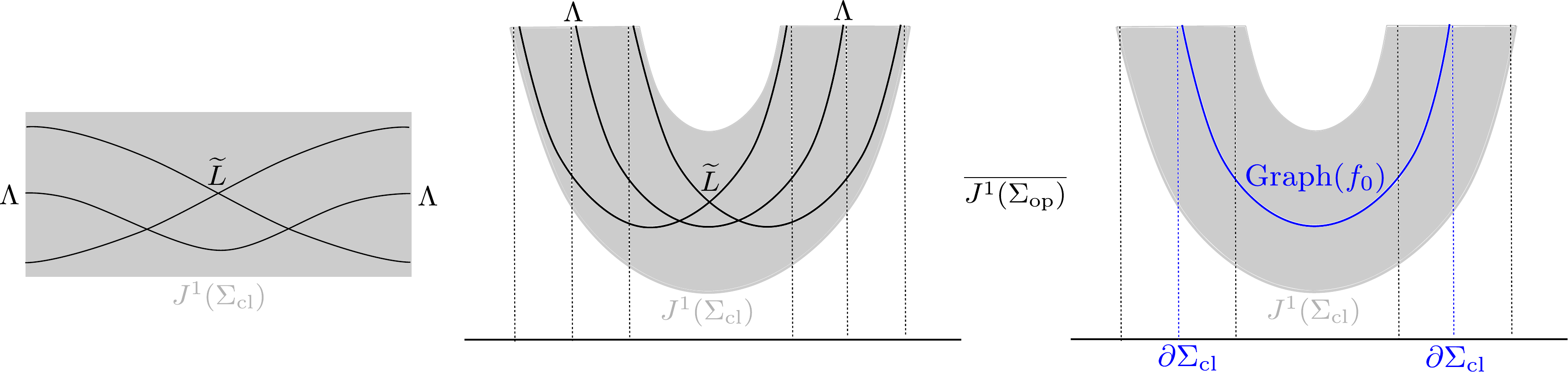}\\
  \caption{The Legendrian embedding $\Sigma_\text{cl} \hookrightarrow \ol{J^1\Sigma_\text{op}}$ via the graph of $f_{\Sigma_\text{cl}}$ and the contact embedding $J^1\Sigma_\text{cl} \hookrightarrow \ol{J^1\Sigma_\text{op}}$. This embedding enables one to view Legendrian weaves $\widetilde{L} \subset J^1\Sigma_\text{cl}$ as Legendrian weaves with boundary at infinity in $\ol{J^1\Sigma_\text{op}}$.}\label{fig:weave-bdy-at-infty}
\end{figure}

    In addition, we can consider the Legendrian lift of the exact Lagrangian filling. Compactifying $J^1\Sigma_\text{op}$ to $\ol{J^1\Sigma_\text{op}} := \ol{T^*\Sigma}_\text{op} \times \ol{\bR}$, the following diagram of Lagrangian projections and front projections commutes
    \[\xymatrix{
    \ol{T^*\Sigma}_\text{cl} \ar[d] & J^1\Sigma_\text{cl} \ar[d] \ar[l]_{\pi_\text{Lag}} \ar[r]^{\pi_\text{front}} & \Sigma_\text{cl} \times \bR \ar[d] \\
    \ol{T^*\Sigma}_\text{op} & \ol{J^1\Sigma_\text{op}} \ar[l]_{\pi_\text{Lag}} \ar[r]^{\pi_\text{front}} & \Sigma_\text{op} \times \ol{\bR},
    }\]
    where the vertical embedding on the right is determined by the Legendrian front $\mathrm{Graph}(f_{\Sigma_\text{cl}})$ of the graphical Lagrangian embedding $\ol{L}_{df_0} \subset \ol{T^{*,\infty}\Sigma}_\text{op}$. As a result, the fronts of Legendrian surfaces $\pi_\text{front}(\widetilde{L})$ in $J^1\Sigma_\text{op}$ are exactly the embeddings of the fronts of the original corresponding free Legendrian weave in $J^1\Sigma_\text{cl}$ via the embedding $\Sigma_\text{cl} \times \bR \hookrightarrow \Sigma_\text{op} \times \ol{\bR}$, and hence we can just encode the crossings in the front projection $\pi_\text{front}(\widetilde{L})$ by an $n$-graph as in the case of Legendrian weaves, but the $n$-graph will have boundary on the front of the Legendrian link $\pi(\Lambda) \times \{+\infty\} \subset \Sigma_\text{op} \times \ol{\bR}$.\footnote{Under the identification, when we view the boundary of the Legendrian weave as a Legendrian cylindrical braid closure, then the crossing coming from $G_i$ (labelled from bottom to top) corresponds to the $i$-th crossing $s_i$ (labelled from outside to inside).}

\begin{figure}[h!]
  \centering
  \includegraphics[width=0.6\textwidth]{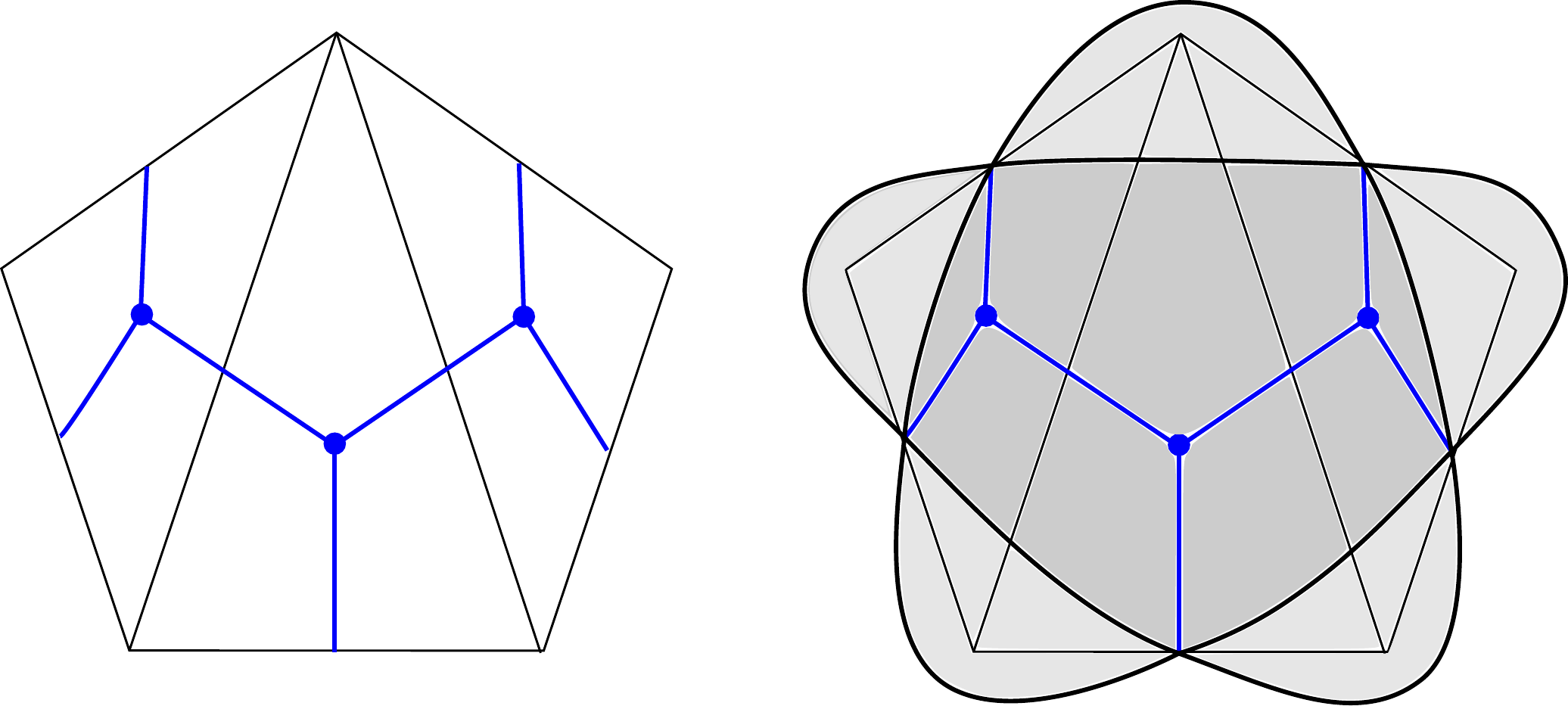}\\
  \caption{The front projection of the Legendrian weave in Example \ref{ex:2,nlink-weave} via the contact embedding $J^1\ol{\bD}^2 \hookrightarrow \ol{J^1\bD^2}$ such that the boundary lies in $J^1(S^1) \hookrightarrow T^{*,\infty}\bD^2$. In the right figure, different grey colors indicate the cardinality of preimages of the front projection in each stratum: darker grey stands for two points in the pre-image and lighter grey indicates one point.}\label{fig:2,nlink-weave-cotan}
\end{figure}
\begin{figure}[h!]
  \centering
  \includegraphics[width=0.7\textwidth]{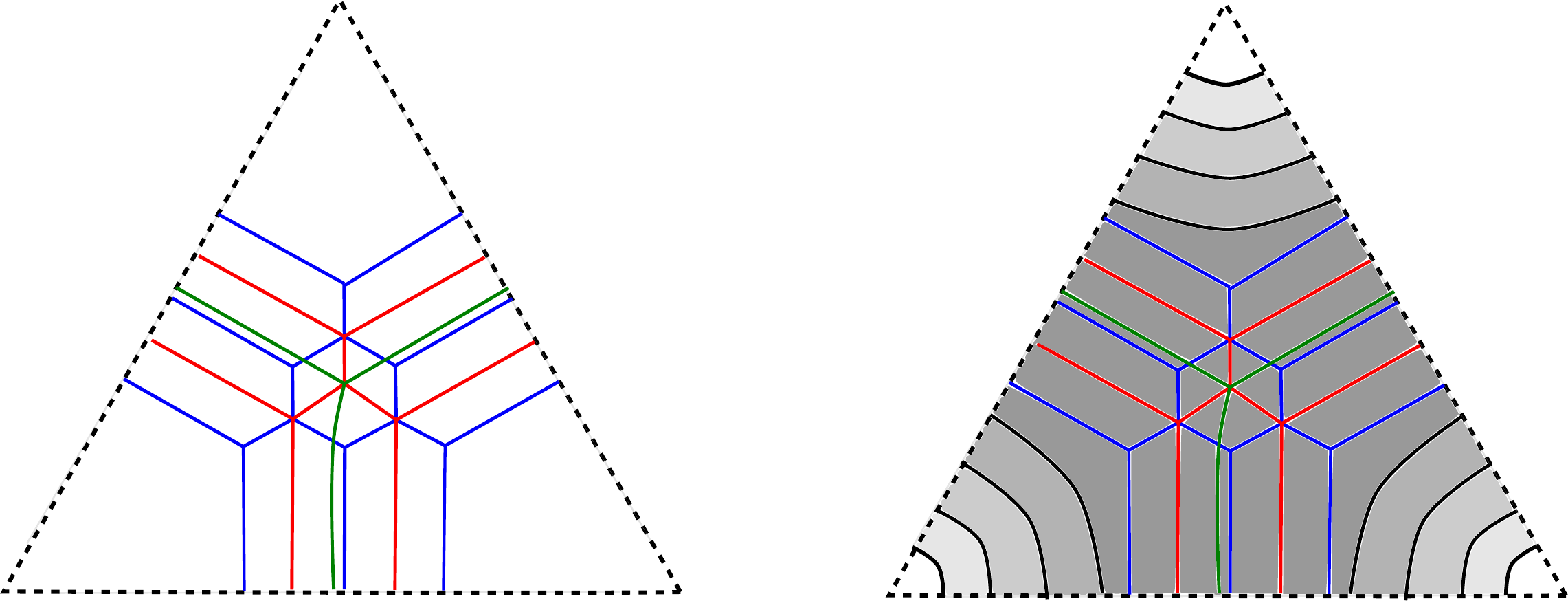}\\
  \caption{The front projection of the Legendrian weave for an $n$-triangulation via the contact embedding $J^1\Sigma_\text{cl} \hookrightarrow \ol{J^1\Sigma_\text{op}}$ such that the boundary lies in $J^1(\partial \Sigma_\text{cl}) \hookrightarrow T^{*,\infty}\Sigma_\text{op}$. Again, different shades of grey colors indicate the different cardinalities of the preimages of the front projection in each stratum.}\label{fig:triangle-weave-cotan}
\end{figure}

\begin{remark}
    For a Legendrian surface $\widetilde{L} \subset \ol{J^1\Sigma_\text{op}}$ with Legendrian boundary $T^{*,\infty}\Sigma_\text{op}$, whose front projection in $\Sigma_\text{op} \times \ol{\bR}$ has only $A_1^2$, $A_1^3$ and $D_4^-$-singularities, we will label the colors of the crossings by the number of sheets of the front below that crossing. For example, we label the crossing by \textcolor{blue}{blue} if there is no sheet below the crossing, by \textcolor{red}{red} if there is one sheet below, by \textcolor{green}{green} if there are two below, by \textcolor{purple}{purple} is there are three below, and so on.
\end{remark}

\subsection{Pinching Sequences}\label{sec:pinching}
    Finally, we recall the concept of pinching sequences of Reeb chords and how they can be used to obtain exact Lagrangian fillings of certain Legendrian knots; see \cite{CasalsNg21} and \cite[Section 6.5]{EHK16} for further details.\\
    
    Let $\Lambda_\pm \subset (\bR^3, \alpha_\text{std})$ be Legendrian links. By definition, an exact Lagrangian cobordism $L$ from $\Lambda_-$ to $\Lambda_+$ is an Lagrangian $L$ properly embedded in the symplectization $(\bR_t \times \bR^3_\text{std}, d(e^t\alpha_\text{std}))$ with $e^t\alpha_\text{std}|_L = df_L$, such that for $T \gg 0$,
    $$L \cap (-\infty, -T) \times \bR^3 = (-\infty, -T) \times \Lambda_-, \,\,\, L \cap (T, +\infty) \times \bR^3 = (T, +\infty) \times \Lambda_+,$$
    and the primitive $f_L$ is constant on $(-\infty, -T) \times \Lambda_-$ and on $(T, +\infty) \times \Lambda_+$. (This latter condition automatically holds if $\Lambda_\pm$ are connected). An exact Lagrangian filling is an exact Lagrangian cobordism from the empty set $\varnothing$ to $\Lambda$. Concatenations of exact Lagrangian cobordisms are still exact Lagrangian cobordisms; this requires the condition of the primitive $f_L$ being constant on either ends. Therefore, we can build Lagrangian fillings by concatenating exact Lagrangian cobordisms.\\
    
    \begin{figure}[h!]
  \centering
  \includegraphics[width=0.8\textwidth]{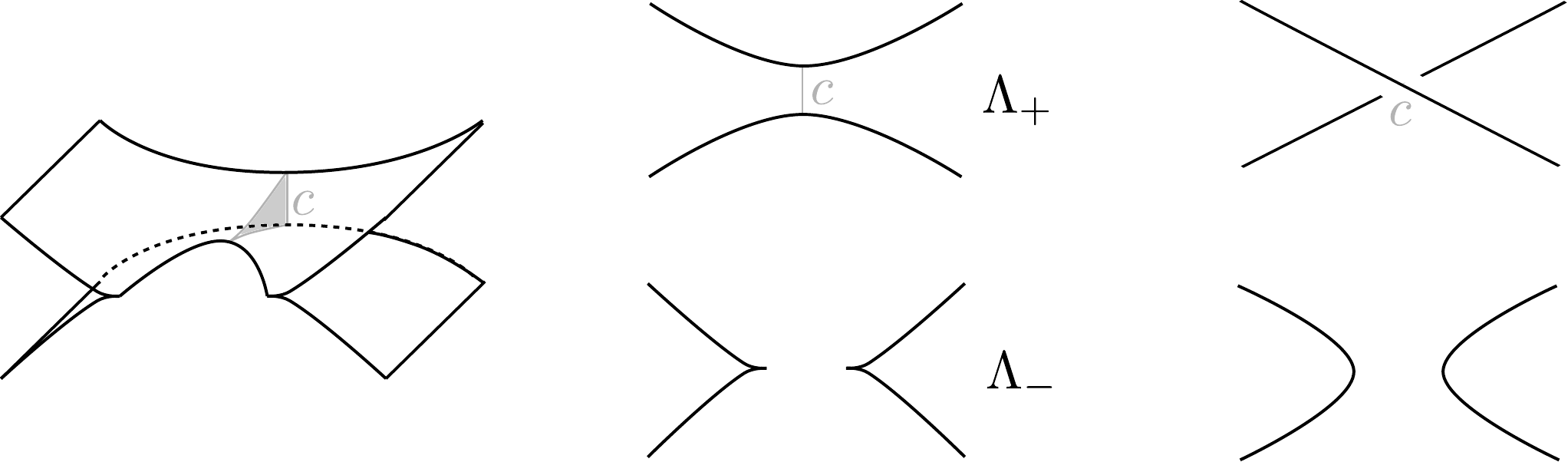}\\
  \caption{A Lagrangian saddle cobordism corresponding to a contractible Reeb chord $c$ on $\Lambda_+$. The picture in the middle is the front projections of $\Lambda_\pm$, while the picture on the right is the Lagrangian projections of $\Lambda_\pm$.}\label{fig:saddle-cob}
\end{figure}

    \noindent Three examples of exact Lagrangian cobordisms between Legendrian links that we use are:
    \begin{enumerate}
      \item the Lagrangian cobordism induced by a Legendrian isotopy \cite{Chan10};
      \item the minimal cobordism, which is the unique Lagrangian 0-handle from $\varnothing$ to the Legendrian unknot with maximal Thurston-Bennequin number \cite{EliPol96};
      \item the saddle cobordism, which is a Lagrangian 1-handle introducing a contractible Reeb chord (i.e.~there is a regular Legendrian homotopy such that the length of the chord converges to 0) as in Figure \ref{fig:saddle-cob}.
    \end{enumerate}
    Concatenations of the three types of cobordisms are called decomposable cobordisms. A pinching sequence is a sequence of minimal cobordisms and saddle cobordisms.

\begin{figure}[h!]
  \centering
  \includegraphics[width=1.0\textwidth]{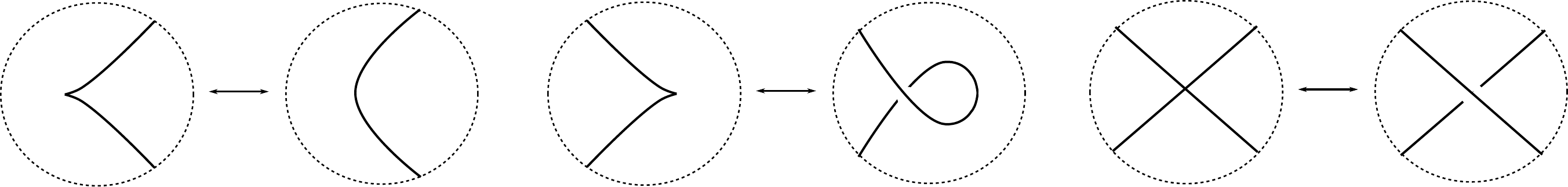}\\
  \caption{Ng's resolution translating the front projection into the Lagrangian projection of a Legendrian link.}\label{fig:Ng-resolution}
\end{figure}

    Since it is much simpler to see (contractible) Reeb chords from the Lagrangian projection than from the front projection, we frequently switch between the Lagrangian projection and front projection of a Legendrian link. For that we use the recipe in \cite{Ng03Reso}*{Proposition 2.2}: given a Legendrian link $\Lambda \subset (\bR^3, \xi_{st})$ with a generic front $\pi_\text{front}(\Lambda)$ is generic, there exists a Legendrian isotopy from $\Lambda$ to $\Lambda'$ such that the Lagrangian projection $\pi_\text{Lag}(\Lambda)$ can be obtained from the front $\pi_\text{front}(\Lambda)$ via the rules of Figure \ref{fig:Ng-resolution}. Note that Ng's resolution does not interchange Reidemeister moves in the front projection and in the Lagrangian projection in general. Therefore, when considering a Lagrangian saddle cobordism resolving a crossing in the Lagrangian projection as in Figure \ref{fig:saddle-cob}, it is possible that in the front projection we need to first apply some Reidemeister moves and then pinch the Reeb chord as in Figure \ref{fig:saddle-cob}. However, in the case of Legendrian positive braid closures, when both the Lagrangian projection and the front projection are identified as satellites of the Legendrian unknot, as in Figure \ref{fig:Lag-front} (right), we can keep track of all the contractible Reeb chords. Namely, a crossing in the Lagrangian projection corresponds to the closed region on the left of the crossing bounded by (the same) two strands in the front projection, where there is a unique Reeb chord between the two strands where they have the same slope. See Figure \ref{fig:Lag-front}, left and center, for a depiction of the Reeb chords and the regions in the front associated to the crossings in the Lagrangian projection in Figure \ref{fig:Lag-front} (right).
    
    \begin{figure}[h!]
  \centering
  \includegraphics[width=1.0\textwidth]{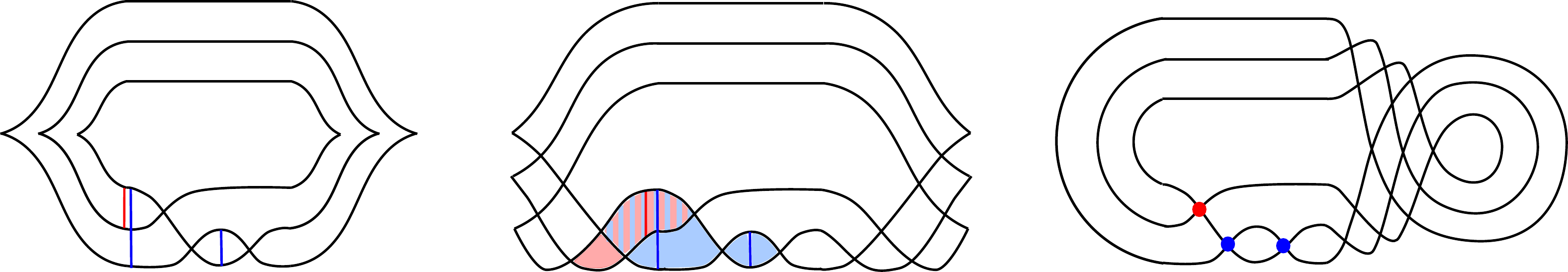}\\
  \caption{The Legendrian positive rainbow braid closures $\Lambda_{\beta_0}^\prec$ (on the left) after Legendrian isotopy to a Legendrian satelitte of the standard unknot (in the middle), and the correspondence between the front and Lagrangian projections (on the right), where each crossing in the Lagrangian projection corresponds to a closed region (on the left of the corresponding crossing) bounded by two strands in the front projection.}\label{fig:Lag-front}
\end{figure}

\begin{example}\label{ex:braid-pinch}
    Let $\beta \in \mbox{Br}_N^+$ be a positive braid and $\Lambda_{\beta}^\prec \subset \bR^3$ the Legendrian rainbow closure of $\beta_0$. Following Ng's resolution, we draw its Lagrangian projection as in Figure \ref{fig:Lag-front}. The crossings coming from cusps in the front are degree~1 Reeb chords while the crossings coming from crossings in the front are degree~0 Reeb chords; see \cite[Section 2.2]{CasalsNg21}. In the Lagrangian projection of $\Lambda_{\beta}^\prec$, all the degree~0 Reeb chords are contractible Reeb chords\footnote{I.e.~there is a regular Legendrian homotopy such that the length of the Reeb chord converges to 0, see \cite{EHK16}*{Lemma 8.1}.} and we can therefore consider concatenations of elementary cobordisms by pinching all the degree~0 Reeb chords in an arbitrary order.\\

    \noindent In general, different pinching sequences of Reeb chords may yield different Lagrangian fillings. For example, for Legendrian $(2, k)$-torus links, i.e.~$\beta = s_1^k$, \cite{EHK16}*{Section 8.1} introduced the following equivalence relation for different pinching sequences. Label the crossings by $1, 2, \ldots, k$ from left to right, and label the sequence by the order of pinching from the positive end $\Lambda_{2,k}$ to the negative end $\varnothing$. Two sequences are isotopy equivalent if they differ by a sequence of transpositions of the following type
    $$\sigma = (i_1 \cdots i_{j-1}\, i_j\, i_{j+1}\, i_{j+2} \cdots i_k), \,\,\, \sigma' = (i_1 \cdots i_{j-1}\, i_{j+1}\, i_j\, i_{j+2} \cdots i_k),$$
    where for $1 \leq i_j, i_{j+1} \leq k$ there exists $i_p\,(p \geq j+2)$ such that either $i_j < i_p < i_{j+1}$ or $i_{j+1} < i_p < i_j$. Conversely, \cite[Theorem 1.1]{Pan17Toruslink} showed that any two pinching sequences that are not isotopy equivalent produce Lagrangian fillings that are not Hamiltonian isotopic. This construction yields $C_k$ different Lagrangian fillings for $\Lambda_{2,k}$, as these sequences (modulo the equivalence above) are in bijection with 132-avoiding permutations; see also \cite{EHK16}*{Lemma 8.3}.\hfill$\Box$
\end{example}

\section{Hamiltonian Isotopies for Hybrid Lagrangians}\label{sec:localmove}

    In this section, we establish the local models of the Hamiltonian isotopies of the Lagrangian fillings needed to prove Theorem \ref{thm:main1}, relating conjugate Lagrangians with Legendrian weaves, and Theorem \ref{thm:braid-weavepinching}, relating Reeb pinching sequences with Legendrian weaves.\\
    
    \noindent By definition, a Lagrangian surface $L\sse T^*\R^2$ is said to be a hybrid Lagrangian surface if it can be in part locally described using the local models for a conjugate surface, and in part locally described using the local models of a (front for) Legendrian weave. A hybrid diagram is a planar superposition of the planar diagrams (locally) describing either conjugate surface or Legendrian weaves. Hybrid Lagrangians will allow us to translate from conjugate Lagrangians to the Lagrangian projection of Legendrian weaves. The main result we prove in this section is the following:
    
    \begin{thm}\label{thm:hybridReidemeister} Let $L\sse T^*\R^2$ be a hybrid Lagrangian surface and $\Pi_L\sse\R^2$ an associated hybrid diagram.  The four exchanges in Table \ref{fig:TableMoves} can be performed to $\Pi_L$ while preserving the Hamiltonian isotopy class of $L$.
    \end{thm}
    
    \begin{figure}[h!]
  \centering
  \includegraphics[width=1.0\textwidth]{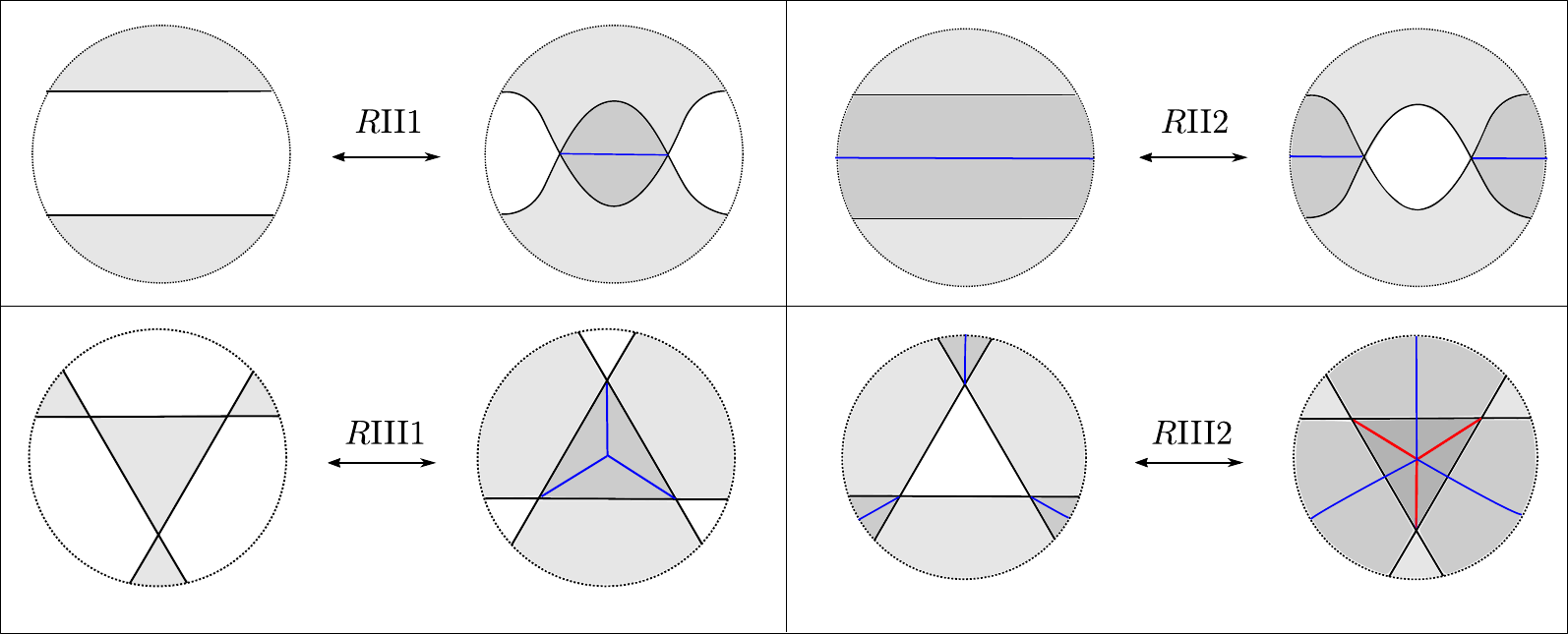}\\
  \caption{Reidemeister-type moves for hybrid Lagrangian surfaces. These moves allow us to transition from a conjugate surface towards a weave.}\label{fig:TableMoves}
\end{figure}

\noindent Theorem \ref{thm:hybridReidemeister} is proven in Subsections \ref{ssec:RIII1}, \ref{ssec:RIII2} and \ref{ssec:RII}. Subsections \ref{ssec:weave_pinching} and \ref{ssec:squaremove_Lagrangianmutation} respectively discuss part of the relation between Legendrian weaves and pinching sequences, and the square move of plabic graphs and Legendrian mutation. In particular, we will explain in detail the meaning of the hybrid diagrams in Figure \ref{fig:TableMoves}. Before delving into the proof of Theorem \ref{thm:hybridReidemeister}, we provide a necessary technical statement on Hamiltonian isotopies in Subsection \ref{ssec:behavior_infinity}.

\subsection{Behaviour at Infinity}\label{ssec:behavior_infinity}
In the Hamiltonian isotopies we construct, we have a smooth family of Lagrangian surfaces $L_t\subset T^*\Sigma$ with Legendrian boundaries $\ol{L}_t \cap T^{*,\infty}\Sigma = \Lambda_t$ and, for any collar neighbourhood $U$ of the boundary $\Lambda_t \subset \ol{T^*\Sigma}$, there is a 1-parametric family of Hamiltonian isotopies $\varphi_{t,s}$, $s\in[0,1]$, supported on $U$ so that $\varphi_{t,1}(L_t) \cap \{(x, \xi)\in T^*\Sigma | |\xi| > r\} = \Lambda_t \times (r,\infty)$, for $r$ sufficiently large. Note that the Hamiltonian isotopy $\varphi_{t,s}$ is fixed at $t = 0, 1$. Therefore, first we need to prove that there is a contractible choice of Hamiltonian isotopies $\varphi_{t,s}$.

\begin{lemma}\label{lem:graph-at-infty}
    Let $L_t \subset T^*\Sigma$, $t\in \bD^k$, be a smooth family of Lagrangian submanifolds such that $\ol{L}_t \cap T^{*,\infty}\Sigma = \Lambda_t$, and for $r_{t,0}$ sufficiently large, $L_t$ lies in a Weinstein neighbourhood of $\Lambda_t \times (r_0,+\infty)$ as a graphical Lagrangian.

    Suppose there are Hamiltonian isotopies $\varphi_{t,s}$, $t \in \partial \bD^k$, as in the proof of Proposition \ref{prop:conj-exist} such that for $r > r_0$ sufficiently large, $\varphi_{t,1}(L_t) \cap \{(x, \xi) | |\xi| > r\} = \Lambda_t \times (r,+\infty)$. Then there exists a extension for the family of Hamiltonian isotopies to $\varphi_{t,s}$, $t\in \bD^k$, such that for $r > r_0$ sufficiently large, $\varphi_{t,1}(L_t) \cap \{(x, \xi) | |\xi| > r\} = \Lambda_t \times (r,+\infty)$.
\end{lemma}
\begin{proof}
    As in the proof of Proposition \ref{prop:conj-exist}, we need to choose a family of sufficiently large $r_{t,0} > 0$, $\varphi_{t,1}(L_t) \cap \{(x, \xi) | |\xi| > r_{t,0}\}$ is contained in a Weinstein tubular neighbourhood $U_t$ of $\Lambda_t \times (r_{t,0},+\infty)$ as a graphical Lagrangian
    $$L_{dF_t} \subset U_t \cong \{(y, \eta) \in T^*(\Lambda_t \times (r_{t,0},+\infty)) \,|\, |\eta| < \epsilon_t\},$$
    and choose a family of cut-off functions $\beta_t: \mathbb{R}_+ \rightarrow [0, 1]$ for $t \in \partial \bD^k$ such that $\beta_t = 1$ when $r \leq r_{t,0}$, $\beta = 1$ when $r$ is sufficiently large, and $\beta_t'(r)$ is sufficiently small comparing to $\epsilon$. Then $L_{dF_t}$ is Hamiltonian isotopic to $L_{d(\beta_t F_t)}$, and for some $r_t \gg r_{t,0}$,
    $$L_{d(\beta_t F_t)} \cap \{(x, \xi) | |\xi| > r_t\} = \Lambda \times (r_t, +\infty).$$

    Note that since $\bD^k$ is compact, one can assume that there is a uniform $r_0$ such that $\varphi_{t,1}(L_t) \cap \{(x, \xi) | |\xi| > r_0\}$ is contained in a Weinstein tubular neighbourhood $U_t$ of $\Lambda_t \times (r_0,+\infty)$ as a graphical Lagrangian. We can also assume that there is a uniform $\epsilon > 0$ such that the neighbourhoods of $\Lambda_t \times (r_0,+\infty)$ are of the form
    $$U_t \cong \{(y, \eta) \in T^*(\Lambda_t \times (r_0,+\infty)) \,|\, |\eta| < \epsilon\}.$$
    Then we need to extend the family of cut-off functions $\beta_t: \mathbb{R}_+ \rightarrow [0, 1]$ from $\partial \bD^k$ to $t \in \bD^k$ such that $\beta_t = 1$ when $r \leq r_0$, $\beta = 1$ when $r$ is sufficiently large, and $\beta_t'(r)$ is sufficiently small comparing to $\epsilon$. However, note that the space of such cut-off functions is convex, and hence contractible. Therefore, the family of Hamiltonian isotopies for $t \in \partial \bD^k$ can be extended to $t \in \bD^k$.
\end{proof}

\subsection{Local Move~1 for Reidemeister III}\label{ssec:RIII1}
    For an alternating Legendrian and its conjugate Lagrangian filling, associated to a bipartite graph, we first consider a Reidemeister~III move near a trivalent black vertex. Figure \ref{fig:localmove1} depicts such a move, denoted by RIII1. This type of Reidemeister~III moves on a conjugate Lagrangian give rise to the Lagrangian projection of a $D_4^-$-singularity, which allows us to move a step closer to the Lagrangian projection of a Legendrian weave.
    
    Let $\Lambda_0$ and $\La_1$ be the alternating Legendrians depicted in Figure \ref{fig:localmove1}, left and right respectively. Then $\Lambda_0$ and $\Lambda_1$ are connected by a Legendrian Reidemeister~III Move and are thus Legendrian isotopic. The Legendrian $\La_0$, being local alternating, admits a conjugate Lagrangian surface. Figure \ref{fig:localmove1} depicts a hybrid Lagrangian surface which bounds $\La_1$. We now show that these two Lagrangian surfaces are Hamiltonian isotopic:

\begin{figure}[h!]
  \centering
  \includegraphics[width=0.7\textwidth]{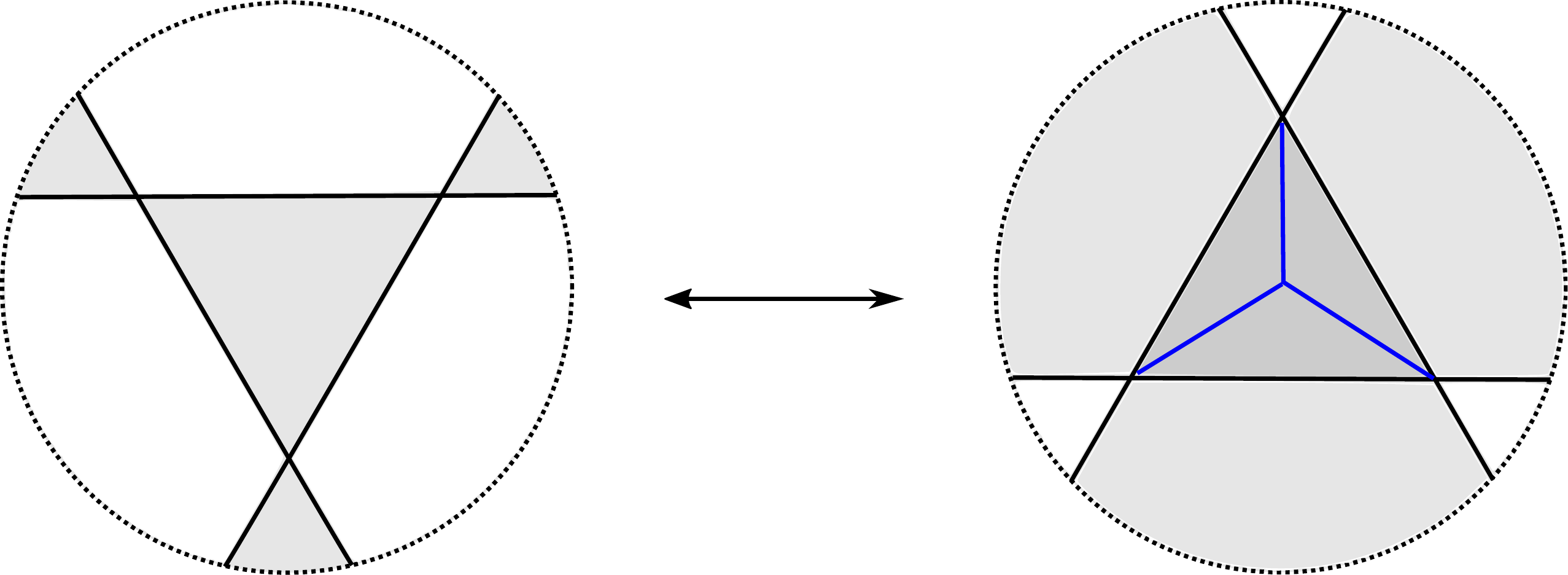}\\
  \caption{The local move~1 of Lagrangian fillings for a Legendrian Reidemeister~III move. The covectors of the Legendrian on the left are pointing towards the grey triangular region in the middle.}\label{fig:localmove1}
\end{figure}

\begin{prop}\label{lem:localmove1}
    Let $\Lambda_0$ and $\La_1$ be the alternating Legendrians depicted in Figure \ref{fig:localmove1}, left and right respectively, and $L_0$ the conjugate Lagrangian filling associated to $\La_0$. Then there exists a hybrid Lagrangian surface $L_1$ which is Hamiltonian isotopic to $L_0$ and satisfies that:
    
    \begin{itemize}
        \item[(i)] The projection to the base $\pi|_{L_1}: L_1 \rightarrow \pi(L_1)$ is a 2-fold branched covering inside the dark grey triangle, with branch set a point -- the trivalent vertex in Figure \ref{fig:localmove1} (right) -- and injective in the three unbounded light grey regions.\\
        
        \item[(ii)] The front projection of the Legendrian lift of $L_1$ is, when restricting to the dark grey triangle, a $D_4^-$-singularity as in Figure \ref{fig:localmove1} (right).
        \end{itemize} 
\end{prop}
\begin{proof}
    Consider the Legendrian link $\Lambda_{1/2}$ with a triple point in its front projection, which is the time slice in the middle of the Reidemeister~III move from $\Lambda_0$ to $\Lambda_1$. We first construct a Lagrangian filling $L_{1/2}$ of $\Lambda_{1/2}$, and then show that $L_{1/2}$ is Hamiltonian isotopic to both $L_0$ and $L_1$. Define the Lagrangian embedding $\Phi_{L_{1/2}}:\R^2\lr\R^2$ of $\mathbb{R}^2$ with polar coordinates $(r,\theta)\in\R^2$, $r\in\R^{\geq0}, \theta\in[0,2\pi)$, in the domain $\bR^2$, to be
    $$\Phi_{L_{1/2}}: (r, \theta) \mapsto \Big(r^2\sin \frac{3\theta}{2} \cos \frac{\theta}{2}, r^2\sin \frac{3\theta}{2} \sin \frac{\theta}{2}, -r\sin \theta, -r\cos \theta \Big).$$
    This is an exact Lagrangian and the primitive of (the image of) $\Phi_{L_{1/2}}$ is
    $$f_{L_{1/2}}(r, \theta) = -\frac{2}{3}r^3\sin^2 \frac{3\theta}{2}.$$
    Alternatively, using polar coordinates $(\rho, \varphi)$ on the base $\mathbb{R}^2$ of $T^*\mathbb{R}^2$, we can write in Cartesian coordinates
    $$L_{1/2} = \left\{\Big(\rho\cos \varphi, \rho\sin \varphi, -\rho^{1/2}(\sin 3\varphi)^{-1/2}\sin 2\varphi, -\rho^{1/2}(\sin 3\varphi)^{-1/2}\cos 2\varphi \Big) \Big| \rho \geq 0\right\}.$$
    Notice that the coordinate change is given by $\rho = r^2\sin ({3\theta}/{2}), \,\,\, \varphi = {\theta}/{2}$. In this proof, we use the polar coordinates $(\rho,\varphi)$ for the zero-section of $\R^2$ and polar coordinates $(r,\theta)$ for the cotangent fibers of $T^*\R^2$.\\

\begin{figure}[h!]
  \centering
  \includegraphics[width=0.8\textwidth]{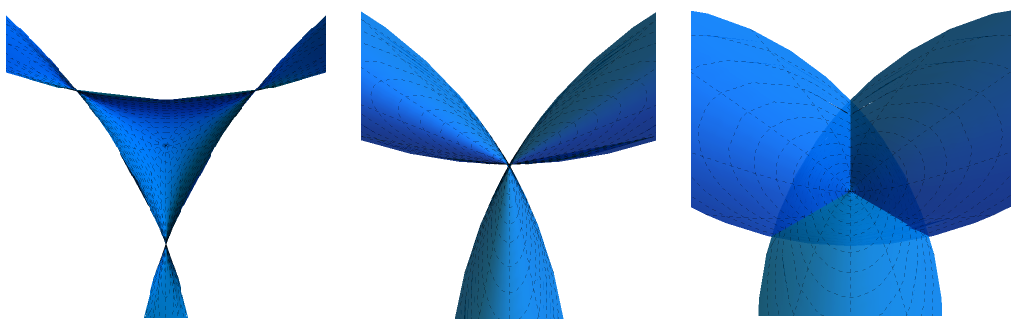}\\
  \caption{The projection of the Lagrangian fillings $L_0$, on the left, $L_{1/2}$, at the center, and $L_1$, on the right, onto the base $\bR^2$ and the level sets $L_t^r = L_t \cap T^{*,r}\mathbb{R}^2$ in dashed lines.}\label{fig:localmove1_levelset}
\end{figure}

    Consider Hamiltonians $H:T^*\R^2\lr\R$ of the form $H = H(r)$ on $T^*\mathbb{R}^2$ so that the level sets are exactly the radius-$r$ (co)circle bundles $T^{*,r}\mathbb{R}^2$ and $H'(r) < 0$. The projections onto the zero section of the image of $L_{1/2}$ under the Hamiltonian flows of these Hamiltonians are
    $$\pi(L_{1/2+t}) = \left\{\Big(r^2\sin \frac{3\theta}{2}\cos \frac{\theta}{2}, r^2\sin \frac{3\theta}{2} \sin\frac{\theta}{2}\Big) - tH'(r)(\sin \theta, \cos \theta)\right\}.$$
    Now we construct two explicit Hamiltonians $H_\pm$ of this form, so that:
    
    \begin{itemize}
        \item[(1)] The image $L_0 = \varphi_{H_-}^{-1/2}(L_{1/2})$ of $L_{1/2}$ under the  $(-1/2)$-time flow of $H_-$ is a conjugate Lagrangian as in Figure \ref{fig:localmove1} (left).
        
        \item[(2)] The image $L_1 = \varphi_{H_+}^{1/2}(L_{1/2})$ of $L_{1/2}$ under the  $1/2$-time flow of $H_+$ is the Lagrangian projection of a $D_4^-$-Legendrian front, as in the dark grey region of Figure \ref{fig:localmove1} (right).
    \end{itemize}
     
    \textit{Part 1.~Construction of $H_-$.} We construct the Hamiltonian $H_-$ so that $L_0 = \varphi_{H_-}^{-1/2}(L_{1/2})$ is a conjugate Lagrangian. First, consider a Hamiltonian $H_-:T^*\R^2\lr\R$, $H_-(\rho,\varphi;r,\theta) = H_-(r)$ and we will find constraints so as to determine it exactly. Under the $(-1/2)$-time Hamiltonian flow, the Lagrangian $L_{1/2}$ becomes
    $$L_0= \left\{\Big(r^2\sin \frac{3\theta}{2}\cos \frac{\theta}{2} + \frac{1}{2}H_-'(r)\sin \theta, r^2\sin \frac{3\theta}{2} \sin\frac{\theta}{2} + \frac{1}{2}H_-'(r)\cos \theta, r\cos\theta, r\sin\theta\Big) \right\}.$$
    For $L_0$ to be the required conjugate Lagrangian, we need to show that
    
    \begin{enumerate}
        \item the projection of $L_0$ onto $\mathbb{R}^2$ is contained in the light grey regions;
        \item the projection of $L_0$ onto $\bR^2$ is injective in the light grey region.
    \end{enumerate} 

\begin{figure}[h!]
  \centering
  \includegraphics[width=1.0\textwidth]{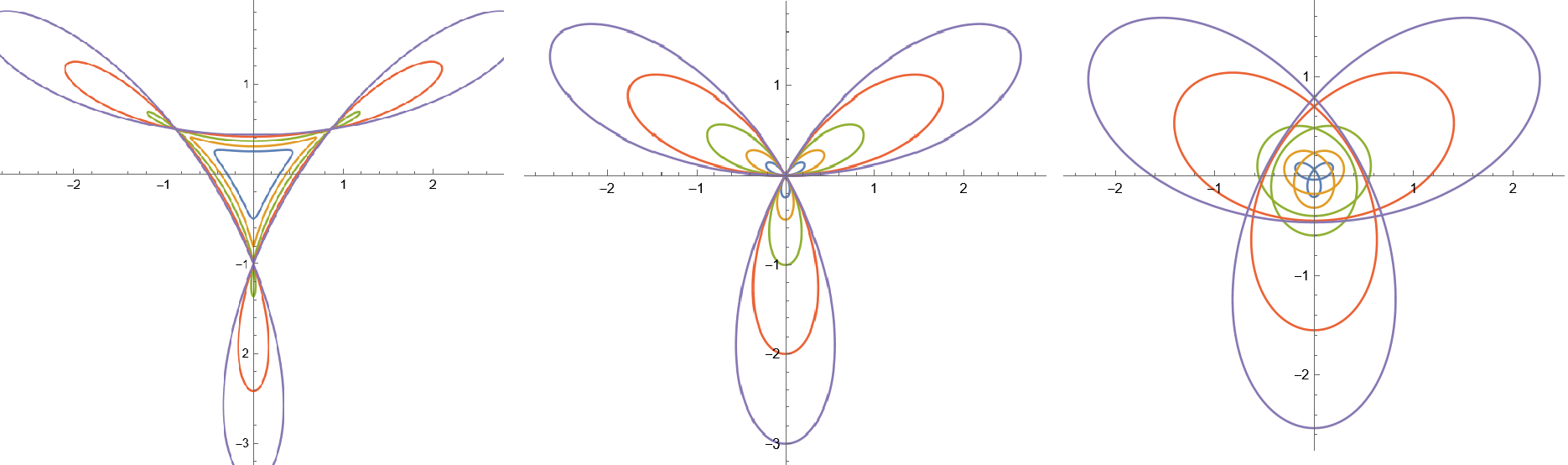}\\
  \caption{The projection of level curves of the Lagrangian fillings $L_{t}^r = L_{t} \cap T^{*,r}\bR^2$ onto $\bR^2$ for $t = 0$ (left), $t = 1/2$ (middle) and $t = 1$ (right), and $r^2 = 1/4, 1/2, 1, 2$ and $3$. Here $L_0 = \varphi_{H_-}^{1/2}(L_{1/2})$ and $L_1 = \varphi_{H_+}^{1/2}(L_{1/2})$.}\label{fig:levelset}
\end{figure}

\noindent Consider the intersection $L_0^r := L_0 \cap T^{*,r}\bR^2$ of the Lagrangian surface $L_0$ with the radius-$r$ cocircle bundle of $T^*\R^2$. Then the projection $\pi(L_0^r)$ is an immersed curve with 3 crossings on the line $\varphi = \pi/6, 5\pi/6$ and $3\pi/2$. The coordinates $(r, \theta(r))$ of the crossings at $\varphi = 3\pi/2$ satisfy
    $$r^2\sin\frac{3\theta}{2}\cos \frac{\theta}{2} + \frac{1}{2}H_-'(r)\sin \theta = \frac{1}{2}r^2\sin 2\theta + \frac{1}{2}(r^2 + H_-'(r))\sin \theta = 0.$$
    $$\cos\theta(r) = -\frac{1}{2}\Big(1 + \frac{H_-'(r)}{r^2}\Big).$$
In order to get a conjugate Lagrangian, this crossing of $\pi(L_0^r) \subset \bR^2$ should coincide with the crossing of the Legendrian $\pi(\Lambda_0)$. The Cartesian coordinates of the crossing of $\pi(\Lambda_0)$ are $(0, -1)$ and thus we require
    $$r^2\sin\frac{3\theta}{2}\sin \frac{\theta}{2} + \frac{1}{2}H_-'(r)\cos \theta = \frac{1}{2}r^2\cos 2\theta + \frac{1}{2}(r^2 + H_-'(r))\cos \theta = -1.$$
    By using the formula of $\cos\theta(r)$, we obtain the expression
    $$H_-'(r) = r^2\big(-1 + \sqrt{1 + {2}/r^2}\big)$$
    for the derivative of $H_-(r)$, which uniquely determines the required Hamiltonian $H(r)$, up to a constant. It is readily verified that ~$H_-'(r) > 0$ for $r \in (0, +\infty)$ and also that the limit $\lim_{r \rightarrow \infty}H_-'(r) = 1$ is a finite number.\footnote{Intuitively, $H'(r)$ is the speed of the flow on the cocircle level set $r$, which determines by how much the Legendrians are moved.} For this particular Hamiltonian $H_-(r)$, we can see explicitly from Figure \ref{fig:localmove1_levelset} that $L_0$ is a conjugate Lagrangian satisfying condition~(1) \& (2).\\

    \textit{Part 2.~Construction of $H_+$ such that $\varphi_{H_+}^{1/2}(L_{1/2})$ is the Lagrangian projection of $D_4^-$.} We use the same polar coordinates $(r,\theta)$ in the cotangent fibers of $T^*\R^2$ and $(\rho,\varphi)$ in the base. We want to construct a Hamiltonian $H_+:T^*\R^2\lr\R$ so that, following Figure \ref{fig:localmove1} (right), $L_1 := \varphi_{H_+}^{1/2}(L_{1/2})$ is the Lagrangian projection of a $D_4^-$-Legendrian front in the dark grey region, and an injective projection in the light grey region. First, note that for any such Hamiltonian the primitive of $L_1$ is
    $$f_{L_1}(r, \theta) = f_{L_{1/2}}(r, \theta) + G_+(r) = -\frac{2}{3}r^3\sin^2 \frac{3\theta}{2} + G_+(r), \,\,\, \mbox{where }dG_+(r) = \frac{1}{2}rdH_+'(r).$$
    In Cartesian coordinates $(\rho\cos\varphi,\rho \sin \varphi)$ on the base of $T^*\R^2$, the projection $\pi|_{L_1}: L_1 \rightarrow \bR^2$ is given by
    $$(\rho \cos\varphi, \rho \sin \varphi) = \Big(r^2\sin \frac{3\theta}{2}\cos \frac{\theta}{2} - \frac{1}{2}H_+'(r)\sin \theta, r^2\sin \frac{3\theta}{2}\sin \frac{\theta}{2} - \frac{1}{2}H_+'(r)\cos \theta\Big).$$
    The exact Lagrangian $L_1\subset T^*\mathbb{R}^2$ can be lifted to a Legendrian surface $\widetilde{L}_1\subset J^1\mathbb{R}^2$ using the primitive $f_{L_1}:L_1\lr\R$. For the Legendrian lift $\widetilde{L}_1$ to a weave with a $D_4^-$-front, we need to show that
    
    \begin{enumerate}
        \item the front projection of $\widetilde{L}_1$ near $\rho = 0$ agrees with the standard $D_4^-$-singularity;
        \item the crossings in the front projection $\widetilde{L}_1$ are the three rays $\varphi = \pi/2, 7\pi/6$ or $11\pi/6$;
        \item the projection onto $\bR^2$ in the light grey regions is injective.
    \end{enumerate} 

\begin{figure}[h!]
  \centering
  \includegraphics[width=0.7\textwidth]{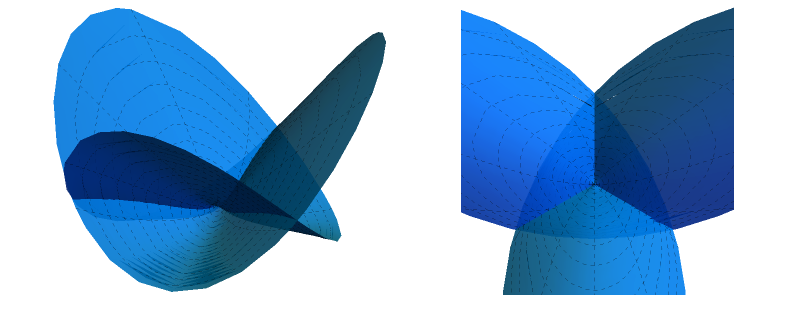}\\
  \caption{The front projection of the Legendrian lift $\widetilde{L}_1$ when we consider the Hamiltonian $H(r) = \arctan(\pi r^2) / \pi$.}\label{fig:Localmove1_weave}
\end{figure}

    For $\rho \geq 0$ sufficiently small, or equivalently $r \geq 0$ sufficiently small, we define the Hamiltonian $H_+(r)$ by the condition $H_+'(r) = r^2$. Then at time $t = 1/2$, the front projection of $\widetilde{L}_1$ is
    \[\begin{split}
    \Big(r^2\sin \frac{3\theta}{2}&\cos \frac{\theta}{2} + \frac{1}{2}H_+'(r)\sin \theta, r^2\sin \frac{3\theta}{2}\sin \frac{\theta}{2} + \frac{1}{2}H_+'(r)\cos \theta, -\frac{2}{3}r^3\sin^2 \frac{3\theta}{2} + G_+(r)\Big) \\
    =& \Big(r^2\sin \frac{3\theta}{2}\cos \frac{\theta}{2} - \frac{1}{2}r^2\sin \theta, r^2\sin \frac{3\theta}{2}\sin \frac{\theta}{2} - \frac{1}{2}r^2\cos \theta, -\frac{2}{3}r^3\sin^2 \frac{3\theta}{2} + \frac{1}{3}r^3\Big) \\
    =& \Big(\frac{1}{2}r^2 \sin 2\theta, -\frac{1}{2}r^2 \cos 2\theta, \frac{1}{3}r^3 \cos 3\theta \Big).
    \end{split}\]
    This coincides with the standard local model of the front projection of the $D_4^-$-Legendrian weave in Definition \ref{def:weave} given by
    $$\Big(\mathrm{Re}(w^2), \mathrm{Im}(w^2), \frac{2}{3}\mathrm{Im}(w^3)\Big)$$
    only up to a scalar multiple. Thus condition~(1) is satisfied.

    In general, $H_+(r)$ is chosen to be such that
    \begin{itemize}
        \item[(i)] $H_+'(r) = r^2$ when $r \ll 1$, as explained above;
        \item[(ii)] $H_+'(r) > 0$ for $r \in (0, +\infty)$;
        \item[(iii)]~$\lim_{r \rightarrow \infty}H_+'(r) = 1$.
    \end{itemize}
    \noindent Then the only double points in the front projection $\widetilde{L}_1$ are the three rays $\varphi = \pi/2, 7\pi/6$ or $11\pi/6$, thus satisfying condition (2) above. For instance, we can choose a Hamiltonian given by the condition
    $$H_+'(r) = \frac{2}{\pi} \arctan(\pi r^2), \,\,\, r \geq 0.$$
    With this choice, one can then see explicitly from Figure \ref{fig:Localmove1_weave} that conditions~(2) \& (3) are satisfied, thus concluding the proof.\footnote{The reader may wonder why not take $H_+'(r) = H_-'(r) = r^2(-1 + \sqrt{1 + {2}/r^2})$. This is because the $D_4^-$-singularity is non-generic: taking $H_+'(r) = r^2(-1 + \sqrt{1 + {2}/r^2})$ yields a generic perturbation of the $D_4^-$-singularity consisting of three $A_3$-swallowtails.}
\end{proof}

\noindent Proposition \ref{lem:localmove1} is stated only for the Reidemeister~III move at a black vertex, but the same argument shows that Reidemeister~III moves at white vertices work equally. Note also that the intermediate local model $L_{1/2}$ constructed in the proof of Proposition \ref{lem:localmove1} realizes Lagrangian fillings of the Legendrian links for Thurston's triple point diagrams \cite{Thurs17}.\\

We can also keep track of the 1-cycles in the Lagrangian fillings in the Reidemeister III1 move in Figure \ref{fig:Localmove1_weave}. The following corollary follows directly from Figure \ref{fig:levelset} or Figure \ref{fig:localmove1_levelset}.

\begin{cor}\label{cor:1cycle-loc}
    Consider the Hamiltonian isotopy of Lagrangian fillings in (the proof of) Proposition \ref{lem:localmove1}. Then the relative 1-cycle $[\gamma_0]\in H_1(L_0, L_0 \cap \partial T^*\bD^2)$ of the conjugate Lagrangian, as drawn in Figure \ref{fig:localmove-cycle} (left), is smoothly isotopic to the relative 1-cycle $[\gamma_1]\in H_1(L_1, L_1 \cap \partial T^*\bD^2)$ of the hybrid Lagrangian surface $L_1$ in Figure \ref{fig:localmove-cycle} (right).
\end{cor}

\begin{figure}[h!]
  \centering
  \includegraphics[width=0.7\textwidth]{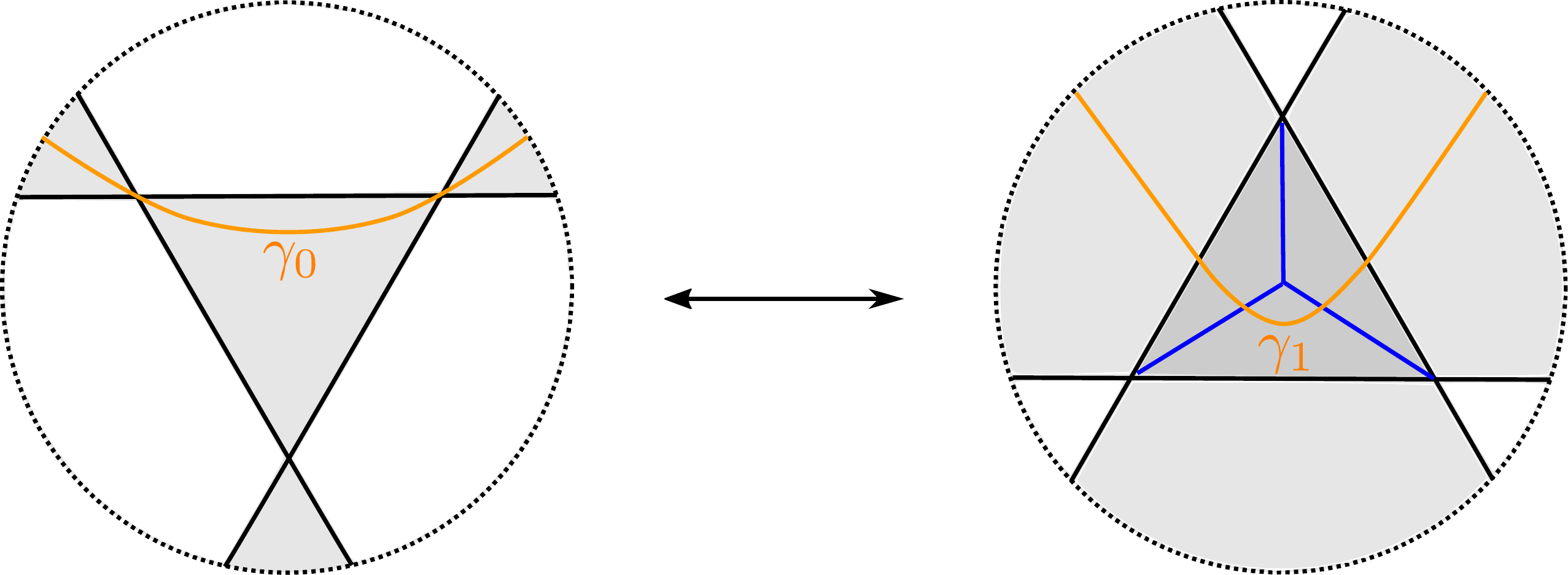}\\
  \caption{The relative 1-cycle $\gamma_0 \in H_1(L_0, L_0 \cap \partial T^*\bD^2)$ and the corresponding relative 1-cycle $\gamma_1 \in H_1(L_1, L_1 \cap \partial T^*\bD^2)$. Note that the 1-cycle $\gamma_1$ on the right starts at the unique sheet (above the light grey regions) and enters through the bottom sheet (of the two) in the dark grey regions; it then moves up to the upper sheet in the dark grey region only when it crosses the blue edges (thus when it is below the trivalent vertex).}\label{fig:localmove-cycle}
\end{figure}

\subsection{Local Move~2 for Reidemeister III}\label{ssec:RIII2}
The second local move that is required to interpolate between conjugate Lagrangians and the Lagrangian projection of Legendrian weaves is the Reidemeister III2 move, as depicted in Figure \ref{fig:TableMoves} and also Figure \ref{fig:localmove2}.

\begin{figure}[h!]
  \centering
  \includegraphics[width=0.7\textwidth]{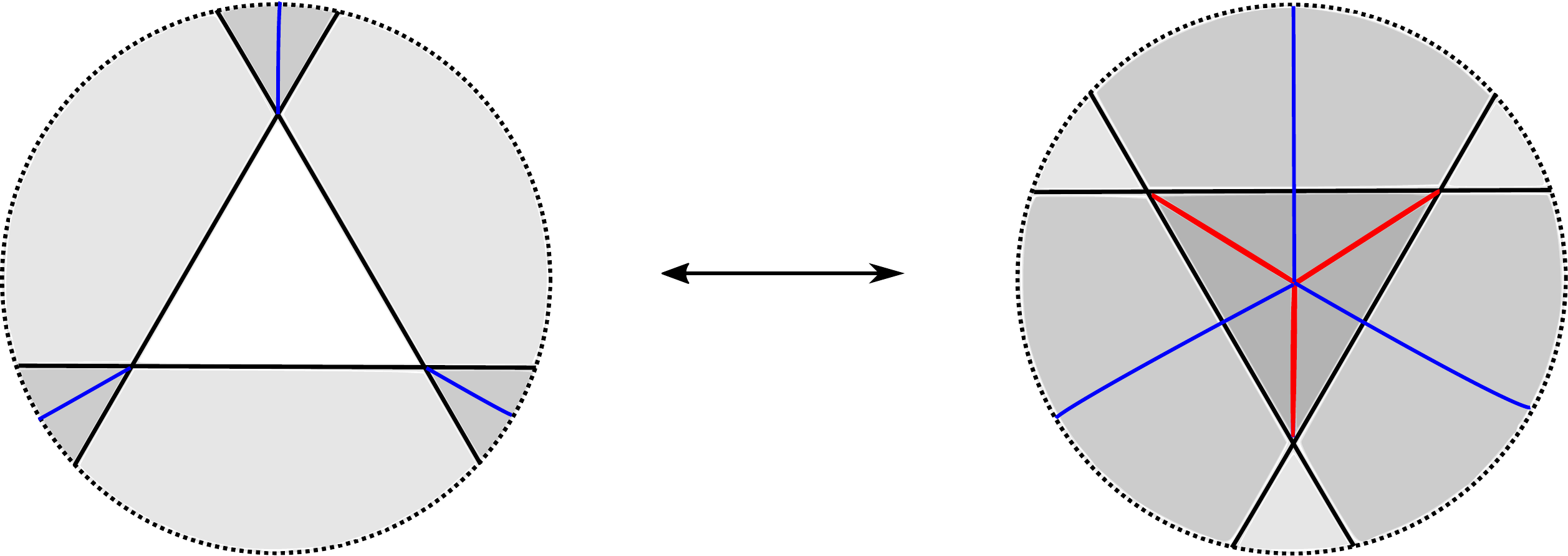}\\
  \caption{The second local move RIII2 of Lagrangian fillings for a Legendrian Reidemeister III move. The covectors of the Legendrian on the left are pointing towards the white region in the middle.}\label{fig:localmove2}
\end{figure}

\begin{prop}\label{lem:localmove2}
    Let $\Lambda_0$ and $\La_1$ be a Legendrian tangles depicted in Figure \ref{fig:localmove2}, left and right respectively. Let $L_0$ be the hybrid Lagrangian surface in Figure \ref{fig:localmove2} (left), whose projection onto the plane is injective in the light grey region, and a 2-fold branched covering in the three dark grey regions, where the front projection of the Legendrian lift has $A_1^2$-singularities along the blue edges.\\
    
    \noindent Then $L_0$ is Hamiltonian isotopic to the hybrid Lagrangian surface $L_1$ depicted in Figure \ref{fig:localmove2} (right), which satisfies that:
    
    \begin{itemize}
        \item[(i)] the projection to the base $\pi|_{L_1}: L_1 \rightarrow \pi(L_1)$ is an injection in the light grey region;
        
        \item[(ii)] the projection is a 2-fold covering inside the dark grey triangle and a 3-fold covering in the triangle in the middle;
        
        \item[(iii)] the front projection of Legendrian lift of $L_1$ is a $A_1^3$-singularity at the hexavalent vertex when restricted to the center triangle.
    \end{itemize}
\end{prop}

\begin{proof}
    Let $\cR_\theta:\R^2\lr\R^2$ be the rotation by angle $\theta$ centered at the origin of the plane $\mathbb{R}^2$. Let us parametrize the front projection of the Legendrian $\Lambda_t$ by the union of
    $$\Lambda_{t,j} := \left\{\cR_{2\pi j/3}\Big(x, t - \frac{1}{2}\Big) \Big| x\in \mathbb{R}\right\} \subset \mathbb{R}^2,\,\,0\leq j\leq 2.$$
    Let $d\cR_{\theta}:T^*\R^2\lr T^*\R^2$ be the differential of the rotation $\cR_\theta$ and consider the exact Lagrangian $L_t$ to be the union of
    $$L_{t,j} := \left\{d\cR_{2\pi j/3}\Big(x, t - w - \frac{1}{2}, 0, -\frac{1}{w}\Big) \Big| x\in \mathbb{R}, w \in (0,+\infty) \right\}, \,\, 0\leq j\leq 2,$$
    For each $t$, the Lagrangian $L_t$ is exact and the primitive of $L_{t,j}$ is $f_t(x, w) = \ln w$. Therefore, their corresponding Legendrian lifts in the 1-jet bundle are defined by
    $$\widetilde{L}_{t,j} = \left\{j^1{\cR}_{2\pi j/3}\Big(x, t - w - \frac{1}{2}, 0, -\frac{1}{w}, \ln w\Big) \Big| x\in \mathbb{R}, w \in (0,+\infty) \right\}, \,\, 0\leq j\leq 2.$$
    This defines the required Legendrian isotopy from $L_0$ to $L_1$, as the Legendrian lifts satisfy the conditions in the statement and no Reeb chords are created during this process either. Indeed, one can check that there is a unique hexavalent vertex at the center of $\widetilde{L}_1$ since the fronts of the three Legendrian surfaces $\widetilde{L}_{1,j}$ for $0 \leq j \leq 2$ intersect transversely.
    
    Finally, note that the exact Lagrangian isotopy can be induced by the Hamiltonian flow of $H(r)$, where $r$ is the radius coordinate of the fiber in $T^*\mathbb{R}^2$, such that $H'(r) \geq 0$, $H'(r) = 0$ when $r$ is sufficiently small, and $H'(r) = 1$ when $r$ is large. Therefore the proof is completed.
\end{proof}
\begin{remark}\label{rem:moresheats-localmove2}
In general, following the terminology in Section \ref{sec:weave-at-infty}, when there are already $i$ sheets below the Legendrian front, then we can apply the Reidemeister local move RIII2 to the $s_i$-edges and get a hexagonal vertex emanating $s_i$ and $s_{i+1}$-edges in the weave. 
\end{remark}

\subsection{Local Moves for Reidemeister II}\label{ssec:RII}
Let us prove the Reidemeister RII1 and RII2 moves in Figure \ref{fig:TableMoves} describing how certain Lagrangian fillings are isotoped under Reidemeister II moves applied to their Legendrian boundaries.

\begin{figure}[h!]
  \centering
  \includegraphics[width=0.7\textwidth]{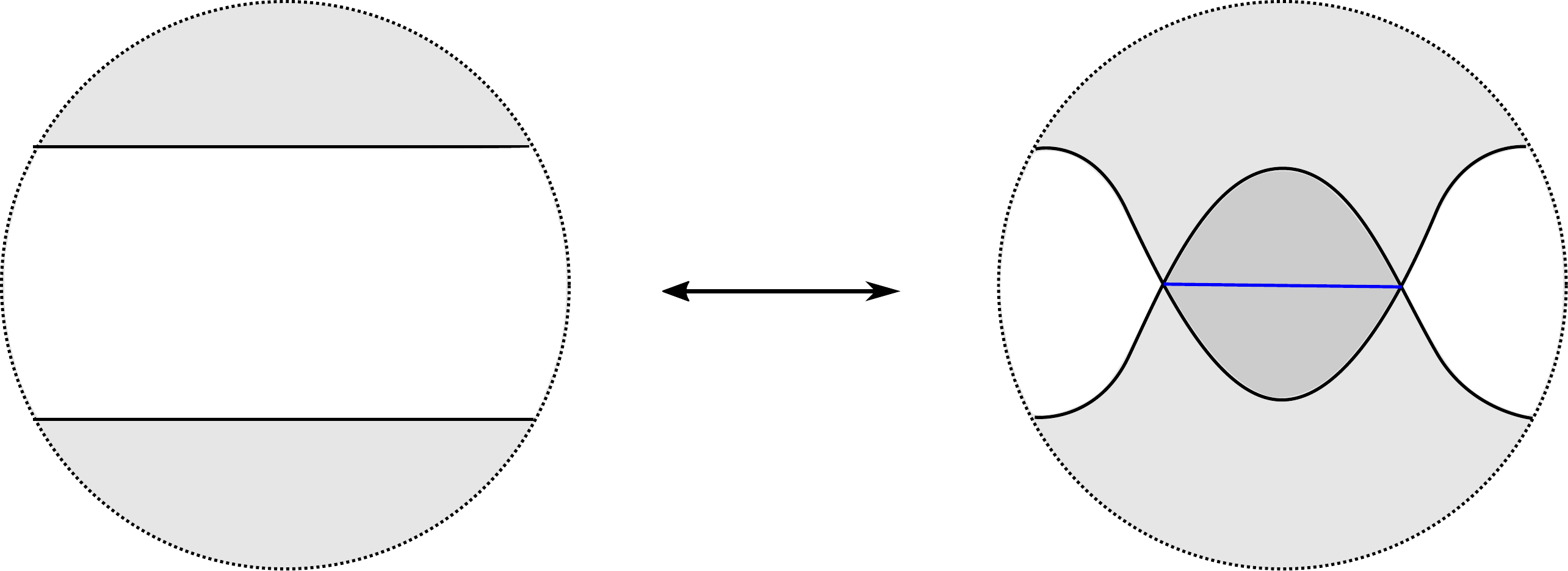}\\
  \caption{The local move~1 of Lagrangian fillings for a Legendrian Reidemeister II move. The covectors of the Legendrian are pointing towards the white region.}\label{fig:RII-move1}
\end{figure}

\begin{lemma}\label{lem:RII-move1}
    Let $\Lambda_0$ and $\La_1$ be the Legendrian tangles depicted in Figure \ref{fig:RII-move1}, left and right respectively. Let $L_0$ be the conjugate Lagrangian filling for $\La_0$ in Figure \ref{fig:RII-move1} (left), whose projection onto the plane is injective in the light grey region.\\

    \noindent Then $L_0$ is Hamiltonian isotopic to the hybrid Lagrangian surface  $L_1$ depicted in Figure \ref{fig:RII-move1} (right), which satisfies that:
    
    \begin{itemize}
        \item[(i)] the projection onto the plane is injective in the light grey region and a double covering in the dark grey region;
        
        \item[(ii)] the front projection of the Legendrian lift has a $A_1^2$-singularity in the bi-gon as in Figure \ref{fig:RII-move1} (right).
    \end{itemize}
\end{lemma}

\begin{proof}
    Let $\beta: [0,+\infty) \rightarrow [0,1]$ be a cut-off function such that $\beta(x) = 1$ when $x$ is sufficiently small, $\beta(x) = 0$ when $x$ is sufficiently large, and $\beta'(x) \leq 0$. Define $f_t(x) = -1 + 2t\beta(|x|)$. Suppose that the front projection of the Legendrian link $\Lambda_t$ is defined by the graph $\{(x, \pm f_t(x) | x \in \mathbb{R}^2\}$. Let $(x, w)$ be coordinates in the domain $\bR \times (0, +\infty)$ and consider the family of exact Lagrangians of two components
    $$L_{t,\pm} = \left\{\Big(x, \pm(w - f_t(x)), \pm\frac{f'_t(x)}{w}, \frac{1}{w}\Big)\Big|x \in \mathbb{R}, w \in (0,+\infty)\right\}.$$
    Since the primitive of the exact Lagrangians are $f(x, w) = \ln w$, their corresponding Legendrian lifts are thus defined by
    $$\widetilde{L}_{t,\pm} = \left\{\Big(x, \pm(w - f_t(x)), \pm\frac{f'_t(x)}{w}, \frac{1}{w}, \ln w \Big)\Big|x \in \mathbb{R}, w \in (0,+\infty)\right\}.$$
    It is clear that this defines an exact Lagrangian isotopy from $L_0$ to $L_1$ whose Legendrian lifts satisfy the conditions in the lemma (indeed, $L_1$ has a family of $A_1^2$-singularities in the bigon because the fronts of the two components $\widetilde{L}_{1,\pm}$ intersect transversely in the region). Finally, note that the exact Lagrangian isotopy can be induced by the Hamiltonian flow of ($(x, y)$ are the coordinates of $\mathbb{R}^2$ and $(\xi, \eta)$ the coordinates of the cotangent fiber)
    $$H(x, y, \xi, \eta) = 2\beta(|x|)\eta.$$
    Hence the proof is completed.
\end{proof}

\begin{lemma}\label{lem:RII-move2}
    Let $\Lambda_0$ and $\La_1$ be the Legendrian tangles depicted in Figure \ref{fig:RII-move2}, left and right respectively. Consider the hybrid Lagrangian surface $L_0$ in Figure \ref{fig:RII-move2} (left), whose projection onto the plane is injective in the light grey region and a 2-fold covering in the dark grey region, where the front projection of the Legendrian lift has $A_1^2$-singularities.\\

    \noindent Then $L_0$ is Hamiltonian isotopic to the hybrid Lagrangian surface  $L_1$ depicted in Figure \ref{fig:RII-move2} (right), which satisfies that:
    
    \begin{itemize}
        \item[(i)] the projection onto the plane is injective in the light grey region and is a double covering in the dark grey region;
        
        \item[(ii)] the front projection of the Legendrian lift has two connected families of $A_1^2$-singularities as specified in Figure \ref{fig:RII-move2} (right).
    \end{itemize}
\end{lemma}

\begin{figure}[h!]
  \centering
  \includegraphics[width=0.7\textwidth]{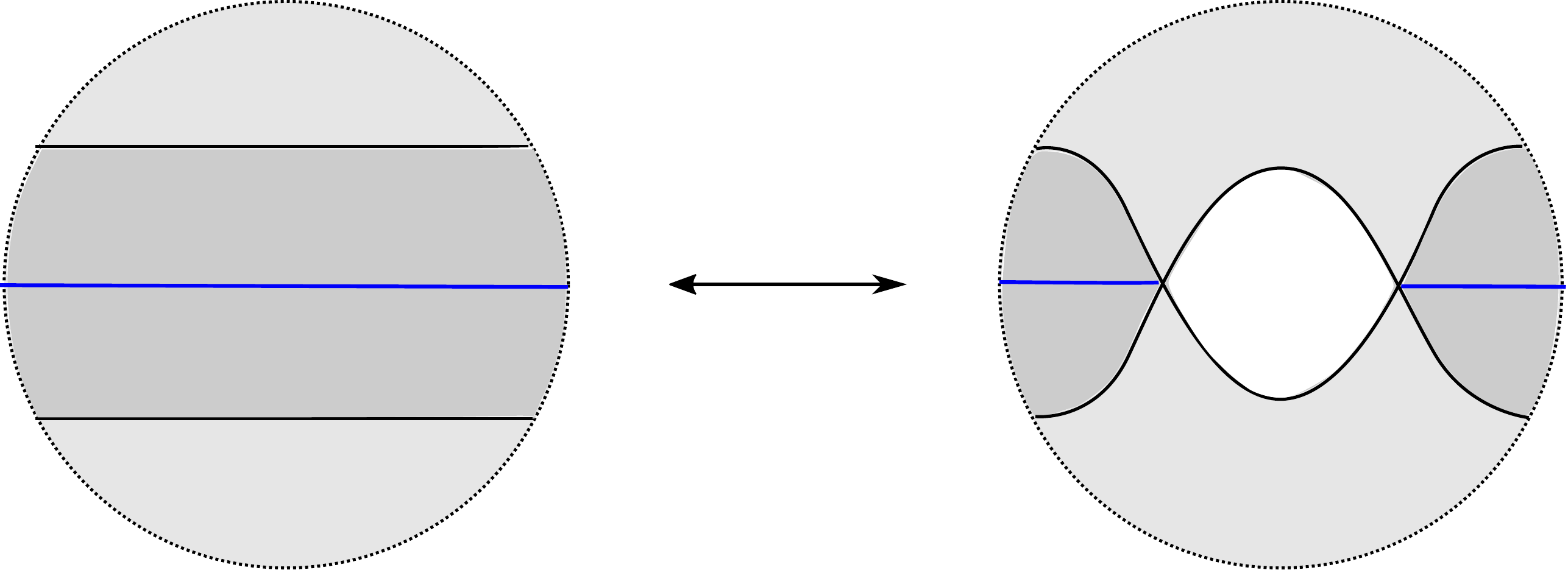}\\
  \caption{The lcoal move~2 for Lagrangian fillings under a Legendrian Reidemeister II move. The covectors of the Legendrian points towards the dark grey region.}\label{fig:RII-move2}
\end{figure}

\begin{proof} 
    Let $\beta: [0,+\infty) \rightarrow [0,1]$ be a cut-off function such that $\beta(x) = 1$ when $x$ is sufficiently small, $\beta(x) = 0$ when $x$ is sufficiently large, and $\beta'(x) \leq 0$. Define
    $$g_t(x) = 1 - 2(1 - t)\beta(|x|).$$
    Suppose the front projection of the Legendrian link $\Lambda_t$ be defined by $\{(x, \pm g_t(x) | x \in \mathbb{R}\}$. Consider the family of exact Lagrangians ($(x, w)$ are coordinates in the domain of $L_\pm$)
    $$L_{t,\pm} = \left\{\Big(x, \pm(w - g_t(x)), \pm\frac{g'_t(x)}{w}, \frac{1}{w}\Big)\Big|x \in \mathbb{R}, w \in (0,+\infty)\right\}.$$
    Since the primitive of the exact Lagrangians are $g(x, w) = \ln w$, their corresponding Legendrian lifts are thus defined by
    $$\widetilde{L}_{t,\pm} = \left\{\Big(x, \pm(w - g_t(x)), \pm\frac{g'_t(x)}{w}, \frac{1}{w}, \ln w \Big)\Big|x \in \mathbb{R}, w \in (0,+\infty)\right\}.$$
    It is clear that this defines an exact Lagrangian isotopy from $L_0$ to $L_1$ whose Legendrian lifts satisfy the conditions in the lemma (the $A_1^2$-singularities of $L_0$ and $L_1$ come from the transverse intersection of the fronts of $\widetilde{L}_{0,\pm}$ and $\widetilde{L}_{1,\pm}$ in the corresponding regions). Finally, note that the exact Lagrangian isotopy can be induced by the Hamiltonian flow of ($(x, y)$ are the coordinates of $\mathbb{R}^2$ and $(\xi, \eta)$ coordinates of the cotangent fiber)
    $$H(x, y, \xi, \eta) = 2\beta(|x|)\eta.$$
    Hence the proof is completed.
\end{proof}
\begin{remark}\label{rem:moresheats-localmoveII}
Similar to Remark \ref{rem:moresheats-localmove2} if there are already $i$ sheets below the Legendrian front, then we can apply the above two local moves to the $s_i$-edges in the weave. 
\end{remark}

\begin{figure}[h!]
  \centering
  \includegraphics[width=0.4\textwidth]{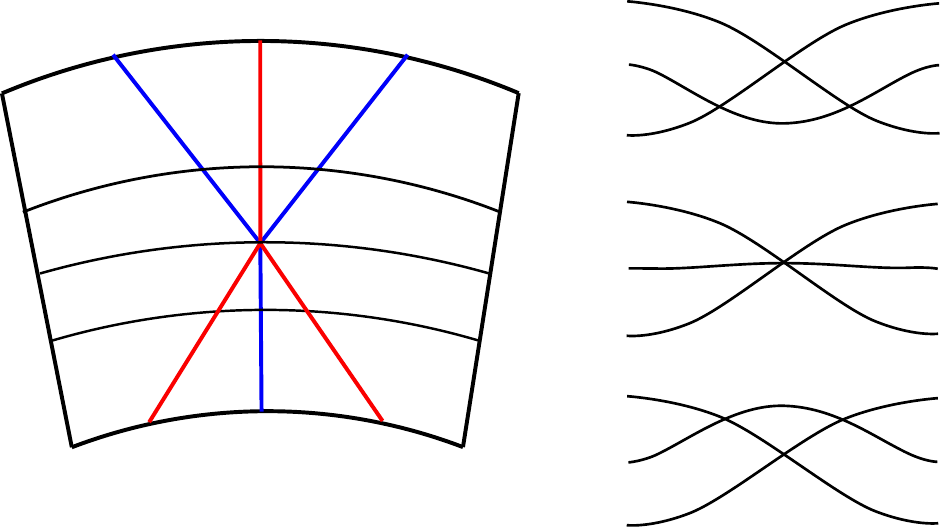}\\
  \caption{Realizing a free Legendrian weave with an $A_1^3$-singularity in the front as a cobordism induced by a Reidemeister~III move in the Legendrian front projection.}\label{fig:A_1_3-RIII}
\end{figure}

\subsection{Legendrian weaves and pinching of Reeb chords}\label{ssec:weave_pinching}
    For the relation between free Legendrian weaves and decomposable Lagrangian cobordisms, we briefly recall the relation between free Legendrian weaves and  elementary Lagrangian cobordisms.\\
    
    First, observe that a hexagonal vertex in the Legendrian weave, when sliced according to Figure \ref{fig:A_1_3-RIII} (left), corresponds to a Lagrangian cobordism induced by a Reidemeister~III move; see \cite{CasalsZas20}*{Remark 4.3}. The front braid slices are drawn in Figure \ref{fig:A_1_3-RIII} (right). Second, a trivalent vertex in the Legendrian weave, sliced horizontally top-to-bottom with two edges being intersected at the top and one edge intersected at the bottom (as in the left of Figure \ref{fig:D4-saddle}), corresponds to an elementary saddle cobordism. Indeed, this follows from the fact that the generic Legendrian perturbation of the front for the $D_4^-$-singularity contains three swallowtails as depicted in Figure \ref{fig:D4-saddle} (center); Figure \ref{fig:D4-saddle} (right) depicts the front braid slices illustrating what Reeb chord is precisely pinched under this (elementary) exact Lagrangian cobordism. See \cite[Section 4.3]{CasalsZas20} for further details.


\begin{figure}[h!]
  \centering
  \includegraphics[width=0.9\textwidth]{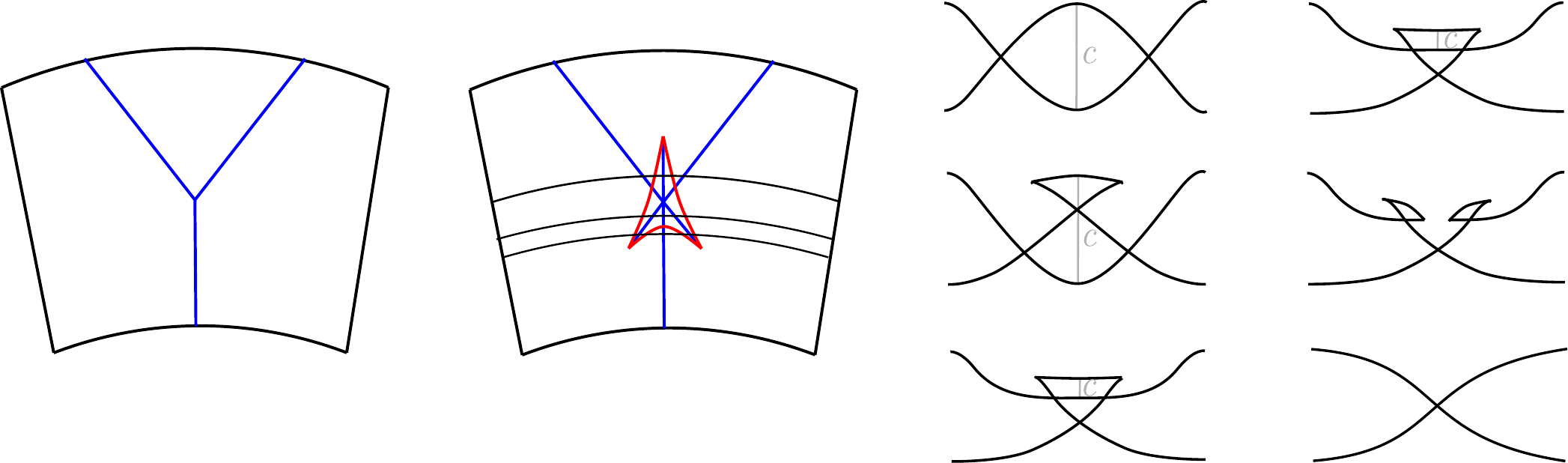}\\
  \caption{Realizing a $D_4^-$-singularity in the Legendrian front as a Lagrangian saddle cobordism. The graph on the left represents the $D_4^-$-singularity while the graph in the middle represents a generic perturbation of the $D_4^-$-singularity.}\label{fig:D4-saddle}
\end{figure}

\subsection{Square moves on conjugate Lagrangians and weave mutations}\label{ssec:squaremove_Lagrangianmutation} Let us illustrate the reason why a square move of conjugate Lagrangians has the same effect as a Legendrian mutation of Legendrian weaves in the local picture. On the one hand, this starts to show why the weave calculus includes the plabic graph calculus. On the other hand, it also serves as a starting exercise for the diagrammatic proofs that are presented in Section \ref{sec:main_proofs}.\\

\begin{figure}[h!]
  \centering
  \includegraphics[width=0.9\textwidth]{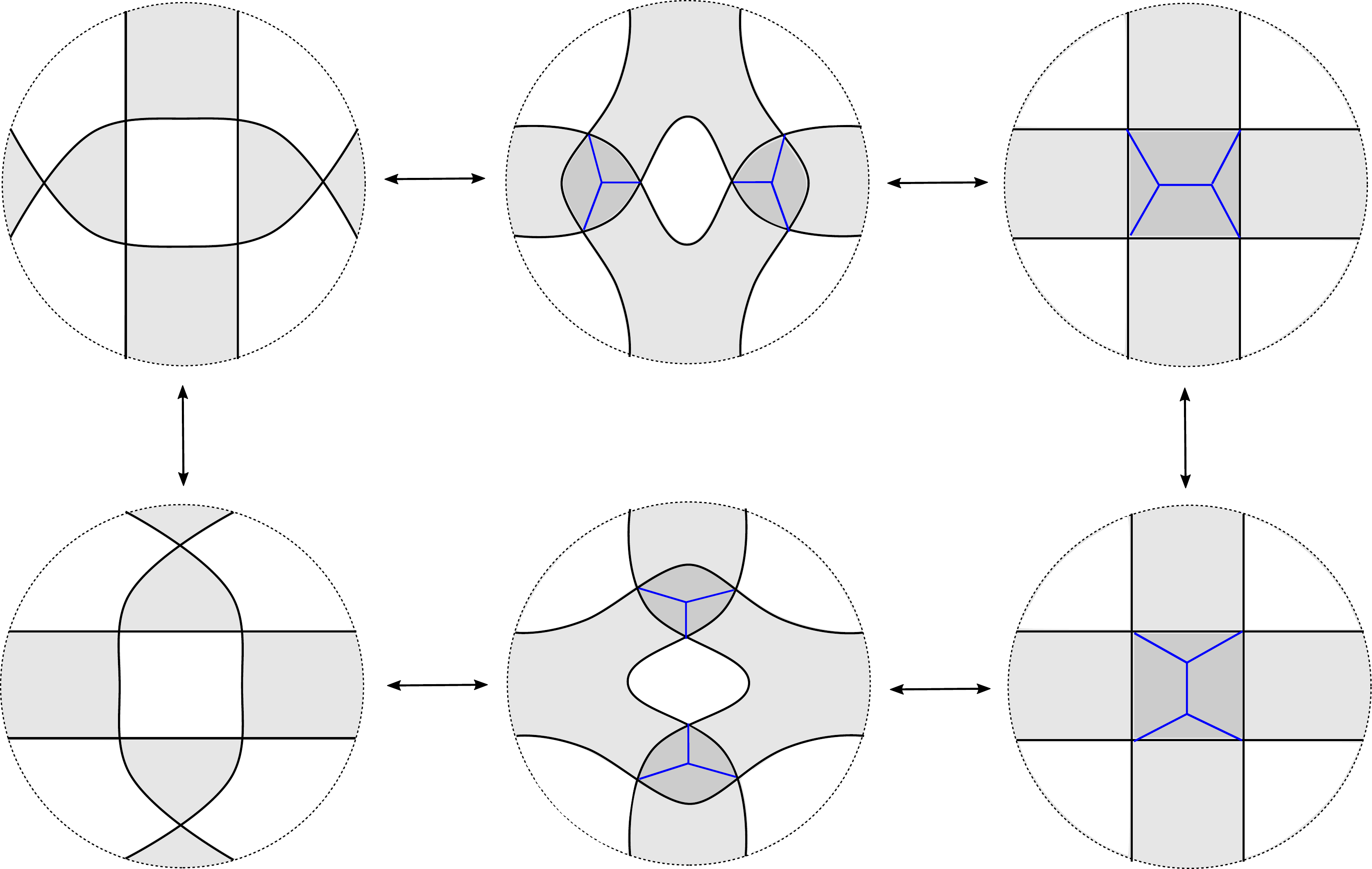}\\
  \caption{The conjugate Lagrangians related by a square move (on the left) and the Legendrian weaves related by a Legendrian mutation (on the right). On the left the covectors are pointing inward the two light grey triangles. On the right the covectors are pointing outward the dark grey squares.}\label{fig:SqMoveMutation}
\end{figure}

    Let $L_0, L'_0$ be conjugate Lagrangian surfaces related by a square move, as depicted in the two local models of Figure \ref{fig:SqMoveMutation} (left). Let $L_1, L_1'$ be hybrid Lagrangian surfaces depicted in Figure \ref{fig:SqMoveMutation} (right), whose projection onto the base is a 1-fold covering in the light grey regions, a 2-fold covering in the dark grey region, whose Legendrian lifts have crossings labeled by the blue edges. Then $L_0$ is Hamiltonian isotopic to $L_1$, as illustrated by the first row of Figure \ref{fig:SqMoveMutation}. This sequence of moves first uses two RIII1 moves, from $L_0$ to the hybrid Lagrangian in Figure \ref{fig:SqMoveMutation} (top center), and then an (inverse) RII2 move. Similarly, the Lagrangian $L_0'$ is Hamiltonian isotopic to $L_1'$. In addition, the 1-cycle coming from the null region in the middle of Figure \ref{fig:SqMoveMutation} (left) corresponds to the weave $\sf I$-cycle in the middle of Figure \ref{fig:SqMoveMutation} (right).\\

\begin{figure}[h!]
    \centering
    \includegraphics[width=0.55\textwidth]{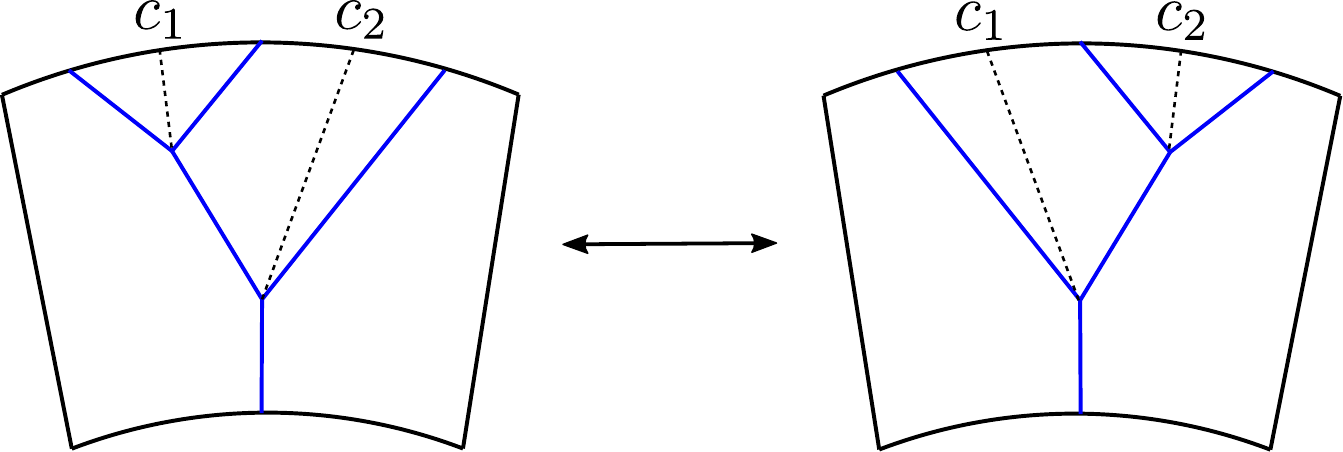}
    \caption{The Lagrangian cobordism of the pinching sequence $c_1, c_2$ (left) and $c_2, c_1$ (right), and the fronts of their Legendrian lifts.}
    \label{fig:weave-pinch-mutate}
\end{figure}

    Finally, let us emphasize that Legendrian weave mutation can also be seen as different choices of pinching sequences, as illustrated in Figure \ref{fig:weave-pinch-mutate}. Indeed, consider the two free Legendrian weaves from the (long) Legendrian braid $s_1 \subset J^1\bR$ to the Legendrian braid $s_1^3 \subset J^1\bR$ in that figure, which are related by a Legendrian mutation. The free Legendrian weave on the left corresponds to the concatenation of elementary cobordisms by first pinching the Reeb chord on the left and then pinching the Reeb chord on the right, while the free Legendrian weave on the right corresponds to the concatenation of first pinching the Reeb chord on the right and then pinching the Reeb chord on the left. That is, changing the order of two adjacent Reeb chords in the pinching sequence corresponds exactly to a Lagrangian mutation along the corresponding $\sf I$-cycle.


\section{Conjugate Fillings and Reeb Pinchings as Legendrian weaves}\label{sec:main_proofs}

Let us now prove Theorem \ref{thm:main1}: conjugate Lagrangian fillings associated to plabic graphs $\bG\in\mathcal{C}(\Sigma)$ are Hamiltonian isotopic to Lagrangian projections of (free) Legendrian weaves, and so is any Lagrangian filling obtained by a pinching sequence, as shown in Theorem \ref{thm:braid-weavepinching} below. The proofs use the moves proven in Theorem \ref{thm:hybridReidemeister}.

\subsection{From conjugate Lagrangian fillings of $(2, k)$-torus links to weaves}\label{ssec:conjugate_to_weave_2klinks}
Let us address a simple class of examples by proving Theorem \ref{thm:main1} for the (reduced) plabic graphs in $\mathcal{C}(\Sigma)$, associated to $\mbox{Gr}(2,k+2)$, relating conjugate Lagrangians to the free Legendrian weaves for Legendrian $(2, k)$-torus links. In this case  $\Sigma=\D^2$ with $(k+2)$ marked points at the boundary.

\begin{proof}[\pfo Theorem \ref{thm:main1}: reduced plabic graphs for $\mbox{Gr}(2,k+2)$]
    Given a triangulation of an $(k+2)$-gon, we label all the vertices of the triangulation as white vertices, all faces of the triangulation as black vertices, and connect a pair of black and white vertices if the corresponding vertex lies in the closure of the corresponding face (the top left part of Figure \ref{fig:2,nlink}). This reduced plabic graph gives an alternating Legendrian isotopic to the $(2, k)$-torus link and a corresponding conjugate Lagrangian filling; see Section \ref{sec:conjugate}, especially Example \ref{ex:2,nlink-conj}.\\
    
    \begin{figure}[h!]
  \centering
  \includegraphics[width=0.9\textwidth]{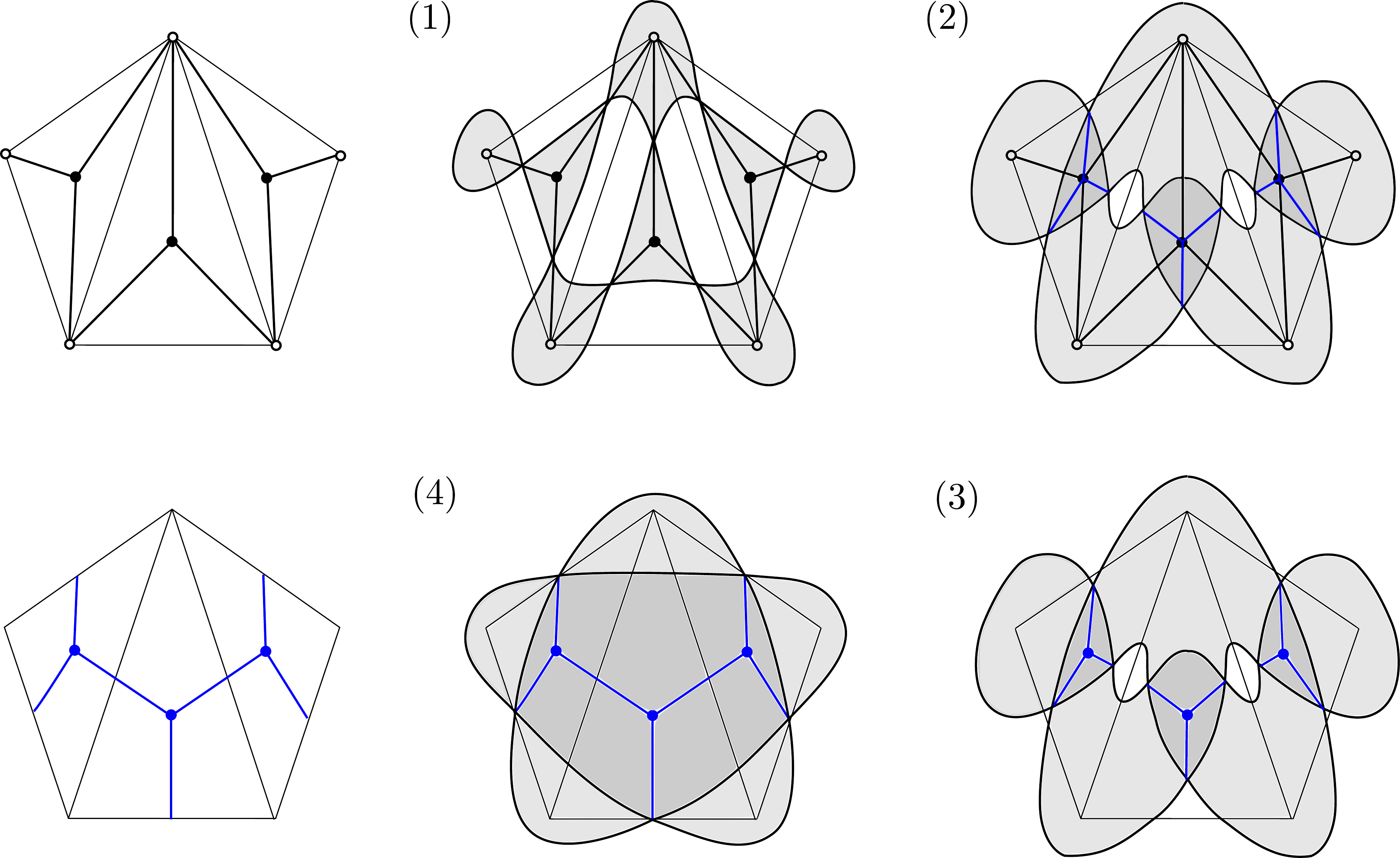}\\
  \caption{The correspondence between conjugate Lagrangian fillings and Legendrian 2-weaves coming from the same triangulation of an $(k+2)$-gon, depicted in the case $k=3$.}\label{fig:2,nlink}
\end{figure}

    Let us start with the alternating Legendrian and conjugate Lagrangian filling associated to this reduced plabic graph. An instance of such a triangulation for $k=3$ is depicted in Figure \ref{fig:2,nlink}.(1). Each black vertex of the plabic is connected with 3 white vertices, as the black vertex corresponds to a triangle and the white vertices correspond to its vertices. Now we apply one Reidemeister~III1 move at each black vertex, as drawn in Figure \ref{fig:TableMoves}. This creates a hybrid Lagrangian surface with a trivalent blue vertex at each black vertex. This is depicted in Figure \ref{fig:2,nlink}.(2)--(3), where Figure \ref{fig:2,nlink}.(3) is the hybrid Lagrangian surface without the plabic graph superimposed. At this stage, for each pair of black vertices that share two adjacent white vertices, we can perform the Reidemeister~II1 move in Figure \ref{fig:TableMoves}. After performing these $(k-1)$ RII1 moves, each pair of blue edges from two trivalent blue vertices separated by a diagonal are connected. This yields a free weave filling of the $(2, k)$-torus link at the boundary. Figure \ref{fig:2,nlink}.(4) illustrates that in the example.

    \noindent Finally, for the resulting Legendrian weave, there exists a blue edge if and only if two black vertices share a pair of adjacent white vertices, which means the corresponding two faces in the original triangulation share a common edge. Hence the weave we have obtained is exactly the one determined by the same triangulation of the $(k+2)$-gon as in Example \ref{ex:2,nlink-weave}; see the bottom left part of Figure \ref{fig:2,nlink} for an instance of such weave dual to a triangulation.
\end{proof}

\subsection{From conjugate fillings of positive braid closures to weaves}\label{sec:braid-conjweave}
    In this section we prove Theorem \ref{thm:main1} for conjugate surfaces of grid plabic fences, identifying conjugate Lagrangian associated to such plabic fences with Lagrangian projections of Legendrian weaves.
    
\begin{figure}[h!]
  \centering
  \includegraphics[width=1.0\textwidth]{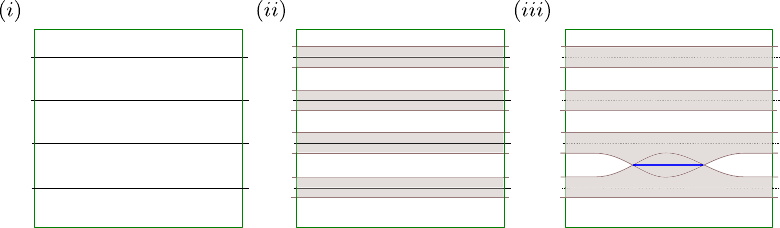}\\
  \caption{The start of the process from the conjugate surface of a Type 1 column towards its associated weave. (i) A Type 1 column for $n=4$. (ii) The conjugate surface associated to that Type 1 column. (iii) Inserting a Reidemeister II1 move between the lowest two strips of the conjugate surface.}\label{fig:Proof_Type1ToWeave_Part1}
\end{figure}

\begin{proof}[\pfo Theorem \ref{thm:main1} (continued): grid plabic fences] The proof is local in the three types of columns of a plabic fence: Type 1 columns, consisting of $n$ parallel horizontal edges, and Type 2 columns, $n$ parallel horizontal edges with a vertical edge (black on top) between the $i$th and $(i+1)$st strands, $i\in[1,n-1]$, and Type 3 columns, containing lollipops. Figure \ref{fig:Proof_Type1ToWeave_Part1}.(i) depicts a Type 1 column for $n=4$, Figure \ref{fig:Proof_CrossingToWeave_N5_Part1}.(1) depicts a Type 2 column with $n=5$ and $i=1$, and Figure \ref{fig:lolipop-conj-weave}.(1) depicts a Type 3 column. See \cite[Section 2]{CasalsWeng22} for this notation and more details.\\

\begin{figure}[h!]
  \centering
  \includegraphics[width=1.0\textwidth]{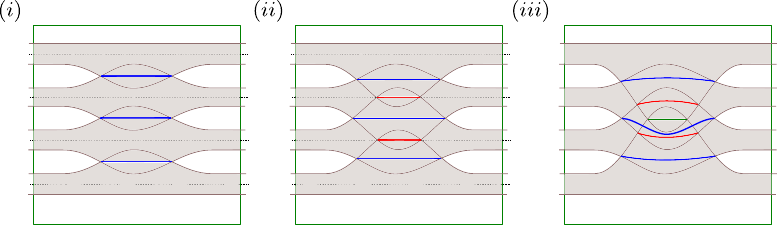}\\
  \caption{(i) The conjugate surface for Type 1 column with $n=4$ after inserting the initial $(n-1)$ Reidemeister II1 moves that create $(n-1)$ horizontal weave $s_1$-edges (in blue). (ii) The conjugate surface after the second iteration of Reidemeister II1 moves, where two such moves are performed to insert $(n-2)$ horizontal $s_2$-edges (in red) in between the previous $(n-1)$ horizontal $s_1$-edges. (iii) The result of performing the last Reidemeister II1 move between the two $s_2$-edges, introducing the last weave $s_3$-edge (in green).}\label{fig:Proof_Type1ToWeave_Part2}
\end{figure}

First, we address Type 1 columns. Locally, the conjugate surface is given by $n$ horizontal strips, each containing one of the $n$ horizontal strands: Figure \ref{fig:Proof_Type1ToWeave_Part1}.(ii) depicts the horizontal strips in grey and the horizontal strands (from the plabic fence) in black, for $n=4$. The sequence of moves from this conjugate surface to a (local) weave starts with the insertion of $(n-1)$ Reidemeister II1 moves, one per each adjacent pair of horizontal strips. Figure \ref{fig:Proof_Type1ToWeave_Part1}.(iii) illustrates one such moves, inserted between the bottom two horizontal strips, and Figure \ref{fig:Proof_Type1ToWeave_Part2}.(i) depicts the result of inserting the $(n-1)$ Reidemeister II1 moves, for $n=4$. After these moves have been inserted, there are $(n-1)$ horizontal weave $s_1$-edges, drawn in blue in Figures \ref{fig:Proof_Type1ToWeave_Part1} and \ref{fig:Proof_Type1ToWeave_Part2}. Right above and below each of such $s_1$-edge, the conjugate surface has a piece of its boundary component. In particular, in between two such adjacent $s_1$-edges there are two such pieces of boundary. The next step is to insert $(n-2)$ Reidemeister II1 moves in between these pieces of boundary: this creates $(n-2)$ horizontal weave $s_2$-edges, as the sheets of the conjugate surface that intersect are now the second and third, counted from below. Figure \ref{fig:Proof_Type1ToWeave_Part2}.(ii) depicts the result of this process for $n=4$, with the weave $s_2$-edges drawn in red.\\

\noindent The process can now be iterated among the newly creates weave edges. Indeed, performing $(n-3)$ Reidemeister II1 moves introduces $(n-3)$ horizontal $s_3$-edges, and iteratively performing $(n-i)$ Reidemeister II1 moves between the $(n-i+1)$ existing $s_{i-1}$-edges, which creates $(n-i)$ horizontal weave $s_i$-edges. After $(n-2)$ iterations, the last iteration consists of performing a unique Reidemeister II1 inserting a unique $s_{n-1}$-edge. Figure \ref{fig:Proof_Type1ToWeave_Part2}.(iii) illustrates the result for $n=4$ and the $s_3$-edge drawn in green. The resulting piece of weave consists of ${n\choose 2}$ horizontal weave edges which spell an expression for $w_0\in S_n$, the longest word in the symmetric group $S_n$. For concreteness, this expression can be taken to be $\mathtt{w}_{0,n}:=s_1(s_2s_1)(s_3s_2s_1)\ldots (s_{n-1}s_{n-1}\cdots s_2s_1)$ when the $s_i$-edges of the weave are read bottom to top. For reference, note that this weave is denoted by $\mathfrak{n}(\mathtt{w}_{0,n})$ in Subsection \ref{ssec:ex-free-weave} and \cite[Section 3.3.1]{CasalsWeng22}.\\

\begin{figure}[h!]
  \centering
  \includegraphics[width=1.0\textwidth]{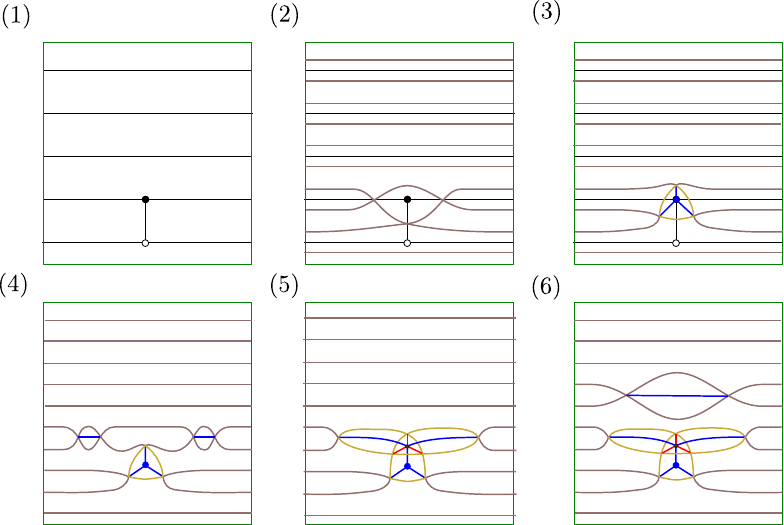}\\
  \caption{(1) A Type 2 column of a plabic fence with its crossing at the lowest possible level. (2) The conjugate Lagrangian surface associated to a Type 2 column. (3) The result of applying a Reidemeister RIII3 to the conjugate surface. (4) The resulting hybrid Lagrangian obtained from (3) after applying two RII2 moves. (5) The result after applying a RIII2 move. (6) The subsequent hybrid Lagrangian after using a Reidemeister RII1 move. }\label{fig:Proof_CrossingToWeave_N5_Part1}
\end{figure}

Second, we address Type 2 columns. Consider a Type 2 column with $n$ horizontal strands and a vertical edge (black on top) between the $i$th and $(i+1)$st strands. The case of $i=1$ is the most interesting, so let us set $i=1$. Figure \ref{fig:Proof_CrossingToWeave_N5_Part1}.(1) depicts the case $n=5$ and $i=1$ and Figure \ref{fig:Proof_CrossingToWeave_N5_Part1}.(2) draws the conjugate surface for this Type 2 column.\footnote{In Figures \ref{fig:Proof_CrossingToWeave_N5_Part1} and \ref{fig:Proof_CrossingToWeave_N5_Part2} we have not filled the interior of the conjugate surface (and its hybrid cousins) in grey in order to increase clarity and ease readibility of the weave lines. Even if not drawn, the corresponding regions {\it are} filled in the same manner as usual.} The first step for transforming this conjugate surface into a weave is to perform a Reidemeister III1 move at the (bounded) triangular region of the conjugate surface that contains the black vertex. This results in the hybrid surface drawn in Figure \ref{fig:Proof_CrossingToWeave_N5_Part1}.(3), inserting a trivalent vertex with three weave $s_1$-edges (in blue) emanating from it. The next step is to introduce two Reidemeister II1 moves in between the third strip, counting from the bottom, and the second strip. These two moves are inserted at the right and left of the trivalent vertex (and above it): this is drawn in Figure \ref{fig:Proof_CrossingToWeave_N5_Part1}.(4). The insertion of these two Reidemeister II1 moves creates an empty triangular region right above the trivalent vertex, with three $s_1$-edges being incident at its vertices. Therefore, we can perform a Reidemeister III2, introducing a hexavalent vertex with $s_1$- and $s_2$-edges emanating from it. Figure \ref{fig:Proof_CrossingToWeave_N5_Part1}.(5) illustrates this. This sequence of moves, first a Reidemeister III1, then two Reidemeister II1 and then a Reidemeister III2, is the {\it starting} step of the iteration; it is performed once. The iterative step, to be performed $(n-2)$ times, is similar but no identical.\\

\noindent Let us now continue the process by first performing a Reidemeister II1 move between the top of the third horizontal strip and the bottom of the fourth horizontal strip (counting from the bottom). This produces a horizontal $s_1$-edge, depicted in blue in Figure \ref{fig:Proof_CrossingToWeave_N5_Part1}.(6). That move itself moves a piece of the boundary of the conjugate surface (which used to be at the bottom of the fourth strip) towards a piece of the boundary of the conjugate surface near the hexavalent vertex. This allows for the introduction of two Reidemeister II1 moves which insert two horizontal $s_2$-edges. The precise location is drawn in Figure \ref{fig:Proof_CrossingToWeave_N5_Part2}.(7), and it is readily seen that a triangular region with only one sheet is created in the middle. Therefore, we can and do perform a Reidemeister III2 move creating a hexavalent vertex in that triangular region which emanates $s_2$ and $s_3$-edges. This is depicted in Figure \ref{fig:Proof_CrossingToWeave_N5_Part2}.(8).\\

\begin{figure}[h!]
  \centering
  \includegraphics[width=1.0\textwidth]{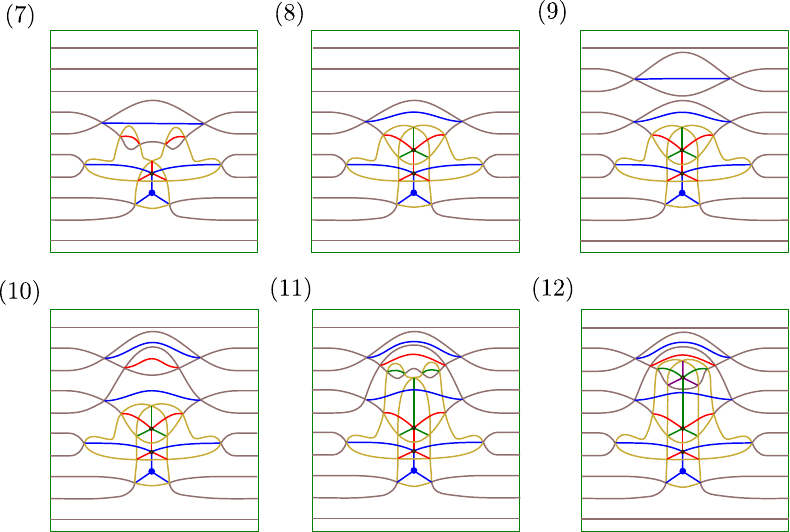}\\
  \caption{(7) Two Reidemeister RII1 being applied to the top two sheets, creating the two additional red edges. (8) The hybrid Lagrangian obtained after applying a Reidemeister RIII2 move to the second and third sheets, creating the hexavalent vertex. (9)-(10)-(11) Consists of a sequence of Reidemeister RII2 moves, one creating the blue edge in (8), another red edge in (9) and the two green edges in (11). (12) The hybrid Lagrangian surface after applying the final Reidemeister RIII2 move.}\label{fig:Proof_CrossingToWeave_N5_Part2}
\end{figure}

\noindent This process of inserting Reidemeister II1 moves and then a Reidemeister III2 move can now be iterated as follows, always inserting the moves above. At the $j$th step, we perform $j$ Reidemeister II1 moves, creating horizontal weave $s_i$-edges that spell $s_1s_2\ldots s_j$ top to bottom. Figure \ref{fig:Proof_CrossingToWeave_N5_Part2}.(9) and (10) show this for $n=5$ at the step $j=2$; the first iterative step $j=1$ was Figures \ref{fig:Proof_CrossingToWeave_N5_Part1}.(6) and \ref{fig:Proof_CrossingToWeave_N5_Part2}.(7) and (8).\footnote{In a sense, Figures \ref{fig:Proof_CrossingToWeave_N5_Part1}.(4) and (5) are the 0th iterative step.} Then, two Reidemeister II1 moves are created at the right and left (and on top) of the unique vertical weave $s_{j+1}$-edge, giving rise to two horizontal $s_{j+1}$-edges which are incident to vertices of a (newly created) triangular region. This is drawn in Figure \ref{fig:Proof_CrossingToWeave_N5_Part2}.(11). To finalize the $j$th iterative step, a Reidemeister III2 is performed at that triangular region, creating a hexavalent vertex emanating $s_{j+1}$- and $s_{j+2}$-edges. The iteration ends when $j=n-3$, creating a weave with $s_1$- up to $s_{n-1}$-edges. Figure \ref{fig:Proof_CrossingToWeave_N5_Part2}.(12) depicts the result of these iterations. Finally, the unique vertical long $\sf I$-cycle on the weave, with a vertical $s_{n-1}$-edge on towards the top, can be vertically lengthened beyond all weave edges above it. Indeed, all the weave edges above that vertical $s_{n-1}$-edge are $s_j$-edges with $j\leq n-3$ and thus the  vertical $s_{n-1}$-edge can cross them with tetravalent vertices. This concludes the construction of the weave for $i=1$ and any $n\in\N$.\\

\noindent The case of arbitrary $i\in[1,n-1]$ readily follows from the case $i=1$ by observing that the construction for $i=1$ only involves the horizontal edges of the plabic fence right (at and) above the vertical edge. In consequence, for any $i\in[1,n-1]$, we proceed as in the $i=1$ case for the portion of the plabic fence consisting of the horizontal edges $i$th through $n$th, counting from the bottom, and apply the Type 1 column construction for the horizontal edges $1$st through $(i-1)$st.\\

\begin{figure}[h!]
    \centering
    \includegraphics[width=0.7\textwidth]{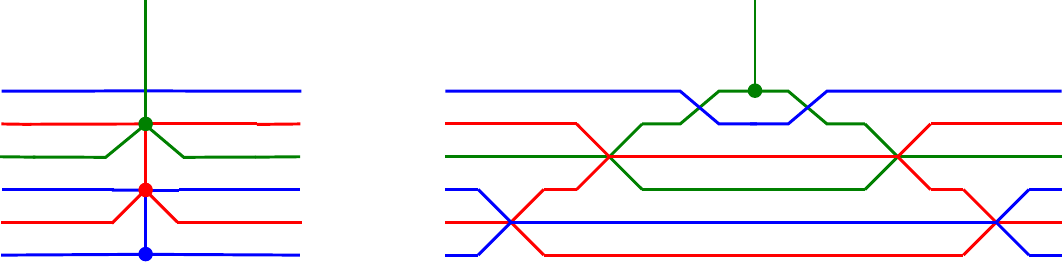}
    \caption{On the left, the weave we obtain for Type~II columns. On the right, the weave $\mathfrak{c}_{i}^{\uparrow}(\w_{0,n})$ for Type~II columns (see \cite[Sec.~3.3.2]{CasalsWeng22}). These two weaves differ by a sequence of push-through moves, the move II in Figure \ref{fig:ReidemeisterWeave}.}
    \label{fig:braid-weave-pushthru}
\end{figure}

 Finally, note that at the right and left the ends of these hybrid surfaces, hybrids between a conjugate surface and a weave, coincide, regardless of whether they are obtained from a Type 1 or a Type 2 column. Thus, these hybrid surface can be horizontally concatenated according to the plabic fence, since all the pieces have the same lateral boundary conditions. In order to remove the white (empty) areas and obtain a weave in the end, one then applies (the inverse of) Reidemeister II2 moves. This concludes the construction of the Hamiltonian isotopy from a conjugate surface to (the Lagrangian projection of) a Legendrian weave. This concludes the case of a Type 2 column.\\
 
Third, the case of a Type 3 column follows an argument analogous to the above and it is left for the reader. In fact, the resulting weaves are equivalent to the weaves $\mathfrak{l}_i^{b}$, or $\mathfrak{l}_i^{w}$ introduced in \cite[Sec.~3.3.3]{CasalsWeng22}. An explicit instance of the necessary sequence of moves (from Figure \ref{fig:TableMoves}) is depicted in Figure \ref{fig:lolipop-conj-weave}. The general case is a direct generalization of this sequence of moves.\\

At this stage we have found weaves for the Type 1, Type 2 and Type 3 columns and their lateral ends (both at the right and left) match. Therefore, we can horizontally concatenate these weaves to obtain a hybrid surface. In order to turn the hybrid surface entirely into a weave, it now suffices to apply ${n-1\choose 2}$ inverse Reidemeister II2 moves, which will remove the white (empty) areas and obtain a weave in the end. In addition, a sequence of push-through moves, pushing the unique trivalent vertex upwards, shows that this weave is equivalent to the weave $\mathfrak{c}_{i}^{\uparrow}(\w_{0,n})$ in Subsection \ref{ssec:ex-free-weave} Figure \ref{fig:braid-weave-cross2}, see also \cite[Sec.~3.3.2]{CasalsWeng22}; Figure \ref{fig:braid-weave-pushthru} illustrates this equivalence in an example.
\end{proof}

\begin{figure}[h!]
    \centering
    \includegraphics[width=1.0\textwidth]{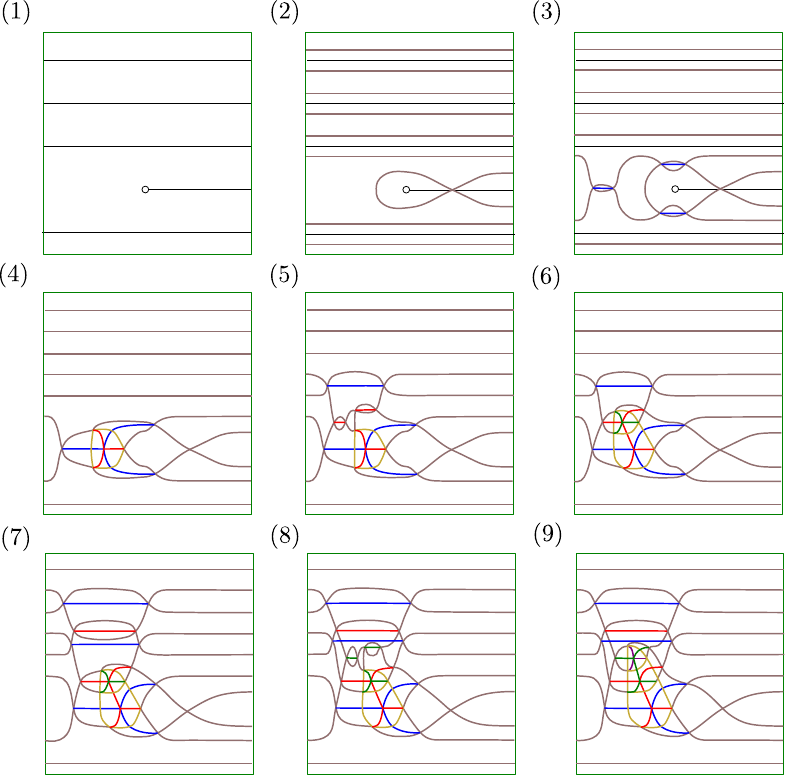}
    \caption{(1) A Type 3 column in a grid plabic graph. (2) The conjugate surface associated to a Type 3 column. (3) Performing three RII1 moves. (4) Inserting a hexavalent vertex with an RIII2 move. (5) Creating three more weave edges (one blue and two red) with three RII1 moves. (6) Performing a RIII2 move that creates a hexavalent vertex. (7)-(8) Introducing a sequence of RII1 moves. (9) The result of applying the final RIII2 move.}\label{fig:lolipop-conj-weave}
\end{figure}

\noindent Given a plabic fence with white vertices on top of its vertical edges, the proof described above can also be applied with the modification that now the process creates the weave downwards. Thus the base case is $i=n$, at the top horizontal edge of the plabic fence. The resulting weave is equivalent to $\mathfrak{c}_{i}^{\downarrow}(\w_{0,n})$ from \cite[Sec.~3.3.2]{CasalsWeng22}.\\

Note that in the correspondence between conjugate Lagrangians and Legendrian weaves, each black vertex on the $i$-th level of the bipartite graph in a Type~II column corresponds uniquely to an $s_i$-trivalent vertex in the corresponding weave $\mathfrak{c}_i^\uparrow(\mathtt{w}_{0,n})$. In the conjugate Lagrangian, the 1-cycles are given by the null components in the plabic fence between two vertical edges on the same level, while in the Legendrian weave, the 1-cycles are given by long $\sf I$-cycles, which connect two trivalent vertices adjoining edges on the same level. Keeping track of the 1-cycles of the Lagrangian fillings under the Hamiltonian isotopy, we can show that these 1-cycles correspond to each other, as explained in the following corollary:

\begin{figure}[h!]
    \includegraphics[width=0.95\textwidth]{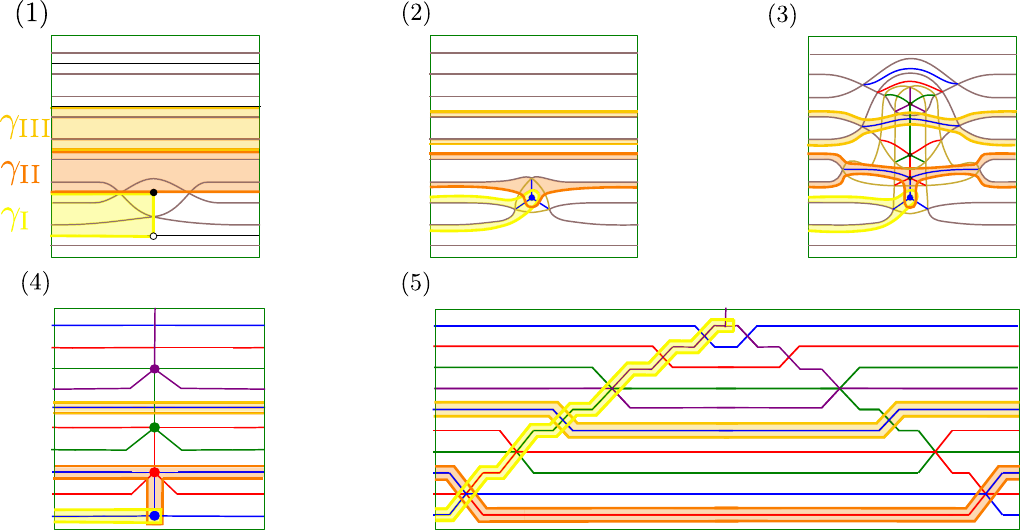}
    \caption{(1)~The relative 1-cycles in the Type~II column corresponding to null regions with boundaries. (2)~The relative 1-cycles after a Reidemeister~III1 move. (3)~The relative 1-cycles after all Reidemeister~III2 moves. (4)~The relative 1-cycles in the weave we obtain for Type~II columns before applying push throughs. (5)~The relative $\sf I$-cycles in the weave $\mathfrak{c}_i^\uparrow(\mathtt{w}_{0,n})$ after applying push throughs.}\label{fig:1cycle-braid}
\end{figure}

\begin{cor}\label{cor:1cycle-braid}
    Consider the Hamiltonian isotopies of the hybrid Lagrangians in Theorem \ref{thm:main1}. Then the following holds:
    
    \begin{enumerate}
    \item Each 1-cycle in the conjugate Lagrangian from the rectangular null regions between two vertical edges on $i$-th row and $j$ and $j'$-th columns corresponds to the long $\sf I$-cycle in the Lagrangian projection of the Legendrian weave connecting the two trivalent vertices on the $j$ and $j'$-th columns adjoining $s_i$-edges.\\
    
    \item Each relative 1-cycle in the conjugate Lagrangian from the half-open rectangular null regions from leftmost/rightmost vertical edges on $i$-th to the left/right boundary of the plabic fence corresponds to the relative long $\sf I$-cycle in the Lagrangian projection of the Legendrian weave from leftmost/rightmost the trivalent vertex adjoining $s_i$-edges to the boundary of the weave.
    \end{enumerate}
\end{cor}
\begin{proof}
    This follows from the algorithm in the proof of Theorem \ref{thm:main1} together with the local characterization of 1-cycles in Corollary \ref{cor:1cycle-loc}. In fact, since both the null regions in the plabic fence and the long $\sf I$-cycles in the Legendrian weave is the union of null regions and $\sf I$-cycles with boundaries in each Type~II column, it suffices for us to keep track of the relative 1-cycles in each Type~II column.\\
    
    In each Type~II column, there are three different types of relative 1-cycles coming from null regions with boundaries, as shown in Figure \ref{fig:1cycle-braid}.(1). $\gamma_\text{I}$ are the two null regions on the $i$-th level of the plabic fence which are bounded from one side by the vertical edge in the plabic fence, $\gamma_\text{II}$ are the null regions on the $(i-1)$-th level of the plabic fence right above the unique vertical edge in the plabic fence, and $\gamma_{\text{III},i'}$ are the other null regions on the $i'$-th level of the plabic fence with boundaries.\\
    
    \noindent Then we keep track of the three classes of relative 1-cycles in the algorithm of Theorem \ref{thm:main1} using Corollary \ref{cor:1cycle-loc}. Figures \ref{fig:1cycle-braid}.(2) shows that after the Reidemeister~III1 move, both $\gamma_\text{I}$ and $\gamma_\text{II}$ winds around the blue trivalent vertex by $2\pi/3$, and Figures \ref{fig:1cycle-braid}.(3) depicts the resulting 1-cycles after all Reidemeister~III2 moves. In Figure \ref{fig:1cycle-braid}.(4), $\gamma_\text{I}$ becomes the relative $\sf I$-cycle ending at the blue trivalent vertex, $\gamma_\text{II}$ becomes the relative $\sf Y$-cycle starting from the $(i-1)$-th blue edges, bifurcating at the blue-red hexavalent vertex and ending at the blue trivalent vertex, and $\gamma_{\text{III},i'}$ becomes the relative $\sf I$-cycles starting and ending at the $i'$-th blue edges.\\
    
    \noindent Finally we apply the push through moves in Subsection \ref{ssec:ex-free-weave} Figure \ref{fig:braid-weave-cross2}. After applying push through once, $\gamma_\text{I}$ becomes an $\sf I$-cycle ending at the red trivalent vertex adjoining $s_2$-edges and $\gamma_\text{II}$ becomes an $\sf I$-cycle starting and ending at the $(i-1)$-th blue edges. Therefore, by iteratively applying push throughs, we can conclude that all these 1-cycles become the $\sf I$-cycles as in Figure \ref{fig:1cycle-braid}.(5).
\end{proof}

\noindent Similar to Corollary \ref{fig:1cycle-braid}, we can prove the correspondence between null regions and long $\sf I$-cycles for a grid plabic graph, which is set up combinatorially in \cite{CasalsWeng22}.


\subsection{From conjugate fillings of $n$-triangulations to weaves}
    In this section we prove Theorem \ref{thm:main1} for $n$-triangulations, showing that the conjugate Lagrangian fillings associated to $A_n^*$-bipartite graphs of $n$-triangulations are Legendrian weave fillings.

\begin{proof}[\pfo Theorem \ref{thm:main1}: $n$-triangulations]
    The proof is locally in each triangle inside the ideal triangulation. Figure \ref{fig:triangle-conjweave2} explains the proof for $n = 2$, Figure \ref{fig:triangle-conjweave3} explains the proof for $n =3$, and Figure \ref{fig:triangle-conjweaveN} explains the inductive part of the proof. Let us now provide the necessary details for the general argument. We construct Hamiltonian isotopies fixing the boundary of the triangle and proceed by induction. First, consider the case $n = 2$ in Figure \ref{fig:triangle-conjweave2}.(1). we apply a Reidemeister~III1 move for the black vertex at the center, and thus obtain a trivalent vertex at the center; the resulting hybrid Lagrangian is drawn in Figure \ref{fig:triangle-conjweave2}.(2). Then, for any pair of adjacent triangles we perform two (inverse) Reidemeister II moves for conjugate surfaces\footnote{This is a local move between conjugate surfaces, with no hybrid Lagrangian involved.} as in Proposition \ref{prop:RII-move0}, so that the hybrid Lagrangian surface in each triangle becomes that depicted in Figure \ref{fig:triangle-conjweave2}.(3). Finally, we consider a Hamiltonian isotopy transforming Figure \ref{fig:triangle-conjweave2}.(3) into Figure \ref{fig:triangle-conjweave2}.(4), which is the required final weave in the $n=2$ case.\\
    
\begin{figure}[h!]
  \centering
  \includegraphics[width=1.0\textwidth]{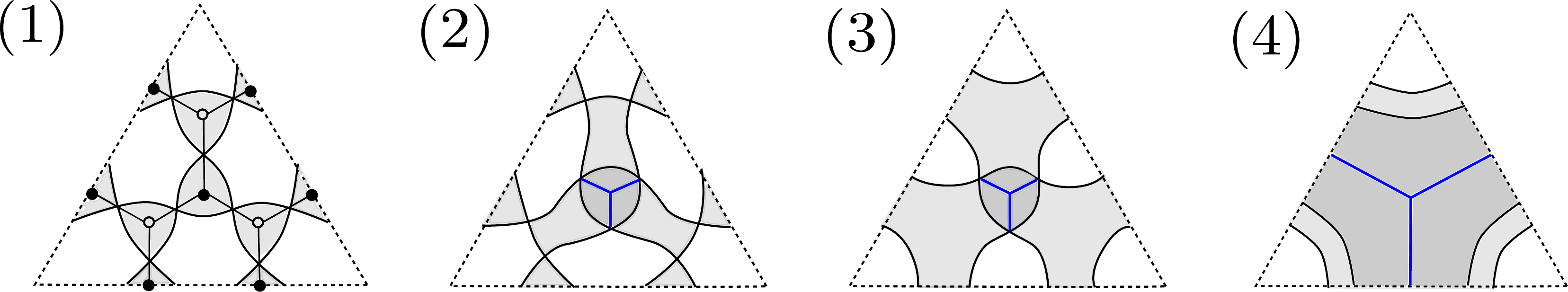}\\
  \caption{The correspondence between conjugate Lagrangian fillings of an $2$-triangulation and the Legendrian weave of a $2$-triangulation.}\label{fig:triangle-conjweave2}
\end{figure}


    
\begin{figure}[h!]
  \centering
  \includegraphics[width=0.9\textwidth]{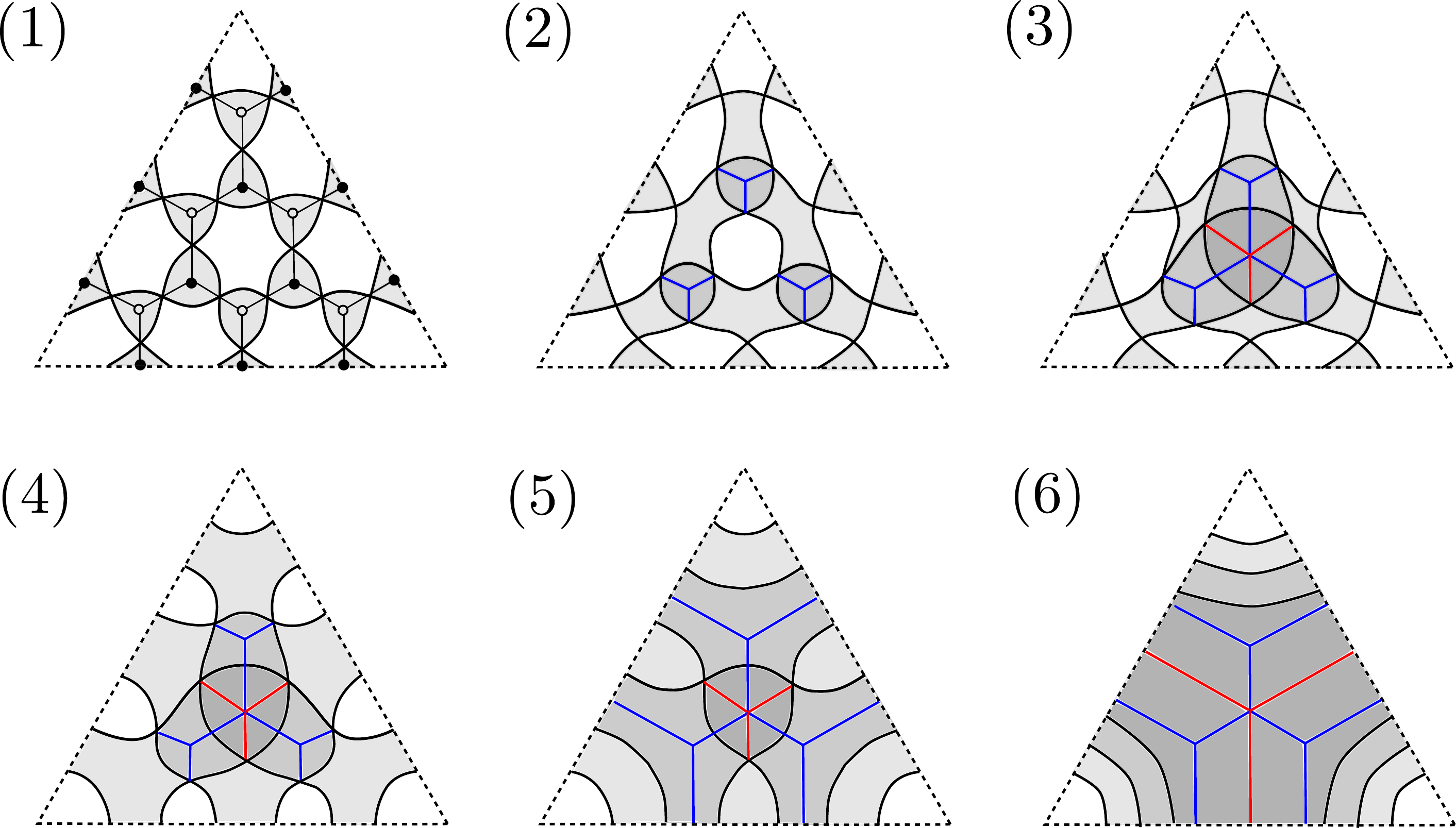}\\
  \caption{The correspondence between conjugate Lagrangian fillings of an $3$-triangulation and the Legendrian weave of a $3$-triangulation.The first row is obtained by Hamiltonian isotopies fixing the boundary. The second row is how we do Hamiltonian isotopies near the boundary.}\label{fig:triangle-conjweave3}
\end{figure}

    Let us proceed with the induction step and assume the correspondence between conjugate Lagrangians and Legendrian weaves for $n$-triangulations. For the $(n+1)$-triangulation, we start with the alternating Legendrian (and the corresponding conjugate Lagrangian) and we first consider the top $n$ rows of the black and white vertices which form the bipartite graph of an $n$-triangulation, as depicted in Figure \ref{fig:triangle-conjweaveN}.(1). By the induction hypothesis, we can apply Hamiltonian isotopies, relative to the boundary, that introduce the Legendrian weave of an $n$-triangulation, where all the strands in the $n$-graph end at the crossings in the front projection of the Legendrian link; we have depicted this in Figure \ref{fig:triangle-conjweaveN}.(2).\\

    In order to construct the weave for the $(n+1)$-triangulation, we first apply Reidemeister~III1 moves at all the $n$ triangles corresponding to the interior black vertices in the last row (the bottom row) of the bipartite graph, which remained fixed under the previous Hamiltonian isotopy. This introduces $n$ trivalent vertices in the last row, as depicted in Figure \ref{fig:triangle-conjweaveN}.(3). Second, consider the $n-1$ triangles above the $n$ trivalent vertices on the bottom row, whose three vertices are the two $s_1$-edges in blue from the adjacent trivalent vertices on the bottom and another $s_1$-edge in blue from the weave of the $n$-triangulation on the top. We then apply Reidemeister~III2 moves at each of these $n-1$ triangles. This introduces $n-1$ hexagonal vertices emanating $s_1$ and $s_2$-edges, as depicted in Figure \ref{fig:triangle-conjweaveN}.(4).

    \noindent We keep working iteratively. Consider the $n-2$ null triangles above the $n-1$ hexagonal vertices, whose three vertices are the two $s_2$-edges in red from the adjacent hexagonal vertices on the bottom and another $s_2$-edge in red from the weave of the $n$-triangulation on the top. We then apply Reidemeister~III2 moves at all these $n-2$ triangles: this introduces $n-1$ hexagonal singularities emanating $s_2$ and $s_3$-edges, as depicted in Figure \ref{fig:triangle-conjweaveN}.(5). Iteratively, we can perform Reidemeister~III2 moves at $n-j$ triangles whose three vertices are two $s_j$-edges from the bottom and another $s_j$-edge from the weave of the $n$-triangulation on the top. These move introduces $n-j$ new hexagonal vertices emanating $s_j$ and $s_{j+1}$-edges. This process yields the Legendrian weave of an $(n+1)$-triangulation, except that now all the strands in the $(n+1)$-graph end at some crossings in the front projection of the Legendrian link; see Figure \ref{fig:triangle-conjweaveN}.(5). At this stage, this is the Legendrian weave that that we obtain by performing Hamiltonian isotopies fixing the boundary of each triangle. In order to produce a weave, we now further perform Hamiltonian isotopies which are not compactly supported in each triangle.\\

    \noindent Indeed, we perform Reidemeister~II moves, which are Hamiltonian isotopies near the boundary of each triangle. This yields a Legendrian weave of an $(n+1)$-triangulation that fills the Legendrian $(n+1)$-satellite of the outward unit conormal of the circles at the marked points. Continuing, we apply Reidemerister~II moves (of conjugate surfaces) near all the remaining black vertices on the boundary of the triangle where we still have locally conjugate surface fillings, and by Proposition \ref{prop:RII-move0} we can resolve these crossings, as shown in Figure \ref{fig:triangle-conjweaveN}.(6).\\

    \noindent Now we apply Reidemerister~II2 moves near all the crossings where the $s_1$-edges in the weaves (colored in blue) end, and resolve the crossings and extend the blue strands to the boundary. This is illustrated in Figure \ref{fig:triangle-conjweaveN}.(7). Then inductively we apply Reidemeister~II2 moves at the crossings where all the $s_j$-edges end, to resolve these crossings and extend the edges to the boundary of the triangle, as in Figure \ref{fig:triangle-conjweaveN}.(8)--(9). This finishes the proof of the general case.
\begin{figure}[h!]
  \centering
  \includegraphics[width=1.0\textwidth]{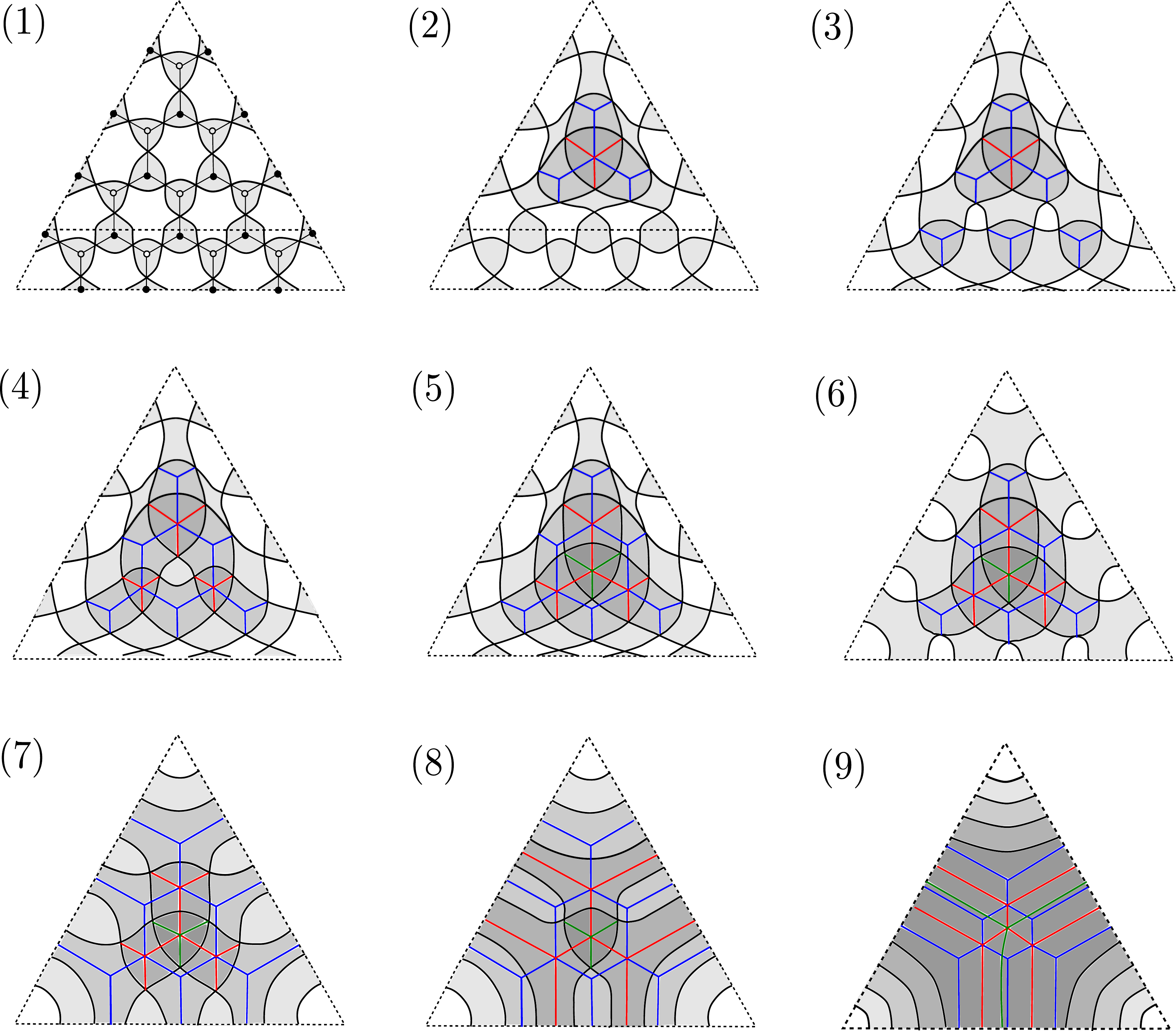}\\
  \caption{Given the correspondence between conjugate Lagrangian fillings of an $n$-triangulation and the Legendrian weave of a $n$-triangulation, how we build the correspondence for $(n+1)$-triangulations. The first two rows are obtained by Hamiltonian isotopies fixing the boundary. The last row is how we do Hamiltonian isotopies near the boundary.}\label{fig:triangle-conjweaveN}
\end{figure}
\end{proof}

    Note that in the correspondence between conjugate Lagrangians and Legendrian weaves, each black vertex in the interior of the bipartite graph corresponds uniquely to a blue trivalent vertex. In the conjugate Lagrangian, the 1-cycles in the interior are given by hexagonal null regions, consisting of three adjacent black vertices, while in the Legendrian weave, the 1-cycles in the interior are given by $\sf Y$-cycles arising from three adjacent blue trivalent vertices. We can also keep track of the 1-cycles of the Lagrangian filling under the Hamiltonian isotopy, as concluded in the following corollary:

\begin{cor}\label{cor:1cycle-triangle}
    Consider the Hamiltonian isotopies of the hybrid Lagrangians in Theorem \ref{thm:main1}. Then the following holds:
    
    \begin{enumerate}
        \item Each relative 1-cycle in the conjugate Lagrangian from a rectangular null region (yellow cycles in Figure \ref{fig:triangle-conjweave3-cycle} left) corresponds to the relative $\sf I$-cycle in the Lagrangian projection of the Legendrian weave (yellow cycles in Figure \ref{fig:triangle-conjweave3-cycle} right) arising from the corresponding blue trivalent vertex.\\
        
        \item Each 1-cycle in the conjugate Lagrangian from hexagonal null regions (orange cycles in Figure \ref{fig:triangle-conjweave3-cycle} left) corresponds to the $\sf Y$-cycle in the Lagrangian projection of the Legendrian weave (orange cycles in Figure \ref{fig:triangle-conjweave3-cycle} right) arising from the corresponding three adjacent blue trivalent vertices.
    \end{enumerate} 
    
    \noindent Furthermore, the null regions around the marked points on $\Sigma$ corresponds to the 1-cycles with lowest levels in the Legendrian weave around the punctures.
\end{cor}
\begin{proof}
    This follows from the algorithm in the proof of Theorem \ref{thm:main1} together with the local characterization of 1-cycles in Corollary \ref{cor:1cycle-loc}. An example is shown in Figure \ref{fig:triangle-conjweave3-cycle}.
\end{proof}
\begin{figure}[h!]
  \centering
  \includegraphics[width=1.0\textwidth]{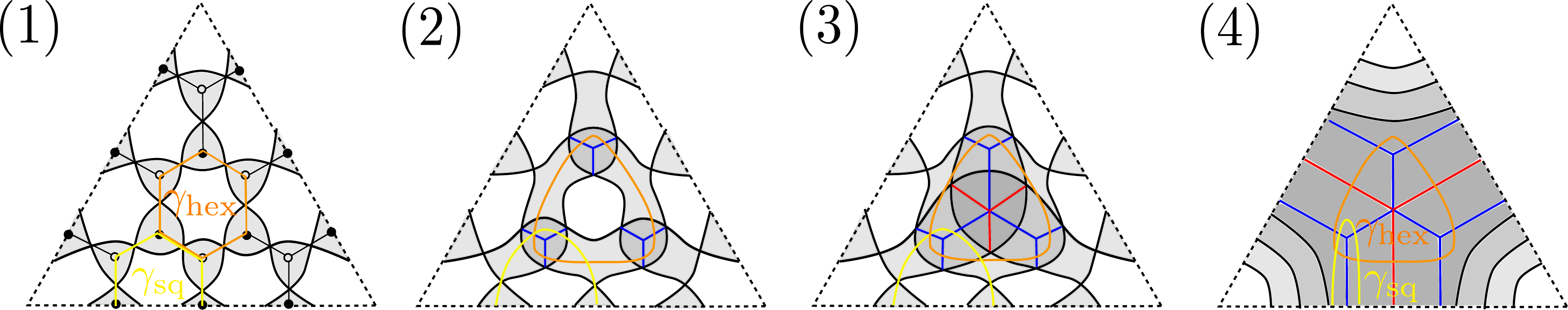}\\
  \caption{The correspondence between 1-cycles in the conjugate Lagrangian and in the Legendrian weave for a 3-triangulation. The yellow cycles arising from the rectangular null regions on the boundary correspond to the relative $\sf I$-cycles, while the orange cycles arising from the hexagonal null regions in the interior correspond to the $\sf Y$-cycles.}\label{fig:triangle-conjweave3-cycle}
\end{figure}

Note that we can also consider the conjugate Lagrangian filling $L(\bG_n)$ associated to the $A_n$-bipartite graph $\bG_n$ of the $n$-triangulation. Since $L(\bG_n) \not\subset T^{*}\Sigma_\text{op}$, this conjugate Lagrangian does not directly come from a free Legendrian weave in $J^1\Sigma_\text{op}$. Nevertheless, we are still able to describe the corresponding Legendrian surface filling of the link $\bigcirc_{n-1}$ in $J^1\Sigma$ as a Legendrian weave  away from neighbourhoods of the marked points. The statement reads:
    
\begin{thm}\label{thm:triangle-partweave}
    For an ideal triangulation on a surface $\Sigma$ with marked points, the conjugate Lagrangian associated to the $A_n$-bipartite graph $\bG_n$ associated to the $n$-triangulation is Hamiltonian isotopic to the Lagrangian projection of the Legendrian $\widetilde{L}(\mathfrak{w}_n)$ associated to the $n$-triangulation as in Figure \ref{fig:triangle-conj-weave-part}.
\end{thm}

\noindent In particular, should we consider the compactification of the conjugate Lagrangian surface, by adding the necessary capping disks at each of its boundary components, then Theorem \ref{thm:triangle-partweave} implies that the resulting Hamiltonian isotopy class of exact Lagrangian in $T^*\Sigma$ admits a representative whose lift is a (non-free) Legendrian weave.

\begin{cor}\label{cor:1cycle-triangle-part}
    Consider the Hamiltonian isotopies of the Lagrangians in Theorem \ref{thm:triangle-partweave}. Then the following holds:
    
    \begin{enumerate}
        \item Each relative 1-cycle in the conjugate Lagrangian from a rectangular null region corresponds to the relative $\sf I$-cycle Legendrian surface.
        
        \item Each 1-cycle in the conjugate Lagrangian from hexagonal null regions corresponds to the $\sf Y$-cycle in the Legendrian surface.
    \end{enumerate} 
\end{cor}
\begin{figure}[h!]
    \centering
    \includegraphics[width=0.85\textwidth]{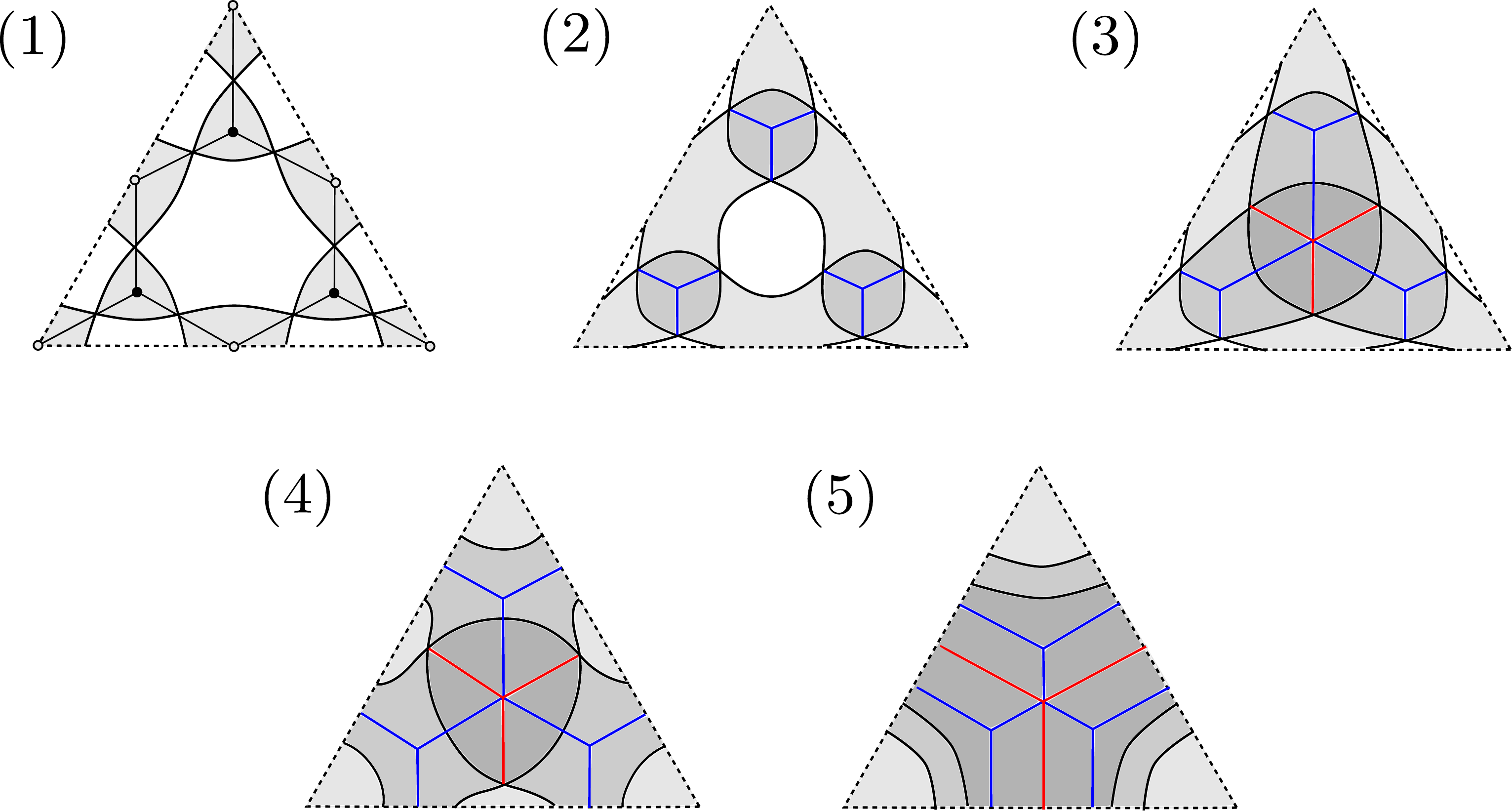}
    \caption{The conjugate Lagrangian filling $L(\bG_n) \subset T^{*}\Sigma$ and the corresponding hybrid Legendrian surface filling of $\bigcirc_{n-1}$, being a Legendrian weave away from neighbourhoods of the marked points.}
    \label{fig:triangle-conj-weave-part}
\end{figure}

\noindent The proofs are exactly the same as Theorem \ref{thm:main1} and Corollary \ref{cor:1cycle-triangle}.
    
    \subsection{From pinching sequences to weaves}\label{ssec:proof_pinching_to_weave}    Let us now prove that embedded exact Lagrangian fillings for Legendrian positive braid closures obtained through pinching sequences are Hamiltonian isotopic to (Lagrangian projections of) Legendrian weaves. The result reads as follows:
    
    \begin{thm}\label{thm:braid-weavepinching}
    Let $\beta \in \mbox{Br}_n^+$ be a positive braid and $\Delta$ be the half twist, $\Lambda_{\beta\Delta^2} \subset T^{*,\infty}\bD^2$ be the Legendrian cylindrical braid closure of $\beta\Delta^2$, and $\Lambda_{\beta}^\prec \subset \bD^3 \subset T^{*,\infty}\bD^2$ the Legendrian rainbow braid closure of $\beta$. Consider a pinching sequence $\sigma\in S_{\ell(\beta)}$.\\
    
    Then the embedded exact Lagrangian filling $L_{\sigma}$ associated to the pinching sequence $\sigma$ is Hamiltonian isotopic to the Lagrangian projection of a Legendrian weave. In fact, the weave associated to the Hamiltonian isotopy class given by $\sigma=\mbox{id}$, i.e.~left to right pinching, coincides with the weave associated to the plabic fence $\bG_\beta$ given by $\beta$.

\end{thm}

\begin{proof}
    Consider $\sigma=\mbox{id}$. Following Example \ref{ex:braid-pinch}, each crossing in $\beta$ in the Lagrangian projection corresponds to the Reeb chord in the unique region on the left of the corresponding crossing in $\beta$ in the front projection. We prove that each Legendrian weave $\mathfrak{c}_i^\uparrow(\mathtt{w}_{0,n})$ defines the elementary cobordism pinching the leftmost Reeb chord.
    
    In the front projection, in order to pinch the Reeb chord corresponding to the crossing $s_i$ in the leftmost closed bigon region, we first need to perform a sequence of Reidemeister~III moves to the strands crossing the closed region such that they are moved away from the closed region. The closed bigon region is bounded by the $i$-th $s_1$-crossing in the half twist and the $s_i$-crossing. The strands crossing this region are the 1-st to the $i$-th strand. See Figure \ref{fig:braid-pinch}~(1). These Reidemeister~III moves introduce a number of hexagonal vertices in the weave following Figure \ref{fig:A_1_3-RIII}, bringing the $i$-th $s_1$-strand to the top, as depicted in Figure \ref{fig:braid-pinch}~(2)--(3). The resulting weave is $\mathfrak{n}_i^\uparrow(\mathtt{w}_{0,n})$. Then we pinch the Reeb chord in the closed region, as explained in Subsection \ref{ssec:weave_pinching} (see Figure \ref{fig:D4-saddle}), which introduces an $s_i$-trivalent vertex in the weave, as in Figure \ref{fig:braid-pinch}~(4). The resulting weave for the elementary cobordism is exactly $\mathfrak{c}_i^\uparrow$. Finally, we need Reidemeister~III moves to bring the braid word back to the half twist, whose resulting weave is $\mathfrak{n}_i^\uparrow(\mathtt{w}_{0,n})^\text{op}$, as in Figure \ref{fig:braid-pinch}~(5). The concatenation of these weaves is exactly $\mathfrak{c}_i^\uparrow(\mathtt{w}_{0,n})$.
    
\begin{figure}[h!]
    \centering
    \includegraphics[width=1.0\textwidth]{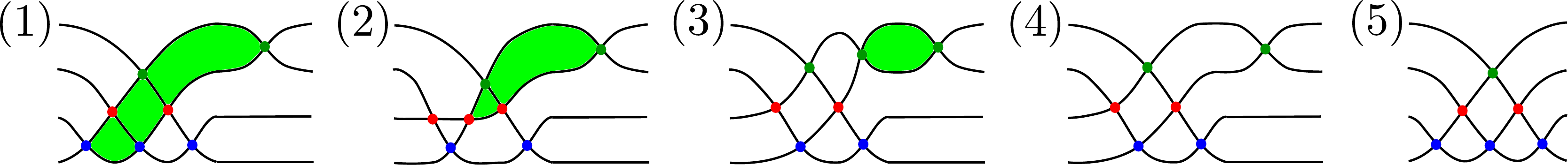}
    \caption{The Reeb chord pinching in the front projection. (1)~The Legendrian front of the braid word $\Delta s_i$. (2)~The front after a Reidemeister~III move emanating $s_1$ and $s_2$-crossings (introducing a hexagonal vertex in the weave). (3)~The front after a Reidemeister~III move emanating $s_2$ and $s_3$-crossings (introducing a hexagonal vertex). (4)~The front after Reeb chord pinching (introducing a trivalent vertex). (5)~The front after applying Reidemeister~III moves backwards.}\label{fig:braid-pinch}
\end{figure}
    
    In the Lagrangian projection, following \cite[Section 4]{Hughes21A}, the above elementary cobordism can be realized by a single Reeb chord pinching with no Reidemeister moves in the Lagrangian projection. The key observation is that in the front projection, none of the bigons (in which there exists a Reeb chord) are contained in the bigon on the left of the first crossing in $\beta$ (in which there exists a Reeb chord corresponding to the first crossing of $\beta$ in the Lagrangian projection). Therefore, under small perturbation of the front in Figure \ref{fig:braid-pinch}.(1), we may assume that the Reeb chords in other bigons are moved away from the bigon on the left of the first crossing in $\beta$. That means all Reeb chords corresponding to the crossings in the half twist $\Delta$ are beyond the left of the figure while all chords corresponding to the crossings after $s_i$ in $\beta$ are beyond the right of the figure. Hence we may assume that the Lagrangian projection is on the left of Figure \ref{fig:braid-pinch-lag}. Therefore, all the Reidemeister~III moves in the front do not change the Lagrangian projection. At this stage, the Reeb chord pinching as presented in in Subsection \ref{ssec:weave_pinching} does not introduce any additional Reidemeister moves in the Lagrangian projection. This is shown explicitly on the right of Figure \ref{fig:braid-pinch-lag}. Therefore, no Reidemeister moves are introduced in the Lagrangian projection.\\
    
\begin{figure}[h!]
    \centering
    \includegraphics[width=0.9\textwidth]{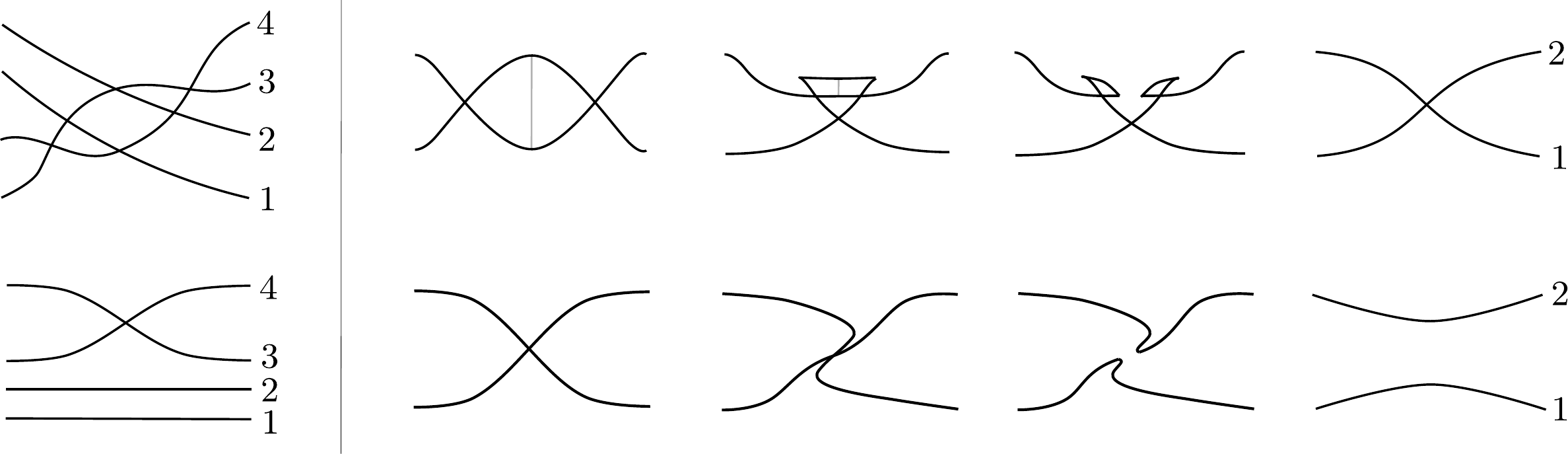}
    \caption{The front projection (on the top) and the corresponding Lagrangian projection (on the bottom) in Reeb chord pinching.}\label{fig:braid-pinch-lag}
\end{figure}
    
    \noindent Finally, we claim that the weave $\mathfrak{n}(\mathtt{w}_{0,n})$ defines the minimal cobordism of the Legendrian unlink. Indeed, since in $T^{*,\infty}\bD^2$, the Legendrian knot  on the boundary is isotopic to the Legendrian unlink, and it follows from \cite{EliPol96} that $\mathfrak{n}(\mathtt{w}_{0,n})$ has to define the minimal embedded Lagrangian cobordism.\\
    
    For a general permutation $\sigma \in S_{\ell(\beta)}$, to show that $L_\sigma$ is Hamiltonian isotopic to a Lagrangian projection of a Legendrian weave, it suffices to notice that a pinching sequence defines a decomposable Lagrangian filling, which in the front projection is a sequence of pinching saddle cobordisms, Lagrangian 0-handes, Legendrian Reidemeister III moves. Since each of these elementary cobordisms can be realized by Legendrian weaves (Section \ref{ssec:weave_pinching}), we can conclude that the Lagrangian filling is Hamiltonian isotopic to the projection of a weave.
\end{proof}

\noindent Reeb pinching sequences can be combined with Legendrian isotopies to produce new embedded exact Lagrangian fillings \cite{CasalsNg21}. The hexavalent vertices in a weave capture Legendrian Reidemeister III moves. Thus, in general, the above argument shows that a decomposable Lagrangian filling obtained via pinching saddle cobordisms, Lagrangian 0-handes, Legendrian Reidemeister III moves and cyclic rotation (of the braid word) is Hamiltonian isotopic to the Lagrangian projection of a Legendrian weave.\footnote{In general, Reidemeister moves might be required in the Lagrangian projection.}


\subsection{Numerical comparisons on Lagrangian fillings with each method} For any Legendrian $(2, k)$-torus link, the number of graphical conjugate Lagrangian fillings is the same as the number of free Legendrian weaves. In general, the number of conjugate Lagrangian fillings can be strictly less than the number of free Legendrian weaves.

 \noindent For a Legendrian positive $(n, k)$-torus link, the number of graphical conjugate Lagrangian fillings corresponds to the number of certain reduced plabic graphs of an $(k+n)$-gon (Example \ref{ex:braid-conj-red}), which is also the number of cluster charts in $\mathrm{Gr}(n, k+n)$ defined by Pl\"{u}cker coordinates \cite{OPS15WeakSep} (see Section \ref{sec:framed-locsys}). Indeed, each Lagrangian filling defines a distinguished cluster chart in the moduli space of sheaves (which we discuss in Section \ref{sec:quantize}), and fillings that define different charts can be shown to be different; see \cite{GaoShenWeng20,CasalsWeng22}.\footnote{A precise relation between Lagrangian fillings and cluster varieties is conjectured in \cite[Conjecture 5.1]{Casals20Sk}; we will discuss it partially in Section \ref{sec:cluster}.}. The number of Legendrian weaves is also well studied, here are some explicit examples:

\begin{ex}\label{ex:k,ntorus}
    For Legendrian $(n, k)$-torus links, where either $n \geq 3, k \geq 6$ or $n, k \geq 4$, the number of graphical conjugate Lagrangian fillings (which is the number of reduced plabic graphs of a $(k+n)$-gon) is always finite. In contrast, we showed in  \cite{CasalsGao20} that there are infinitely many free Legendrian weaves that fill the torus link; see also \cite{CasalsZas20}*{Theorem 7.6} and \cite{CasalsWeng22}. Hence the number of graphical conjugate Lagrangians is strictly smaller than the number of free Legendrian weaves.\hfill$\Box$
\end{ex}

\begin{ex}\label{ex:3,3torus}
    For the Legendrian $(3, 3)$-torus link, there are 34 graphical conjugate Lagrangian fillings, corresponding to the number of Pl\"{u}cker-type cluster coordinates in $\mathrm{Gr}(3, 6)$ \cite{KK21Gr37}*{Theorem 4.2}. In constrast, \cite{Hughes21D} shows that the number of different free Legendrian weaves is (at least) 50, which is the number of all cluster charts in $\mathrm{Gr}(3, 6)$. Conjecturally, all conjugate Lagrangian fillings should be Hamiltonian isotopic to the Lagrangian projection of a Legendrian weave.\\

    \noindent We can also consider Lagrangian fillings coming from pinching sequences, i.e.~concatenations of standard local models of 1-handle attachments and 0-handle attachments (with no Reidemeister moves in the Lagrangian projection). For different Lagrangian projections of the Legendrian link, i.e.~braid representatives of the $(3, 3)$-torus link, the number of Lagrangian fillings coming from pinching sequences will be different. For example, for the braid $(s_1s_2)^3$, the number of different fillings given by pinching sequences is 46, while for the braid $(s_1s_2s_2)^2$ or $(s_2s_1s_1)^2$, the number is 42. Again, all these fillings coming from pinching sequences are Hamiltonian isotopic to the Lagrangian projection of a Legendrian weave, but there are strictly more Legendrian weave fillings.\hfill$\Box$
\end{ex}

\begin{ex}\label{ex:3,4torus}
    For the Legendrian $(3, 4)$-torus link, there are 259 graphical conjugate Lagrangian fillings, in bijection with the number of Pl\"{u}cker-type cluster coordinates in $\mathrm{Gr}(3, 7)$; see \cite{KK21Gr37}*{Theorem 5.2}. In contrast, \cite{AnBaeLee21ADE} shows that the number of different free Legendrian weaves is (at least) 833, which is the number of all cluster charts in $\mathrm{Gr}(3, 7)$). Finally, the number of Lagrangian fillings coming from pinching sequences for the braid $(s_1s_2)^4$ is 633, which we have verified by computer.\footnote{We thank D.~Weng for sharing with us his code computing the different clusters that can be obtained by pinching Reeb chords in positive braids.} In consequence, free Legendrian weaves still provide the maximal number of Lagrangian fillings.
\end{ex}

\section{Sheaf Quantization of Conjugate Fillings and Legendrian weaves}\label{sec:quantize}

    Sections \ref{sec:prelim}, \ref{sec:localmove} and \ref{sec:main_proofs} prove results about Hamiltonian isotopy classes of Lagrangian fillings. In order to connect the comparison results between Lagrangian fillings to the comparison result of cluster coordinates in Section \ref{sec:cluster}, we need to first enhance the geometric objects, i.e.~Lagrangian fillings, with some algebraic structures, i.e.~constructible sheaves or microlocal sheaves. We use the notation from Appendix \ref{appen:cat-sheaf}.\\
    
    Let $L\sse (T^*\Sigma,\la_{st})$ be an exact Lagrangian orientable surface with non-empty Legendrian boundary $\dd_\infty L\sse(T^{*,\infty}\Sigma,\ker(\la_{st}))$, and $\mathcal{L}\in \Loc^{pp}(L)$ a local system with perfect stalks. Consider the category of sheaves singularly supported on $\widetilde{L}$ with perfect stalks $\Sh_{\widetilde{L}}^{pp}(\Sigma \times \bR)$ and the microlocal functor $m_{\widetilde{L}}:\Sh_{\widetilde{L}}^{pp}(\Sigma \times \bR)\lr\Loc^{pp}(L)$, where $\widetilde{L}\sse (J^1\Sigma,\ker(dz-\la_{st}))$ is the Legendrian lift of $L$.\footnote{Here we have trivially identified $\Loc(\widetilde{L})=\Loc(L)$, as local systems only depend on $L$ itself and not its Hamiltonian isotopy class.} The first precise definition we need is the following:
    
    \begin{definition}\label{def:sheafquant}
        Let $L\sse (T^*\Sigma,\la_{st})$ be an exact Lagrangian orientable surface with non-empty boundary, with Legendrian lift $\widetilde{L} \in (J^1\Sigma, \xi_{st})$, and $\mathcal{L}\in \Loc^{pp}(L)$ a local system. A sheaf $\mathcal{F}\in \Sh_{\widetilde{L}}^{pp}(\Sigma \times \bR)$ is said to be a {\it sheaf quantization} of $(L,\mathcal{L})$ if $m_{\widetilde{L}}(\mathcal{F})=\mathcal{L}$. By definition, a sheaf quantization functor of $L$ is a functor of dg-categories
        $$\psi_{\widetilde{L}}: \Loc^{pp}(L) \longrightarrow \Sh_{\widetilde{L}}^{pp}(\Sigma \times \bR)$$
        such that $\psi_{\widetilde{L}} \circ m_{\widetilde{L}} = \mathrm{id}_{\Loc(\Lambda)}$.
    \end{definition}
    
    \noindent By Theorem \ref{thm:JinTreumann}.(1) below, there exists a sheaf quantization functor for Legendrian lifts of Lagrangian fillings. In addition, such functor is fully faithful. This sheaf quantization functor admits a left-adjoint which is between the corresponding categories of compact objects (in the unbounded dg-categories) 
    $$\psi_{\widetilde{L}}^L:\Sh^{cpt}_{\widetilde{L}}(\Sigma \times \bR)\lr\Loc^{cpt}(L).$$
    Thus, a sheaf quantization of $L$ leads to an open embedding
    $$\bR \mathcal{M}\big(\psi_{\widetilde{L}}^L\big):\bR\cL oc(L) \hookrightarrow \bR\mathcal{M}\big(\Sigma \times \bR, \widetilde{L}\big)$$
    between the associated derived stacks of pseudo-perfect objects in $\Sh^{cpt}_{\widetilde{L}}(\Sigma \times \bR)$ and $\Loc^{cpt}(L)$ \cite[Section 3]{ToenVaquie07}. Note that these derived stacks parametrize objects in the subcategories $\Loc^{pp}(L)$ and $\Sh^{pp}_{\widetilde{L}}(\Sigma \times \bR)$.
    

    \noindent By Theorem \ref{thm:JinTreumann}.(2), the proper push-forward $\pi_{\Sigma,!}$ of the projection $\Sigma\times\R\lr\R$ yields a fully faithful functor $$\pi_{\Sigma,!}:\Sh_{\widetilde{L}}^{pp}(\Sigma \times \bR)\lr\Sh_\La^{pp}(\Sigma),$$
    where $\La:=\dd_\infty L\sse(T^{*,\infty}\Sigma,\ker(\la_{st}))$ is the ideal Legendrian boundary of $L$. The functor $\pi_{\Sigma,!}$ admits a left-adjoint which is between the corresponding categories of compact objects (in the unbounded dg-categories)  $$\pi_{\Sigma,!}^L:\Sh^{cpt}_\La(\Sigma)\lr\Sh^{cpt}_{\widetilde{L}}(\Sigma \times \bR)$$
    and thus we obtain an open embedding
    of the corresponding derived stacks:
    $$\bR \mathcal{M}\big(\pi_{\Sigma,!}^L\big): \bR \mathcal{M}\big(\Sigma \times \bR, \widetilde{L}\big) \hookrightarrow \bR \mathcal{M}(\Sigma, \La).$$
    Again, these moduli parametrize objects in the subcategories $\Sh^{pp}_{\widetilde{L}}(\Sigma\times\R)$ and $\Sh^{pp}_\La(\Sigma)$. By restricting to Abelian local systems of (differential graded) perfect $\Bbbk$-modules, the composition $\mathcal{M}\big(\psi_{\widetilde{L}}^L\circ \pi_{\Sigma,!}^L\big)$ yields two open embeddings of algebraic stacks
    $$H^1(L; \Bbbk^\times) \hookrightarrow \cM_1(\Sigma, \Lambda)_0, \;\mbox{and}\; H^1(L\setminus T, \Lambda\setminus T; \Bbbk^\times) \hookrightarrow \cM_1^{\mu,\textit{fr}}(\Sigma, \Lambda)_0,$$
    \noindent where $\cM_1(\Sigma, \Lambda)_0\sse \bR\mathcal{M}(\Sigma, \La)$ is the moduli space of microlocal rank~1 sheaves singular support on the Legendrian and vanishing stalk near them marked points of $\Sigma$, and $\cM_1^{\mu,\textit{fr}}(\Sigma, \Lambda)_0$ its framed version with base points $T\sse \La$, see Appendix \ref{appen:sheaf-moduli}. In Section \ref{sec:cluster} we explain that these embeddings define cluster charts.\\
    
    
    \begin{remark}
        In general, the definition of sheaf quantization (and a sheaf quantization functor) is more general than Definition \ref{def:sheafquant}, in that $\Loc^{pp}(L)$ is substituted by the global sections of the Kashiwara-Schapira stack on the Lagrangian cone over the Legendrian lift of $L$; see Appendix \ref{appen:sheaf}. Indeed, it is not generally the case that this category of global sections is equivalent to $\Loc^{pp}(L)$: see the obstructions discussed in \cite[Section 10]{Gui19}. That said, for an orientable exact Lagrangian surface with non-empty Legendrian boundary these obstruction vanish, and we can equivalently work with $\Loc^{pp}(L)$. (In the main body of the section, in order to better keep track of signs, we will work over dg-derived category of twisted local systems $\Loc^{pp}_\epsilon(\cI_{\widetilde{L}})$ on the frame bundle $\cI_{\widetilde{L}}$. See Section \ref{ssec:muloc_coherentsigns}.)\hfill$\Box$
    \end{remark}
    
    In brief, the goal in this section is to describe sheaf quantizations for Lagrangian fillings through explicit combinatorial models in the two cases of conjugate Lagrangian surfaces and Legendrian weaves. The necessary subtleties and results about sheaf categories and constructible sheaves are provided in Appendices \ref{appen:cat-sheaf} and \ref{appen:sheaf}.
    
\subsection{Generalities on constructible sheaves and sheaf quantization}
Let $\Sh(M)$ be the dg-derived category of sheaves on a smooth manifold $M$. In \cite[Chapter V]{KSbook}, M.~Kashiwara and P.~Schapira introduced the notion of singular support of a sheaf in $\Sh(M)$ (Definition \ref{def:ss}) and show in \cite[Theorem 6.5.4]{KSbook} that the singular support of is a coisotropic subset in $T^{*,\infty}M$. By \cite[Theorem 8.4.2]{KSbook}, if the singular support of a sheaf is a (possibly singular) Legendrian in the ideal contact boundary $T^{*,\infty}M$, then that sheaf itself must be a constructible sheaf.\\
    
Let $\Lambda \subset T^{*,\infty}M$ be a Legendrian submanifold. The front projection $\pi(\Lambda) \subset M$ gives a stratification of $M$ and a (constructible) sheaf in $\Sh_\Lambda^{pp}(M)$ is a stratified locally constant sheaf with respect to the stratification $\pi(\Lambda)$. The co-directions of the Legendrian $\Lambda$ determine directions of transport maps between (derived) sections of such sheaves. Discrete, rather combinatorial, descriptions of such sheaves can be obtained by applying work of R.~MacPherson, see \cite[Section 3.3]{STZ17}, \cite[Section 4.1]{CasalsWeng22} and references therein. For instance, \cite{STZ17}*{Section 3} discusses the following examples, which cover the case of Legendrian links:\\

\begin{ex}\label{ex:sheafcombin}
    In the case that $\pi(\Lambda)$ is locally a hyperplane in $\bR^{n+1}$ as in Figure \ref{fig:sheaf-combin} (left), with covectors pointing upward, then the sheaf is locally determined by the following diagram 
    $$F_d \longrightarrow F_u$$
    of sheaves of chain complexes of $\Bbbk$-modules. In the case that $\pi(\Lambda)$ is locally two transversely intersecting hyperplanes in $\bR^{n+1}$ as in Figure \ref{fig:sheaf-combin} (middle), covectors both pointing upward, then the sheaf is locally determined by a diagram 
    \[\xymatrix@R=2.5mm{
    & F_l \ar[dr]^{\phi_{lu}} & \\
    F_d \ar[ur]^{\phi_{dl}} \ar[dr]_{\phi_{dr}} & & F_u \\
    & F_r \ar[ur]_{\phi_{ru}} &
    }\]
    where the total complex $\mathrm{Tot}\big(F_{d} \xrightarrow{(\phi_{dl}, \phi_{dr})} F_{l} \oplus F_{r} \xrightarrow{(-\phi_{lu}, \phi_{ru})} F_{u}\big) \simeq 0$ is required to be acyclic; see \cite{STZ17}*{Theorem 3.12}. Finally, in the case that $\pi(\Lambda)$ is locally a smooth family of cusps in $\bR^{n+1}$, as in Figure \ref{fig:sheaf-combin} (right), with covectors pointing upward, the sheaf is locally determined by a diagram 
    $$F_o \xrightarrow{\phi_d} F_i \xrightarrow{\phi_u} F_o$$
    where $\phi_u \circ \phi_d$ is a quasi-isomorphism. In the examples we study, namely Legendrian rainbow braid closures, we can also assume that $\phi_u \circ \phi_d = \mathrm{id}$ \cite{STZ17}*{Proposition 3.22}.\hfill$\Box$
\end{ex}
\begin{figure}[h!]
  \centering
  \includegraphics[width=0.8\textwidth]{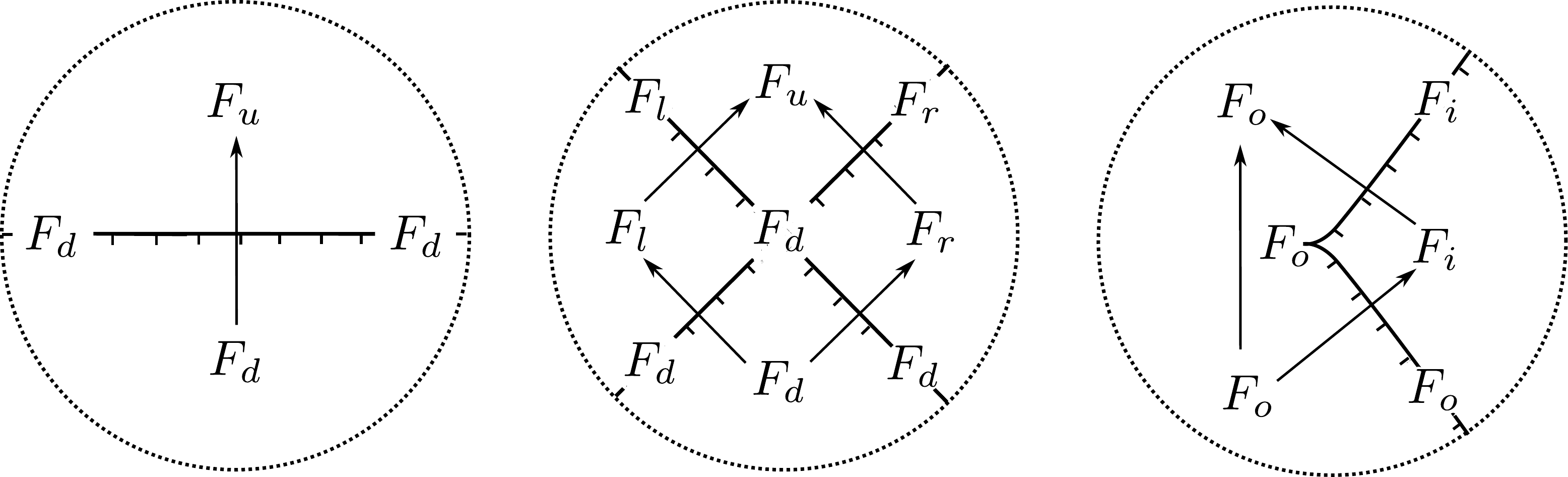}\\
  \caption{The local combinatoric models for sheaves in $Sh^b_\Lambda(\bR^{n+1})$ where on the left $\pi(\Lambda)$ is a hyperplane, in the middle $\pi(\Lambda)$ is two transverse hyperplanes, and on the right $\pi(\Lambda)$ is a smooth family of cusps.}\label{fig:sheaf-combin}
\end{figure}

\subsubsection{Combinatorial microlocal functor}\label{sec:microlocal}
    In the combinatorial description of Example \ref{ex:sheafcombin}, the functor $m_\La:\Sh_\La^{pp}(M)\lr\Loc^{pp}(\La)$ can be also computed explicitly, following \cite{STZ17}*{Section 5}, as follows. (For the general framework, see Appendix \ref{appen:microlocal}). Consider a Legendrian $\Lambda \subset T^{*,\infty}M$, and a generic point $p \in \Lambda$ such that $\pi(\Lambda)$ is a smooth hypersurface near $\pi(p)$, as in Figure \ref{fig:sheaf-combin} (left). Then the microstalk of $\mathcal{F}\in\Sh_\La^{pp}(M)$ at $p$ can be computed as
    $$m_{\Lambda,p}(\cF) = \mathrm{Cone}(F_d \rightarrow F_u).$$
    Given a path $\gamma\sse\Lambda$ connecting two points $p, q\in \Lambda$, the cone construction above yields a local system $m_{\Lambda, \gamma}(\cF)$, often referred to as the microlocalization of $\cF$ along $\gamma$, and thus the microstalk of $\cF$ is independent of the point we choose in $\Lambda$. By definition, a sheaf $\cF$ is said to be of microlocal rank $r$ sheaf if the microstalk is concentrated in a single degree and has rank $r$.

\begin{ex}\label{ex:mumon-STZ}
    Consider a path $\gamma$ in $\Lambda$ such that $\pi(\gamma)$ passes through a smooth family of transverse double points, as in Figure \ref{fig:sheaf-combin} (right). Then the parallel transport map from bottom right to top left is computed by
    $$m_{\Lambda,dr}(\cF) =  \mathrm{Cone}\big(F_{d} \xrightarrow{\phi_{dr}} F_{r}\big) \xrightarrow{\sim} \mathrm{Cone}\big(F_{l} \xrightarrow{\phi_{lu}} F_{u}\big) = m_{\Lambda,ul}(\cF).$$
    Similarly, the parallel transport map from top right to bottom left is computed by
    $$m_{\Lambda,ur}(\cF) =  \mathrm{Cone}\big(F_{r} \xrightarrow{\phi_{ru}} F_{u}\big) \xrightarrow{\sim} \mathrm{Cone}\big(F_{d} \xrightarrow{\phi_{dl}} F_{l}\big) = m_{\Lambda,dl}(\cF).$$
    In general, we equip the front diagram $\pi(\Lambda)$ with a Maslov potential and the microstalk carries a degree shift by the Maslov potential. 
\end{ex}

    In the sheaf quantization of Lagrangian fillings, we need to study both the case that $\La$ is a Legendrian link and the case that $\La$ is a Legendrian surface (the Legendrian lift of a Lagrangian filling of a Legendrian link). In higher dimension, including Legendrian surfaces, the signs of microlocal monodromies need to be fixed and that requires fixing a spin structure on the 2-skeleton of the Legendrian. In particular, the sign issue appears  when we compute microlocal merodromy along relative 1-cycles (see \cite[Section 4.5]{CasalsWeng22}) so as to obtain cluster variables over $\Z$. This is discussed in Appendix \ref{appen:microlocal}.

\subsubsection{Existence of sheaf quantization for Lagrangian fillings}
 Let $\Sh_\Lambda^{pp}(M)\sse\Sh_\Lambda(M)$ be the full subcategory of constructible sheaves with singular support in $\Lambda$ and $\Bbbk$-perfect stalks, and consider the dg-category $\Loc^{pp}_\epsilon(\cI_{\widetilde{L}})$ of pseudo-perfect twisted local systems on the (relative) frame bundle $\cI_{\widetilde{L}}$ defined in Appendix \ref{appen:microlocal}. The superscript {\it pp} means pseudo-perfect and see Appendix \ref{appen:cat-sheaf} for descriptions of these categories. By \cite[Theorem 3.7]{GKS12}, the dg-category $\Sh^{pp}_\Lambda(M)$ is an invariant of a Legendrian $\La$ under Legendrian isotopies; see Theorem \ref{thm:GKS} and Appendix \ref{appen:cat-sheaf}.(b).\\
    
    \noindent Results of X.~Jin and D.~Treumann \cite{JinTreu17}, following S.~Guillermou's \cite{Gui12}, can be adapted to show that a sheaf quantization result exists for Lagrangian fillings, explaining how Lagrangian fillings of $L$ equipped with local systems gives rise to sheaves in $\Sh^{pp}_\Lambda(M)$.
    
\begin{thm}[Theorem \ref{thm:JinTreu-app}]\label{thm:JinTreumann}
    Let $\Lambda \subset T^{*,\infty}M$ be a Legendrian and $L \subset T^*M$ a relatively spin exact Lagrangian filling of $\La$ with zero Maslov class. Then there exists Legendrian lift $\widetilde{L} \subset J^1(M)$ of $L \subset T^*M$ whose primitive is bounded from below and\\
\begin{enumerate}
  \item There exists a fully faithful functor $\Psi_{\widetilde{L}}: \Loc^{pp}_\epsilon(\cI_{\widetilde{L}}) \rightarrow \Sh^{pp}_{\widetilde{L}}(M \times \bR)$ such that the stalk at $M \times \{-\infty\}$ is acyclic and $m_{\widetilde{L}} \circ \Psi_{\widetilde{L}} \simeq \mathrm{id}$.\\
  \item The proper push forward functor $\pi_{M,!}: \Sh^{pp}_{\widetilde{L}}(M \times \bR) \rightarrow Sh^{pp}_\Lambda(M)$ via the projection $\pi_M: M \times \bR \rightarrow M$ is fully faithful.
\end{enumerate}
\end{thm}

Let $\Pi_M:(J^1M,\xi_{st})\lr (T^*M,\la_{st})$ be the projection forgetting the $j_0$-value of a section.

\begin{prop}[Theorem \ref{prop:JinTreu-inv-app}]\label{prop:JinTreu-inv} Let $L \subset (T^*M,\la_{st})$ be a relatively spin exact Lagrangian filling of a Legendrian submanifold $\La\subset T^{*,\infty}M$, $\widetilde{L}\subset(J^1M,\xi_{st})$ its Legendrian lift and $\cL \in \Loc^{pp}_\epsilon(\cI_{\widetilde{L}})$ a twisted local system. Consider a Hamiltonian  $H:T^*M\lr\R$ that is homogeneous at infinity, $\varphi_t\in \mbox{Ham}(T^*M,d\la_{st})$ its time-$t$ flow, $t\in[0,1]$, which extends to an homonymous contactomorphism $\varphi_t\in \mbox{Cont}(T^{*,\infty}M,\ker\la_{st})$ of the ideal contact boundary, $t\in[0,1]$, $\widetilde H=\Pi^*_M(H):J^1M\lr\R$ the pull-back and $\widetilde{\varphi}_t\in\mbox{Cont}(J^1M,\xi_{st})$ its time-$t$ flow.\\

\noindent Suppose that $\widetilde{\cF}_t \in \Sh^{pp}_{\widetilde{L}_t}(M \times \bR)$ is a sheaf quantization of $(\varphi_t(L),\varphi_{t *}\cL)$ and consider the proper push-forward $\cF_t := \pi_{M,!}(\widetilde{\cF}_t) \in \Sh^{pp}_{\varphi_t(\Lambda)}(M)$, $t\in[0,1]$. Then:

\begin{itemize}
    \item[(i)] The sheaf kernel convolution $\Phi_{\widetilde{H}}:\Sh^{pp}_{\widetilde{L}}(M \times \bR)\lr\Sh^{pp}_{\varphi_1(\widetilde{L})}(M \times \bR)$ associated to $\widetilde{H}$ satisfies $\Phi_{\widetilde{H}}(\widetilde{\cF}_0) \simeq \widetilde{\cF}_1$.\\
    
    \item[(ii)] The sheaf kernel convolution $\Phi_H:\Sh^{pp}_{\Lambda}(M)\lr\Sh^{pp}_{\varphi_1(\Lambda)}(M)$ associated $H$ satisfies $\Phi_H(\cF_0) \simeq \cF_1.$
\end{itemize}

\end{prop}

\noindent The convolutions of sheaf kernels $\Phi_H$ and $\Phi_{\widetilde{H}}$ from \cite{GKS12} are described in Appendix \ref{appen:sheaf-quan}. In brief, Theorem \ref{thm:JinTreumann} states that a sheaf quantization for a Lagrangian filling exists, and Proposition \ref{prop:JinTreu-inv} shows that such sheaf quantization is invariant under Hamiltonian isotopies. The two statements above are slightly different from the results presented in \cite{JinTreu17}. In particular, we assume that the front of the Legendrian $\tilde{L}$ is bounded from below, while Jin-Treumann assume boundedness from above, and we use the proper push forward by the projection $\pi_{M !}$, while Jin-Treumann use the push forward $\pi_{M *}$. (Our choices are given by the geometry of Lagrangian surfaces in our setting.) Therefore, we technically need to provide proofs of the above results, which can be obtained by appropriately modifying the arguments in \cite{JinTreu17}. These proofs are written in Appendix \ref{appen:sheaf-quan}.
    
\begin{remark}
The Hamiltonian invariance in Proposition \ref{prop:JinTreu-inv}, which allows for the Legendrian boundary condition $\varphi_t(\La)$ to vary, is needed in our context. The fixed boundary condition in \cite[Section 3.20]{JinTreu17} is not enough because in our setting (see Sections \ref{sec:localmove} and \ref{sec:main_proofs}) the Legendrian boundaries change, via a contact isotopy, when the Lagrangian fillings change, via a Hamiltonian isotopy. In addition, we must also use the Hamiltonian invariance statement for $\cF_t \in \Sh^{pp}_{\varphi_t(\Lambda)}(M)$, rather than just the invariance of the sheaf quantization $\widetilde{\cF}_t \in \Sh^{pp}_{\widetilde\varphi_t(\widetilde{L})}(M)$. Indeed, the toric charts for the cluster structures we use are defined by a composition
    $$H^1(\varphi_t(L); \Bbbk^\times) \hookrightarrow \cM_1(M \times \bR, \widetilde\varphi_t(\widetilde{L}))_0 \hookrightarrow \cM_1(M, \varphi_t(\Lambda))_0,$$
into a derived stack of moduli of (pseudo-perfect) objects of $\Sh^{cpt}_{\varphi_t(\Lambda)}(M)$, and not just in $\Sh^{cpt}_{\widetilde{\varphi}_t(\widetilde{L})}(M\times\R)$.
  Thus, we need to show that the toric chart itself is a Hamiltonian invariant and we do so by showing the invariance of the composition above.
\end{remark}
    
\subsubsection{The moduli space of microframed sheaves}\label{sec:sheaf-moduli}
    Let $\Sigma$ be a smooth surface and $\Lambda \subset T^{*,\infty}\Sigma$ a Legendrian link. By \cite{ToenVaquie07}, there exists a moduli derived stack $\cM(\Sigma, \Lambda)$ which parametrizes pseudo-perfect objects in the smooth dg-category $\Sh^{cpt}_\Lambda(\Sigma)$. These can be identified with objects in $\Sh^{pp}_\Lambda(\Sigma)$, which are constructible sheaves with perfect stalks; see Appendix \ref{appen:cat-sheaf} and Definition \ref{def:moduli-sh}. In particular, if we fix a Maslov potential, we can consider the locus $\cM_r(\Sigma, \Lambda)_0 \sse \cM(\Sigma, \Lambda)$ of microlocal rank~$r$ sheaves, concentrated in a given degree, with acyclic stalks at the marked points of $\Sigma$, and respectively the locus $\cM_r(\Sigma \times \bR, \widetilde{L})_0 \sse \cM(\Sigma \times \bR, \widetilde{L})$ of microlocal rank~$r$ sheaves, concentrated in a given degree, with acyclic stalks at $M \times \{-\infty\}$. For the Legendrian links we study, these microlocal rank~$r$ sheaves have no negative self-extension groups, and thus the moduli space $\cM(\Sigma, \Lambda)_0$ is an Artin stack, though typically not an algebraic variety. By Theorem \ref{thm:JinTreumann}, one can construct an open embedding of moduli stacks
    $$\cL oc_r(L) \hookrightarrow \cM_r(\Sigma \times \bR, \widetilde{L})_0 \hookrightarrow \cM_r(\Sigma, \Lambda)_0$$
    which define open toric charts on the moduli $\cM_r(\Sigma, \Lambda)_0$; see Appendix \ref{appen:sheaf-moduli}. As shown in \cite{CasalsWeng22}, this Artin stack does only typically admit a cluster $\mathcal{X}$-structure. In order to obtain a better suited moduli space, we add the data of certain framings, as follows.\\

   Let $\Lambda \subset T^{*,\infty}\Sigma$ be a Legendrian link equipped with a Maslov potential and choose a finite set of base points $T = \{p_1, \dots , p_n\} \subset \Lambda$. By definition, the microframed moduli of microlocal rank~$r$ sheaves is defined as
    $$\cM_r^{\mu,\textit{fr}}(\Sigma, \Lambda)_0 := \left\{(\cF, t_1, \dots, t_n) | \cF \in \cM_r(\Sigma, \Lambda)_0,\, t_i: m_\Lambda(\cF)_{p_i} \xrightarrow{\sim} \Bbbk^r \right\}.$$
    It parametrizes those sheaves in $\cM_r(\Sigma, \Lambda)_0$ with the additional data of a fixed trivialization of microstalks at the base points.  See Definition \ref{def:moduli-sh-mufr} and Appendix \ref{appen:sheaf-moduli}.\footnote{The terminology microframing is used to distinguish the notion of framing in \cite{STWZ19}*{Section 2.4}. Their relation will be discussed in Appendix \ref{appen:sheaf-moduli}.} A first advantage is that $\cM_r^{\mu,\textit{fr}}(\Sigma, \Lambda)_0$ is an algebraic variety in the cases we study, whereas $\cM_r(\Sigma, \Lambda)_0$ is not. From the perspective of cluster algebras, $\cM_r^{\mu,\textit{fr}}(\Sigma, \Lambda)_0$ is better behaved than the unframed $\cM_r(\Sigma, \Lambda)_0$ in that $\C[\cM_1^{\mu,\textit{fr}}(\Sigma, \Lambda)_0]$ is actually a commutative cluster algebra.\footnote{We expect $\C[\cM_r^{\mu,\textit{fr}}(\Sigma, \Lambda)]$ to be a non-commutative cluster algebra for $r\geq2$, in some sense of the definition of a {\it non-commutative} cluster algebra, see \cite{GonKon21}.}
    
\begin{figure}[h!]
  \centering
  \includegraphics[width=0.7\textwidth]{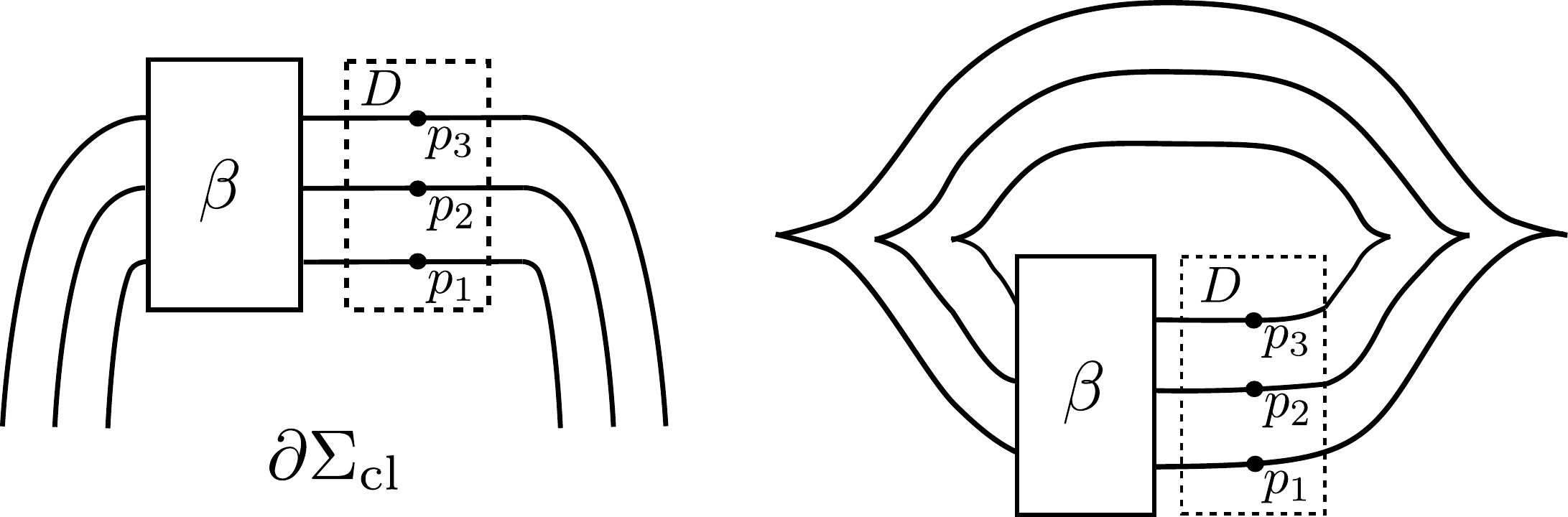}\\
  \caption{The rectangular regions $D$ and the set of base points $T = \{p_1, \dots, p_n\}$ for a Legendrian (cylindrical or rainbow) $n$-braid  closure.}\label{fig:braid-markpt}
\end{figure}

The symplectic geometric definition of $\cM_r^{\mu,\textit{fr}}(\Sigma, \Lambda)_0$ can be described in a Lie theoretic manner by using affine flags. The following result illustrates this:

\begin{prop}\label{prop:frame-markpt}
    Let $\Lambda \subset T^{*,\infty}\Sigma$ be a Legendrian positive braid closure, either cylindrical closure or rainbow closure, equipped with a Maslov potential, $D \subset \Sigma$ an open disk such that $\Lambda \cap T^*D$ consist of $n$ parallel strands. Let $T = \{p_1, ..., p_n\} \subset \Lambda$ be a collection of base points in $D$, one on each strand. Then the restriction $\cF|_D$ of a microframed sheaf $(\cF, t_1, \dots, t_n) \in \cM^{\mu,\textit{fr}}_r(\Sigma, \Lambda)_0$ determines the complete flag
    $$0 = F_0 \subset F_1 \subset F_2 \subset \dots \subset F_n \cong \Bbbk^n,$$
    and the microframing $(t_1, \dots , t_n)$ determines an assigned volume $\omega_i$ on each $F_i\,(0 \leq i \leq n)$.
\end{prop}    
\begin{proof}
    Consider the region $D = D_0 \cup D_1 \cup ... \cup D_n$, microlocal rank~1 sheaves on $D$ is characterized by a complete flag
    $$B_k: 0 = V_0 \subset V_1 \subset ... \subset V_n \cong \Bbbk^n.$$
    Given the microframing data, we have specified isomorphisms of microstalks
    $$t_i: \, m_{\Lambda^\prec_{\beta_0}}(\cF)_{p_i} = V_i/V_{i-1} \cong \Bbbk.$$
    Note that the preimage of the the unit $1 \in \Bbbk$ determines a vector $\ol{v}_i \in V_i/V_{i-1}$. In particular, when $i = 1$ we have a volume form $\omega_1 = v_1$ on $V_1$. Inductively, a vector $\ol{v}_i \in V_i/V_{i-1}$ and the given volume form $\omega_{i-1}$ on $V_{i-1}$ determines a unique volume form on $V_i$ by
    $$\omega_i = \omega_{i-1} \wedge v_i,$$
    where $v_i$ is any representative in the equivalence class $\ol{v}_i \in V_i/V_{i-1}$. Thus we get a collection of volume forms by induction.
    
    Conversely, consider a collection of volume forms $\omega_1, \dots, \omega_n$ on $V_0, V_1, \dots, V_n$. There exists a vector $v_i  \in V_i/V_{i-1}$ such that
    $$\omega_i = \omega_{i-1} \wedge v_i.$$
    In addition, the equivalence class $\ol{v}_i \in V_i / V_{i-1}$ is well defined. The vector $\ol{v}_i$ on $V_i/V_{i-1}$ then determines a trivialization $t_i: V_i/V_{i-1} \cong \Bbbk$ by $t_i(v_i) = 1$.
\end{proof}

    In Appendix \ref{appen:sheaf-moduli}, we discuss different notions of framings of sheaves in the literature \cites{STWZ19,CasalsWeng22} and explain how microlocal merodromy along relative 1-cycles gives rise to functions the above moduli space.

\subsection{Sheaf quantization of conjugate Lagrangians}\label{sec:quan-conj}
    For an alternating Legendrian $\Lambda \subset T^{*,\infty}\Sigma$ and its conjugate Lagrangian filling $L\sse T^*\Sigma$, the sheaf quantization of rank~$r$ local systems $\Loc^r(L)$ in the conjugate surface is studied in \cite{STWZ19}*{Section 4.3}. This subsection briefly summarizes that construction.

\subsubsection{Alternating sheaves}
Let $\Lambda \subset T^{*,\infty}\Sigma$ be an  alternating  Legendrian. Following Theorem \ref{thm:JinTreumann} (see also Theorem \ref{thm:JinTreu-app} Appendix \ref{appen:sheaf-quan}), which adapts the results of \cite{JinTreu17}, a conjugate Lagrangian filling $L$ gives rise to a fully faithful sheaf quantization functor
    $$\Loc^{pp}(L) \hookrightarrow \Sh^{pp}_\Lambda(\Sigma).$$
   In \cite{STWZ19}*{Definition 4.13}, the sheaves in $\Sh^{pp}_{\Lambda}(\Sigma)$ which are the sheaf quantizations of the corresponding conjugate Lagrangians are characterized as follows:

\begin{definition}\label{def:altsheaf}
    Let $\Lambda \subset T^{*,\infty}\Sigma$ be an alternating Legendrian link and $L$ its conjugate Lagrangian filling. An alternating sheaf is an object in $\Sh_\Lambda^{pp}(\Sigma)$ whose support is contained in the closure of the union of white and black regions.
\end{definition}

\noindent In \cite{STWZ19}*{Theorem 4.17 \& Proposition 4.18} it is shown that this definition indeed captures the required notion:

\begin{thm}[\cite{STWZ19}]\label{thm:altsheaf}
Let $\Lambda \subset T^{*,\infty}\Sigma$ be an alternating Legendrian link and $L$ its conjugate Lagrangian filling. The subcategory of alternating sheaves in $\Sh_\Lambda^{pp}(\Sigma)$ is equivalent to the category of local systems on $L$. In fact, $\cF \in \Sh_\Lambda^{pp}(\Sigma)$ is the image of a rank~$r$ local system on $L$ under the Jin-Treumann sheaf quantization functor (Theorem \ref{thm:JinTreumann})
    $$\Loc^{pp}(L) \hookrightarrow \Sh^{pp}_{\widetilde{L}}(\Sigma \times \bR) \hookrightarrow \Sh^{pp}_\Lambda(\Sigma)$$
    if and only if it is an alternating sheaf with microlocal rank~$r$.
\end{thm}

    Indeed, an alternating sheaf can be characterized in the following way, following \cite{STWZ19}*{Proposition 4.15 \& 16}. Locally near a crossing, suppose that the white region $W$ is the first quadrant while the black region $B$ is the third quadrant. Then the alternating sheaf $\cF$ fits into an exact triangle
    $$\cF_{\ol{W}}[1] \to \cF \to \cF_{\ol{B}} \xrightarrow{+1}.$$
    Therefore, when $\cF_{\ol{W}}$ and $\cF_{\ol{B}}$ are given, the sheaf $\cF$ is classified by $Ext^1(\cF_{\ol{B}}, \cF_{\ol{W}}[1])$. In addition, the singular support condition $SS^\infty(\cF) \subset \Lambda$ is satisfied only when the element in $Ext^1(\cF_{\ol{B}}, \cF_{\ol{W}}[1])$ is invertible.

\begin{ex}
    Let $\cF_{\ol{W}} = \Bbbk_W^r[1]$, $\cF_{\ol{B}} = \Bbbk_{\ol{B}}^r$. Any nonzero element in $Ext^1(\cF_{\ol{B}}, \cF_{\ol{W}}[1]) = M_{r\times r}(\Bbbk)$ defines an alternating sheaf $\cF$, which in this case is the sheaf quantization of a rank~$r$ local system on the conjugate Lagrangian $L$.
\end{ex}

\subsubsection{Microlocal holonomies}
By Theorem \ref{thm:altsheaf}, an alternating sheaf comes from a local system of a conjugate Lagrangian filling. The holonomies of the local system determine the microlocal holonomies of the sheaf. Let us now explain how to compute the unsigned microlocal monodromies in Section \ref{sec:microlocal} and Example \ref{ex:mumon-STZ} (see also Appendix \ref{appen:microlocal}) along a 1-cycle $\gamma_F \in H_1(L; \bZ)$ for a null region $F$, following \cite{STWZ19}*{Section 5.2}.\\

Suppose that the front $\pi(\Lambda)$ is the union of the $x$-axis and $y$-axis on the plane, and $\cF$ is an alternating sheaf supported in the first and third quadrant of the plane. Then the microlocal merodromy $m_{L,\gamma}(\cF)$ along the path $\gamma$ from the positive $x$-axis to the negative $y$-axis is the composition of
\begin{enumerate}
  \item the parallel transport of the microstalk along the horizontal strand of $\pi(\Lambda)$;
  \item the isomorphism between the microstalk on the left of the horizontal strand of $\pi(\Lambda)$ and the stalk in the third quadrant;
  \item the isomorphism between the stalk in the third quadrant and the microstalk at the bottom of the vertical strand of $\pi(\Lambda)$.
\end{enumerate}
    Note that this is in fact the microlocal holonomy of the corresponding sheaf $\widetilde{\cF} \in \Sh^{pp}_{\widetilde{L}}(\Sigma \times \bR)$ from Theorem \ref{thm:JinTreumann} Part~(1). In an abuse of notation, so as to follow \cite{STWZ19}, we still refer to it as the microlocal holonomy of the alternating sheaf. In Section \ref{sec:cluster} we study cluster algebras over $\Z$. Therefore, we need to fix the signs when computing these microlocal holonomies on Legendrian surfaces. The details of how to achieve this are explained in Appendix \ref{appen:microlocal}. In consequence, we need to choose a coherent collection of sign curves (following Appendix \ref{appen:microlocal} Definition \ref{def:sign-curve}), which accounts for fixing a relative spin structure on $L$. The definition of such a coherent collection reads as follows:
    
\begin{definition}
    Let $L \subset T^*\Sigma$ be a conjugate Lagrangian filling of $\Lambda \subset T^{*,\infty}\Sigma$. Consider line segments $\gamma_v \subset L$ connecting two components of $\Lambda$ at each crossing $v \in \pi(\Lambda)$ such that $\pi(\gamma_v)$ are embedded arcs in the white region. Let the 1-skeleton $(L, \Lambda)_{\leq 1}$ of $(L, \Lambda)$ be the union of the boundary $\Lambda$ and the line segments $\gamma_v$. By definition, a coherent collection of sign curves associated to the 1-skeleton is a graph $P\sse\Sigma$ whose vertices are the black vertices and boundary points in $\Lambda$, such that the number of curves starting from the black vertices equals their degrees in the bipartite graph.
\end{definition}
    
    We leave it to the readers to verify that this definition is compatible with Appendix \ref{appen:microlocal}, accounting for relative spin structures on $L$, and also the computation \cite{GonKon21}*{Section 3.2} where the $-1$'s in the monodromies appear.\\
    
    Finally, in the context of conjugate Lagrangian surfaces, the first hint of the appearance of cluster $\mathcal{X}$-structures, observed for partial $\mathcal{X}$-structures in \cite[Section 5.1]{STWZ19} and proven in general in \cite[Corollary 1.2]{CasalsWeng22}, is illustrated by the precise rational transformation that the microlocal monodromy along 1-cycles undergoes upon performing a square move:
    
\begin{prop}[\cite{STWZ19}*{Theorem 5.8}]\label{prop:mumon-square}
    Let $\Lambda_0, \Lambda_1 \subset T^{*,\infty}\Sigma$ be alternating Legendrians that differ by a square move at the null region $F$, and $L_0, L_1$ the corresponding conjugate Lagrangians. For $\cF_0 \in \Sh^{pp}_{\Lambda_0}(\Sigma)$ an alternating sheaf of microlocal rank~$1$, its image $\cF_1 \in \Sh^{pp}_{\Lambda_1}(\Sigma)$ is an alternating sheaf of microlocal rank~$1$ if and only if $m_{L_0, \gamma_F}(\cF_0) \neq -1$. Under this assumption, for any null region $C$,
    \[m_{L_1, \gamma_C}(\cF_1) = \begin{cases}
    m_{L_0, \gamma_C}(\cF_0)\left(1 + m_{L_0, \gamma_F}(\cF_0)\right)^{\left<\gamma_F,\gamma_C\right>},& C \neq F, \\
    m_{L_0, \gamma_C}(\cF_0)^{-1}, & C=F.
    \end{cases}\]
\end{prop}

    As we will see in Section \ref{sec:cluster}, this is a cluster $\cX$-transformation in a cluster $\cX$-variety defined by the intersection quiver of $H_1(L; \bZ)$.

\subsection{Sheaf quantization of Legendrian weaves}\label{sec:quan-weave}
    Let $\Lambda \subset T^{*,\infty}\Sigma$ be a  Legendrian link and $L \subset T^*\Sigma$ a Lagrangian filling  of $\La$ whose Legendrian lift is a Legendrian weave $\widetilde{L} \subset J^1\Sigma$. In this context of Legendrian weaves, the sheaf quantization of $L$ can be obtained explicitly as follows.

\subsubsection{Sheaves singularly supported on weaves}
    Given the Legendrian weave $\widetilde{L} \subset J^1\Sigma$, a microlocal rank~$r$ sheaf $\widetilde{\cF} \in \Sh^{pp}_{\widetilde{L}}(\Sigma \times \bR)$ with singular support on the Legendrian weave can be described by a framed flag on the associated $n$-graph $G$, generalizing Example \ref{ex:sheafcombin}. This is a consequence of the discussion in \cite{CasalsZas20}*{Section 5}. The framed flag moduli for a weave is described as follows. By definition, a flag of local systems on $X$ is a local system $E \rightarrow X$ with a complete filtration $E_\bullet$ by local systems $E_k \rightarrow X$ such that the total monodromy preserves the filtration. For $U \subset X$, a flag of sub-local systems $F_\bullet$ on $U$ is said to be compatible with $E_\bullet$ if for $\gamma \in \pi_1(U, u)$ and $v \in F_k$, $i_k(\gamma \cdot v) = i_{k*}\gamma \cdot i_k(v)$, where $i_k: F_k(u) \hookrightarrow E_k(u)$. In \cite{CasalsZas20}*{Section 5}, a framed flag for a weave (given by an $n$-graph $G$) is given by the following data:
    \begin{enumerate}
      \item a rank $rn$ local system $E \rightarrow \Sigma$;
      \item for each face $F$, a flag of local systems $0 = V_0(F) \subset V_1(F) \subset ... \subset V_n(F) \cong \Bbbk^{rn}$;
      \item for each pair of adjacent faces $F_1$ and $F_2$ sharing a common edge $e \in G_i$, only the $i$-th vector space of $V_\bullet(F_1)$ and $V_\bullet(F_2)$ are transverse, and there are chosen isomorphisms $V_j(F_1) \cong V_j(F_2)\, (j \neq i)$.
      \item by gluing, the isomorphisms define local systems in each region, and these local systems in each region are compatible with $E$.
    \end{enumerate}

\subsubsection{Microlocal holonomies}\label{sec:quan-weave-mon}
    Given a Legendrian weave $\widetilde{L} \subset J^1\Sigma_\text{cl}$ with boundary $\Lambda \subset J^1(\partial \Sigma_\text{cl})$, and a sheaf $\widetilde{\cF} \in \Sh^{pp}_{\widetilde{L}}(\Sigma_\text{op} \times \bR)$, the following examples suffice in order to compute microlocal holonomies for the weaves we need. We follow Section \ref{sec:microlocal} and Example \ref{ex:mumon-STZ} (see also Appendix \ref{appen:microlocal}). Given a sheaf $\widetilde{\cF} \in \Sh^{pp}_{\widetilde{L}}(\Sigma_\text{op} \times \bR)$ which is the sheaf quantization of a local system on the free Legendrian weave $\widetilde{L}$, the monodromy of the local system exactly determines the microlocal monodromy of the sheaf. The computations of unsigned microlocal holonomies, according to \cite{STZ17}*{Section 5} and \cite{CasalsZas20}*{Section 7}, are as follows:

\begin{ex}\label{ex:mero-2strand}
    Suppose $\gamma$ is a path passing transversely through an edge $e_i$ starting from the $i$-th sheet of the Legendrian front. Then the corresponding flags in the two adjacent regions have the same $V_0, ..., V_{i-1}$, and also the same $V_{i+1}, ..., V_n$. Consider the quotient $V := V_{i+1}/V_{i-1}$. By taking $V_i/V_{i-1} \subset V_{i+1}/V_{i-1}$, we get two lines $l_a, l_b$ (in microlocal rank $r$, to $r$-dimensional subspaces), equipped with volume forms $\omega_a$ and $\omega_b$. The parallel transport map is given by the isomorphism $l_a \xrightarrow{\sim} V/l_b$, and hence is computed by $\omega_a \wedge \omega_b \in \det(V)$.
\end{ex}

\begin{ex}\label{ex:mero-3strand}
    Suppose $\gamma$ is a path passing transversely through the edges $e_i, e_{i+1}, e_i$ starting from the $i$-th sheet of the Legendrian front. Then the corresponding flags in the four regions have the same $V_0, ..., V_{i-1}$, and also the same $V_{i+2}, ..., V_n$. Consider the 3-dimensional quotient $V := V_{i+2}/V_{i-1}$. By taking $V_i/V_{i-1} \subset V_{i+1}/V_{i-1} \subset V_{i+2}/V_{i-1}$, we get four pairs of lines and planes
    $$(l_a, l_A), \,\, (l_{AB}, l_A), \,\, (l_{AB}, l_B), \,\, (l_b, l_B),$$
    where $l_{AB} = l_A \cap l_B$. The pairs $(l_a, l_A), (l_b, l_B)$ are equipped with volume forms $(\omega_a, \omega_A)$ and $(\omega_b, \omega_B)$. The parallel transport map is given by the composition
    $$l_a \xrightarrow{\sim} l_A/l_{AB} \xrightarrow{\sim} V/l_B,$$
    which is computed by $\omega_a \wedge \omega_B \in \det(V)$.
\end{ex}

    Given the microlocal merodromy computations in Examples \ref{ex:mero-2strand} and \ref{ex:mero-3strand}, we can compute the microlocal monodromy of $\widetilde{\cF}$ along (the Legendrian lift $\widetilde{L}$ of) the 1-cycles $\gamma \in H_1(\widetilde{L}; \bZ)$, see \cite{CasalsZas20}*{Section 7}. For simplicity, here we only describe the microlocal monodromy along $\sf I$-cycles and $\sf Y$-cycles.

\begin{ex}[$\sf I$-cycles]\label{ex:mumon-Icycle}
    Suppose $\gamma \in H_1(\widetilde{L}; \bZ)$ is an $\sf I$-cycle that locally lies on the $i$ \& $(i+1)$-th sheet of $\pi_\text{front}(\widetilde{L})$. Thus the flags in the four regions have the same $V_0, ..., V_{i-1}$, and also the same $V_{i+1}, ..., V_n$. Consider $V = V_{i+1}/V_{i-1}$. Then by taking $V_i/V_{i-1} \subset V_{i+1}/V_{i-1}$, we get four lines $l_a, l_b, l_c, l_d \subset V$ in the four regions, such that any two lines in adjacent regions are transverse, as in Figure \ref{fig:mumonI}. Then the microlocal monodromy $m_{\widetilde{L},\gamma}(\widetilde{\cF})$ along $\gamma$ is the composition
    $$l_a \xrightarrow{\sim} V/l_b \xrightarrow{\sim} c \xrightarrow{\sim} V/l_d \xrightarrow{\sim} l_a.$$
    Therefore, the monodromy is the cross ratio which, assigning each $l_i$ a volume form $\omega_i$, by Example \ref{ex:mero-2strand}, can be written as
    $$\left<l_a, l_b, l_c, l_d\right> = \frac{\omega_a \wedge \omega_b}{\omega_b \wedge \omega_c} \frac{\omega_c \wedge \omega_d}{\omega_d \wedge \omega_a}.$$
\end{ex}
\begin{figure}[h!]
	\begin{tikzpicture}
	\pgfmathsetmacro{\A}{.866}
	\draw[blue,thick] (-1/2,\A) -- (0,0) -- (-1/2,-\A);
	\draw[blue,thick] (0,0) -- (2,0);
	\draw[blue,thick] (2.5,\A) -- (2,0) -- (2.5,-\A);
	\node at (1,.6) {$a$};
	\node at (-.5,0) {$b$};
	\node at (1,-.6) {$c$};
	\node at (2.5,0) {$d$};
	\node[blue] at (.7,.15) {$e$};
	\draw[orange,thick] (1,0) ellipse (2 and 1);
	\end{tikzpicture}
	\caption{Neighborhood of a monochromatic edge $e$ with the data determining a constructible sheaf $F$. As we show, the microlocal monodromy $m_{\Lambda,\gamma}(F)$ along the 1-cycle $\gamma(e)$ is given by the cross-ratio $\langle l_a, l_b, l_c, l_d\rangle$.}
	\label{fig:mumonI}
\end{figure}
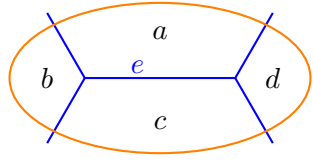

\begin{ex}[$\sf Y$-cycles]\label{ex:mumon-Ycycle}
    Suppose $\gamma \in H_1(\widetilde{L}; \bZ)$ is an $\sf Y$-cycle that locally lies on the $(i-1)$, $i$ \& $(i+1)$-th sheet of $\pi_\text{front}(\widetilde{L})$. Then consider $V = V_{i+2}/V_{i-1}$. Then by taking $V_i/V_{i-1}, V_{i+1}/V_{i-1} \subset V_{i+1}/V_{i-1}$, we get a couple of flags as in Figure \ref{fig:mumonY}. In particular we have 1-dimensional vector spaces $l_a, l_b, l_c$ and 2-dimensional vector spaces $l_A, l_B, l_C$. Then by assigning each $l_i$ a volume form, by Example \ref{ex:mero-3strand} the monodromy is
    $$\left<(l_a, l_A), (l_b, l_B), (l_c, l_C) \right> = \frac{(\omega_B \wedge \omega_a)(\omega_C \wedge \omega_b)(\omega_A \wedge \omega_c)}{(\omega_B \wedge \omega_c)(\omega_C \wedge \omega_a)(\omega_A \wedge \omega_b)}.$$
\end{ex}
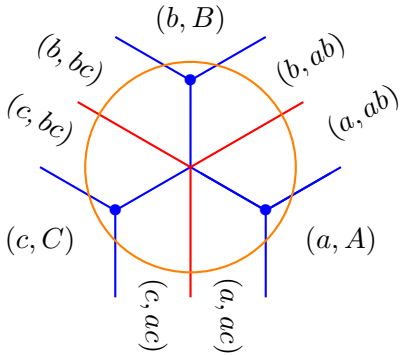
\begin{figure}[h!]
	\begin{tikzpicture}
	\pgfmathsetmacro{\A}{1.732}
	\pgfmathsetmacro{\a}{.866}
	\draw[blue,thick] (1-1,\A+1/\A)--(1,\A)--(1+1,\A+1/\A);
	\draw[blue,thick] (0-1,0+1/\A)--(0,0)--(0+1,0+1/\A);
	\draw[blue,thick] (2-1,0+1/\A)--(2,0)--(2+1,0+1/\A);
	\draw[blue,thick] (2-\a,0+.5)--(2,0)--(2+\a,0+.5);
	\draw[blue,thick] (0,0)--(0,-2/\A);
	\draw[blue,thick] (2,0)--(2,-2/\A);
	\draw[blue,thick] (1,\A)--(1,\A-2/\A);
	\draw[red,thick] (1,\A-2/\A)--(1,\A-2/\A-\A);
	\draw[red,thick] (0+1,0+1/\A)--(0+1+1.5,0+1/\A+\a);
	\draw[red,thick] (0+1,0+1/\A)--(0+1-1.5,0+1/\A+\a);
	\node[blue] at (0,0) {$\bullet$};
	\node[blue] at (2,0) {$\bullet$};
	\node[blue] at (1,\A) {$\bullet$};
	\node at (3,-.4) {$(a,A)$};
	\node at (-1,-.4) {$(c,C)$};
	\node at (1,\A+.8) {$(b,B)$};
	\node[rotate=30] at (3.3,1.3) {$(a,ab)$};
	\node[rotate=30] at (2.6,2) {$(b,ab)$};
	\node[rotate=-30] at (-1,1.2) {$(c,bc)$};
	\node[rotate=-30] at (-.6,2) {$(b,bc)$};
	\node[rotate=-90] at (.5,-1.4) {$(c,ac)$};
	\node[rotate=-90] at (1.5,-1.4) {$(a,ac)$};
	\draw[orange,thick] (1,1/\A) ellipse (1.4 and 1.4);
	\end{tikzpicture}
	\caption{Neighborhood of a $\sf Y$-cycle with the data determining a constructible sheaf $F$. As we compute, the microlocal monodromy $m_{\Lambda,\gamma}(F)$ along the associated 1-cycle $\gamma$ is given by the triple ratio of the three transverse flags. Here $l_{ab} = l_a + l_b$, while $l_{AB} = l_A \cap l_B$.}
	\label{fig:mumonY}
\end{figure}

    We need to fix the signs when computing microlocal holonomies on Legendrian surfaces (as explained in Appendix \ref{appen:microlocal}). Hence we choose a coherent collection of sign curves (following Appendix \ref{appen:microlocal} Definition \ref{def:sign-curve} and \cite{CasalsWeng22}*{Section 4.5}), which accounts for fixing a relative spin structure on $\widetilde{L}$.
    
\begin{definition}[\cite{CasalsWeng22}*{Definition 4.15}]\label{def:sign}
    Let $\widetilde{L}(\mathfrak{w}) \subset J^1\Sigma_\text{cl}$ be a Legendrian weave. Then a coherent collection of sign curves is a graph $P$ whose vertices are the trivalent vertices in $\mathfrak{w}$ and boundary points on $\Lambda$, such that each trivalent vertex has degree 1.
\end{definition}

    For simplicity, consider the free Legendrian weave $ \widetilde{L}(\mathfrak{w}_\beta) \subset J^1\bD^2$ with boundary on a Legendrian positive braid closure $\Lambda_{\beta\Delta^2} \subset J^1S^1$ as in Subsection \ref{ssec:ex-free-weave}.\footnote{Following the identification in Subsection \ref{sec:weave-at-infty}, we do not distinguish $\Lambda_{\beta\Delta^2}$ and $\Lambda(\beta \Delta^2)$ used in Subsection \ref{ssec:triangle-positive}.} Then we can choose a coherent collection of sign curves such that each trivalent vertex is connected to the boundary crossing in $\beta$ as in Figure \ref{fig:braid-sign}.
    
\begin{figure}
    \centering
    \includegraphics[width=0.65\textwidth]{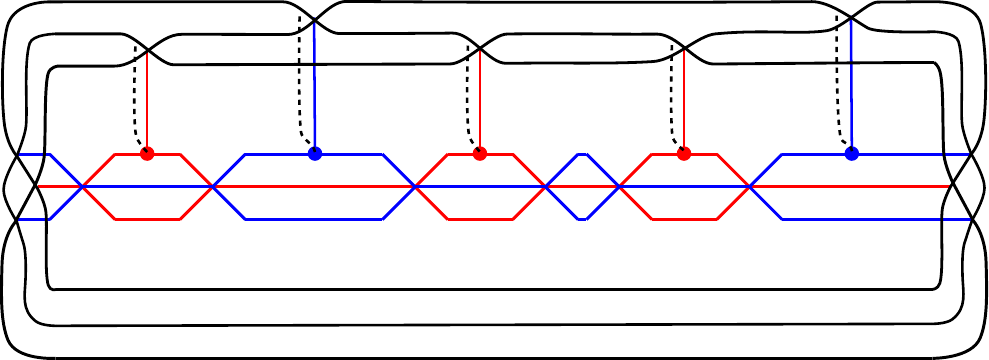}
    \caption{The sign curves (in dashed lines) associated to the Legendrian weave $ \widetilde{L}(\mathfrak{w}_\beta) \subset J^1\bD^2$ with boundary being the Legendrian cylindrical braid closure $\Lambda_{\beta\Delta^2} \in J^1S^1$.}
    \label{fig:braid-sign}
\end{figure}

    In order to compute microlocal merodromies along relative 1-cycles, we need to recall the following lemma. We follow the notations in \cite{ShenWeng21,GaoShenWeng20}: a pair of flags $B, B'$ are in relative position $s_i \in S_n$ if
    $$B_i \neq B'_i; \;\; B_j = B'_j, \, \forall\,j \neq i.$$
    In addition, two decorations $A, A'$ over $B, B'$, are called compatible if there exists $v_i \neq v_i' \in \Bbbk^n$ such that
    $$\omega_{i} = \omega_{i-1} \wedge v_i, \, \omega'_{i} = \omega'_{i-1} \wedge v'_i; \;\; \omega_{i+1} = \omega_{i-1} \wedge v_i \wedge v'_i; \;\; \omega_j = \omega'_j, \, \forall\, j \neq i.$$

\begin{lemma}[\cite{ShenWeng21}*{Lemma 2.10}, \cite{GaoShenWeng20}*{Lemma 4.14}]\label{lem:decorate-flag}
    Let $w \in S_n$ be an element of the symmetric group. Given $B \xrightarrow{w} B'$, for any decoration $A$ over $B$, there exists a unique decoration $A'$ over $B'$ such that $A \xrightarrow{w} A'$ is compatible.
\end{lemma}

    As in Proposition \ref{prop:frame-markpt}, we associate $n$ base points on each Legendrian link $\Lambda_{\beta\Delta^2}$ on the top right region in Figure \ref{fig:braid-sign}. By Lemma \ref{lem:decorate-flag}, given a system of flags associated to $\beta\Delta^2$, we can start with a decorated flag on the top right region and determine a compatible decoration on the flags associated to $\Lambda(\beta\Delta^2)$ counterclockwise. By Proposition \ref{prop:frame-markpt}, the compatible decorations are equivalent to trivialization data of the microstalks along each strand. These trivializations coincide with the ones determined by the parallel transport maps along the paths specified above.\footnote{These paths should be viewed as capping paths, connecting arbitrary points to the base points (on the top right region of $\Lambda(\beta\Delta^2)$) as in the computation of the Legendrian contact dg-algebra.} 

\begin{lemma}\label{lem:frame-decorate}
    Let $\beta \in \mbox{Br}_n^+$ be a positive braid and $ \widetilde{L}(\mathfrak{w}_{\beta}) \subset J^1\bD^2$ the associated free Legendrian weave with boundary on a Legendrian positive braid closure $\Lambda_{\beta \Delta^2} \in J^1S^1$. Consider the choice of (micro)framings in Proposition \ref{prop:frame-markpt} and the coherent collection of sign curves such that each trivalent vertex is connected to the boundary crossing in $\beta$ (Figure \ref{fig:braid-sign}).\\
    
    Then the decorations of flags over $\beta$ determined by microlocal parallel transport maps along $\beta$ starting from the microframing on the right coincide with the decorations determined by Lemma \ref{lem:decorate-flag} by following the decoration on the right.
\end{lemma}
\begin{proof}
    At a crossing $s_i$ in the braid $\beta$, suppose the flag on the right is $B$ and the flag on the left is $B'$. Suppose the decoration on $B$ (determined by microstalks) is
    $$\omega_j = v_1 \wedge v_2 \wedge \dots \wedge v_j, \; \; 1\leq j\leq n.$$
    At the crossing, the strand going up induces the parallel transport from $v_i$ to $v_i$, while the strand going down passes through the sign curve, and thus induces the parallel transport from $v_{i+1}$ to $-v_{i+1}$. Hence the decoration on $B'$ is
    \[\begin{array}{c}
    \omega_j' = v_1' \wedge v_2' \wedge \dots \wedge v_j', \; \; 1\leq j\leq n, \\
    v_j' = v_j, \, \forall\, j \neq i, i+1; \;\; v_i' = -v_{i+1}, \; v_{i+1}' = v_i.
    \end{array}\]
    Hence Lemma \ref{lem:decorate-flag} indeed provides a way to compute parallel transport maps and microlocal merodromies.
\end{proof}

    In parallel to Proposition \ref{prop:mumon-square}, the effect of a Legendrian mutation is also readily computed. By direct computation in \cite{STWZ19}*{Proposition 5.8} and \cite{CasalsZas20}*{Lemma 7.11}, the change of microlocal monodromies under a Legendrian mutation along an $\sf I$-cycle or $\sf Y$-cycle $\gamma \in H_1(\widetilde{L}; \bZ)$ reads as follows.

\begin{prop}[\cite{CasalsWeng22,CasalsZas20}]\label{prop:mumon-weave} Let $\widetilde{L}_0, \widetilde{L}_1 \subset J^1\Sigma$ be Legendrian weaves that differ by a Legendrian mutation along an $\sf I$-cycle or $\sf Y$-cycle $\gamma \in H_1(\widetilde{L}_0; \bZ)$. Then for microlocal rank~$1$ sheaves $\widetilde{\cF}_0 \in \Sh^{pp}_{\widetilde{L}_0}(\Sigma \times \bR)$ and $\widetilde{\cF}_1 \in \Sh^{pp}_{\widetilde{L}_1}(\Sigma \times \bR)$ that are identical away from the standard $3$-ball in $\Sigma \times \bR$ where Legendrian mutation is applied, we have
    \[m_{L_1,\xi}(\widetilde{\cF}_1) = \begin{cases}
    m_{L_0,\xi}(\widetilde{\cF}_0)\big(1 + m_{L_0,\gamma}(\widetilde{\cF}_0)\big)^{\left<\gamma, \xi\right>},& \xi \neq \gamma, \\
    m_{L_0,\xi}(\widetilde{\cF}_0)^{-1}, & \xi = \gamma.
    \end{cases}\]
\end{prop}

    Again, as we will see in Section \ref{sec:cluster}, this is a cluster $\cX$-transformation in a cluster $\cX$-variety defined by the intersection quiver of $H_1(\widetilde{L}; \bZ)$.

\subsubsection{Sheaves singularly supported on links via weaves}\label{sec:quan-weave-link}
    Following the construction (in the proof of) Theorem \ref{thm:JinTreumann}, the sheaf quantization of the Lagrangian projection of the $\widetilde{L}$ in $\Sh^{pp}_{\Lambda}(\Sigma)$ in the case of Legendrian weaves is described as follows.\\

    First, by Theorem \ref{thm:JinTreumann}.(1), there is a unique microlocal rank~1 sheaf $\widetilde{\cF} \in \Sh_{\widetilde{L}}^{pp}(\Sigma \times \bR)$ whose microlocal monodromy along $\widetilde{L}$ is the prescribed rank~1 local system. Then, by Theorem \ref{thm:JinTreumann}.(2), the projection $\pi_\Sigma: \Sigma \times \bR \rightarrow \Sigma$ defines the sheaf quantization $\cF$ of the Lagrangian filling $L \subset T^*\Sigma$ by
    $$\cF = \pi_{\Sigma!}\widetilde{\cF} \in \Sh_\Lambda^{pp}(\Sigma).$$
    One can compute that the stalk of $\cF$ is determined by the stalk of $\widetilde{\cF}$ at $M \times \{+\infty\}$; see also Theorem \ref{thm:JinTreu-app}.(2). Hence wherever the projection $\widetilde{L} \rightarrow \Sigma$ or $L \rightarrow \Sigma$ is a (branched) $k$-fold covering, the rank of the stalk of $\cF$ is $k$. Indeed, $\cF$ can be described by the diagram as in Figure \ref{fig:weavesheafproj}. Taking proper push forward by the projection $\pi_\Sigma$ does not lose any information of the sheaf, as is promised by the full faithfulness in Theorem \ref{thm:JinTreumann}.\\

    In particular, let $\Lambda_{\beta\Delta^2} \subset T^{*,\infty}\bD^2$ be a cylindrical positive braid closure, satellited along the unit outward conormal bundle of a disk. If we view $\widetilde{L} \subset J^1\bD^2 \cong T^{*,\infty}_{\tau>0}(\bD^2 \times \bR)$ as a free Legendrian weave which forms a multiple cover of $\bD^2$, and $\widetilde{\cF} \in \Sh^{pp}_{\widetilde{L}}(\bD^2 \times \bR)$, then $\cF|_{\bD^2 \backslash \{0\}} = \widetilde{\cF}|_{\partial\bD^2 \times \bR}$.

\begin{figure}[h!]
  \centering
  \includegraphics[width=0.6\textwidth]{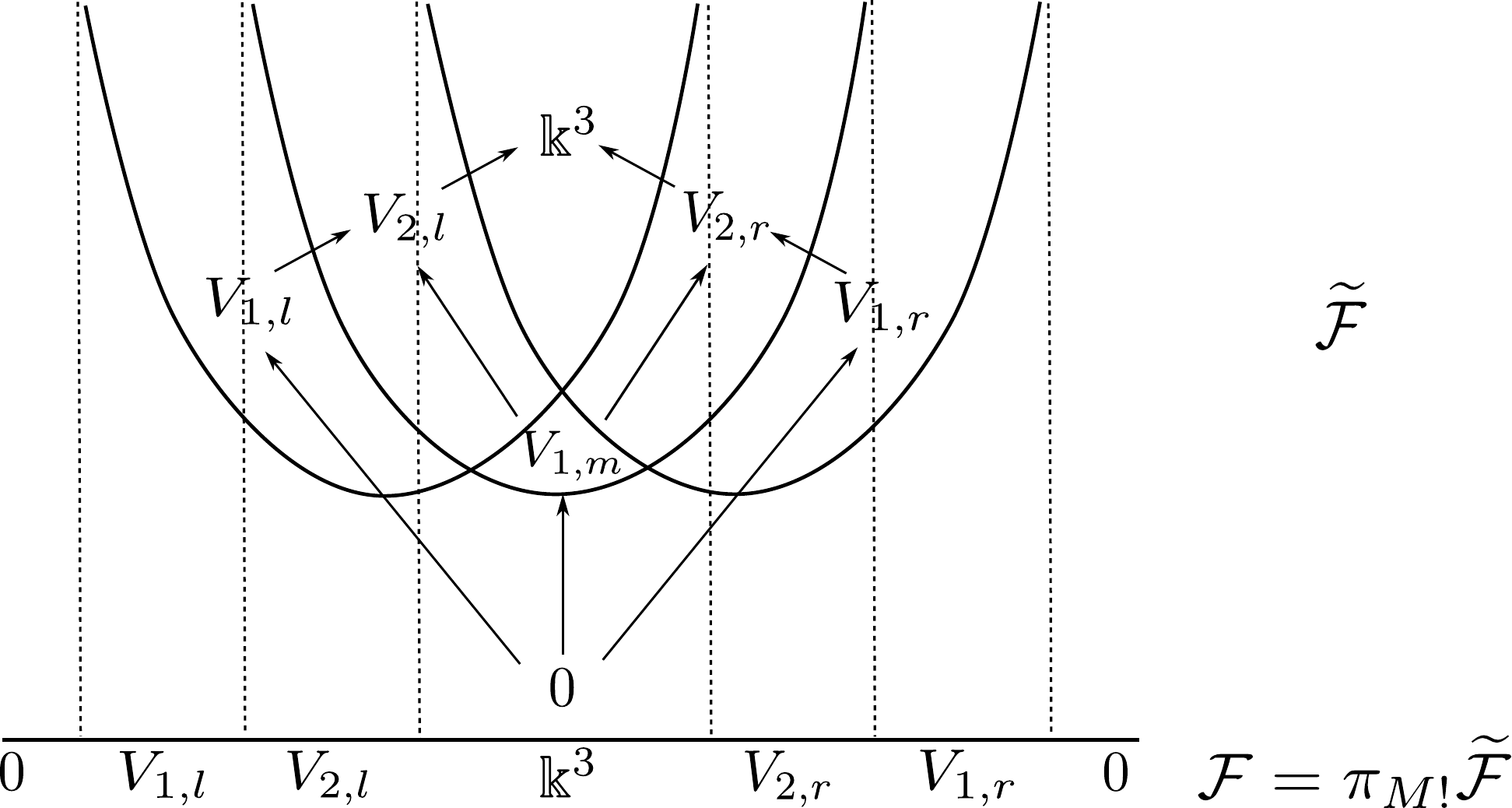}\\
  \caption{The sheaf quantization of a Legendrian weave on $M \times \bR$ and the proper push forward on $M$ following Theorem \ref{thm:JinTreumann}.}\label{fig:weavesheafproj}
\end{figure}

\begin{ex}\label{ex:weave-flag} Consider the special case of the Legendrian weave associated to an $n$-triangulation described in Subsections \ref{ssec:ex-free-weave} and \ref{sec:weave-at-infty}. Let $\Sigma$ be a triangle with vertices $x_1, x_2$ and $x_3$, and $\bigcirc_n$ the inward unit conormal bundle of $n$ concentric sectors at the vertices of the triangle. Let $A_0(x)$ be the sector centered at $x$ bounded by the smallest concentric arc, and $A_i(x)$ the region between the $i$-th and $(i+1)$-th concentric arc (from inside to outside) centered at $x$. Consider any microlocal rank~$1$ sheaf $\cF \in \Sh^{pp}_{\bigcirc_n}(\Sigma)$ whose stalks in $A_0(x_j)$ are acyclic. The stalks of $\cF$ in $A_0(x_j), A_1(x_j), \dots , A_n(x_j)$ define a flag
    $$F_\bullet(x_j): \,\, 0 = F_0(x_j) \subset F_1(x_j) \subset ... \subset F_n(x_j) \cong \Bbbk^n.$$
    Then, by \cite[Theorem 8.1]{CasalsZas20}, $\cF$ is the image of a rank~$1$ local system on the Legendrian weave $\widetilde{L}(\mathfrak{w}_n^*)$ associated to an $n$-triangulation under the Jin-Treumann functor (Theorem \ref{thm:JinTreumann})
    $$\Loc^{pp}(\widetilde{L}(\mathfrak{w}_n^*)) \hookrightarrow \Sh^{pp}_{\widetilde{L}(\mathfrak{w}_n^*)}(\Sigma \times \bR) \hookrightarrow \Sh^{pp}_{\bigcirc_n}(\Sigma)$$
    if and only if $F_\bullet(x_1), F_\bullet(x_2)$ and $F_\bullet(x_3)$ are transverse flags.
\end{ex}

\section{Comparison of Cluster Coordinates and Related Applications}\label{sec:cluster}

    By using the results from Section \ref{sec:quantize}, a Lagrangian filling $L$ of a Legendrian link $\Lambda \subset T^{*,\infty}\Sigma$ determines fully faithful embeddings
    $$\Loc^1(\widetilde{L}) \hookrightarrow \Sh^1_{\widetilde{L}}(\Sigma \times \bR) \hookrightarrow \Sh^1_\Lambda(\Sigma).$$
    These result in the following open embeddings of moduli spaces
    $$H^1(L; \Bbbk^\times ) \hookrightarrow \cM_1(\Sigma, \Lambda)_0, \;\text{and}\; H^1(L \backslash T, \Lambda \backslash T; \Bbbk^\times ) \hookrightarrow \cM_1^{\mu,\textit{fr}}(\Sigma, \Lambda)_0.$$
    By \cite{CasalsWeng22}, microlocal holonomies along certain curves on the Lagrangian fillings give rise to cluster $\mathcal{X}$-coordinates on $\cM_1(\Sigma, \Lambda)_0$, and microlocal merodromies give rise to cluster $\mathcal{X}$-coordinates on $\cM_1^{\mu,\textit{fr}}(\Sigma, \Lambda)_0$. With the Hamiltonian isotopies in Section \ref{sec:main_proofs}, between conjugate Lagrangian surfaces and (Lagrangian projections of) Legendrian weaves, we can relate cluster structures that arise from different combinatoric data, either plabic graphs or weaves, and identify the different cluster coordinates.\\
    
    First, in Section \ref{sec:clust-A}, we compare the cluster $\cA$-structures on the double Bott-Samelson cell $\mathrm{Conf}^e_\beta(\cC)$ defined in \cite{ShenWeng21,GaoShenWeng20} using plabic graphs and genralized minors, and the framed moduli of sheaves $\cM_1^{\mu,\textit{fr}}(\bD^2, \Lambda_{\beta\Delta^2})_0$, defined in \cite{CasalsWeng22} using microlocal merodromies along Legendrian weaves, and hence prove Theorem \ref{thm:main_algebra1}. As a corollary, we explain in Section \ref{sec:clust-X-sheaf} how to deduce the identifications between cluster $\cX$-coordinates.
    
    Second, in Section \ref{sec:framed-locsys}, we compare the cluster $\cX$-structures on the moduli space of framed local systems on punctured surfaces $\cX_{\GL_n}(\Sigma)$, comparing the cluster coordinates defined in \cite{FockGon06a} using plabic graphs, the ones defined in \cite{GMN14} using non-abelianization maps, relating them to weaves, and the ones defined using microlocal monodromies of sheaves \cite{STWZ19,CasalsZas20}. In particular, this proves Theorem \ref{cor:FG=GMN}.

\subsection{Cluster coordinates on flag moduli over $\bD^2$ and Theorem \ref{thm:main_algebra1}}\label{sec:clust-A}
   Consider the moduli space of microlocal sheaves for a positive cylindrical braid closure $\Lambda_{\beta\Delta^2} \in T^{*,\infty}\bD^2$ for $\beta \in \mbox{Br}_n^+$.\footnote{Following the convention in Subsection \ref{ssec:triangle-positive}, the strands in the braids are labeled from bottom to top, i.e.~the crossing $s_1$ is on the bottom and $s_{n-1}$ is on the top.}. Let us study the cluster $\cA$-coordinates on the moduli space of (micro)framed microlocal rank~1 sheaves $\cM_1^{\mu,\textit{fr}}(\bD^2, \Lambda_{\beta\Delta^2})_0$, which, according to \cite{BFZ05}, also determine the cluster $\cX$-variables on the moduli of (unframed) sheaves $\cM_1(\bD^2, \Lambda_{\beta\Delta^2})_0$.\\
   

\subsubsection{Flag moduli over $\bD^2$ and double Bott-Samelson cells}
    Proposition \ref{prop:frame-markpt}, see also Appendix \ref{appen:sheaf-moduli} (and Proposition \ref{prop:frame-markpt-app}), describes the moduli space of framed microlocal rank~1 sheaves in $\cM_1^{\mu,\textit{fr}}(\bD^2, \Lambda_{\beta}^\prec)_0$ for the Legendrian rainbow closure of a positive $n$-braid $\beta$, see \cite{ShenWeng21}. The moduli of (unframed) microlocal rank~1 sheaves $\cM_1(\bD^2, \Lambda_{\beta}^\prec)_0$ has been studied in \cite{ShenWeng21} as double Bott-Samelson cells. Denote by $B_0$ a flag
    $$B_0: 0 = V_0 \subset V_1 \subset \dots \subset V_n \cong \Bbbk^n,$$
    and denote by $B^0$ a flag
    $$B^0: \Bbbk^n \cong V'_n \twoheadrightarrow \dots \twoheadrightarrow V'_1  \twoheadrightarrow V'_0 = 0.$$
    They are said to be in opposite positions if $V_j = V_j'$.
    
\begin{definition}
    Let $\beta \in \mbox{Br}_n^+$ be a positive $n$-braid. The double Bott-Samelson cell $\mathrm{Conf}_{\beta}^e(\cB)$ is the moduli stack of sequences of  $\PGL_n$-flags
    \[\xymatrix@R=7mm{
    & & B^0 \ar@{-}[dll] \ar@{-}[drr] & & \\
    B_0 \ar[r]_{s_{i_1}} & B_1 \ar[r]_{s_{i_2}} & B_2 \ar[r]_{s_{i_3}} & \dots \ar[r]_{s_{i_k}} & B_k.
    }\]
    where $B_0, B_1, \dots, B_k$ are $\PGL_n$-flags such that $B_{j-1}, B_j$ are in relative position $s_{i_j} \in S_n$, and $B_0, B^0$ are in opposite positions.\\
    
    \noindent By definition, the half-decorated double Bott-Samelson cell $\mathrm{Conf}_{\beta}^e(\cC)$ is the moduli stack of sequences of  $\PGL_n$-flags
    \[\xymatrix@R=7mm{
    & & B^0 \ar@{-}[dll] \ar@{-}[drr] & & \\
    B_0 \ar[r]_{s_{i_1}} & B_1 \ar[r]_{s_{i_2}} & B_2 \ar[r]_{s_{i_3}} & \dots \ar[r]_{s_{i_k}} & B_k.
    }\]
    together with a decoration $A_k$ over $B_k$, where $B_0, B_1, \dots, B_k$ are $\PGL_n$-flags such that $B_{j-1}, B_j$ are in relative position $s_{i_j} \in S_n$, and $B_0, B^0$ are in opposite positions.
\end{definition}
\begin{thm}[\cite{ShenWeng21}*{Theorem 6.11}]\label{thm:ShenWeng-sheaf}
    Let $\beta \in \mbox{Br}_n^+$ be a positive braid and $\cM_1(\bD^2, \Lambda_{\beta}^\prec)_0$ the moduli of microlocal rank~1 sheaves on the Legendrian rainbow closure $\Lambda_{\beta}^\prec$. Then $\cM_1(\bD^2, \Lambda_{\beta}^\prec)_0$ is isomorphic to the double Bott-Samelson cell $\mathrm{Conf}_{\beta}^e(\cB)$.
\end{thm}

    The framed version is as follows. Consider the (micro)framed moduli of sheaves, following Section \ref{sec:sheaf-moduli} and Appendix \ref{appen:sheaf-moduli} (see Definition \ref{def:moduli-sh-mufr}), parametrizing sheaves with fixed trivializations of microstalks at base points. We fix the framing data consisting of $n$ base points $T = \{p_1, \dots, p_n\}$ on the right of the braid $\beta$, with one point in each strand, as in Proposition \ref{prop:frame-markpt} (Figure \ref{fig:braid-markpt}). By Proposition \ref{prop:frame-markpt}, the (micro)framing data is equivalent to a decoration on the flags; these are oftentimes called affine flags, decorated flags or principal flags as well.

\begin{cor}\label{cor:frame-braid-mod}
    Let $\beta \in \mbox{Br}_n^+$ be a positive braid and $\cM_1^{\mu,\textit{fr}}(\bD^2, \Lambda_{\beta}^\prec)_0$ the framed moduli space of microlocal rank~$1$ sheaves on the Legendrian rainbow closure $\Lambda_{\beta}^\prec$, with framing data determined by $n$ base points $T = \{p_1, \dots, p_n\}$. Then $\cM_1^{\mu,\textit{fr}}(\bD^2, \Lambda_{\beta}^\prec)_0$ is isomorphic to the half decorated double Bott-Samelson cell $\mathrm{Conf}_{\beta}^e(\cC)$.
\end{cor}
\begin{remark}
    Note that our convention of half decorated Bott-Samelson cells is different from \cite{GaoShenWeng20}, as in our case the decorated flags below the cusps lie in $\SL_n/\mathrm{U}_+$ instead of $\SL_n/\mathrm{U}_-$, where $\mathrm{U}_\pm$ are the unipotent subgroup of upper/lower triangular matrices. It can be shown that different conventions actually give isomorphic moduli spaces.
\end{remark}

    Let us identify the sheaf moduli for the rainbow closure $\Lambda_{\beta}^\prec$ and the cylindrical closure $\Lambda_{\beta\Delta^2}$ where $\Delta$ is the half twist, we will compute microlocal holonomies over the latter moduli space by considering Legendrian weaves.

\begin{prop}\label{prop:sheaf-X-rain=sate}
    Let $\beta \in Br_n^+$ be a positive braid, $\Delta$ be the half twist, $\Lambda_{\beta\Delta^2} \subset T^{*,\infty}\bD^2$ the Legendrian cylindrical braid closure of $\beta\Delta^2$ and $\Lambda_{\beta}^\prec \subset T^{*,\infty}_{\eta < 0}\bD^2$ the Legendrian rainbow braid closure of $\beta$. Then there is an isomorphism
    $$\cM_1(\bD^2, \Lambda_{\beta}^\prec)_0 \cong \cM_1(\bD^2, \Lambda_{\beta\Delta^2})_0.$$
\end{prop}
\begin{proof}
    Consider $\cF_{rain} \in \cM_1(\bD^2, \Lambda_{\beta}^\prec)_0$. In the leftmost region, below the cusps we have a complete flag
    $$B_0: \,\, 0 = V_0 \hookrightarrow V_1 \hookrightarrow V_2 \hookrightarrow \dots \hookrightarrow V_n \cong \Bbbk^n,$$
    while above the cusps we have a complete flag
    $$B^0: \,\, \Bbbk^n \cong V_n \twoheadrightarrow \dots \twoheadrightarrow V_2 \twoheadrightarrow V_1 \twoheadrightarrow V_0 = 0,$$
    such that compositions $V_i \hookrightarrow \dots \hookrightarrow V_n \twoheadrightarrow \dots \twoheadrightarrow V_i$ is the identity. Focus only on the $n$ lower half strands below the cusps. Applying Reidemeister~II moves (as in Subsection \ref{ssec:triangle-positive} Figure \ref{fig:NgBraidclosure}) we introduce extra crossings between $n$ lower half strands that form a half twist $\Delta$. When pulling out the 2nd lowest strand by a Reidemeister~II move, we introduce a new region in the stratification by the front projection with stalk $\ker(V_2 \rightarrow V_1)$. After pulling out the $j$-th lowest strand, we introduce $j-1$ new regions in the stratification with stalks
    $$\ker(V_j \rightarrow V_{j-1}) \hookrightarrow \dots \hookrightarrow \ker(V_j \rightarrow V_2) \hookrightarrow \ker(V_j \rightarrow V_1) \hookrightarrow V_j.$$
    Hence after pulling out all strands via Reidemeister~II moves, on the left of the half twist $\Delta$, we now have the complete flag
    \footnote{We denote it by $B_\text{bot}$ following the convention in Subsection \ref{ssec:triangle-positive} and \ref{ssec:ex-free-weave}, that $\Lambda_{\beta}^\prec$ has the braid $\beta$ on the bottom while the cylindrical closure $\Lambda_{\beta\Delta^2}$ has to be flipped upside down with the braid $\beta$ on the top. Indeed, the flag $B_\text{bot}$ will correspond to the sheaf $\cF_{cyl}$ in the region on the bottom of $\Lambda_\beta$.}
    $$B_\text{bot}: \,\, 0 \hookrightarrow \ker(V_n \rightarrow V_{n-1}) \hookrightarrow \dots \hookrightarrow \ker(V_n \rightarrow V_2) \hookrightarrow \ker(V_n \rightarrow V_1) \hookrightarrow V_n \cong \Bbbk^n.$$
    With the Reidemeister~II moves, we have introduced a half twist $\Delta$ on both sides of $\Lambda_{\beta}^\prec$ and a corresponding system of flags. In addition, the flag in the leftmost region and in the rightmost region are both $B_\text{bot}$ because of the conditions coming from cusps on both sides. Therefore, we get a sheaf $\cF_{cyl} \in \cM_1(\bD^2, \Lambda_{\beta\Delta^2})_0$ determined by the system of flags, as on the right of Figure \ref{fig:rain-cylin}.

\begin{figure}[h!]
  \centering
  \includegraphics[width=0.8\textwidth]{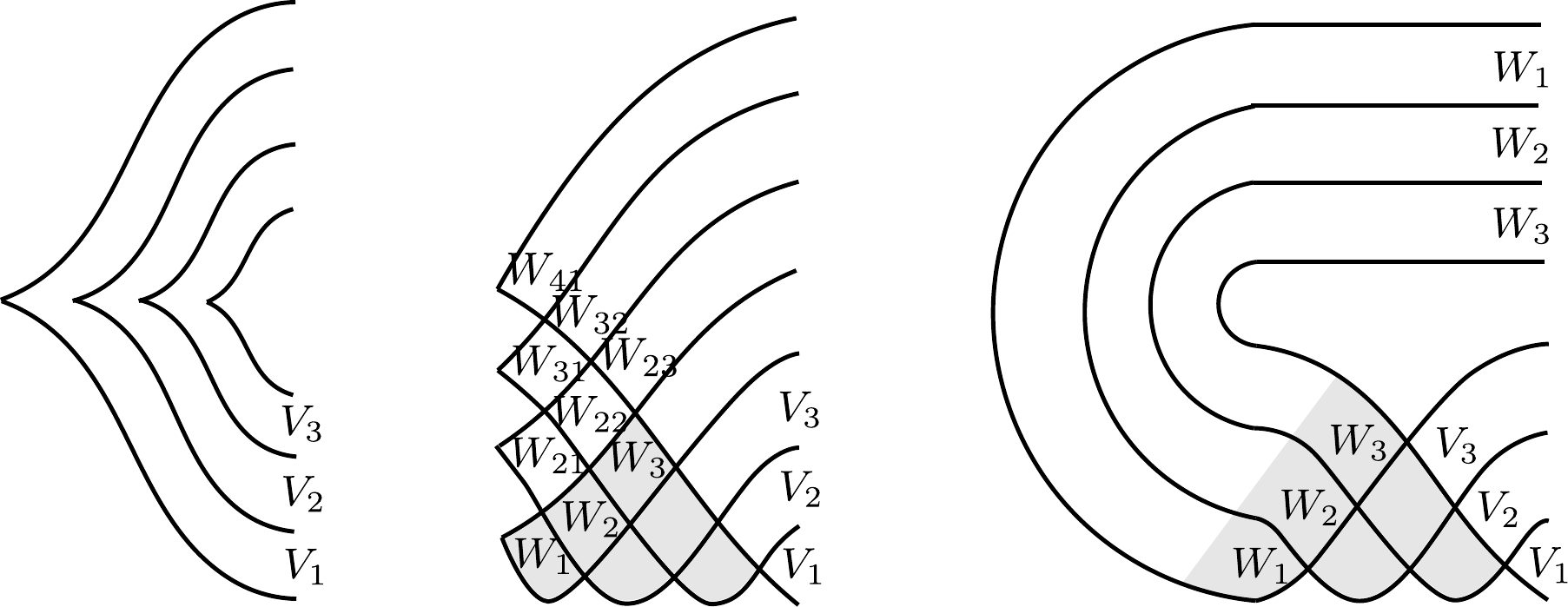}\\
  \caption{The identification between sheaves with singular support in the rainbow closure $\Lambda_{\beta}^\prec$ (left) and the cylindrical closure $\Lambda_{\beta\Delta^2}$ before flipped upside down (right). The Legendrian in the middle is obtained by applying Reidemeister~II moves to the Legendrian rainbow closure on the left. The grey regions are where the half twist $\Delta$ is introduced.}\label{fig:rain-cylin}
\end{figure}

    Conversely, given $\cF_{cyl} \in \cM_1(\bD^2, \Lambda_{\beta\Delta^2})_0$ as on the right of Figure \ref{fig:rain-cylin}, we have a sequence of flags in the regions formed by the braid $\Delta\beta\Delta$. Hence we can define a corresponding sheaf $\cF_{rain}'$ in the regions below all the upper half strands in the middle of Figure \ref{fig:rain-cylin}, which are formed by the braid $\Delta\beta\Delta$. In order to define the sheaf $\cF_\text{rain}'$ on the whole plane in the middle of Figure \ref{fig:rain-cylin} we just need to note that the stalks in all upper strata are determined by the given data below all the upper half strands. Indeed, suppose the stalks of the sheaf $\cF_{cyl}$ on the top region is determined by the flag
    $$B_\text{bot}: \,\, 0 \hookrightarrow W_1 \hookrightarrow W_2 \hookrightarrow \dots \hookrightarrow W_n \cong \Bbbk^n.$$
    Then the stalks of $\cF'_{rain}$ on the leftmost region of $\Delta$, i.e.~in the $n$ strata right below the lowest upper half strand, are determined by the same flag $B_\text{bot}$. The stalks of the sheaf in the $n-1$ regions (from bottom to top) in between the lowest and 2nd lowest upper half strands have to be the flag $W_{2,\bullet}: 0 \hookrightarrow W_{21} \hookrightarrow W_{22} \hookrightarrow \dots \hookrightarrow W_{2,n-1}$ where
    \[\begin{split}
    W_{21} = \mathrm{coker}(W_1 &\rightarrow W_2), W_{22} = \mathrm{coker}(W_2 \rightarrow W_3 \oplus W_{21}),\dots,\\
    W_{2,n-1} &= \mathrm{coker}(W_{n-1} \rightarrow W_n \oplus W_{2,n-2}).
    \end{split}\]
    Suppose the stalks in between the $(j-1)$-th and $j$-th lowest upper half strands are determined by the flag $W_{j-1,\bullet}$. For the stalks in between the $j$-th and $(j+1)$-th lowest upper half strands have to be the flag $W_{j,\bullet}$ where
    \[\begin{split}
    W_{j1} = \mathrm{coker}(W_{j-1,1} &\rightarrow W_{j-1,2}), W_{j2} = \mathrm{coker}(W_{j-1,2} \rightarrow W_{j-1,3} \oplus W_{j1}),\dots,\\
    W_{j,n-j+1} =&\, \mathrm{coker}(W_{j-1,n-j+1} \rightarrow W_{j-1,n-j+2} \oplus W_{j,n-j}).
    \end{split}\]
    Therefore inductively we are able to define the sheaf $\cF_{rain}'$ on the whole plane in the middle of Figure \ref{fig:rain-cylin}. By Reidemeitser~II moves we get a sheaf $\cF_{rain} \in \cM_1(\bD^2, \Lambda_{\beta}^\prec)_0$ on the left of Figure \ref{fig:rain-cylin}, and the stalks in the region stratified by $\beta$ are the same as $\cF_{cyl}$. Hence in particular we know that this is inverse to the previous assignment.
\end{proof}

    Since the framing data is simply a decoration on one of the flags, it is preserved in the above bijection. Hence, similar to Corollary \ref{cor:frame-braid-mod}, we can show the following corollary from Proposition \ref{prop:sheaf-X-rain=sate}.

\begin{cor}\label{cor:sheaf-rain=sate}
    Let $\beta \in \mbox{Br}_n^+$ be a positive braid, $\Delta$ be the half twist, $\Lambda_{\beta\Delta^2} \subset T^{*,\infty}\bD^2$ the Legendrian cylindrical braid closure of ${\beta\Delta^2}$ and $\Lambda_{\beta}^\prec \subset T^{*,\infty}\bD^2$ the Legendrian rainbow braid closure of $\beta$. Then there is an isomorphism
    $$\cM_1^{\mu,\textit{fr}}(\bD^2, \Lambda_{\beta}^\prec)_0 \cong \cM_1^{\mu,\textit{fr}}(\bD^2, \Lambda_{\beta\Delta^2})_0.$$
\end{cor}
\begin{remark}
    More conceptually, one can view $\Lambda_{\beta\Delta^2}, \Lambda_{\beta}^\prec$ as Legendrians in the contact boundary $S^3$ of the Weinstein manifold $\bD^4$, and adapt the proof in \cite{GPS18a,GPS18b} or \cite{NadShen20} by translating the sheaf category with singular support condition into the microlocal sheaf category over the relative Lagrangian skeleton of $(\bD^4, \Lambda)$ and apply Weinstein deformation invariance, keeping track of the microlocal rank. However, the proof above gives an explicit bijection that could be of independent interest and makes the following concrete computation possible.
\end{remark}


    Moreover, we can choose relative spin structures by assigning sign points and sign curves as in Definition \ref{def:sign} (Appendix \ref{appen:microlocal} Definition \ref{def:sign-knot} and \ref{def:sign-curve}) on $\Lambda_{\beta\Delta^2}$ and $\Lambda_{\beta}^\prec$ as in Lemma \ref{lem:decorate-flag}, and consider the decoration of the flags.\\
    
    Note that a complete flag of vector spaces is uniquely determined by a coset $x\mathrm{B}_+ \in \PGL_n/\mathrm{B}_+$, where $\mathrm{B}_+$ is a Borel subgroup; e.g.~ the subgroup of all upper triangular invertible matrices. A decorated complete flag of vector spaces with a common ambient volume form is uniquely determined by a coset  $x\mathrm{U}_+ \in \SL_n/\mathrm{U}_+$, where $\mathrm{U}_+$ is a unipotent subgroup, e.g.~ the subgroup of all strictly upper triangular matrices.

\begin{lemma}\label{lem:decorate-top}
    Let $\beta \in \mbox{Br}_n^+$ be a positive braid, $\Delta$ be the half twist, $\Lambda_{\beta\Delta^2} \subset T^{*,\infty}\bD^2$ the Legendrian cylindrical braid closure of $\beta \Delta^2$ and $\Lambda_{\beta}^\prec \subset T^{*,\infty}\bD^2$ the Legendrian rainbow braid closure of $\beta$. For each crossing in $\beta$, assign a sign point on the lower right half strand as in Definition \ref{def:sign}. Suppose that the unique compatible decorations of flags over $\beta$ in Lemma \ref{lem:decorate-flag} is given by
    $$U_+ \xrightarrow{s_{i_1}} x_1U_+ \xrightarrow{s_{i_2}} \dots \xrightarrow{s_{i_{k-1}}} x_{k-1}U_+ \xrightarrow{s_{i_k}} x_kU_+.$$
    Then the decoration of the flag $B_\text{bot}$ in the bottom region of $\Lambda_\beta$\footnote{The convention follows the one in Subsection \ref{ssec:triangle-positive} and \ref{ssec:ex-free-weave}.} determined by microlocal parallel transport maps is
    $$\omega_{\text{bot},j} = e_n \wedge e_{n-1} \wedge \dots \wedge e_{n-j+1}, \;\; 1 \leq j \leq n.$$
\end{lemma}
\begin{proof}
    Lemma \ref{lem:frame-decorate} shows that the decorations for the flags over $\beta$ determined by compatible decorations coincide with the decorations determined by microlocal parallel transport maps. Therefore we may assume that decoration of the flag $B_0$ on the leftmost region of $\beta$ is determined by volume forms 
    $$\omega_{0,j} = e_1 \wedge e_2 \wedge \dots \wedge e_j, \;\; 1\leq j\leq n,$$
    or equivalently, by Proposition \ref{prop:frame-markpt} the microstalks on the $n$ strands in the leftmost region are $e_1, e_2, \dots, e_n$. Then the microlocal parallel transport maps in the region $\Delta$ determines the microstalks on the top by $e_n, e_{n-1}, \dots , e_1$. Therefore, Proposition \ref{prop:frame-markpt} shows that the decoration of $B_{bot}$ is determined by
    $$\omega_{\text{bot},j} = e_n \wedge e_{n-1} \wedge \dots \wedge e_{n-j+1}, \;\; 1 \leq j \leq n.$$
    This completes the proof.
\end{proof}

\subsubsection{Cluster $\cA$-coordinates on double Bott-Samelson cells}
    The cluster $\cA$-coordinates on half decorated double Bott-Samelson cells $\mathrm{Conf}_{\beta}^e(\cC)$ constructed in \cite{ShenWeng21,GaoShenWeng20} are computed by generalized minors. Let us show that these cluster coordinates coincide with certain microlocal merodromies of framed sheaves (so as to compare with \cite{CasalsWeng22,CasalsZas20}).

    Following Lemma \ref{lem:decorate-flag}, given any point in the half-decorated Bott-Samelson cell, we can uniquely determine compatible decorations clockwise from right to left and then to the top as in Gao-Shen-Weng \cites{GaoShenWeng20}. In particular, we get a unique decorated sequence of flags
    \[\xymatrix@R=7mm{
    & & A^0 \ar@{-}[dll] \ar@{-}[drr] & & \\
    A_0 \ar[r]_{s_{i_1}} & A_1 \ar[r]_{s_{i_2}} & A_2 \ar[r]_{s_{i_3}} & \dots \ar[r]_{s_{i_k}} & A_k.
    }\]
    (Note that by Lemma \ref{lem:decorate-flag}, the decorations from $A_0$ to $A_k$ are indeed the ones determined by microlocal parallel trasnport maps.) Since decorated flags with a common ambient volume form can be characterized by cosets $x\mathrm{U}_+ \in \SL_n/\mathrm{U}_+$, we can fix a trivialization for the decorated flag $A_0$ and write the sequence of flags as
    \[\xymatrix@R=7mm{
    & & U_- \ar@{-}[dll] \ar@{-}[drr] & & \\
    U_+ \ar[r]_{s_{i_1}} & x_1U_+ \ar[r]_{s_{i_2}} & x_2U_+ \ar[r]_{s_{i_3}} & \dots \ar[r]_{s_{i_k}} & x_kU_+.
    }\]

    Consider the triangulation associated to the pair $(e, \beta)$ on top of Figure \ref{fig:BS-cluster}. On top of the triangulation, we draw $n-1$ parallel lines, labeled from top to bottom.\footnote{Following the conventions in Subsection \ref{ssec:triangle-positive}, in the string diagram and the plabic fence, the strands are labeled from top to bottom, (i.e.~$s_1$ is on the top while $s_{n-1}$ is on the bottom), which is opposite to the labeling of the crossings in the Legendrian braid closure.} Depending on the labeling of the base, each triangle places a node at one of the lines, cutting it into segments called strings. For each string $a$ on the $i$-th line, it will at least cross one diagonal in the triangulation, say, the $j$-th diagonal connecting $\mathrm{U}_-$ and $x_j\mathrm{U}_+$. Then we define the cluster $\cA$-coordinate associated to $a$ as the $i$-th principal minor of $x_j$, i.e.
    $$A_a = \Delta_i(x_j).$$
    (Note that there may be more than one choice for the diagonal $1 \leq j \leq k$, but it turns out that any diagonal that intersects the string $a$ will define the same function $A_a$, see \cite{ShenWeng21}*{Proposition 3.21}.)

\begin{figure}
  \centering
  \includegraphics[width=0.6\textwidth]{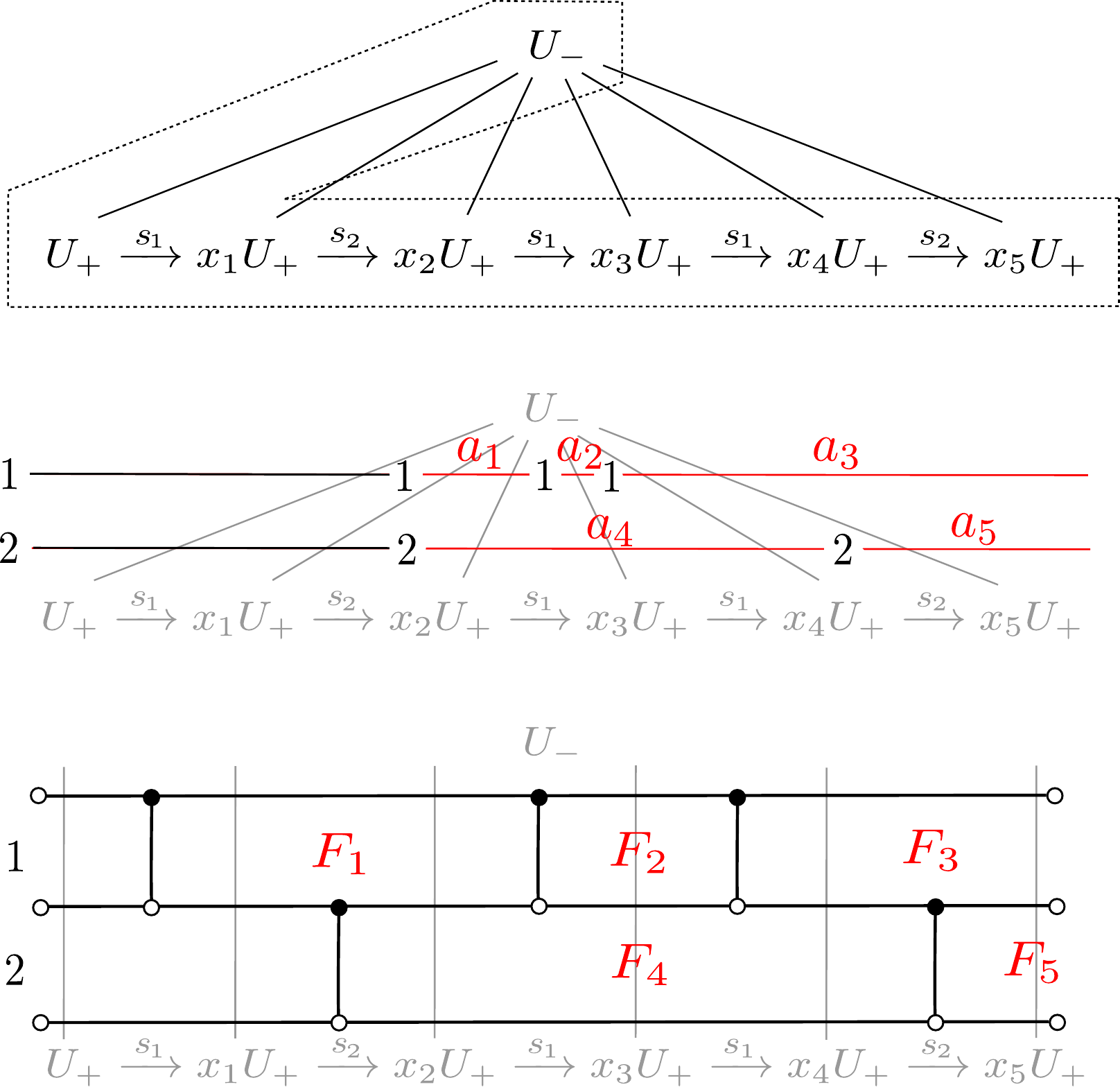}\\
  \caption{The diagram of flags on the top illustrates a point in $\mathrm{Conf}^e_{\beta}(\cC)$ for the positive braid $\beta = s_1s_2s_1^2s_2$. The decoration of the flags are determined clockwise from right to left and then to the top. The diagram in the middle is the string diagram (consisting of all red strings, either closed or extendable to the right) associated to the grey triangulation. The one on the bottom is the plabic fence $\bG_{\beta}$, where closed faces correspond to closed strings, and half-open faces correspond to the half-open strings.}\label{fig:BS-cluster}
\end{figure}

    In order to interpret the geometric meaning of the above minors, we first identify the set of strings with the set of null regions (i.e.~faces) in the corresponding plabic fence associated to the braid $\bG_{\beta}$. The following lemma is immediate from Figure \ref{fig:BS-cluster}.

\begin{lemma}
    Let $\beta \in \mbox{Br}_n^+$ be a positive braid. Then the set of closed strings $I_\text{str}^\text{cl}$ in the string diagram of the triangulation associated to $(e, \beta)$ corresponds bijectively to the set of closed null regions $I_\text{fence}^\text{cl}$ in the plabic fence associated to $\beta$.
\end{lemma}

    As a result, the set of closed strings corresponds bijectively to the basis of $H_1(L(\bG_{\beta}); \bZ)$ for the conjugate Lagrangian in Subsection \ref{sec:H1cycle-conj}, and by Corollary \ref{cor:1cycle-braid}, is identified with the basis of $H_1(\mathfrak{w}_{\beta}; \bZ)$ for the Legendrian weave $\mathfrak{w}_{\beta}$ by the corresponding long $\sf I$-cycles in Subsection \ref{sec:H1cycle-weave}\footnote{In this section, abusing notations, we write $\mathfrak{w}_\beta$ for the weave instead of $\widetilde{L}(\mathfrak{w}_\beta)$.}. Moreover, the half open strings are relative 1-cycles on the conjugate Lagrangian, and by Corollary \ref{cor:1cycle-braid} they are identified with the relative $\sf I$-cycles in the weave.

    Based on the string diagram and the plabic fence, we now explain a choice of dual relative 1-cycles in $(\mathfrak{w}_{\beta}^*, \Lambda_{\beta\Delta^2}^*)$ corresponding to all strings in order to define cluster $\cA$-coordinates from sheaves on $\mathfrak{w}_{\beta}$, where
    $$\mathfrak{w}_{\beta}^* = \mathfrak{w}_{\beta} \backslash T = \mathfrak{w}_\beta \backslash \{p_1, \dots, p_n\}, \;\; \Lambda_{\beta\Delta^2}^* = \Lambda_{\beta\Delta^2} \backslash T = \Lambda_{\beta\Delta^2} \backslash \{p_1, \dots, p_n\}.$$
    The $k$ grey edges in the triangulation in Figure \ref{fig:BS-cluster} connecting the top decorated flag $\mathrm{U}_-$ and $x\mathrm{U}_+$ determine $n$ vertical slices that split the front projection of $\mathfrak{w}_\beta$ (in particular, we choose the $k$ slices so that they avoid all the candy twists in the Legendrian weave), as depicted in Figure \ref{fig:braid-rel-cycle}. On each slice we have the Legendrian braid $\Delta$ connecting $n$ points on both ends.

    For a string $a$ on the $i$-th level that intersects the $j$-th diagonal, we consider the $j$-th vertical slice in the Legendrian weave, whose front consists of $n$ strands. Denote by $\delta_{a,j,r}$ the strand that starts from the $r$-th level on the top of the slice to the $(n-r)$-the level on the bottom of the slice in Figure \ref{fig:braid-rel-cycle}. Define 
    $$\eta_{a,j} = \sum_{r=1}^i \xi_{a,j,r} \in H_1(\mathfrak{w}_\beta^*, \Lambda_{\beta \Delta^2}^*; \bZ).$$
    From the definition, if a string $a$ intersects the $j$-th and $j'$-th diagonal in the triangulation, then $\eta_{a,j}$ and $\eta_{a,j'}$ are isotopic relative to $\Lambda_{\beta \Delta^2}^*$.
    
\begin{figure}
    \centering
    \includegraphics[width=1.0\textwidth]{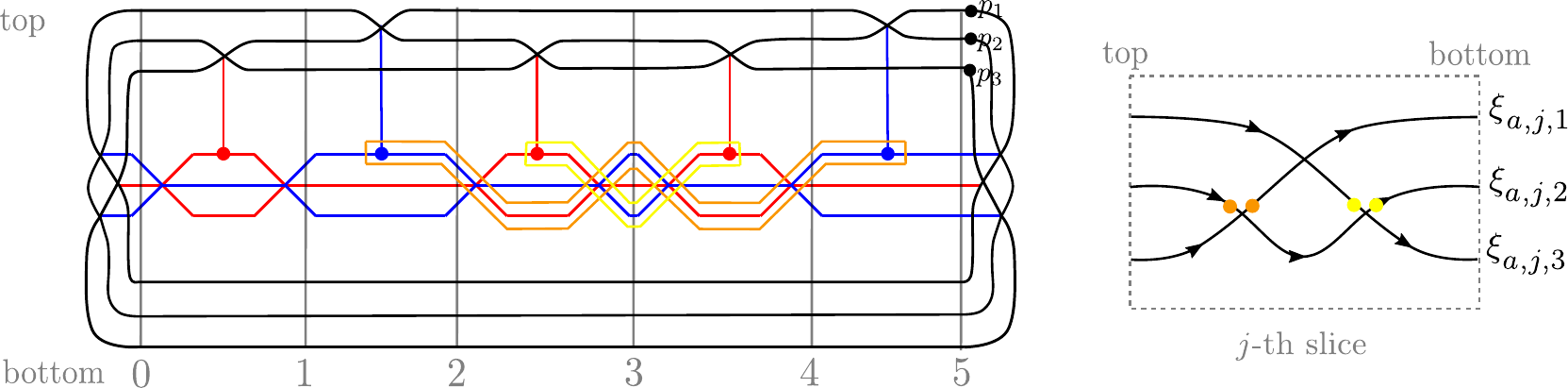}
    \caption{The vertical slices in the Legendrian weave $\mathfrak{w}_\beta$ in grey and the strands $\xi_{a,j,r}$ on each vertical slice. Two closed $1$-cycles on $\mathfrak{w}_\beta$ corresponding to the closed strings on the 1st and 2nd levels are respectively drawn in orange and yellow.}
    \label{fig:braid-rel-cycle}
\end{figure}

    Denote by $\eta_a$ the homology class in $H_1(\mathfrak{w}_\beta^*, \Lambda_{\beta\Delta^2}^*; \bZ)$ associated to the string $a$.

\begin{lemma}\label{lem:H1cycle-cluster}
    Let $\beta \in \mbox{Br}_n^+$ be a positive braid and $\Delta$ be the half twist. Then $\{\eta_a\, | \,a \in I_\text{str}\}$ forms a basis of $H_1(\mathfrak{w}_\beta^*, \Lambda_{\beta\Delta^2}^*; \bZ)$, and the subset $\{\eta_a | a \in I_\text{str}^\text{cl}\}$ forms a basis of $H_1(\mathfrak{w}_\beta, \Lambda_{\beta\Delta^2}; \bZ)$ dual to $\{\gamma_a | a \in I_\text{str}^\text{cl}\}$ in $H_1(\mathfrak{w}_\beta; \bZ)$ via the intersection product.
\end{lemma}
\begin{proof}
    We first check that $\{\eta_a | a \in I_\text{str}^\text{cl}\}$ forms a dual basis of $\{\gamma_a | a \in I_\text{str}^\text{cl}\}$ via the intersection product. We need to check that
    $$\left<\eta_a, \gamma_c \right> = \delta_{ac}.$$
    First, suppose $a \in I_\text{str}^\text{cl}$ is on the $i$-th level of the string diagram, and $c \in I_\text{str}^\text{cl}$ on the $i'$-th level. Note that according to the description of the 1-cycles $\{\gamma_c | c \in I_\text{str}^\text{cl}\}$, such a cycle $\gamma_c$ will intersect the slice of the weave at 2 points near the $i'$-th blue crossing with opposite signs as in Figure \ref{fig:braid-rel-cycle} and \ref{fig:braid-cycle-int}. Hence
    $$\left<\xi_{a,j,i}, \gamma_c \right> = 1, \,\,\, \left<\xi_{a,j,i+1}, \gamma_c \right> = -1, \,\,\, \left<\xi_{a,j,r}, \gamma_c \right> = 0, \, \forall\,r \neq i, i+1.$$
    This shows that indeed $\left<\eta_a, \gamma_c \right> = \delta_{ac}$. Hence $\{\eta_a | a \in I_\text{str}^\text{cl}\}$ forms a dual basis of $\{\gamma_a | a \in I_\text{str}^\text{cl}\}$.

\begin{figure}[h!]
    \centering
    \includegraphics[width=1.0\textwidth]{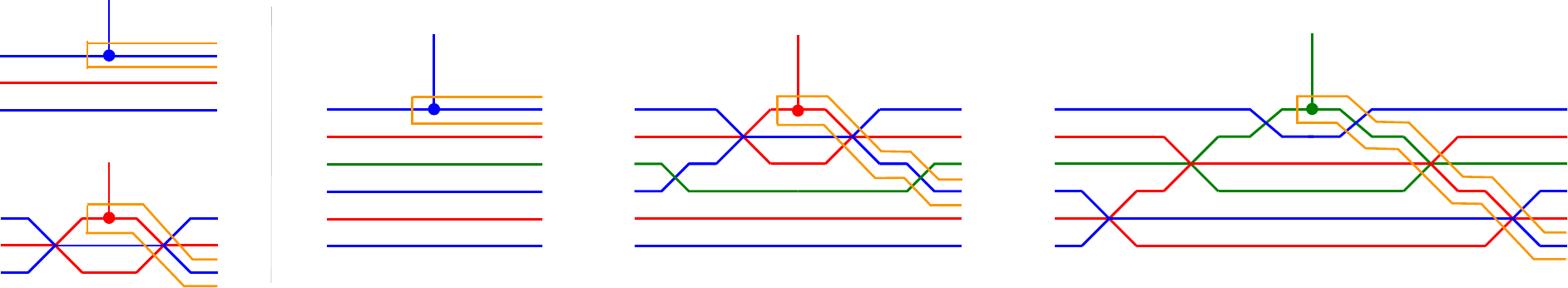}
    \caption{The closed 1-cycles on $\mathfrak{w}_\beta$ that start from a trivalent vertex in $G_i\,(1 \leq i \leq n-1)$, which correspond to the closed 1-cycles on the plabic fence $\bG_{\beta}$ on the $i$-the level, intersect the vertical slices at the $i$-th blue crossing.}
    \label{fig:braid-cycle-int}
\end{figure}

    Now we show that $\{\eta_a | a \in I_\text{str}\}$ is a basis of $H_1(\mathfrak{w}_\beta^*, \Lambda_{\beta\Delta^2}^*; \bZ)$. Consider the exact sequence
    $$0 \rightarrow H_1(\mathfrak{w}_\beta, \Lambda_{\beta\Delta^2}; \bZ) \rightarrow H_1(\mathfrak{w}_\beta^*, \Lambda_{\beta\Delta^2}^*; \bZ) \rightarrow \widetilde{H}_0(T; \bZ) \rightarrow 0.$$
    We have shown that $\{\eta_a |\, a \in I_\text{str}^\text{cl}\}$ forms a dual basis of $\{\gamma_a |\, a \in I_\text{str}^\text{cl}\}$, and thus it suffices to show that there is a backward map
    $$\widetilde{H}_0(T; \bZ) \rightarrow H_1(\mathfrak{w}_\beta^*, \Lambda_{\beta\Delta^2}^*; \bZ)$$
    that splits the exact sequence, and send generators to the set $\{\eta_a |\, a \in I_\text{str} \backslash I_\text{str}^\text{cl}\}$. Indeed, consider $s_i'$ to be to be the semicircle in $\mathfrak{w}_\beta$ around the base point $p_i \in \Lambda_{\beta\Delta^2}$. This defines such a backward map, and clearly
    $$\xi_i' \in \mathrm{coker}(H_1(\mathfrak{w}_\beta , \Lambda_{\beta\Delta^2}; \bZ) \rightarrow H_1(\mathfrak{w}_\beta^*, \Lambda_{\beta\Delta^2}^*; \bZ)).$$
    Note that $\xi_i'$ is isotopic to the strand $\xi_i$ on the rightmost slice of the weave relative to $\Lambda_{\beta\Delta^2}^*$. For $a \in I_\text{str} \backslash I_\text{str}^\text{cl}$ the unique half open string on the $i$-th level of the string diagram, since $\eta_a = \sum_{r=1}^i\xi_i$, we conclude that they form a basis of $\widetilde{H}_0(T; \bZ) \hookrightarrow H_1(\mathfrak{w}_\beta^*, \Lambda_{\beta\Delta^2}^*; \bZ)$.
\end{proof}

    Finally, under the isomorphism in (the proof of) Corollary \ref{cor:sheaf-rain=sate}, we show how to identify the cluster $\cA$-coordinates in $\mathrm{Conf}_{\beta}^e(\cC)$ defined in \cite{ShenWeng21} with the microlocal merodromies along relative 1-cycles in $H_1(\mathfrak{w}_\beta^*, \Lambda_{\beta\Delta^2}^*; \bZ)$ specified in the above construction.

\begin{proof}[\pfo Theorem \ref{thm:main_algebra1}]
    Consider a framed sheaf $\cF \in \cM_1^{\mu,\textit{fr}}(\bD^2, \Lambda_{\beta\Delta^2})_0$ that is the sheaf quantization of a framed local system on $\mathfrak{w}_\beta$. By Corollary \ref{cor:sheaf-rain=sate}, it is represented by the decorated system of flags
    \[\xymatrix{
    & & U_- \ar@{-}[dll] \ar@{-}[drr] & & \\
    U_+ \ar[r]_{s_{i_1}} & x_1U_+ \ar[r]_{s_{i_2}} & x_2U_+ \ar[r]_{s_{i_3}} & \dots \ar[r]_{s_{i_k}} & x_kU_+,
    }\]
    We show that for a given string $a \in I_\text{str}$ on the $i$-th level of the string diagram that intersects the $j$-th diagonal, the microlocal merodromy
    $$m_{L(\mathfrak{w}_\beta),\eta_{a,j}}(\widetilde{\cF}) = \Delta_i(x_j).$$
    Consider the $j$-th vertical slice of the Legendrian weave $\mathfrak{w}_\beta$ as in Figure \ref{fig:braid-rel-cycle}. On the bottom of the slice, by Proposition \ref{prop:sheaf-X-rain=sate}, we have a decorated flag represented by 
    $$(e_n, e_{n-1}, \dots, e_2, e_1).$$
    the trivialization data of microlocal stalks from top to bottom are given by $\Bbbk e_n, \Bbbk e_{n-1},\dots,$ $\Bbbk e_2, \Bbbk e_1$. On the top of the slice, we have a decorated flag represented by the matrix
    $$x_j = \left(v_{j1}, v_{j2}, \dots, v_{jn}\right),$$
    i.e.~the trivialization data of microlocal stalks from top to bottom are given by $\Bbbk v_{j1}, \Bbbk v_{j2}, \dots,$ $\Bbbk v_{j,n-1}, \Bbbk v_{jn}$. Note that $\eta_{a,j} = \sum_{r=1}^i\xi_{a,j,r}$. Along the strand $\xi_{a,j,r}$, the parallel transport map is the isomorphism
    $$t_{a,j,r}: \Bbbk\, v_{jr} \xrightarrow{\sim} \Bbbk\, e_r.$$
    Whenever there is a crossing where on the left the assigned vectors associated to the flag are $v_1, v_2$ and one the right the vectors are $u_2, u_1$, we may always assume that the isomorphism $\Bbbk \left<v_1, v_2\right> \xrightarrow{\sim} \Bbbk \left<u_2, u_1\right>$ are given by an upper triangular matrix $R \in \GL_2(\Bbbk)$ with
    $$R(v_1) = p_1u_1, \,\,\, R(v_2) = q_1u_1+q_2u_2.$$
    Therefore, under the above convention of trivializations, we may assume by induction that the isomorphism $R_{a,j,12...i-1,i}: \Bbbk \left<v_{j1}, v_{j2}, \dots, v_{jn}\right> \xrightarrow{\sim} \Bbbk \left<e_n, e_{n-1}, \dots, e_1\right>$ satisfies
    $$R_{a,j,12\dots i-1,i}(v_{jr}) = \sum_{s=1}^r t_{a,j,sr}e_s, \,\,\, t_{a,j,rr} = t_{a,j,r}, \,\,\, 1 \leq r \leq n.$$
    The microlocal merodromy along $\eta_{a,j}$ is $t_{a,j,12\dots i-1,i} = \prod_{r=1}^it_{a,j,r}$, so since $R_{a,j,12\dots i-1,i}$ is an upper triangular matrix, we can conclude that $t_{a,j,12\dots i-1,i}$ is indeed the isomorphism
    $$t_{a,j,12\dots i-1,i} = \det(R_{a,j,12\dots i-1,i}): \Bbbk\, v_{j1} \wedge v_{j2} \wedge \dots \wedge v_{ji} \xrightarrow{\sim} \Bbbk\, e_1 \wedge e_2 \wedge \dots \wedge e_i.$$
    Write $t_{12\dots i-1,i} = t_{a,j,12\dots i-1,i}$ for short. Write $v_{j1} \wedge v_{j2} \wedge \dots \wedge v_{ji} = \sum_{k_1 < k_2 < \dots < k_i} t_{k_1k_2\dots k_i} e_{k_1} \wedge e_{k_2} \wedge \dots \wedge e_{k_i}$. Take the wedge with $e_{i+1} \wedge \dots \wedge e_n$, we can conclude that
    $$v_{j1} \wedge v_{j2} \wedge \dots \wedge v_{ji} \wedge e_{i+1} \wedge \dots \wedge e_n = t_{12\dots i-1,i} e_1 \wedge e_2 \wedge \dots \wedge e_n = t_{12\dots i-1,i}.$$
    However, the left hand side is exactly the $i$-th principal minor of the matrix $x_j$, i.e.~$\Delta_i(x_j)$. Hence this shows that the microlocal merodromy is the cluster $\cA$-coordinates defined by principal minors of matrices.
\end{proof}

\subsubsection{Cluster $\cX$-coordinates on Bott-Samelson cells}\label{sec:clust-X-sheaf}
    Following Section \ref{sec:clust-A}, we discuss the cluster $\cX$-variables on $\cM_1(\bD^2, \Lambda_{\beta\Delta^2})_0$ and show that they are exactly the microlocal monodromies along the 1-cycles in $H_1(\mathfrak{w}_\beta; \bZ)$. In \cite{ShenWeng21}, cluster $\cX$-variables on the double Bott-Samelson cell $\mathrm{Conf}_{\beta}^e(\cB)$ are defined independently and they showed in \cite[Proposition 3.43]{ShenWeng21} that the cluster $\cX$ and $\cA$-variables satisfy the standard duality relation: 

\begin{thm}[\cite{ShenWeng21}]\label{thm:ShenWeng-X}
    Let $(I_\text{str}, I_\text{str}^\text{cl}, \left<-, -\right>_\text{str})$ be the intersection quiver associated to the string diagram of $\beta$. Let $\{A_a\}_{a \in I_\text{str}}$ be the set of cluster $\cA$-variables on $\mathrm{Conf}_{\beta}^e(\cC)$. Then the set of cluster $\cX$-variables $\{X_a\}_{a \in I_\text{str}}$ satisfy
    $$X_a = \prod_{c \in I_\text{str}^\text{cl}}A_c^{\left<a, c\right>}.$$
\end{thm}

    Our goal is to show that $X_a$ is the microlocal monodromy along $\gamma_a \in H_1(\mathfrak{w}_\beta; \bZ)$. It follows from Theorem \ref{thm:main_algebra1} that $A_a$ is the microlocal merodromy along $\delta_a \in H_1(\mathfrak{w}_\beta^*, \Lambda_{\beta\Delta^2}^*; \bZ)$. Hence we will prove the theorem by determining the relation between the corresponding 1-cycles. By Lemma \ref{lem:H1cycle-cluster}, $\{\delta_a\}_{a \in I_\text{str}^\text{cl}}$ form a basis of $H_1(\mathfrak{w}_\beta, \Lambda_{\beta\Delta^2}; \bZ) \subset H_1(\mathfrak{w}_\beta^*, \Lambda_{\beta\Delta^2}^*; \bZ)$. Consider the exact sequence
    $$0 \rightarrow H_1(\mathfrak{w}_\beta; \bZ) \rightarrow H_1(\mathfrak{w}_\beta^*, \Lambda_{\beta\Delta^2}^*; \bZ) \rightarrow \widetilde{H}_0(\Lambda_{\beta\Delta^2} \backslash T; \bZ) \rightarrow 0.$$
    We can view $H_1(\mathfrak{w}_\beta; \bZ)$ as a submodule in $H_1(\mathfrak{w}_\beta^*, \Lambda_{\beta\Delta^2}^*; \bZ)$. Then the relation between $\{\gamma_a\}$ and $\{\delta_c\}$ is as follows:

\begin{lemma}\label{lem:1cycle-AX}
    Under the identification $H_1(\mathfrak{w}_\beta; \bZ) \subset H_1(\mathfrak{w}_\beta^*, \Lambda_{\beta\Delta^2}^*; \bZ)$, we have
    $$\gamma_a = \sum_{c \in I_\text{str}^\text{cl}}\left<a, c \right>\delta_c.$$
\end{lemma}
\begin{proof}
    Similar to the proof of Lemma \ref{lem:H1cycle-cluster}, we can show that the exact sequence
    $$0 \rightarrow H_1(\mathfrak{w}_\beta; \bZ) \rightarrow H_1(\mathfrak{w}_\beta^*, \Lambda_{\beta\Delta^2}^*; \bZ) \rightarrow \widetilde{H}_0(\Lambda_{\beta\Delta^2}^*; \bZ) \rightarrow 0$$
    splits, where the backward map $\widetilde{H}_0(\Lambda_{\beta\Delta^2}^*; \bZ) \rightarrow H_1(\mathfrak{w}_\beta^*, \Lambda_{\beta\Delta^2}^*; \bZ)$ maps each connected component $b$ of $\Lambda_{\beta\Delta^2}^*$ to a path connecting $b$ and a fixed component $b_0$ of $\Lambda_{\beta\Delta^2}^*$.

    Denote the image of $b \in \pi_0(\Lambda_{\beta\Delta^2}^*)$ by $\gamma_b$. Then it is obvious that $\left< \gamma_a, \gamma_b \right> = 0$ for any $a \in I_\text{str}^\text{cl}$ and $b \in \pi_0(\Lambda_{\beta\Delta^2}^*)$. On the other hand, for any $b \in I_\text{str}^\text{cl}$, since $\{\eta_c | c\in  I_\text{str}^\text{cl}\}$ forms a dual basis of $\{\gamma_a | a \in I_\text{str}^\text{cl}\}$ via the intersection product,
    $$\sum_{c \in I_\text{str}^\text{cl}}\left<\left<a, c \right>\eta_c, \gamma_b \right> = \sum_{c \in I_\text{str}^\text{cl}}\left<a, c \right>\delta_{cb} = \left<\gamma_a, \gamma_b \right>.$$
    This proves the lemma.
\end{proof}

    In consequences, it follows from Lemma \ref{lem:1cycle-AX}, Theorem \ref{thm:main_algebra1}, and Theorem \ref{thm:ShenWeng-X} above, that the cluster $\cX$-variable $X_a$ coincides with the microlocal monodromy functions of the microlocal rank~1 sheaf $\widetilde{\cF} \in \Sh^{1}_{\mathfrak{w}_\beta}(\bD^2 \times \bR)$ along $\gamma_a$.
    

\subsection{Cluster coordinates on the moduli of framed local systems via ideal triangulations and Theorem \ref{cor:FG=GMN}}\label{sec:framed-locsys}
    Let $\Sigma$ be a smooth surface with marked points and consider the moduli space of sheaves $\cM_1(\Sigma, \bigcirc_n)_0$ for the trivial Legendrian $n$-satellite $\bigcirc_n \subset T^{*,\infty}\Sigma$ of outward unit conormals around marked points, with acyclic stalks at the marked points. The moduli space of sheaves $\cM_1(\Sigma, \bigcirc_n)_0$ is isomorphic to the cluster varieties $\cX_{\GL_n}(\Sigma)$ of rank~$n$ framed local systems on $\Sigma$, studied by Fock-Goncharov \cite{FockGon06a,Gon17Web}.
    
    First, let us review two classes of cluster $\cX$-variables associated to an ideal $n$-triangulation, those from bipartite graphs considered in \cite{GonKen13,Gon17Web}, from non-abelianization maps considered in \cite{GMN13,GMN14}, and those from sheaf quantizations of conjugate Lagrangians \cite{STWZ19}. Then we prove Theorem \ref{cor:FG=GMN}, comparing these two classes of $\cX$-variables, at the end of the section.

\begin{remark}
    Note that \cite{FockGon06a} considered cluster coordinates on the variety $\cX_{\PGL_n}(\Sigma)$ of rank~$n$ $\PGL_n$-local systems on $\Sigma$. However, we focus on $\cX_{\GL_n}(\Sigma)$, as cluster $\cX$-variables there determine the $\cX$-variables on $\cX_{\PGL_n}(\Sigma)$ via a projection map.\footnote{In \cite{FockGon06a}, there is also a corresponding cluster $\cA$-variety $\cA_{\GL_n}(\Sigma)$,  resp.~$\cA_{\SL_n}(\Sigma)$, whose coordinates determine the cluster $\cX$-coordinates of $\cX_{\GL_n}(\Sigma)$, resp.~$\cX_{\PGL_n}(\Sigma)$. A similar construction with microlocal holonomies, following in Section \ref{sec:clust-A} should likely exist but it has not yet been constructed in the literature. We thus only focus on the cluster $\cX$-structure.}
\end{remark}

\subsubsection{Framed local systems on a surface}
    The moduli space of sheaves $\cM_1(\Sigma, \bigcirc_n)_0$, see Appendix \ref{appen:sheaf-moduli}, can be identified with the following moduli space of framed local systems:

\begin{definition}[\cite{FockGon06a}*{Definition 1.2}]\label{def:framed-loc}
    Let $\Sigma$ be a surface with marked points $\{x_1, ..., x_m\}$, $\Sigma_\text{op} = \Sigma \backslash \{x_1, ..., x_m\}$ and $\Sigma_\text{cl} = \Sigma \backslash \bigcup_{i=1}^mB_\epsilon(x_i)$. A framed $\GL_n$-local system $(\cL, \sigma)$ on $\Sigma_\text{op}$ is a $\GL_n$-local system $\cL$ with a flat section $\sigma: \partial \Sigma_\text{cl} \rightarrow \cL/\mathrm{B}_+$. The space $\cX_{\GL_n}(\Sigma)$ is the moduli space of framed local systems on $\Sigma_\text{op}$.
    
    A framed $\PGL_n$-local system $(\cL, \sigma)$ on $\Sigma_\text{op}$ is a $\PGL_n$-local system $\cL$ with a flat section $\sigma: \partial \Sigma_\text{cl} \rightarrow \cL/\mathrm{B}_+.$ The space $\cX_{\PGL_n}(\Sigma)$ is the moduli space of framed local systems on $\Sigma_\text{op}$.
\end{definition}

\begin{remark}\label{rem:frameloc-proj}
    There exists a natural projection map $\cX_{\GL_n}(\Sigma) \rightarrow \cX_{\PGL_n}(\Sigma)$ with fiber $(\Bbbk^\times)^m$. The fiber over $(\cL_0, \sigma_0)$ parametrizes framed $\GL_n$-local systems $(\cL, \sigma)$ where the flat section $\sigma: \partial \Sigma_\text{cl} \rightarrow \cL/\mathrm{B}_+$ is reduce to $\sigma_0$ under the quotient map $\GL_n \rightarrow \PGL_n$. Note that the flat section lives in the Cartan subgroup, and $\GL_n \rightarrow \PGL_n$ factors through the corresponding Cartan subgroups $(\Bbbk^\times)^n \rightarrow (\Bbbk^\times)^{n-1}$. Remarks \ref{rem:frozen-abelian} and \ref{rem:frozen-dim} explain that, from the point of view of cluster varieties, the projection map is exactly defined by forgetting the frozen variables.
\end{remark}

    \noindent The following Proposition \ref{prop:framedloc-sheaf} is folklore, see \cite{GonKon21}*{Lemma 8.11} for a similar result in a slightly different setting, and we include its proof so it is explicitly available:

\begin{prop}\label{prop:framedloc-sheaf}
    Let $\Sigma$ be a surface with marked points $\{x_1, ..., x_m\}$. Let $\bigcirc_n$ be the union of the inward conormal bundle of $n$ concentric circles around $x_1, ..., x_m$. Let $\cM_1(\Sigma, \bigcirc_n)_0$ be the moduli space of microlocal rank 1 sheaves in $\Sh^1_{\bigcirc_n}(\Sigma)$ with acyclic stalks at $x_1, ..., x_m$. Then there is an equivalence of moduli spaces
    $$\cX_{\GL_n}(\Sigma) \xrightarrow{\sim} \cM_1(\Sigma, \bigcirc_n)_0.$$
\end{prop}
\begin{proof}
    We fix the following notations. Consider the stratfication of $\Sigma$ defined by $\pi(\Lambda_n) \subset \Sigma$. Write $A_0(x_j)$ for the closed disks inside the smallest concentric circle around $x_j\,(1\leq j\leq m)$. Write $A_i(x_j)$ for the annulus between the $i$-th and $(i+1)$-th concentric circles (from inside to outside) around $x_j\,(1\leq j\leq m)$ containing the $i$-th concentric circle. Without loss of generality, let $B_\epsilon(x_j) = \bigcup_{i=0}^{n-1}A_i(x_j)$ and $\Sigma_\text{cl} = \Sigma \backslash \bigcup_{j=1}^mB_\epsilon(x_j)$.

    For a $GL_n$-local system $\cL$, a section $\sigma$ near $\partial \Sigma_\text{cl}$ determines a framing $\{s_1, ..., s_n\}$ near $\cL|_{\partial B_\epsilon(x_i)}$ for each marked point $x_i\,(1\leq i\leq m)$. Therefore, we can consider the flag of $S^1$-local systems near each marked point $x_i\,(1\leq i\leq m)$
    $$0_{S^0} = \cL_0 \subset \cL_1 \subset \cL_2 \subset ... \subset \cL_n \cong \cL|_{\partial B_\epsilon(x_j)}.$$
    Therefore we can define a sheaf $\cF \in \Sh^1_{\bigcirc_n}(\Sigma)$ such that (1)~in each small disk $A_0(x_j)$, $\cF|_{A_0(x_j)} = 0_{A_0(x_j)}$; (2)~in the annulus $A_i(x_j)$, the restriction of $\cF|_{A_i(x_j)}$ is a local system determined by (the pullback of) a $S^1$-local system $\cL_i$ to $A_i(x_j)$; (3)~finally in $\Sigma_\text{cl}$, the restriction $\cF|_{\Sigma_\text{cl}} \simeq \cL|_{\Sigma_\text{cl}}$. Locally when we pass the $i$-th concentric circle, we assign the parallel transport map (see Example \ref{ex:sheafcombin}) from $\cF|_{A_{i-1}(x_j)}$ to $\cF|_{A_i(x_j)}$ to be the natural inclusion
    $$\cL_{i-1} \hookrightarrow \cL_i.$$
    Since along the $i$-th concentric circle around $x_j\,(1\leq j\leq m)$, the microlocal monodromy is a rank~1 local system $\cL_i/\cL_{i-1}$. Therefore the sheaf $\cF$ lies in $\cM(\Sigma, \bigcirc_n)_0$.

    Conversely, consider any sheaf $\cF \in \Sh^1_{\bigcirc_n}(\Sigma)$ with acyclic stalks at $x_1, ..., x_m$. Then define $\cL = \cF|_{\Sigma_\text{cl}^\circ}$ where $\Sigma_\text{cl}^\circ$ is the interior of $\Sigma_\text{cl}$, diffeomorphic to $\Sigma_\text{op}$. We now try to define a flag of $S^1$-local systems
    $$0_{S^1} = \cL_0 \subset \cL_1 \subset \cL_2 \subset ... \subset \cL_n \cong \cL|_{\partial B_\epsilon(x_i)}$$
    along $\cL|_{\partial B_\epsilon(x_j)}$ for each marked point $x_j\,(1\leq j\leq m)$. Note that for any $1 \leq i \leq n-1$, $\cF|_{A_i(x_j)}$ is a locally constant sheaf, and thus come from the pullback of an $S^1$-local system $\cL_i$. We can check that $\cL_i\,(1\leq i\leq n-1)$ defines a flag of $S^1$-local systems. Note that along the $i$-th concentric circle, the microlocal monodromy
    $$\mathrm{Cone}(\cL_{i-1} \rightarrow \cL_i)$$
    is a rank~1 local system. When $i = 0$, $\cL_0 \cong 0_{S^1}$, using this fact we can tell that $\cL_1$ is a rank~1 local system (in degree 0) and $\cL_0 \rightarrow \cL_1$ is an injection. Suppose $\cL_{i-1}$ is a rank $(i-1)$ local system (in degree 0), then using the same fact we can tell that $\cL_i$ is a rank $i$ local system (in degree 0), and $\cL_{i-1} \rightarrow \cL_i$ is an injection. Therefore we can get a flag of $S^1$-local systems, which determines the framing around $\cL|_{\partial B_\epsilon(x_j)}\,(1\leq j\leq m)$.

    These two constructions from framed local systems to microlocal sheaves and back are inverses to each other, and thus we have completed the proof.
\end{proof}
\begin{remark}
    Similarly, one can show that given a collection of braids $\beta = \{\beta_1, ..., \beta_m\}$ around each puncture, the moduli stack of $\beta$-framed local systems $\cX_{\GL_n}(\Sigma,\beta)$ introduced in \cite{GonKon21}*{Section 9} is isomorphic to $\cM_1(\Sigma, \Lambda_\beta)_0$, which is also referred to as the moduli of $\beta$-filtered local system in \cite{STWZ19}*{Section 3.1}.
\end{remark}

\subsubsection{Non-abelianization map and sheaves over weaves}
    As has been noticed in \cite{CasalsZas20}*{Section 3.1}, the Legendrian $n$-weave of an $n$-triangle is closedly related to spectral covers introduced by D.~Gaiotto, G.~Moore and A.~Neitzke \cite{GMN13,GMN14}, at least on the smooth level. Here we interpret the cluster $\cX$-variables on the moduli of framed local systems coming from non-abelianization \cite{GMN13} using the language of Legendrian weaves and microlocal sheaves.

    When defining cluster $\cX$-coordinates, \cite{GMN13,GMN14} only use the information of the Lagrangian as smooth branched covers and define the coordinates as monodromies of local systems on that Lagrangian spectral curve. The spectral covers they considered are Lagrangian branched covers in the cotangent bundle of the triangle with exactly $n^2$ simple branched points, and hence there is a diffeomorphism between these spectral covers and (Lagrangian projections of) the weaves coming from the $n$-triangulation, that intertwines with the projection maps.

    Starting from a Riemann surface $\Sigma$ with punctures, consider a free Legendrian weave $\widetilde{L} \subset J^1(\Sigma_\text{op})$, with compactification  in $J^1(\Sigma_\text{cl})$. Then the boundary $\partial \widetilde{L} \subset \partial \Sigma_\text{cl} \times \bR$ determines a Legendrian $n$-braid closure $\Lambda_\beta \subset J^1(\partial \Sigma_\text{cl})$, which, under the convention in Subsection \ref{sec:weave-at-infty}, can also be identified with Legendrian $n$-braid closures in $T^{*,\infty}\Sigma$ as satellites around the inward unit conormals of small disks around the punctures.
    
    Let $\cX_{\mathfrak{w}} = H^1(\mathfrak{w}; \Bbbk^\times)$ be the moduli space of rank~1 local systems on the Legendrian weave $\mathfrak{w}$.\footnote{In the literature, for example \cite{GMN13}*{Section 10}, \cite{Kuwagaki20WKB} and \cite{GonKon21}*{Section 8}, twisted local systems are used instead of local systems. Here we identify twisted local systems with local systems by choosing a basis of smooth 1-cycles together with a lifting in the circle bundle (by keeping track of their tangent vectors). This is equivalent to fixing a (relative) spin structure. See Appendix \ref{appen:microlocal} and also \cite{Kuwagaki20WKB}*{Section 8.2} or \cite{GonKon21}*{Section 8.3}.} Note that any 1-cycle in $H_1(\mathfrak{w}; \bZ)$ defines a regular function on $\cX_{\mathfrak{w}}$ by simply taking monodromy along the 1-cycle.

\begin{prop}\label{prop:weave-cluster}
    Let $\mathfrak{w} \subset J^1\Sigma_\text{op}$ be a free Legendrian weave with boundary being a Legendrian braid closure $\Lambda \subset J^1(\partial \Sigma_\text{cl}) \subset T^{*,\infty}\Sigma_\text{op}$. Then there is an embedding of algebraic stacks
    $$\cX_{\mathfrak{w}} \hookrightarrow \cM_1(\Sigma, \Lambda)_0.$$
\end{prop}
\begin{proof}
    This follows from Theorem \ref{thm:JinTreumann} and Subsection \ref{sec:quan-weave-link}. Indeed, the category of rank~1 local systems $\Loc^1(\mathfrak{w})$ embeds into $\Sh^1_{\mathfrak{w}}(\Sigma \times \bR)$, and taking proper push forward via the projection defines a fully faithful embedding into $\Sh^1_\Lambda(\Sigma)$. Thus we know that
    $$H^1(\mathfrak{w}, \Bbbk^\times) \hookrightarrow \cM_1(\Sigma, \Lambda)_0$$
    is an embedding of algebraic stacks.
\end{proof}

    As a corollary, we obtain the non-abelianization map from sheaf quantization of Legendrian weaves. The word non-abelianization \cite{GMN13} refers to the process of getting a framed (non-abelian) $\GL_n$-local system on $\Sigma$ from an (abelian) $\GL_1$-local system on $\mathfrak{w}$.

\begin{cor}\label{cor:weave-framedloc}
    Let $\Sigma$ be a surface with marked points and $\mathfrak{w}_n^* \subset J^1\Sigma_\text{op}$ be the free Legendrian weave associated to the ideal $n$-triangulation. Then there is a non-abelianization map
    $$\cX_{\mathfrak{w}_n^*} \hookrightarrow \cX_{\GL_n}(\Sigma).$$
\end{cor}
\begin{proof}
    This follows from Proposition \ref{prop:framedloc-sheaf} and Proposition \ref{prop:weave-cluster}. Indeed, the condition that the $n$-graph $\bG_n^*$ has no vertices at the punctures ensures that $\mathfrak{w}_n^*$ is a Legendrian weave with boundary $\bigcirc_n \subset J^1(\partial \Sigma_\text{cl})$ being the trivial braid. Then the result follows from the fact that $\cM_1(\Sigma, \bigcirc_n)_0 \cong \cX_{\GL_n}(\Sigma)$.
\end{proof}
\begin{remark}\label{rem:frozen-abelian}
    We can consider a basis of a subspace in $H_1(\mathfrak{w}_n^*; \bZ)$ (see Lemma \ref{lem:H1cycle-conj}) given by the $\sf I, \sf Y$-cycles and the small loops around the marked points, forming the vertices of the quiver defining the cluster variety. The small loops around marked points are frozen variables (which cannot be mutated), and by forgetting the frozen variables one gets the cluster coordinates on $\cX_{\PGL_n}(\Sigma)$ via the projection map in Remark \ref{rem:frameloc-proj} (compare Remark \ref{rem:frozen-dim}). Indeed, one can compare the basis of $H_1$ for the Legendrian surface in Corollary \ref{cor:1cycle-triangle} and \ref{cor:1cycle-triangle-part}, where the difference is exactly the loops around the marked points.
\end{remark}
\begin{remark}\label{rem:abelianization}
    If we forget the data of the flags near the punctures, this will produce a map from $\cX_{\mathfrak{w}}$ to the moduli space of rank~$n$ local systems on $\Sigma_\text{cl}$, or equivalently, the moduli space of microlocal rank~$n$ sheaves with singular support on the 1-strand braid $\bigcirc_1$ around the punctures $\cM_n(\Sigma, \bigcirc_1)_0$. This is just the microlocal non-abelianization map introduced in \cite{STWZ19}*{Section 3.1} or the Legendrian degeneration in \cite{TWZ19Kasteleyn}*{Section 7}.
\end{remark}

    The cluster $\cX$-variables \cite{GMN13,GMN14} in the study of spectral networks are defined via the non-abelianization map, by the monodromies of the rank~1 local system in $\cX_{\mathfrak{w}}$ along a specific collection of 1-cycles. These 1-cycles are defined by three adjacent branched points of the spectral covering \cite{GMN14}*{Section~5.9, Figure~8 \& 9}, and under the diffeomorphism between the spectral covers and (Lagrangian projections of) the weaves, are exactly the $\sf Y$-cycles in the weave defined by three adjacent blue trivalent vertices in Corollary \ref{cor:1cycle-triangle}.

    By Theorem \ref{thm:JinTreumann}, the monodromies of the local systems in $\Loc^1(\mathfrak{w})$ are the same as the microlocal monodromies of the sheaf quantization in $\Sh^1_{\mathfrak{w}}(\Sigma \times \bR)$. Therefore, Gaiotto-Moore-Neitzke's cluster $\cX$-variables are the same as the microlocal monodromies of the sheaf quantizations along $\sf Y$-cycles specified in Corollary \ref{cor:1cycle-triangle}.

\subsubsection{Dimers and sheaves over conjugate Lagrangians}
    Dimer models in cluster varieties, based on bipartite graphs, are discussed in \cite{FockGon06a,Gon17Web} and they are shown to yield cluster $\cX$-variables on the moduli of framed $\GL_n$-local systems on punctured surfaces. In this context, we start with a Riemann surface $\Sigma$ and a bipartite graph $\bG \subset \Sigma$. Let $\cX_\bG = H^1(\bG; \Bbbk^\times)$ be the moduli space of rank~1 local systems on $\bG$. Any null region $F$ (meaning neither black nor white) in $\Sigma \backslash \bG$, which are called faces in \cite{Gon17Web}, defines a cycle $\gamma_F \in H_1(\bG; \bZ)$. They can be viewed as regular functions on $\cX_\bG$. At this stage the following is rather immediate:

\begin{prop}\label{prop:conj-cluster}
    Let $\bG$ be a bipartite graph on a surface $\Sigma$, and $\Lambda(\bG)$ be the alternating Legendrian link determined by $\bG$. Let $\cM_1(\Sigma, \Lambda(\bG))_0$ the moduli space of microlocal rank~1 sheaves in $\Sh^1_{\Lambda(\bG)}(\Sigma)$ with acyclic stalks at the marked points in $\Sigma$. Then there is an embedding of algebraic stacks
    $$\cX_\bG \hookrightarrow \cM_1(\Sigma, \Lambda(\bG))_0.$$
\end{prop}
\begin{proof}
    Note that the conjugate Lagrangian $L(\bG)$ associated to $\Lambda(\bG)$ admits a deformation retraction to the graph $\bG$. Thus $\cX_\bG \cong H^1(L(\bG); \Bbbk^\times)$, both equipped with the coordinate system given by monodromies along $\gamma_F$ for null regions $F$. Following Theorem \ref{thm:JinTreumann}, we get a fully faithful embedding of $\Loc^1(L(\bG))$ into $\Sh^1_{\Lambda(\bG)}(\Sigma)$, and hence an embedding of algebraic stacks
    $$H^1(L(\bG); \Bbbk^\times) \hookrightarrow \cM_1(\Sigma, \Lambda(\bG))_0,$$
    and this completes the proof.
\end{proof}

    With the above preparations established, we now turn to the following specific setting in Corollary \ref{cor:FG=GMN}. Following \cite{Gon17Web}, we can consider a bipartite graph $\bG_n^*$ that come from an ideal $n$-triangulation on a Riemann surface $\Sigma$. It is implicit in \cite{Gon17Web}, and explicitly written (in a more general setting) in \cite{GonKon21} that the moduli space of rank 1 local systems on $\Gamma_n$ embeds into the moduli space of rank~$n$ framed local systems. Now we can use the Hamiltonian isotopies from Theorem \ref{thm:main1} so as to conclude the comparison between the bipartite graph dimer model and the spectral cover reads as follows:

\begin{cor}\label{cor:dimer-framedloc}
    Let $\Sigma$ be a closed surface with marked points $\{x_1, ..., x_m\}$ and an associated ideal $n$-triangulation. Let $\bG_n^*$ be the $A_n^*$-bipartite graph associated to the ideal $n$-triangulation. Then there is an embedding of algebraic varieties
    $$\cX_{\bG_n^*} \hookrightarrow \cX_{\GL_n}(\Sigma).$$
    In addition, for $\mathfrak{w}_n^*$ the spectral cover of the $n$-triangulation, there is a canonical isomorphism $\cX_{\bG_n^*} \cong \cX_{\mathfrak{w}_n^*}$ such that the embedding is identified with the embedding
    $$\cX_{\mathfrak{w}_n^*} \hookrightarrow \cX_{\GL_n}(\Sigma)$$
    defined by sheaf quantization of Legendrian weaves in Corollary \ref{cor:weave-framedloc}.
\end{cor}
\begin{proof}
    It follows immediately from Theorem \ref{thm:main1} that the Hamiltonian isotopy from the conjugate Lagrangian $L(\bG_n^*)$ to the projection $w_n^*$ of the Legendrian weave $\mathfrak{w}_n^*$ defines an isomorphism $H^1(\bG_n^*, \Bbbk^\times) \cong H^1(\mathfrak{w}_n^*, \Bbbk^\times)$. Proposition \ref{prop:framedloc-sheaf} and Corollary \ref{cor:weave-framedloc} therefore implies the existence of the embedding
    $$H^1(\bG_n^*, \Bbbk^\times) \xrightarrow{\sim} H^1(\mathfrak{w}_n^*, \Bbbk^\times) \hookrightarrow \cX_{\GL_n}(\Sigma).$$
    By Hamiltonian invariance in Proposition \ref{prop:JinTreu-inv} we know that this embedding coincides with the embedding in Proposition \ref{prop:conj-cluster}.
\end{proof}
\begin{remark}\label{rem:frozen-dim}
    One can also consider a basis of $H_1(\bG_n^*; \bZ)$ given by all the null regions of the bipartite graph, forming the vertices of the quiver defining the cluster variety. Then the null regions around marked points are frozen variables (which cannot be mutated), and by forgetting the frozen variables one gets the cluster coordinates on $\cX_{\PGL_n}(\Sigma)$ via the projection map in Remark \ref{rem:frameloc-proj} (compare Remark \ref{rem:frozen-abelian}). Indeed, one can compare the basis of $H_1$ for the corresponding $A_n^*$-bipartite graphs as in Subsection \ref{ssec:triangle-positive}, or Corollary \ref{cor:1cycle-triangle} and \ref{cor:1cycle-triangle-part}, where the difference is exactly the null regions around the marked points.
\end{remark}

    Corollary \ref{cor:dimer-framedloc} shows that in fact, the cluster coordinates $\cX_{\bG_n^*} \hookrightarrow \cX_{\GL_n}(\Sigma)$ defined by microlocal monodromies of sheaf quantizations of conjugate Lagrangians are identical as the cluster coordinates $\cX_{\mathfrak{w}_n^*} \hookrightarrow \cX_{\GL_n}(\Sigma)$ defined by microlocal monodromies of sheaf quantizations of the Legendrian weaves.

\subsubsection{Proof of Theorem \ref{cor:FG=GMN}}
    Corollary \ref{cor:weave-framedloc} implies that the cluster coordinates in (3) \& (4) are identical, and Corollary \ref{cor:dimer-framedloc} implies that the cluster coordinates in (1) \& (3) are identical. Therefore it suffices to identify the Fock-Goncharov cluster coordinates in (2) with any other class of cluster coordinates.

    For that, note that Fock-Goncharov's cluster $\cX$-coordinates on the moduli of framed local systems $\cX_{\GL_n}(\Sigma)$ \cite{FockGon06a,Gon17Web} are exactly defined by an embedding $\cX_{\bG_n^*} \hookrightarrow \cX_{\GL_n}(\Sigma)$. For each null component (or face) in the alternating coloring, they assigned a rational function, whose explicit formulas are the cross ratios and triple products as in Section \ref{sec:quan-weave-mon}. The 1-cycles are in $H_1(\bG_n^*; \bZ)$ but the functions are microlocal monodromies of sheaves over $H_1(\mathfrak{w}_n^*; \bZ)$. By Corollary \ref{cor:1cycle-triangle}, the isomorphism $\cX_{\bG_n^*} \cong \cX_{\mathfrak{w}_n^*}$ identifies standard monodromy functions along the standard 1-cycles on both sides. Therefore the microlocal monodromies along 1-cycles in the Legendrian weave are exactly the ones along the corresponding 1-cycles in the conjugate Lagrangian. Hence the Fock-Goncharov cluster coordinates in (2) are indeed the cluster coordinates defined by microlocal monodromies.

\newpage
\appendix
\section{DG-Categories in the microlocal theory of sheaves}\label{appen:cat-sheaf}

Let $M$ be a smooth manifold, $\La\sse T^{*,\infty}M$ a real subanalytic Legendrian submanifold, $L\sse T^*M$ an embedded exact Lagrangian such that $\dd_\infty L=\La$ and $\widetilde{L}\sse J^1M$ its Legendrian lift. Let $\Mod(\Bbbk)$ be the dg-category of chain complexes of $\Bbbk$-modules, $\Bbbk$ a field of characteristic zero. The following diagram displays dg-categories that are relevant in the study of the Legendrian isotopy class of $\La$:

\begin{center}
\begin{tikzcd}
  & \Loc^{pp}(\La) \arrow[r,hook] & \Loc^{cpt}(\La) \arrow[dl,swap, "m_\La^L"] \arrow[rr,hook] & & \Loc(\La) \\

  \Sh^{pp}_\La(M)\arrow[r,hook, "\iota_2"] \arrow[ur, "m_\La"] & \Sh^{cpt}_\La(M) \arrow[r,hook, "\iota_1"]\arrow[d,swap,"(i^{-1}_{M \times \{+\infty\}})^L"'] &\Sh^{ctr}_\La(M) \arrow[r,hook, "\iota_0"] & \Sh_\La(M) \arrow[ur, "m_\La"] & \\
 
  \Sh^{pp}_{\widetilde{L}}(M\times\overline{\R})\arrow[r,hook,swap, "\tilde\iota_2"]\arrow[u,"i^{-1}_{M \times \{+\infty\}}"'] \arrow[dr,swap, "m_{\widetilde{L}}"] & \Sh^{cpt}_{\widetilde{L}}(M\times\overline{\R}) \arrow[r,hook,swap, "\tilde\iota_1"] &\Sh^{ctr}_{\widetilde{L}}(M\times\overline{\R}) \arrow[r,hook,swap, "\tilde\iota_0"] \arrow[u,"i^{-1}_{M \times \{+\infty\}}"']& \Sh_{\widetilde{L}}(M\times\overline{\R}) \arrow[dr,swap, "m_{\widetilde{L}}"] \arrow[u,"i^{-1}_{M \times \{+\infty\}}"']& \\
  
  & \Loc^{pp}(\widetilde{L}) \arrow[r,hook] & \Loc^{cpt}(\widetilde{L}) \arrow[ul, "m_{\widetilde{L}}^L"] \arrow[rr,hook] & & \Loc(\widetilde{L}) \arrow[uuu,"i^{-1}_\La"'] \\
\end{tikzcd}
\end{center}

 

\noindent By definition, the chart legend for the above categories reads as follows:\\

\begin{itemize}
    \item[(a)] {\bf $\Loc(\La)$ is the dg-derived category of local systems of $\Mod(\Bbbk)$ on $\La$}, i.e.~locally constant sheaves of $\Bbbk$-modules on $\La$.\\
    
    \noindent Let us assume that $\La$ is connected and $*_\La\in\La$ is a base point. A computationally useful characterization of $\Loc(\La)$ is that it is dg-equivalent to the dg-category of dg-modules on the dg-algebra $C_{-*}(\Omega\La,*_\La;\Bbbk)$ of chains over the loop space, endowed with the Pontryagin product. Indeed, by \cite[Appendix A.4]{Lurie_HigherAlgebra}, $\Loc(\La)$ is equivalent to the category of (quasi)functors $\Fun_{Simp}(\mbox{Sing}(\La,*_\La),N_{dg}(\Mod(\Bbbk)))$ between simplicial sets, where $\mbox{Sing}(\La)$ is the Kan complex of singular simplices of $\La$, $*_\La$ is a base point, and $N_{dg}$ is the dg-nerve from the dg-category of dg-categories to simplicial sets. By adjunction, this is equivalent to the dg-category of dg-functors $\Fun_{dg}(N_{dg}^L(\mbox{Sing}(\La,*_\La)),\Mod(\Bbbk))$ where $N_{dg}^L$ is the left adjoint of $N_{dg}$. It is then computed that the dg-category $N_{dg}^L(\mbox{Sing}(\La,*_\La))$ is dg-equivalent to the dg-category (with one object) associated to $C_{-*}(\Omega\La,*_\La;\Bbbk)$; e.g.~see \cite{RiveraZeinalian20} for details.\\

\item[(b)] {\bf $\Loc^{cpt}(\La)$ is the full dg-subcategory of compact objects in $\Loc(\La)$, and $\Loc^{pp}(\La)$ is the full dg-subcategory of local systems in $\Loc(\La)$ whose stalks are perfect $\Bbbk$-modules.} Let $C_{-*}(\Omega\La,*_\La;\Bbbk)$ denote the algebra of chains on the based loop space, as above, and $\Mod(C_{-*}(\Omega\La,*_\La;\Bbbk))$ its dg-category of dg-modules over $\Bbbk$. Then $\Loc(\La)\cong \Mod(C_{-*}(\Omega\La,*_\La;\Bbbk))$ and we have the two isomorphisms:
$$\hspace{30pt} \Loc^{pp}(\La)\cong \Fun^{ex}(C_{-*}(\Omega\La,*_\La;\Bbbk),\Perf(\Bbbk)),\quad \Loc^{cpt}(\La)\cong \Perf(C_{-*}(\Omega\La,*_\La;\Bbbk)).$$
In the case of $\Loc^{cpt}$, the stalks of the local systems are perfect as $C_{-*}(\Omega\La,*_\La;\Bbbk)$-modules (and thus typically infinite rank $\Bbbk$-modules), whereas local systems in $\Loc^{pp}$ have a stalk that is a perfect $\Bbbk$-module.\\

\item[(c)] {\bf $\Sh_\La(M)$ is the dg-derived category of sheaves of $\Mod(\Bbbk)$}.\\

\noindent First, there is the need to at least consider sheaves of chain complexes, and not just $\Bbbk$-modules, for $\Sh_\La(M)$ to be invariant under Legendrian isotopies of $\La$. In that sense, we are lead to consider (classically) derived categories and it is possible to work with the derived category of sheaves without enhancing to dg-categories. This is the approach in \cite{KSbook,GKS12}. Nevertheless, since the invariance result \cite{GKS12} actually builds the equivalences through convolution with a sheaf kernel, it also implies that the dg-category $\Sh_\La(M)$ is a Legendrian isotopy invariant, see \cite{Schnurer18}. Two advantages of working with dg-categories are:\\

\begin{itemize}
    \item[(i)] For a dg-category $\mathcal{C}$ of finite type, \cite{ToenVaquie07} shows that there exists a $D^{-}$-stack locally geometric and locally of finite presentation which parametrizes pseudo-perfect objects in $\mathcal{C}$. This allows for the passage to commutative (homotopical) algebraic geometry \cite{HAGII}, often more manageable than the non-commutative setting of dg-categories.\\
    
    \item[(ii)] Mapping cones are functorial for dg-categories. Note that key constructions in symplectic topology use mapping cones (and cocones), e.g. S.~Guillermou's sheaf quantization \cite[Part 11--12]{Gui19}.\footnote{Indeed, the quantization of the double uses the convolution $k_\gamma\star F$ for the functor $\Psi$, as presented (and with the notation) in \cite[Part 11]{Gui19}. The restriction of $\Bbbk_\gamma\star F$ to $u=\varepsilon$ is a sheaf isomorphism to the mapping cone of the sheaf $F$ with an $\varepsilon$-Reeb push-off of $F$.}  In addition, $\Sh_\La(M)$ is a sheaf (and a cosheaf) on $M$, whereas their homotopical categories, which are the (classically) derived categories, are only pre-stacks.\\
\end{itemize}

\noindent Second, more technically, we can consider the dg-categories of chain complexes of sheaves or sheaves of chain complexes. There is a natural functor from the former to the latter, by associating to the sheafification of the corresponding pre-sheaf of chain complexes to each complex of sheaves on $X$.\footnote{See the Mathoverflow thread ``Sheaves of complexes and complexes of sheaves'' and references therein.} This is an equivalence for a smooth manifold $M$ of finite Lebesgue covering dimension, which is indeed the case for any manifolds in this manuscript. Finally, there is a choice of support for the chain complexes: typical choices are unbounded, bounded, non-negatively graded, or non-positively graded. We work in the unbounded setting, or the locally bounded one, following the results from \cite{MarcoSchapira18MicInfty}.\\

\item[(d)] {\bf $\Sh^{ctr}_\La(M)$ is the dg-derived category of constructible sheaves of $\Mod(\Bbbk)$}. This is a dg-subcategory of $\Sh^{ctr}_\La(M)$ and, in the cases that the Legendrian $\La$ is subanalytic, the inclusion yields an equivalence of dg-categories. Since all fronts for the Legendrian links $\La$ being considered in this article are subanalytic, and thus Whitney stratified, we indistinguishable work with $\Sh^{ctr}_\La(M)$ or $\Sh_\La(M)$. The first statement in \cite[Lemma 3.11]{NadWrapped} shows that $\Sh^{ctr}_\La(M)$ is compactly generated by the microlocal skyscrapers \cite[Definitinion 3.14]{NadWrapped}. (It is not split generated by these objects though.)\\ 

\item[(e)] {\bf $\Sh^{cpt}_\La(M)$ is the full dg-subcategory of compact objects in $\Sh^{ctr}_\La(M)$}. This category is split generated by the microlocal skyscrapers and its ind-completion yields $\Sh^{ctr}_\La(M)$. The category $\Sh^{cpt}_\La(M)$ is a cosheaf on $M$. It follows from \cite{NadWrapped} that $\Sh^{cpt}_\La(M)$ is a dg-category of finite-type, as it is a finite colimit of finite-type dg-categories. This is relevant so as to apply \cite{ToenVaquie07} to conclude that its moduli of pseudo-perfect objects is a locally geometric derived stack (and locally of finite presentation). In the examples of this manuscript, the moduli we consider are even Artin stacks.\\

Note that \cite{NadWrapped} refers to $\Sh^{cpt}_\La(M)$ as {\it wrapped} sheaves. We have opted for $\Sh^{cpt}_\La(M)$ so as to emphasize that wrapping is actually related to the singular support condition $\Sh_\La(M)$ and not the subcategory $\Sh^{cpt}_\La(M)$ itself. That is, the sheaves in $\Sh_\La(M)$ are already geometrically {\it wrapped} in that geometric wrapping (partially stopped to $\La$) truly occurs when taking the left adjoint to the inclusion $\Sh_\La(M)$ into $\Sh(M)$, see \cite[Theorem 1.2]{Kuo_WrappedSheaves}.\\

\item[(f)] {\bf $\Sh^{pp}_\La(M)$ is the full dg-subcategory of sheaves in $\Sh_\La(M)$ whose stalks are perfect $\Bbbk$-modules.} The category $\Sh^{cpt}_\La(M)$ is a sheaf on $M$. It is proven in \cite[Theorem 3.21]{NadWrapped} or \cite[Corollary 4.22]{GPS18b} that the hom-pairing provides an equivalence between $\Sh^{pp}_\La(M)$ and the dg-category of exact functors $\Fun^{ex}(\Sh^{cpt}_\La(M),\Perf(\Bbbk))$, where $\Perf(\Bbbk)$ is the dg-category of perfect $\Bbbk$-modules. In that sense, objects in $\Sh^{pp}_\La(M)$ can be referred to as the pseudo-perfect (equiv.~proper) objects in $\Sh^{cpt}_\La(M)$, thence the notation.\\
\end{itemize}

Finally, the functors $\iota_2,\iota_1$ and $\iota_0$ are the defining inclusions. The fact that
$$\Loc^{pp}(\La) \hookrightarrow \Loc^{cpt}(\La), \;\text{and}\; \Sh^{pp}_\La(M) \hookrightarrow \Sh^{cpt}_\La(M)$$
uses the result that $\Loc^{cpt}(\La)$ and $\Sh^{cpt}_\La(M)$ are smooth categories \cite[Corollary 4.25 \& Lemma A.8]{GPS18a}. The functor $i_{M \times \{+\infty\}}^{-1}$ is given by the restriction from $M \times \ol{\bR}$ (with the stratification determined by $\widetilde{L}$) to its boundary $M \times \{+\infty\}$ (with the stratification determined by $\Lambda$). The functor $m_\La$ is the microlocal functor discussed in Appendix \ref{appen:sheaf}. Note that the functor
$m_\La$ is defined to map to the global sections $\mu\sh_\La(\La)$ of the Kashiwara-Schapira stack. In general, there can be obstructions to identify $\mu\sh_\La(\La)$ with $\Loc(\La)$, and even then there is a subtlety with regards to {\it twisted} local systems. For the Legendrian links in this manuscript, these obstructions vanish and twisting is addressed in Appendix \ref{appen:microlocal}.\\

\noindent Neither of the two functors $i_{M \times \{+\infty\}}^{-1}$ or $m_\La$ preserve compact objects. Nevertheless, each of them admits a left adjoint, e.g.~see \cite[Lemma 4.12]{GPS18a} and references therein. These left adjoints $(i_{M \times \{+\infty\}}^{-1})^L$ and $m_\La^L$ preserve compact objects and thus given rise to functors among $\Sh^{cpt}_\La(M)$, $\Sh^{cpt}_{\widetilde{L}}(M \times \ol\bR)$ and $\Loc^{cpt}(\Lambda)$ in the opposite direction.\\





In this manuscript we focus on $\La$ as a Legendrian submanifold of the ideal contact boundary $T^{*,\infty}M$. That said, from the perspective of Lagrangian skeleta, it is natural to instead use the (singular) Lagrangian given by the union of $M$ and the Lagrangian cone $\mbox{cone}(\La)$ of $\La$. The global sections of the Kashiwara-Schapira stack on $M\cup \mbox{cone}(\La)$ lead to $\Sh_\La(M)$ and thus the categories above appear naturally as well. The perspective of Lagrangian skeleta and (co)sheaves of categories on them are also useful in the study of Weinstein manifolds.


\newpage

\section{Results on the microlocal theory of constructible sheaves}\label{appen:sheaf}

    The goal of this appendix is to review and modify the necessary results in the microlocal theory of sheaves, as it relates to studying Lagrangian fillings, and fill in the gaps and necessary comparisons between the works in the current literature. The appendix consists of three separate parts, as follows.\\
    
    \begin{enumerate}
        \item Firstly, in Section \ref{appen:microlocal} we review S.~Guillermou's result on microlocalization \cite{Gui12}, discuss the role of twisted local systems, and provide a combinatoric approach for the computation with coherent signs refining the construction in \cite{STZ17}, which is not systematically written in literature.\\
        
        \item Secondly, in Section \ref{appen:sheaf-quan} we present Jin-Treumann's result on sheaf quantization \cite{JinTreu17} and make modifications on their original results, which is not compatible with the current constructible sheaf model on Legendrian weaves. In particular, we prove Hamiltonian invariance property without fixing boundary conditions as this is needed for the applications in our setting; this modification was also not written in literature.\\
        
        \item Thirdly, in Section \ref{appen:sheaf-moduli}, we define moduli and framed moduli of sheaves with singular support on the Legendrians, compare different approaches to framed moduli of sheaves, and define microlocal merodromies along relative 1-cycles.
    \end{enumerate}

\subsection{Singular support of sheaves}\label{appen:sheaf-general}
    For sheaves on manifolds, M.~Kashiwara and P.~Schapira introduced in \cite{KSbook} the concept of singular support of (complexes of) sheaves, which detect when derived sections fail to extend along a certain codirection; they refer to it as the derived sections fail to propagate. Let $\Sh(M)$ be the dg-derived category of sheaves on $M$ in the sense of \cite{Keller06DG}*{Section 4.4}, i.e.~dg category of complexes of sheaves with unbounded cohomology localized along quasi-isomorphisms, see Appendix \ref{appen:cat-sheaf}. Following \cite{STZ17}*{Section 3}, Example \ref{ex:sheafcombin} gave basic examples of what sheaves with a Legendrian singular support are like. The general definition reads as follows.
    
\begin{definition}[\cite{KSbook}]\label{def:ss}
    For a complex of sheaves $\cF \in \Sh(M)$, the singular support $SS(\cF) \subset T^*M$ is the closure of points $(x, \xi)$ such that for some neighbourhood $U_x$ of $x$, $\varphi\in C^\infty(M)$, $d\varphi(x) = \xi$,
    $$\mathrm{Cone}\left(\Gamma(U_x, \cF) \rightarrow \Gamma(U_x \cap \varphi^{-1}((-\infty,0)), \cF)\right) \not\simeq 0.$$
    $SS^\infty(\cF) = \ol{SS(\cF)} \cap T^{*,\infty}M$.
\end{definition}

    In \cite[Theorem 6.5.4]{KSbook} it is shown that $SS^\infty(\cF)$ is a (singular) coisotropic subset. By \cite[Theorem 8.4.2]{KSbook}, if $SS^\infty(\cF)$ is a (singular) Legendrian, then $\cF$ is a constructible sheaf. Following Appendix \ref{appen:cat-sheaf}, given a (subanalytic) Legendrian $\Lambda \subset T^{*,\infty}M$, $\Sh_\Lambda(M)$ denotes the full dg-subcategory in $\Sh(M)$ of constructible sheaves $\cF$ such that $SS^\infty(\cF) \subset \Lambda$.

\subsection{Microlocalization functor}\label{appen:microlocal}
    Given a constructible sheaf $\cF \in \Sh_\Lambda(M)$, there is a microlocalization functor $m_\Lambda$ which produces a (twisted) local system on the Legendrian $\Lambda \subset T^{*,\infty}M$. Microlocalization has been studied in \cite{KSbook}*{Chapter 6}, \cite{Gui12}*{Section 6--11} (or \cite{Gui19}*{Part 10}), and also \cite{NadWrapped,JinTreu17}. We summarize their results and in particular give a combinatorial approach of computing microlocal monodromies with coherent signs, refining the combinatorial construction in \cite{STZ17}*{Section 5}.

\subsubsection{Microlocalization with twisted coefficients}
    In \cite{STZ17}, a microlocal monodromy map for a sheaf $\cF \in \Sh_\Lambda(M)$ is  constructed in a combinatorial way, as in Example \ref{ex:mumon-STZ}. Their construction fits into a more general framework of microlocalization, see \cite[Chapter IV]{KSbook}. In general, we can consider a sheaf of categories  $\mu \Sh_\Lambda$ on $T^{*,\infty}M$, the Kashiwara-Schapira stack \cites{Gui12,Gui19} or the brane category \cite{JinTreu17}, as follows. Here, working with dg-categories, we adopt the definitions in \cite[Section 5]{NadShen20}. 
    
\begin{definition}
    Let $\Lambda \subset T^{*,\infty}M$. The presheaf of categories $\mu\Sh_\Lambda^\text{pre}$ on $T^{*,\infty}M$ associated to $\Lambda$ is defined by, for any $\Omega \subset T^{*,\infty}M$,
    $$\mu\Sh_\Lambda^\text{pre}(\Omega) := \Sh_{\Lambda \cup (T^{*,\infty}M \backslash \Omega)}(M) /\Sh_{T^{*,\infty}M \backslash \Omega}(M).$$
    By definition, $\mu\Sh_\Lambda$ is the sheafification of $\mu\Sh_\Lambda^\text{pre}$, which is induced by a sheaf of categories on $\Lambda$.
    
    \noindent Finally, by definition, the microlocalization functor
    $$m_\Lambda: \Sh_\Lambda(M) \rightarrow \mu \Sh_\Lambda(\Lambda)$$
    is the natural quotient functor on the sheaf of categories.
\end{definition}
\begin{remark}
    It can be verified that the definition above is the same as \cite[Section 5]{Gui12} or \cite[Part 10]{Gui19}, using \cite[Proposition 6.1.2]{KSbook}, which proves that the stalks of the two sheaves of categories coincide. Alternatively, one can follow the definition in \cite[Section 3.4]{NadWrapped}, that for a carefully chosen pair of small open balls $B \subset M$, $\Omega \subset T^{*,\infty}M$ such that $\pi(\Omega) \subset B$ and $\Omega$ is a neighbourhood of a component of $\Lambda \cap T^{*,\infty}B$,
    $$\mu\Sh_\Lambda(\Omega) = \Sh_{\Lambda \cup (T^{*,\infty}B \backslash \Omega)}(B).$$
    The equivalence of these two definitions relies on existence of adjoints to the inclusion $\Sh_{\Lambda \cup (T^{*,\infty}B \backslash \Omega)}(B) $ $\hookrightarrow \Sh_{\Lambda \cup (T^{*,\infty}M \backslash \Omega)}(M)$. Similarly, in \cite[Section 3.8--11]{JinTreu17}, the definition
    $$\mu\Sh_\Lambda(\Omega) = \Sh_\Lambda(B) / \Loc(B)$$
    is used. Equivalence between this definition and the above ones is an application of the refined microlocal cut-off lemma \cite[Lemma 6.7]{Gui12}, see also \cite[Lemma 3.8]{JinTreu17}.
\end{remark}

Note that the microlocalization functor $m_\Lambda$ restricts to $\Sh^{pp}_\Lambda(M) \rightarrow \mu\Sh_\Lambda^{pp}(\Lambda)$. In addition, the left adjoint $m_\Lambda^L$ of microlocalization can be restricted to $\mu\Sh_\Lambda^{cpt}(\Lambda) \rightarrow \Sh^{cpt}_\Lambda(M)$; see Appendix \ref{appen:cat-sheaf}.\\


    Let $M$ and $\Lambda \subset T^{*,\infty}M$ be oriented, and $\pi_{\Lambda \times \bR_+}: \Lambda \times \bR_+ \rightarrow M$ the projection where $\Lambda \times \bR_+ \subset T^*M \backslash M$ is a conical Lagrangian. In our context, we have vanishing of Maslov classes. The (stabilized) oriented relative frame bundle over $\Lambda$ is the principal $GL_{n,+}$-bundle defined by
    $$\cI_{\Lambda \times \bR_+} := \mathrm{Iso}_+(\pi_{\Lambda \times \bR_+}^* TM, T(\Lambda \times \bR_+)).$$
    Note that $\pi_\Lambda^* TM$ splits as the direct sum of a Lagrangian subbundle in $\xi \subset T(T^{*,\infty}M)$ and a rank~1 trivial line bundle tangent to the radius $\bR_+$-direction in $TM$. Therefore, there is indeed a stabilization map
    $$\cI_{\Lambda} = \mathrm{Iso}_+(\pi_\Lambda^*((TM \backslash M) /\bR_+), T\Lambda) \hookrightarrow \mathrm{Iso}_+(\pi_\Lambda^*TM, T(\Lambda \times \bR_+)) = \cI_{\Lambda \times \bR_+}.$$

\begin{thm}[ \cite{Gui12}*{Theorem 11.5}]\label{thm:Gui}
    Let $M$ be an oriented smooth manifold, $\Lambda \subset T^{*,\infty}M$ be an oriented Legendrian with vanishing Maslov class, and $\cI_{\Lambda \times \bR_+}$ be the oriented relative frame bundle over $\Lambda$. Then
    $$\mu \Sh_\Lambda(\Lambda) \simeq \Loc_\epsilon(\cI_{\Lambda\times \bR_+}) $$
    where $\Loc_\epsilon(\cI_{\Lambda \times \bR_+})$ is the dg-derived category of twisted local systems on $\cI_{\Lambda \times \bR_+}$, i.e.~local systems whose monodromy along each fiber is $-1$.
\end{thm}

\begin{remark}\label{rem:lag-plane-twist}
    Twisting appears as follows, see also \cite{Gui12}*{Section 9 \& 10.2} or \cite{Gui19}*{Section 10.3 \& 10.6}. For simplicity, assume that $\pi(\Lambda)$ is defined a smooth hypersurface in a neighbourhood of $\pi(p)$. Suppose $\pi(\Lambda)$ is defined by a function $\phi_\Lambda$ where
    $$\phi_\Lambda^{-1}(0) = \pi(\Lambda), \;\; d\phi_\Lambda(\pi(\Lambda)) \subset \Lambda \times \bR_+.$$
    In computing the microstalk of a sheaf $\cF \in Sh^b_\Lambda(M)$ at $p \in \Lambda$, in general we do not simply consider the cone
    \[\mathrm{Cone}(\Gamma(\phi_\Lambda^{-1}((-\infty, 0], \cF) \rightarrow \Gamma(\phi_\Lambda^{-1}((-\infty, 0)), \cF)).\]
    In general, rather, we need to consider a function $\varphi$ such that
    $$\varphi^{-1}(0) \cap \pi(\Lambda) = \pi(p), \;\; d\varphi(\varphi^{-1}(0)) \pitchfork \Lambda \times \bR_+,$$
    and compute the local Morse group
    \[\mathrm{Cone}(\Gamma(\varphi^{-1}((-\infty, 0], \cF) \rightarrow \Gamma(\varphi^{-1}((-\infty, 0)), \cF)),\]
    which depends on the relative Morse index of $\phi_\Lambda$ and $\varphi$ at $\pi(p)$. There are different choices of Lagrangian planes $T_p\Lambda_\varphi$ for $\Lambda_\varphi = \{(x, d\varphi(x)) | x \in M\}$ at a point $p \in \Lambda$, and any $\varphi$ such that $\Lambda_\varphi \pitchfork \Lambda \times \bR_+$ determines a linear isomorphism
    \[T_p(T_{\pi(p)}^*M) \rightarrow T_p\Lambda_\varphi \oplus T_p(\Lambda \times \bR_+) \rightarrow T_p(\Lambda \times \bR_+).\]
    The computation of \cite{STZ17}, where they essentially used the defining function $\phi_\Lambda$ for the Legendrian, is equivalent to the choice of a function $\varphi$ with Morse index 0 relative to $\pi(\Lambda)$ plus the degree shifting by the Maslov potential.
\end{remark}

    Under the unique surjective homomorphism $\pi_1(GL_{n,+}(\bR)) \rightarrow \bZ/2\bZ$, $n \geq 2$, the obstruction to the existence of a rank~1 twisted local system is characterized by a class in $H^2(\Lambda; \bZ/2\bZ)$, which turns out to be the relative second Stiefel-Whitney class, see \cite{Gui19}.

\begin{cor}[\cite{Gui12}*{Theorem 11.5}]\label{cor:Gui}
    Let $\Lambda \subset T^{*,\infty}M$ be oriented with vanishing Maslov class. If $\Lambda$ is relatively spin, then there exists a rank~1 twisted local system on $\cI_{\Lambda \times \bR_+}$, and each choice of a rank~1 twisted local system determines an equivalence
    $$\mu \Sh_\Lambda(\Lambda) \simeq \Loc_\epsilon(\cI_{\Lambda \times \bR_+}) \simeq \Loc(\Lambda).$$
\end{cor}
    
    In particular, this recovers the combinatorial construction of
    $$m_\Lambda: \Sh_\Lambda(M) \rightarrow \Loc(\Lambda)$$
    in \cite{STWZ19} up to signs. However, even though $\Loc_\epsilon(\cI_\Lambda)$ and $\Loc(\Lambda)$ are equivalent, we will still frequently use the twisted version $\Loc_\epsilon(\cI_\Lambda)$ in the following sections, since the twisted version will give the correct signs that one will expect in cluster theory.\footnote{\cite{GonKon21}*{Section 8.3}, for example, explains the reason to consider twisted local systems in cluster theory. \cite{GMN13,Kuwagaki20WKB} also use twisted local systems.}
    
\begin{remark}
    The stabilization process in Theorem \ref{thm:Gui} is important. For Legendrian knots, the oriented frame bundle is a $GL_{1,+}(\bR)$-bundle over $S^1$, hence $\cI_\Lambda \simeq \Lambda$ and there is no information on twisting in $\cI_\Lambda$. However, the twisting data appear in the stabilized bundle $\cI_{\Lambda \times \bR_+} \simeq S^1 \times \Lambda$. In fact, the sheaf categories are going to be stabilized in the sense that the classifying map of the Lagrangian factors through the stable Gauss map so that we obtain a stable $\infty$-category over ring spectra \cite{JinTreu17}, and hence we indeed need to study the Kashiwara-Schapira stack or twisted local systems up to stabilization induced by the embedding $\cI_\Lambda \hookrightarrow \cI_{\Lambda \times \bR} \hookrightarrow \dots \hookrightarrow \cI_{\Lambda \times \bR^N} \hookrightarrow\dots$, see \cite{Gui12}*{Section 10.1} or \cite{Gui19}*{Section 10.6}, which yields
    $$\cdots \rightarrow \Loc_\epsilon(\cI_{\Lambda \times \bR^N}) \rightarrow \cdots \rightarrow \Loc_\epsilon(\cI_{\Lambda \times \bR}) \rightarrow \Loc_\epsilon(\cI_\Lambda).$$
    For example, when $\Lambda$ arises as the boundary of a surface $\widetilde{L}$. Then the stabilization embedding provides a natural restriction map
    $$\Loc_\epsilon(\cI_{\widetilde{L}}) \rightarrow \Loc_\epsilon(\cI_{\Lambda \times \bR}) \rightarrow \Loc_\epsilon(\cI_\Lambda).$$
\end{remark}

\subsubsection{Microlocalization with coherent signs}\label{ssec:muloc_coherentsigns}
    Given a Legendrian $\Lambda \subset T^{*,\infty}M$ that has vanishing Maslov class and vanishing relative second Stiefel-Whitney class, Corollary \ref{cor:Gui} shows a non-canonical equivalence
    $$\mu \Sh_\Lambda(\Lambda) \simeq \Loc_\epsilon(\cI_{\Lambda \times \bR}) \simeq \Loc(\Lambda).$$
    Different choices of rank~$1$ twisted local systems may result in different identifications with the category of local systems. Here we discuss these choice, see also \cite{Gui12}*{Section 9 \& 10.2} or \cite{Gui19}*{Section 10.3 \& 10.6}.
    
\begin{lemma}\label{lem:twist-loc-lift}
    Let $\Lambda \subset T^{*,\infty}M$ be relatively spin. Then there is a bijection between
    
    \begin{enumerate}
        \item relative spin structures on $\Lambda$,
        \item rank~1 twisted local systems on $\Loc_\epsilon(\cI_{\Lambda \times \bR})$,
        
        \item sections of the stable relative oriented frame bundle $\cI_{\Lambda \times \bR^N}$ over the 2-skeleton $\Lambda_{\leq 2}$.
    \end{enumerate}
\end{lemma}
\begin{proof}
    For (1)~\&~(3), since $\Lambda \subset T^{*,\infty}M$ is relative spin, the relative oriented frame bundle
    $$\cI_{\Lambda \times \bR} = \mathrm{Iso}_+(\pi_{\Lambda \times \bR_+}^*TM, T(\Lambda \times \bR_+))$$
    is trivial over the 2-skeleton $\Lambda_{\leq 2}$, and any global section uniquely determines a relative spin structure on $\Lambda$. On the 1-skeleton $\Lambda_{\leq 1} \subset T^{*,\infty}\Sigma$, the section determines a trivialization where vanishing of the obstruction class in $H^2(\Lambda; \bZ/2\bZ)$ implies existence of an extension of the trivialization to the 2-skeleton. Since
    $$\pi_1(\GL_{N,+}(\bR)) \xrightarrow{\sim} \bZ/2\bZ,$$
    the choice of a section on the 2-skeleton is classified by $H^1(\Lambda; \bZ/2\bZ)$. 
    
    For (2)~\&~(3), rank~1 (twisted) local system can be uniquely extended from a rank~1 (twisted) local system on the 2-skeleton $\Lambda_{\leq 2}$. Fix a rank~1 twisted local system $\cL_{0}$ on $\Lambda_{\leq 2}$. Then any rank~1 twisted local system can be obtained by $\cL = \pi^{-1}s^{-1}\cL_0$ for a section $s$ on the 2-skeleton $\Lambda_{\leq 2}$, and this correspondence is bijective.
\end{proof}
\begin{remark}\label{rem:twist-spin}
    When we choose a section on the 2-skeleton $\Lambda_{\leq 2}$, we are still doing computations on twisted local systems $\Loc_\epsilon(\cI_{\Lambda \times \bR})$, even though the choice of the section determines a relative spin structure and an identification with $\Loc(\Lambda)$.
\end{remark}

    Therefore, in order to compute the microlocal monodromy with signs, we will fix a relative spin structure and consider a section over the 2-skeleton $s: \Lambda_{\leq 2} \rightarrow \cI_{\Lambda \times \bR^N}|_{\Lambda_{\leq 2}}$.
    
\begin{definition}\label{def:sign-mumon}
    Let $\Lambda \subset T^{*,\infty}M$ be a Legendrian that has vanishing Maslov class with a fixed relatively spin structure, determined by $s: \Lambda_{\leq 2} \rightarrow \cI_{\Lambda \times \bR^N}|_{\Lambda_{\leq 2}}$. Given a sheaf $\cF \in \Sh_\Lambda(M)$, the (signed) microlocal monodromy along a path $\gamma \in H_1(\Lambda; \bZ)$ is the monodromy along the lifting $s \circ \gamma \in H_1(\cI_{\Lambda \times \bR^N}; \bZ)$
    $$m_{\Lambda,\gamma}(\cF) = m_\Lambda(\cF)_{s \circ \gamma}.$$
\end{definition}

    For a Legendrian knot, there are only 2 relative spin structures since $H^1(\Lambda; \bZ/2\bZ) = \bZ/2\bZ$. For the generator $\gamma \in H_1(\Lambda; \bZ)$, the relative spin structure only depends on the parity of the winding number $s \circ \gamma$ along the fiber $\GL_{2,+}(\bR)$ in $\cI_{\Lambda \times \bR}$. Note that under the stabilization map
    $$\pi_1(\GL_{2,+}(\bR)) \rightarrow \pi_1(\GL_{N,+}(\bR)) \xrightarrow{\sim} \bZ/2\bZ \xrightarrow{\sim} \bZ^\times,$$
    the microlocal monodromy only depends on the parity of the winding number.\\
    
    We now explain how the combinatorial microlocal monodromy of \cite{STZ17}*{Section 5} can be enhanced to involve the choice of a lifting of a path in the relative oriented frame bundle.
    
\begin{ex}
    For a Legendrian knot with no cusps in the front, when computing the microlocal monodromy at $(x, y, \xi, \eta) \in \Lambda$, following Remark \ref{rem:lag-plane-twist}, the approach in \cite{STWZ19} is equivalent to choosing the Lagrangian planes defined by $T_p\Lambda_\varphi$ where
    $$\Lambda_\varphi = \{(x, y, d\varphi(x, y)) | (x, y) \in \bR^2\}, \;\; \varphi(x, y) = 0, \, d\varphi(x, y) = (\xi, \eta)$$
    such that $\varphi^{-1}(0)$ has relative Morse index 0 relative to $\pi(\Lambda)$. One may observe that the winding number of this family of Lagrangian planes coincides with the winding number tangent vector in the fiber of $T\Sigma$ (mod 2).
    
    For a Legendrian $n$-satellite $\Lambda \subset T^{*,\infty}\bR^2$ of the outward unit conormal bundle of a disk, this family of Lagrangian planes determines a section in $\cI_{\Lambda \times \bR}$ with winding number $n$ (mod 2) around the fiber and thus corresponds to the Lie group spin structure when $n$ is even and to the null-cobordant spin structure when $n$ is odd.
\end{ex}

    Still for a Legendrian knot, we can also choose different spin structures on $S^1$ when computing microlocal monodromies. In fact, such a choice can be localized near the set of points $P = \{p_1, \dots, p_n\}$. 

\begin{definition}\label{def:sign-knot}
    For a Legendrian knot $\Lambda \subset T^{*,\infty}\Sigma$, a set of sign points is a discrete set of points $P$, such that for the generator $\gamma \in H_1(\Lambda; \bZ)$, the lifting $s \circ \gamma$ winds around the fiber of $\cI_{\Lambda \times \bR^N}$ once in a small neighbourhood of each sign point and follows the tangent vector of $\pi(\gamma)$ away from sign points.
\end{definition}
\begin{ex}\label{ex:sign-knot}
    For a Legendrian knot, and a set of sign points $P = \{p_1, \dots, p_n\}$, we can modify the sign in the computation of \cite{STZ17} by the number of sign points that we choose.

    For a Legendrian $n$-satellite $\Lambda \subset T^{*,\infty}\bR^2$ of the outward unit conormal bundle of a disk, we may add no sign points when $n$ is odd and add 1 sign point $p$ when $n$ is even, then the section in $\cI_{\Lambda \times \bR}$ determined by the small loop around the fiber in a neighbourhood of $p$ together with the family of Lagrangian planes of \cite{STZ17} will define the null-cobordant spin structure.
\end{ex}
    
    For an oriented Legendrian surface, relative spin structures are classified by $H^1(\Lambda; \bZ/2\bZ)$. For a basis $\gamma_i \in H_1(\Lambda; \bZ)$, the relative spin structure depends on the parity of the winding number $s \circ \gamma_i$ along the fiber $\GL_{3,+}(\bR)$ in $\cI_{\Lambda \times \bR}$. Therefore, when fixing a relative spin structure, we in fact fix the liftings of a basis  $\gamma_i \in H_1(\Lambda; \bZ)$.
    
    The following definition is first made in \cite{CasalsWeng22}*{Section 4.5} for Legendrian weaves, and we are simply borrowing the terminology from their work.
    
\begin{definition}\label{def:sign-curve}
    For a closed Legendrian surface $\widetilde{L} \subset T^{*,\infty}(\Sigma \times \bR)$, a set of sign curves is a graph $P$ whose vertices are the centers of 2-cells, such that for each generator $\gamma \in H_1(\widetilde{L}; \bZ)$, the lifting $s \circ \gamma$ winds around the fiber of $\cI_{\Lambda \times \bR^N}$ once in neighbourhoods of intersections with sign curves, and follows the tangent vector of $\pi(\gamma)$ away from sign points.\\
    
    \noindent For a Legendrian surface $\widetilde{L} \subset T^{*,\infty}(\Sigma \times \bR)$ with boundary on a Legendrian knot $\Lambda \subset T^{*,\infty}(\Sigma \times \{+\infty\})$ (Section \ref{sec:weave-at-infty}), a set of sign curves is correspondingly a graph $P$ whose vertices are the centers of 2-cells and 1-cells in the boundary.\\
    
    Finally, a compatible collection of sign curves on $\widetilde{L} \subset T^{*,\infty}(\Sigma \times \bR)$ is a a graph $P$ whose vertices are the centers of 2-cells such that for every loop $\gamma = 0 \in H_1(\widetilde{L}; \bZ)$ that bounds a 2-cell, the winding number of the lifting $s \circ \gamma$ around the fiber is 1.
\end{definition}

The collections of sign curves in Definition \ref{def:sign-curve} indeed exist:

\begin{lemma}
    For an oriented Legendrian surface $\widetilde{L} \subset T^{*,\infty}(\Sigma \times \bR)$, there exists a compatible collection of sign curves.
\end{lemma}
\begin{proof}
    We consider all the 2-cells attached inductively, which determines an (partial) order on all the 1-cells. For each loop 2-cell in $\Lambda$, we can compute the winding number of the tangent vector in the fiber of $T\Sigma$. Whenever the winding number is $0$ (mod $2$), we add a sign point in any of the 1-cells contained in its boundary. Then connect all the sign points on $\gamma$ with the centers of the adjacent 2-cells. 
    
    When $\Lambda$ is a surface with boundary, we thus define a compatible collection of sign curves. When $\Lambda$ is a closed surface, we can defines a compatible collection of sign curves on the compliment of a 2-cell $\Lambda \backslash D$. On the boundary of the 2-cell $D$, the spin structure determined by sign points is null-cobordant. Hence the parity of sign points determined by $\Lambda \backslash D$ agrees with the parity determined by $D$. Therefore there is a compatible choice of sign curves.
\end{proof}

\begin{ex}\label{ex:sign-curve-unknot}
    For the Legendrian unknot $\Lambda \subset T^{*,\infty}(S^1 \times \bR)$ which is the 2-satallite of the unknot with a Legendrian weave filling $\widetilde{L} \subset T^{*,\infty}(\bD^2 \times \bR)$, we can define the sign curve to be the curve from the center of the disk to a boundary sign point as in \cite{CasalsWeng22}*{Example 4.28}.\\
    
    More generally, for a Legendrian weave $\widetilde{L} \subset T^{*,\infty}(\Sigma \times \bR)$, we can consider a triangulation on $\Sigma$ such that each 2-cell contains 0 or 1 trivalent vertex (in the interior). This lifts to a triangulation of $\widetilde{L}$ such that the parity of number of sign points on the boundary of each 2-cell agrees with the number of trivalent vertices inside (either 0 or 1). Therefore, we can choose sign curves to be a set of curves connecting trivalent vertices and $\partial \widetilde{L}$ such that all trivalent vertices have degree 1. This is exactly the notion of a coherent collection of sign curves in \cite{CasalsWeng22}*{Definition 4.15}.
\end{ex}

\subsection{Sheaf quantization results}\label{appen:sheaf-quan}
    We summarize the necessary results on sheaf quantizations in symplectic topology and contact topology.

\subsubsection{Sheaf quantization of Hamiltonian isotopy}
    The following theorem shows that the category $\Sh_\Lambda(M)$ is an invariant of the Legendrian submanifold, which enables one to study the properties of a Legendrian by studying the constructible sheaf category. First, recall that for $\{\Lambda_t\}_{t \in [0,1]}$ a family of Legendrians in $T^{*,\infty}M$ induced by the contact Hamiltonian $H$, the Legendrian movie is
    $$\Lambda_H = \{(x, \xi, t, \tau) | (x, \xi) = \varphi_H^t(x_0, \xi_0), \tau = -H(x_0, \xi_0), (x_0, \xi_0) \in \Lambda_0\},$$
    where $\varphi_H^t\,(t \in [0, 1])$ is the contact flow induced by $H$ such that $\Lambda_t = \varphi_H^t(\Lambda_0)$.

\begin{thm}[\cite{GKS12}]\label{thm:GKS}
    Let $\{\Lambda_t\}_{t \in [0,1]}$ be a Legendrian isotopy in $T^{*,\infty}M$ induced by the contact Hamiltonian $H$. Then there is an equivalence of categories given by convolution of sheaf kernels
    $$\Phi_H: \, \Sh_{\Lambda_0}(M) \xleftarrow{i_0^{-1}} \Sh_{\Lambda_H}(M \times [0, 1]) \xrightarrow{i_1^{-1}} \Sh_{\Lambda_1}(M),$$
    where $i_t: M \times \{t\} \hookrightarrow M \times [0, 1]$ are the inclusions. In addition, the equivalence preserves microlocal rank~$r$ objects on both sides.
\end{thm}
\begin{remark}
    The equivalence $\Sh_{\Lambda_0}(M) \simeq \Sh_{\Lambda_1}(M)$ automatically implies that $\Sh_{\Lambda_0}^{cpt}(M) \simeq \Sh_{\Lambda_1}^{cpt}(M)$ and $\Sh_{\Lambda_0}^{pp}(M) \simeq \Sh_{\Lambda_1}^{pp}(M)$, see Appendix \ref{appen:cat-sheaf}.
\end{remark}

\subsubsection{Sheaf quantization of exact Lagrangians}
    Given an exact Lagrangian filling $L \subset T^*M$ of a Legendrian submanifold $\Lambda \subset T^{*,\infty}M$, the theorem we are going to state defines a fully faithful functor $\Loc(L) \hookrightarrow \Sh_\Lambda(M)$, which realizes exact Lagrangian fillings, endowed with local systems, as objects in the constructible sheaf category.
    

\begin{lemma}\label{lem:lowerexact}
    Let $L_t \subset T^*M$, $t \in \bD^k$, be a family of exact Lagrangian fillings of a family of Legendrian submanifold $\Lambda_t \subset T^{*,\infty}M$, $t\in \bD^k$. Then $L_t$ is Hamiltonian isotopic to a family of exact Lagrangians $L_t'$ whose primitive $f_{L_t'}$ is proper and bounded from below (where $\lambda_\text{st}|_{L_t'}= df_{L_t'}$).
\end{lemma}
\begin{proof}
    The proof is close to the non-parametric version of the lemma in \cite{JinTreu17}*{Section 3.6}. Since $L_t$, $t\in \bD^k$, is a family of exact Lagrangian fillings of $\Lambda_t$, $t \in \bD^k$, for any $t \in \bD^k$ there is some $r_t \gg 0$ sufficiently large such that $L_t \cap \{(x, \xi) \in T^*M | |\xi| > r_t\} = \Lambda \times (r_t, +\infty)$. Since $D^k$ is compact one can find $r_0 \gg 0$ such that for any $t \in D^k$, $L_t \cap \{(x, \xi) \in T^*M | |\xi| > r_0\} = \Lambda \times (r_0, +\infty)$.

    Let $\beta: [0, +\infty) \rightarrow \bR$ be a function such that $\beta(r) = 0$ when $r$ is sufficiently small and $\beta(r) = -1/r$ when $r > r_0$ is sufficiently large, and the Hamiltonian $H(r) = \beta(r)$. Then consider $L_t' = \varphi^{\epsilon}_H(L_t)$. One can check that when $r > r_0$, $df_{L_t'} = \epsilon/r$ and thus are proper and bounded from below.
\end{proof}

\begin{thm}[\cite{Gui12}; \cite{JinTreu17}]\label{thm:JinTreu-app}
    Let $L \subset T^*M$ be an exact Lagrangian filling of a Legendrian submanifold $\Lambda \subset T^{*,\infty}M$ with zero Maslov class, whose primitive $f_{L}$ is proper and bounded from below (where $\lambda_\text{st}|_{L}= df_{L}$). Let $\widetilde{L} \subset J^1(M) \cong T^{*,\infty}_{\tau < 0}(M \times \bR)$ be the Legendrian lift of $L$. Then
\begin{enumerate}
  \item there exists a fully faithful functor $\Psi_{\widetilde{L}}: \Loc_\epsilon(\cI_{\widetilde{L}}) \rightarrow \Sh_{\widetilde{L}}(M \times \bR)$ such that the stalk at $M \times \{-\infty\}$ is acyclic, and $m_{\widetilde{L}} \circ \Psi_{\widetilde{L}} \simeq \mathrm{id}$;
  \item the proper push forward functor $\pi_{M,!}: \Sh_{\widetilde{L}}(M \times \bR) \rightarrow \Sh_\Lambda(M)$ via the projection $\pi_M: M \times \bR \rightarrow M$ is fully faithful, and for $\widetilde{\cF} \in \Sh_{\widetilde{L}}(M \times \bR)$,
  $$\pi_{M,!}\widetilde{\cF} = i_{M \times \{+\infty\}}^{-1}j_{M \times \bR,*}\widetilde{\cF},$$
  where $i_{M \times \{+\infty\}}: M \times \{+\infty\} \hookrightarrow M \times [-\infty, +\infty]$ and $j_{M \times \bR}: M \times \bR \hookrightarrow M \times [-\infty, +\infty]$ are the inclusions.
\end{enumerate}
\end{thm}
\begin{proof}
    The proof appropriately modifies the arguments from \cite{JinTreu17}*{Section 3}. The differences from \cite{JinTreu17} are that (i)~we are considering exact Lagrangians with primitives bounded from below instead of the ones bounded from above, and (ii)~we are using the proper push forward $\pi_{M!}$ in Step~(2) instead of the standard push forward $\pi_{M*}$.\\

    Firstly, Step~(1) is the same as in \cite{JinTreu17}. Namely, we use the equivalence $\Loc_\epsilon(\cI_{\widetilde{L}}) \simeq \mu \Sh_{\widetilde{L}}(\widetilde{L})$ to obtain a microlocal sheaf $\widetilde{\cF}_\mu \in \mu \Sh_{\widetilde{L}}(\widetilde{L})$. Then consider S.~Guillermou's convolution functor to obtain the doubling sheaf $\widetilde{\cF}_{\text{dbl},c} \in \Sh_{\widetilde{L} \cup T_c(\widetilde{L})}(M \times \bR)$ when $c > 0$ is small. Since $\widetilde{L}$ has no Reeb chords, for any $c, c' > 0$ we have an equivalence induced by Hamiltonian isotopy\footnote{Strictly speaking, non-existence of Reeb chords only implies the existence of a Legendrian isotopy from $\widetilde{L} \cup T_c(\widetilde{L})$ to $\widetilde{L} \cup T_{c'}(\widetilde{L})$. Since the corresponding Hamiltonian is not compactly supported, it is not clear that the integrated flow exists. In our setting this can be verified and is explained in \cite{JinTreu17}*{Section 3.14}.}
    $$T_{c,c'}: \Sh_{\widetilde{L} \cup T_c(\widetilde{L})}(M \times \bR) \xrightarrow{\sim} \Sh_{\widetilde{L} \cup T_{c'}(\widetilde{L})}(M \times \bR); \,\,\, T_{c,c'}(\widetilde{\cF}_{\text{dbl},c}) = \widetilde{\cF}_{\text{dbl},c'}.$$
    Choose $C' \gg C \gg 0$ such that (i)~$T^{*,\infty}(M \times (-\infty, C)) \cap T_{C'}(\widetilde{L}) = \emptyset$, and (ii)~there exists a diffeomorphism $\varphi_C: M \times (-\infty, +\infty) \xrightarrow{\sim} M \times (-\infty, C)$ such that $\varphi_C^*: T^{*,\infty}(M \times (-\infty, C)) \xrightarrow{\sim} T^{*,\infty}(M \times \bR)$ sends $\widetilde{L} \cap T^{*,\infty}_{\tau < 0}(M \times (-\infty, C))$ to $\widetilde{L}$. Write $j_C: M \times (-\infty, C) \hookrightarrow M \times \bR$ for the inclusion map. Then for $C' \gg C$,
    $$\widetilde{\cF} = \varphi_{C}^{-1}(j_C^{-1}\widetilde{\cF}_{\text{dbl},C'}) \in \Sh_{\widetilde{L}}(M \times \bR).$$
    Finally, since the stratification of $M \times \bR$ determined by the front projection of $\widetilde{L}$ extends to a stratification of $M \times [-\infty, +\infty]$, we can get a sheaf
    $$j_{\bR *}\widetilde{\cF} \in \Sh_{\widetilde{L}}(M \times [-\infty, +\infty]).$$

    In Step~(2), there are some differences with \cite{JinTreu17} (see their Section 3.18) when showing full faithfulness, in our case the argument is as follows. Fix two sheaves $\widetilde{\cF}, \widetilde{\cG} \in \Sh_{\widetilde{L}}(M \times \bR)$. By adjunction we have
    $$Hom(\pi_{M!}\widetilde{\cF}, \pi_{M!}\widetilde{\cG}) \simeq Hom(\widetilde{\cF}, \pi_{M}^!\pi_{M!}\widetilde{\cG}).$$
    Choose $a' \gg a \gg 0$. Consider a Hamiltonian flow $\varphi_H$ on $T^{*,\infty}_{\tau<0}(M \times \bR)$ that carries $(\Lambda \times (-a, +\infty)) \cup (\widetilde{L} \cap T^{*,\infty}(M \times (-a, +\infty)))$ to $(T_{-a'}(\widetilde{L}) \cup \widetilde{L}) \cap T^{*,\infty}(M \times (-a, +\infty))$. Write $j_a: M \times (-a, +\infty) \rightarrow M \times \bR$ for the inclusion, and $\pi_{M,a}: M \times (-a, +\infty) \rightarrow M$ the projection. Since $T_{-a'}(\widetilde{L}) \subset T^{*,\infty}_{\tau<0}(M \times \bR)$, by microlocal Morse lemma \cite{KSbook}*{Proposition 5.4.17~(ii)} for proper push forward (instead of standard pushforward as in \cite{JinTreu17}), we know that
    $$\Phi_H(\pi_{M,a}^!\pi_{M!}\widetilde{\cG}) \simeq \Phi_H(\pi_{M,a}^!\pi_{M,a!}j_a^{-1}T_{-a'}(\widetilde{\cG})) \simeq j_a^{-1}T_{-a'}\widetilde{\cG}.$$
    Then by the singular support estimate \cite{KSbook}*{Proposition 5.4.14} and microlocal Morse lemma again we have
    $$Hom(\widetilde{\cF}, \pi_M^{-1}\pi_{M!}\widetilde{\cG}[1]) \simeq \Gamma(M \times (-a, +\infty), \mathscr{H}om(\widetilde{\cF}, T_{-a'}\widetilde{\cG})) \simeq Hom(\widetilde{\cF}, T_{-a'}\widetilde{\cG}).$$
    However since $T_{-a'}\widetilde{L} \subset T^{*,\infty}_{\tau<0}(M \times \bR)$, following \cite{JinTreu17}*{Proposition 3.16} we will conclude that $Hom(\widetilde{\cF}, T_{-a'}\widetilde{\cG}) \simeq Hom(\widetilde{\cF}, \widetilde{\cG})$, which gives the full faithfulness.
    
    Finally, consider the extended sheaf $j_{M \times \bR, *}\widetilde{\cG} \in \Sh_{\widetilde{L}}(M \times [-\infty, +\infty]).$ Write $\ol\pi_M: M \times [-\infty, +\infty] \rightarrow M$ for the proper projection and $i_c: M \times [c, +\infty] \hookrightarrow M \times [-\infty, +\infty]$ for the closed embedding. Then by microlocal Morse lemma, we know that
    $$\pi_{M!}\widetilde{\cG} \simeq \ol\pi_{M!}j_{M \times \bR, *}\widetilde{\cG} \simeq \ol\pi_{M!}\,i_{c!}\,i_c^{-1}\,(j_{M \times \bR, *}\widetilde{\cG}) \simeq i_{+\infty}^{-1}(j_{M \times \bR, *}\widetilde{\cG}).$$
    This completes the proof.
\end{proof}
\begin{remark}\label{rem:Jin-Treu-app}
    The sheaf quantization functor in Theorem \ref{thm:JinTreu-app} can be naturally restricted to pseudoperfect objects, in other words sheaves with perfect stalks,
    $$\Loc_\epsilon^{pp}(\widetilde{L}) \hookrightarrow \Sh^{pp}_{\widetilde{L}}(M \times [-\infty, +\infty]) \hookrightarrow \Sh^{pp}_\Lambda(M).$$
    In addition, by passing to the left adjoints,    as in Appendix \ref{appen:cat-sheaf}., we are able to obtain a functor, which is the sheaf theoretic Viterbo restriction functor
    $$\Sh^{cpt}_\Lambda(M) \rightarrow \Sh^{cpt}_{\widetilde{L}}(M \times [-\infty, +\infty]) \rightarrow \Loc_\epsilon^{cpt}(\widetilde{L}).$$
\end{remark}

\noindent    Moreover, we can show the Hamiltonian invariance of the sheaf quantization in the same spirit as \cite{JinTreu17}*{Section 3.19-20}. \\
    
    Recall that given a family of Lagrangian fillings that are conical at infinity, Lemma \ref{lem:lowerexact} ensures that there exists a Hamiltonian flow such that the image of the family of exact Lagrangians will have primitives that are proper and bounded from below.

\begin{prop}\label{prop:JinTreu-inv-app}
Let $L \subset (T^*M,\la_{st})$ be a relatively spin exact Lagrangian filling of a Legendrian submanifold $\La\subset T^{*,\infty}M$, $\widetilde{L}\subset(J^1M,\xi_{st})$ its Legendrian lift and $\cL \in \Loc_\epsilon(\cI_{\widetilde{L}})$ a twisted local system. Consider a Hamiltonian  $H:T^*M\lr\R$ that is homogeneous at infinity, $\varphi_t\in \mbox{Ham}(T^*M,d\la_{st})$ its time-$t$ flow, $t\in[0,1]$, which extends to an homonymous contactomorphism $\varphi_t\in \mbox{Cont}(T^{*,\infty}M,\ker\la_{st})$ of the ideal contact boundary, $t\in[0,1]$, $\widetilde H=\Pi^*_M(H):J^1M\lr\R$ the pull-back and $\widetilde{\varphi}_t\in\mbox{Cont}(J^1M,\xi_{st})$ its time-$t$ flow.\\

\noindent Suppose that $\widetilde{\cF}_t \in \Sh_{\widetilde{L}_t}(M \times \bR)$ is a sheaf quantization of $(\varphi_t(L),(\varphi_{t})_*\cL)$ and consider the proper push-forward $\cF_t := \pi_{M,!}(\widetilde{\cF}_t) \in \Sh_{\varphi_t(\Lambda)}(M)$, $t\in[0,1]$. Then:

\begin{itemize}
    \item[(i)] The sheaf kernel convolution $\Phi_{\widetilde{H}}:\Sh_{\widetilde{L}}(M \times \bR)\lr\Sh_{\varphi_1(\widetilde{L})}(M \times \bR)$ associated to $\widetilde{H}$ satisfies $\Phi_{\widetilde{H}}(\widetilde{\cF}_0) \simeq \widetilde{\cF}_1$.\\
    
    \item[(ii)] The sheaf kernel convolution $\Phi_H:\Sh_{\Lambda}(M)\lr\Sh_{\varphi_1(\Lambda)}(M)$ associated $H$ satisfies $\Phi_H(\cF_0) \simeq \cF_1.$
\end{itemize}
\end{prop}
\begin{proof}
    Consider the Legendrian isotopy of the Legendrian lifts $\widetilde{L}_t$, $t \in [0, 1]$, induced by the contact Hamiltonian $\widetilde{H}$. We have the following commutative diagram
    \[\xymatrix@C=5mm{
    \Loc_\epsilon(\cI_{\widetilde{L}}) \ar[d] & \Loc_\epsilon(\cI_{\widetilde{L} \times [0, 1]}) \ar[l]_{i_0^{-1}\hspace{15pt}} \ar[r]^{\hspace{15pt}i_1^{-1}} \ar[d] & \Loc_\epsilon(\cI_{\widetilde{L}}) \ar[d] \\
    \Sh_{\widetilde{L}_0}(M \times (-\infty,+\infty)) & \Sh_{\widetilde{L}_{\widetilde{H}}}(M \times (-\infty,+\infty) \times [0, 1]) \ar[l]_{i_0^{-1}\hspace{15pt}} \ar[r]^{\hspace{15pt}i_1^{-1}} & \Sh_{\widetilde{L}_1}(M \times (-\infty,+\infty)).
    }\]
    Note that the composition in the first row is the identity, while the composition in the second row is $\Phi_{\widetilde{H}}$. Therefore we know that $\Phi_{\widetilde{H}}(\widetilde{\cF}_0) \simeq \widetilde{\cF}_1.$ 
    
    Next, consider the following commutative diagram
    \[\xymatrix@C=5mm{
    \Sh_{\widetilde{L}_0}(M \times (-\infty,+\infty)) \ar[d]_{\pi_{M!}} & \Sh_{\widetilde{L}_{\widetilde{H}}}(M \times (-\infty,+\infty) \times [0, 1]) \ar[l]_{i_0^{-1}\hspace{15pt}} \ar[r]^{\hspace{15pt}i_1^{-1}} \ar[d]_{\pi_{M \times [0, 1]!}} & \Sh_{\widetilde{L}_1}(M \times (-\infty,+\infty)) \ar[d]_{\pi_{M!}} \\
    \Sh_{\Lambda_0}(M) & \Sh_{\Lambda_H}(M \times [0, 1]) \ar[l]_{i_0^{-1}\hspace{15pt}} \ar[r]^{\hspace{15pt}i_1^{-1}} & \Sh_{\Lambda_1}(M).
    }\]
    Note that the composition in the first row is $\Phi_{\widetilde{H}}$, while the composition in the second row is $\Phi_H$. By proper base change formula we have
    $$\Phi_H(\cF_0) = \Phi_H(\pi_{M!}\widetilde{\cF}_0) \simeq \pi_{M!}(\Phi_{\widetilde{H}}\widetilde{\cF}_0) \simeq \pi_{M!}\widetilde{\cF}_1 = \cF_1,$$
    which completes the proof. 
    
    Finally, we remark that sheaves in $\Sh_{\widetilde{L}_t}(M \times \bR)$ can be naturally extended to $\Sh_{\widetilde{L}_0}(M \times [-\infty,+\infty])$, and we can still get the commutative diagram as above.
\end{proof}

\begin{remark}
    There is another sheaf quantization functor following \cite{NadShen20} $\Loc_\epsilon(\cI_L) \hookrightarrow \Sh_\Lambda(M)$ independent of the work of S.~Guillermou and X.~Jin-D.~Treumann. In general, we do not know whether the two constructions coincide. However, in all examples in this paper, one can verify that the two constructions yield the same result.
\end{remark}

\subsection{Moduli space and framed sheaves}\label{appen:sheaf-moduli}
    The microlocal rank is a discrete piece of data that can be specified to choose components in the derived stack of pseudo-perfect objects in $\Sh_{\La}^{cpt}(M)$. In fact, we also study the moduli space of sheaves with additional framing data. In this section, we will interpret framing data from different perspectives, focusing on the case of Legendrian links. Given an exact Lagrangian filling $L$, not necessarily embedded, with Legendrian lift $\widetilde{L}$, we will also define microlocal merodromy functions of framed sheaves with singular support on $\widetilde{L}$ the along relative 1-cycles.

\subsubsection{Moduli space of framed microlocal sheaves}
The first moduli space that we focused on is defined as follows: 


\begin{definition}\label{def:moduli-sh}
    Let $\Sigma$ be a surface with marked points\footnote{In the whole section, marked points always mean the punctures on the surface. In Legendrian contact homology, people also use the term marked points for a different meaning. We will call those points base points to avoid confusion.} and $\Lambda \subset T^{*,\infty}\Sigma$ a non-empty Legendrian link equipped with Maslov potential. By definition, the derived stack $\bR \cM(\Sigma, \Lambda)$ is the derived moduli stack of the finite type dg-category $\Sh^{cpt}_\Lambda(\Sigma)$, as defined in \cite{ToenVaquie07}, which parametrizes pseudoperfect modules $\Sh^{pp}_\Lambda(\Sigma)$.\\
    
    Let $\bR \cM_r(\Sigma, \Lambda)_0$ be the locus in $\bR \cM(\Sigma, \Lambda)$ parametrizing microlocal rank~$r$ sheaves in $\Sh_\Lambda^{pp}(\Sigma)$ with acyclic stalks at the marked points. Then the Artin stack $\cM_r(\Sigma, \Lambda)_0$ is the 1-rigid locus of the truncated derived moduli of objects $t_0\bR \cM_r(\Sigma, \Lambda)_0$.
\end{definition}

\noindent The fact that $\Sh^{cpt}_\Lambda(\Sigma)$ is a smooth dg-category when $\Lambda$ is subanalytic isotropic is proved in \cite[Section 4.5]{GPS18b}; see also Appendix \ref{appen:cat-sheaf}.

\begin{remark}
    The moduli $\cM_r(\Sigma, \Lambda)_0$ only parametrizes microlocal rank~$r$ sheaves with singular support on $\Lambda$ with no negative self extensions, as it is truncated and the 1-rigid locus is selected. For instance, the $m(5_2)$ Legendrian knot with maximal Thurston-Bennequin invariant that does not have a binary Maslov potential supports microlocal rank~1 sheaves with negative self extensions \cite{STZ17}*{Proposition 7.6}; these are not parametrized by that above Artin stack.
\end{remark}

    From the sheaf quantization result Theorem \ref{thm:JinTreu-app}, see Remark \ref{rem:Jin-Treu-app}, for an exact Lagrangian filling $L$ of $\Lambda$ we have the left adjoint to the fully faithful sheaf quantization functor
    $$\Sh^{cpt}_\Lambda(\Sigma) \lr \Sh^{cpt}_{\widetilde{L}}(\Sigma \times \bR) \lr \Loc^{cpt}(\widetilde{L}),$$
    as is explained in Appendix \ref{appen:cat-sheaf}. Therefore, due to the full faithfulness, we obtain an open embedding of moduli spaces by \cite[Theorem 0.2]{ToenVaquie07}, see also {\cite[Proposition 2.15]{STWZ19}}, which reads as follows:
    
\begin{prop}\label{prop:embed-moduli}
    Let $L \subset T^*\Sigma_\text{op}$ be an exact Lagrangian filling of $\Lambda \subset T^{*,\infty}\Sigma$. Then there is an open embedding of moduli spaces $\bR \cL oc(L) \hookrightarrow \bR \cM(\Sigma, \Lambda)$ and $\cL oc_r(L) \hookrightarrow \cM_r(\Sigma, \Lambda)_0$.
\end{prop}

    The framed analogues are discussed as follows. First, we define the (micro)framed moduli space using trivializations of microstalks at a point instead of sheaves locally defined near a point. This is closer in relation to both the definition of base points in Legendrian contact homology and the definition of decorated moduli space of sheaves of \cite{CasalsWeng22}*{Definition 2.37}. We introduce the following new definition:

\begin{definition}\label{def:moduli-sh-mufr}
    A microframing $(T, t_T)$ of microlocal rank~$r$ sheaves in $\cM_r(\Sigma, \Lambda)_0$ consists of a finite subset $T \subset \Lambda$ and a rank~$r$ local system $t_T \in \cL oc_r(T)$. The microframed moduli space of microlocal rank~$r$ sheaves with singular support in $\Lambda$ with respect to the framing $(T, \cF_T)$ is the fiber product
    \[\xymatrix{
    \cM^{\mu,\textit{fr}}_r(\Sigma, \Lambda)_0 \ar[r] \ar[d] & \cM_r(\Sigma, \Lambda)_0 \ar[d]^{m_{\Lambda,T}} \\
    \mathrm{Spec}\,\Bbbk \ar[r]^{t_T} & \cL oc_r(T).
    }\]
\end{definition}
\begin{remark}
    The map between moduli spaces $\cM_r(\Sigma, \Lambda)_0 \rightarrow \cL oc_r(T)$ is induced by the left adjoint of the microlocalization functor and restriction functor (see Appendix \ref{appen:cat-sheaf})
    $$m_{\Lambda,T}^L: \Loc^{cpt}(T) \rightarrow \Loc^{cpt}(\Lambda) \rightarrow \Sh_\Lambda^{cpt}(M).$$
\end{remark}
\begin{remark}
    When $\Lambda$ is equipped with a relative spin structure, or equivalently a section $s$ in $\cI_{\Lambda \times \bR}$ over the 2-skeleton, then strictly speaking, the microframing is defined over the section at $T$ as $t_T \in \cL oc_r(s({T}))$.
\end{remark}

    Similarly, given an exact Lagrangian filling $L$ of $\Lambda$, consider the moduli space of (micro)framed local systems $\cL oc_r^\textit{fr}(L)$ as the fiber product
    \[\xymatrix{
    \cL oc^{\textit{fr}}_r(L) \ar[r] \ar[d] & \cL oc_r(L) \ar[d] \\
    \mathrm{Spec}\,\Bbbk \ar[r]^{t_T} & \cL oc_r(T).
    }\]
    Then the sheaf quantization result Theorem \ref{thm:JinTreu-app} and Proposition \ref{prop:embed-moduli} asserts the following open embedding of (micro)framed moduli spaces. 
    
\begin{cor}
    Let $L \subset T^*\Sigma_\text{op}$ be an exact Lagrangian filling of $\Lambda \subset T^{*,\infty}\Sigma$. Then there is an open embedding of moduli spaces $\cL oc^{\textit{fr}}_r(L) \hookrightarrow \cM_r^{\mu,\textit{fr}}(\Sigma, \Lambda)_0$.
\end{cor}
    
    Here we use the terminology of microframings in order to distinguish them from the definition of framing data in \cite{STWZ19}*{Section 2.4}, which reads as follows:\\
  
\begin{definition}[\cite{STWZ19}*{Definition 2.19}]\label{def:moduli-sh-fr}
    A framing $(U, \cF_U)$ of microlocal rank~$r$ sheaves in $\cM_r(\Sigma, \Lambda)_0$ consists of an open subset $U \subset \Sigma$ and a microlocal rank~$r$ sheaf $\cF_U \in \cM_r(U, \Lambda \cap T^{*,\infty}U)_0$. The framed moduli space of microlocal rank~$r$ sheaves with singular support in $\Lambda$ with respect to the framing $(U, \cF_U)$ is the fiber product
    \[\xymatrix{
    \cM^\textit{fr}_r(\Sigma, \Lambda)_0 \ar[r] \ar[d] & \cM_r(\Sigma, \Lambda)_0 \ar[d] \\
    \mathrm{Spec}\,\Bbbk \ar[r]^{\cF_U\hspace{30pt}} & \cM_r(U, \Lambda \cap T^{*,\infty}U)_0.
    }\]
\end{definition}
\begin{remark}
    We can always consider a sufficiently small contractible open subset $U \subset \Sigma$ containing $\pi(T)$, choose $\cF_U$ such that $m_{\Lambda,P}(\cF_U) = \cF_T$, and ask for these framing data and microframing data whether $\cM^{\textit{fr},\mu}_r(\Sigma, \Lambda)_0 \simeq \cM^\textit{fr}_r(\Sigma, \Lambda)_0$. In general, this is false. To wit, we consider the right-handed Legendrian trefoil with maximal Thurston-Bennequin number whose front projection admits a Maslov potential with values $-1, 0, 1$, and let $p \in \Lambda$ as in Figure \ref{fig:frame-muframe}. Then
    $$\cM^\textit{fr}_r(\bD^2, \Lambda)_0 \rightarrow \cM_r(\bD^2, \Lambda)_0$$
    is not surjective, while $\cM^{\mu,\textit{fr}}_r(\bD^2, \Lambda)_0 \rightarrow \cM_r(\bD^2, \Lambda)_0$ is surjective.
\end{remark}
\begin{figure}[h!]
    \centering
    \includegraphics[width=0.55\textwidth]{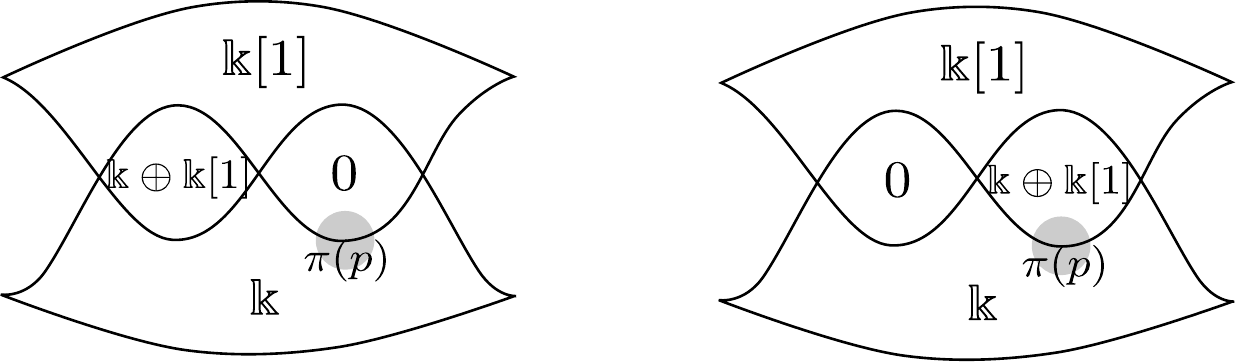}
    \caption{The right-handed Legendrian trefoil with maximal Thurston-Bennequin number whose front projection admits a Maslov potential with values $-1, 0, 1$. The restrictions of two sheaves in the grey region $U$ around $\pi(p) \in \pi(\Lambda)$ are clearly different.}\label{fig:frame-muframe}
\end{figure}

\noindent Despite the fact that a microframing is not always a framing, we can still discuss two simple cases where these notions agree. For connected Legendrians, there is always an appropriate choice of a framing in the sense of Definition \ref{def:moduli-sh-fr} that can be understood as simply choosing base point and specifying a trivialization of the microstalks at the base point.

\begin{lemma}\label{lem:frame-markpt-gen}
    Let $\Lambda \subset T^{*,\infty}\Sigma$ be a Legendrian knot equipped with Maslov potential. Let $\cM_r(\Sigma, \Lambda)_0$ be the moduli space of microlocal rank~$r$ sheaves with acyclic stalk in a specific stratum $V$ of $\Sigma \backslash \pi(\Lambda)$. Let $p \in \Lambda$ be a base point such that $\pi(p)$ is adjacent to $V$, and $U$ be a small neighbourhood of $\pi(p)$ such that $\pi(\Lambda) \cap U$ is smoothly embedded. Then there exists a framing data $(U, \cF_U)$ such that $\cM^\textit{fr}_r(\Sigma, \Lambda)_0 \cong \cM^{\mu,\textit{fr}}_r(\Sigma, \Lambda)_0$.
\end{lemma}
\begin{proof}
    Consider $\cF = \Bbbk^r_{\bR \times \bR_{\geq 0}}[d(p)]$ when the codirection of $p$ is pointing toward the stratum $V$ and $\cF = \Bbbk^r_{\bR \times \bR_{>0}}[d(p)-1]$ otherwise, where $d(p) \in \bZ$ is the Maslov potential. Then for any $\cF \in Sh^b_\Lambda(\Sigma)$ of microlocal rank~$r$, there is always an isomorphism $\cF_U \xrightarrow{\sim} \cF$ and the choice of the isomorphism is equivalent to the the $t: m_\Lambda(\cF)_p \xrightarrow{\sim} \Bbbk^r$.
\end{proof}
\begin{remark}
    Reduction of base points and the relation between our choice of base points and the choice in \cite{STWZ19}*{Section 3.2} for Legendrian $(n,k)$-torus links $\Lambda_{n,k} \subset T^{*,\infty}\bD^2$ with maximal Thurston-Bennequin numbers (where $\Lambda_{n,k}$ are realized as Legendrian satellites of the outward unit conormal bundle of a disk following Example \ref{ex:braid-conj-red}) is as follows. In \cite{STWZ19}*{Section 3.2} the base points are chosen so that their front projection adjacent to the unbounded stratum and the strata corresponding to white vertices in the $(n+k)$-gon in Figure \ref{fig:braidgraphpolygon}). The stalk in the unbounded stratum is assigned to be acyclic. Hence in that case the framing data is also equivalent to the microframing data.
\end{remark}

    The above case is not sufficient for applications. Now we focus on Legendrian braid closures and discuss the case of multiple base points.

\begin{prop}\label{prop:frame-markpt-app}
    Let $\Lambda \subset T^{*,\infty}\Sigma$ be a Legendrian positive braid closure (either cylindrical closure or rainbow closure) equipped with Maslov potential, $D \subset \Sigma$ an open disk such that $\Lambda \cap T^*D$ consist of $n$ parallel strands. Let $T = \{p_1, ..., p_n\} \subset \Lambda$ be a collection of base points in $D$, one on each strand, $U_j$ a sufficiently small neighbourhood of $\pi(p_j) \in \Sigma$ and $U = \bigcup_{j=1}^mU_j$. Then there exists a framing data such that $\cM^\textit{fr}_r(\Sigma, \Lambda)_0 \cong \cM^{\mu,\textit{fr}}_r(\Sigma, \Lambda)_0$.
\end{prop}

\begin{figure}[h!]
    \centering
    \includegraphics[width=0.7\textwidth]{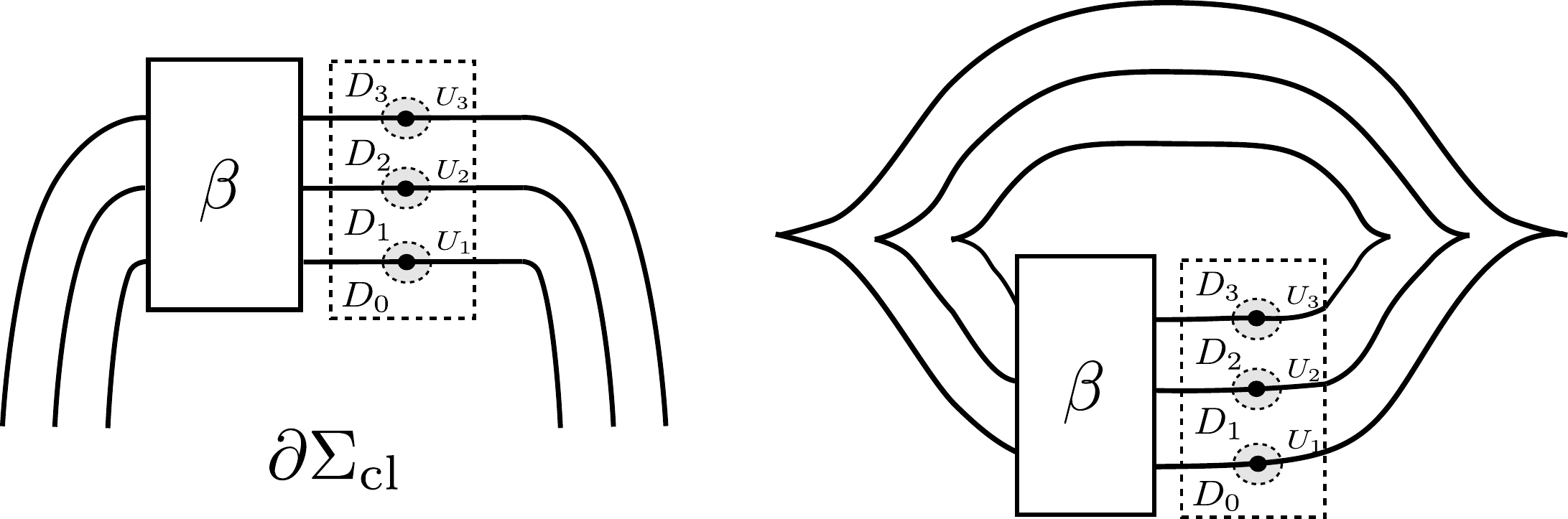}
    \caption{The Legendrian positive braid closure $\Lambda \subset T^{*,\infty}\Sigma$, the open disk $D \subset \Sigma$ specified in the proposition, the base points $T = \{p_1, \dots, p_n\} \subset \Lambda$ which determine the microframing of sheaves, and the small open neighbourhoods $U = \bigcup_{j=1}^mU_j$ around projections of base points which determine the framing of sheaves.}
    \label{fig:braid-markpt-appen}
\end{figure}

\begin{proof}
    Given a Legendrian positive $n$-braid closure $\Lambda$ and suppose $\cF \in Sh^b_\Lambda(\Sigma)$ a microlocal rank 1 sheaf. Fix the base point $p_i$ lying on the $i$-th strand of $\Lambda$ and a small contractible neighbourhood $U_i$ of $\pi(p_i)$. Consider the rectangular regions $D_i\,(0 \leq i \leq n)$ lying in between the $(i-1)$ and $i$-th strand, as shown in Figure \ref{fig:braid-markpt}. Since the stalks of $\cF$ at marked points $\{x_1, ..., x_m\} \subset \Sigma$ are acyclic and the Maslov potential is fixed, following Example \ref{ex:sheafcombin} we know that
    $$m_\Lambda(\cF)_{p_i} \simeq \mathrm{Cone}(\Gamma(D_{i-1}, \cF) \rightarrow \Gamma(D_i, \cF)) \simeq \Bbbk^r,$$
    and hence by induction $\Gamma(D_i, \cF) \simeq \Bbbk^i$. Since $D_i$ are contractible, we have $\cF_{D_i} \simeq \Bbbk^i_{D_i}$.

    Define $\cF_i \in Sh^b_{\Lambda \cap T^{*,\infty}U_i}(U_i)$ to be the (unique) microlocal rank 1 sheaf such that
    $$\Gamma(U_i \cap D_{i-1}, \cF) = \Bbbk^{(i-1)r}, \,\, \Gamma(U_i \cap D_i, \cF) = \Bbbk^{ir},$$
    and the map $\Gamma(U_i \cap D_{i-1}, \cF) \rightarrow \Gamma(U_i \cap D_i, \cF)$ is the inclusion of the first $(i-1)r$ coordinates. We claim that there is a bijection between the trivialization data of the microstalk
    $$t_i: \, m_\Lambda(\cF)_{p_i} \xrightarrow{\sim} \Bbbk^r$$
    and the pairs $(\cF_i, \varphi_i)$ consisting of $\cF_i\,(1 \leq i \leq n)$ and specified quasi-isomorphisms
    $$\varphi_i: \, \cF|_{U_i} \xrightarrow{\sim} \cF_i|_{U_i}.$$

    First, suppose that the data $t_i\,(1 \leq i \leq n)$ are given. Then we prove by induction that there is a unique quasi-isomorphism
    $$T_i^-: \, \Gamma(U_i \cap D_{i-1}, \cF) \xrightarrow{\sim} \Bbbk^{(i-1)r}, \,\,\, T_i^+: \, \Gamma(U_i \cap D_i, \cF) \xrightarrow{\sim} \Bbbk^{ir}.$$
    Then since $U_i \cap D_{i-1}, U_i \cap D_i$ are contractible, there are canonical quasi-isomorphisms
    $$\cF|_{U_i \cap D_{i-1}} \xrightarrow{\sim} \cF_{i-1}|_{U_i \cap D_{i-1}}, \,\,\, \cF|_{U_i \cap D_i} \simeq \cF_i|_{U_i \cap D_i},$$
    which uniquely determine quasi-isomorphisms $\varphi_{i-1}: \, \cF|_{U_i} \xrightarrow{\sim} \cF_i|_{U_i}.$ Therefore, given the trivializations of microstalks $t_i\,(1\leq i \leq n)$, to get unique choices of $\varphi_i\,(1 \leq i\leq n)$, it suffices to prove by induction that the choices of $T_i^\pm\,(1 \leq i \leq n)$ are uniquely determined by $t_i\,(1\leq i \leq n)$. When $i = 0$ this is the data
    $$T_1^-: \,\Gamma(U_1 \cap D_0, \cF) \xrightarrow{\sim} 0,$$
    which is obviously unique. Suppose $T_i^-$ is uniquely determined. Then it follows from the five lemma that there exists a unique quasi-isomorphism $T_i^+$ such that the diagram commutes
    \[\xymatrix{
    \Gamma(U_i \cap D_{i-1}, \cF) \ar[r] \ar[d]^{T_i^-} & \Gamma(U_i \cap D_i, \cF) \ar[r] \ar[d]^{T_i^+} & m_\Lambda(\cF)_{p_i} \ar[r]^{\hspace{15pt}+1} \ar[d]^{t_i} & \\
    \Bbbk^{(i-1)r} \ar[r] & \Bbbk^{ir} \ar[r] & \Bbbk^r \ar[r]^{\hspace{15pt}+1} & .
    }\]
    Suppose now that $T_i^+: \Gamma(U_i \cap D_i, \cF) \xrightarrow{\sim} \Bbbk^{ir}$ is uniquely determined, then it follows from the fact that $D_i, U_i \cap D_i, U_{i+1} \cap D_i$ are contractible that there are canonical quasi-isomorphisms
    $$T_{i+1}^-: \Gamma(U_{i+1} \cap D_i, \cF) \xrightarrow{\sim} \Gamma(D_i, \cF) \xrightarrow{\sim} \Gamma(U_i \cap D_i, \cF) \xrightarrow{\sim} \Bbbk^{ir}.$$
    Therefore, the data $t_i\,(1\leq i \leq n)$ uniquely determines $T_i^\pm\,(0 \leq i\leq n)$ and hence the framing data $\varphi_i\,(1 \leq i \leq n)$.

    Next, assume that the data $\varphi_i\,(1 \leq i\leq n)$ are given. Then from the previous discussion, we know that the data $T_i^\pm\,(0 \leq i\leq n)$ are thus uniquely given. Hence it follows from the five lemma that there exists a unique trivialization
    $$t_i: \, m_\Lambda(\cF)_{p_i} \xrightarrow{\sim} \Bbbk^r$$
    that makes the diagram commutes. This implies that the bijection we have claimed.
\end{proof}
\begin{remark}
    Consider Legendrian $(n,k)$-torus links with maximal Thurston-Bennequin invariant, as in Subsection \ref{ssec:triangle-positive}). The framing data in \cite{STWZ19}*{Section 3.2} is equivalent to the microframing data at $(n+k)$ base points adjacent to the $(n+k)$ vertices coming from the plabic graph on the $(n+k)$-gon by Lemma \ref{lem:frame-markpt-gen}, while here we only choose $n$ base points. It is not difficult to check that when we move the $(n+k)$ base points counterclockwise to the region $D$ as in Proposition \ref{prop:frame-markpt-app} there will be at least one base point on each strand in $\Lambda_{n,k} \cap T^{*,\infty}D$. Therefore there is a forgetful map from the framed moduli of sheaves in \cite{STWZ19}*{Section 3.2} to our framed moduli space, whose fiber is an algebraic torus $(\Bbbk^\times)^k$.
\end{remark}

\subsubsection{Microlocal merodromies of sheaves}
    Our goal in this section is to geometrically introduce the notion of microlocal merodromies along a relative 1-cycle and thus define local coordinate functions on the moduli space $\cM_r^{\mu,\textit{fr}}(\Sigma, \Lambda)_0$; such construction was implicitly mentioned in \cite{STWZ19}*{Section 2.4 \& 5.3}, up to signs, and appeared in detail in \cite{CasalsWeng22}.

    Let $\Lambda \subset T^{*,\infty}\Sigma$ be a Legendrian link with Maslov potential. For a microframed sheaf $\cF \in \cM_r^{\mu,\textit{fr}}(\Sigma, \Lambda)_0$, we have specified trivializations of microstalks at base points on $\Lambda_i$. Given a Lagrangian filling $L$ of $\Lambda$, a relative 1-cycle in $H_1(L, \Lambda; \bZ)$ could end at any point in $\Lambda$. Therefore we need to choose paths in $\Lambda$ that cap off the relative 1-cycles.

\begin{definition}
    Let $\Sigma$ be a surface with marked points, and $\Lambda = \bigsqcup_{i=1}^m\Lambda_i \subset T^{*,\infty}\Sigma$ be a Legendrian link with a chosen relative spin structure $s$ and a $\pi_0$-surjective set of base points $T = \bigsqcup_{i=1}^mT_i$, i.e.~each component $\Lambda_i$ is equipped with a non-empty set of base points $T_i$.
    
    For a base point $p \in \Lambda_i$, the capping path $c_p$ is the smooth path on $\Lambda_i$ that goes following the orientation of $\Lambda_i$, starting from $p$ and ending at some $q$ without intersecting any other base points. 
\end{definition}
\begin{remark}
    Note that by Lemma \ref{lem:twist-loc-lift}, when choosing a spin structure, we fix the a section $s$ of the 2-skeleton from $\Lambda$ to $\cI_{\Lambda \times \bR}$. Therefore, the capping paths are equipped with liftings $s \circ c_p$ in the relative oriented frame bundle $\cI_{\Lambda \times \bR}$.
\end{remark}


    Using capping paths, it follows that, for a sheaf $\cF \in \cM_r^{\mu,\textit{fr}}(\Sigma, \Lambda)_0$, the microlocalization $m_{\Lambda}(\cF)$ is uniquely trivialized in each component of $\Lambda \backslash P$. Conversely, given a trivialization of $m_{\Lambda}(\cF)$ in each of these connected component, which is an open interval from $p$ to $q$), we obtain a trivialization of the microstalk on the rightmost endpoint $q$. Consequently, we conclude the following.

\begin{prop}
    Let $\Lambda \subset T^{*,\infty}\Sigma$ be a Legendrian link with a spin structure $s$ and a $\pi_0$-surjective set of base points $T$. Then points in $\cM_r^{\mu,\textit{fr}}(\Sigma, \Lambda)_0$ consist of the data $\cF \in \cM_r(\Sigma, \Lambda)_0$ together with a trivialization of $m_{\Lambda}(\cF)$ on each connected component of $\Lambda \backslash T$.
\end{prop}
\noindent Note that this is the definition of decorated moduli space of sheaves in \cite{CasalsWeng22}*{Definition 2.37}.

Now consider any decorated sheaf $\cF \in \cM_r^{\mu,\textit{fr}}(\Sigma, \Lambda)_0$. Suppose $\cF$ is the sheaf quantization of a twisted local system on the Lagrangian filling $L$ (with Legendrian lift $\widetilde{L}$). By Theorem \ref{thm:JinTreu-app}, $\cF$ is the proper push forward of a sheaf quantization $\widetilde{\cF} \in \cM_r(\Sigma \times \bR, \widetilde{L})_0$. We want to introduce the notion of microlocal merodromy associated to the sheaf quantization of a (not necessarily embedded) exact Lagrangian filling.
    

\begin{definition}\label{def:micromero}
    Let $\Lambda \subset T^{*,\infty}\Sigma$ be a Legendrian link with the null-cobordant spin structure and a $\pi_0$-surjective set of base points $T \subset \Lambda$ and capping paths, and $\widetilde{L} \subset T^{*,\infty}(\Sigma \times \bR)$ a Legendrian surface with boundary $\Lambda$ with a spin structure $s$. Consider a decorated microlocal rank~$r$ sheaf $\widetilde{\cF} \in \cM_r^{\mu,\textit{fr}}(\Sigma \times \bR, \widetilde{L})_0$.\\

    Let $\gamma \in H_1(L \backslash T, \Lambda \backslash T; \bZ)$ be a relative cycle connecting $p$ to $q$ and $c_p, c_q$ be capping paths of $p, q$. By definition, the microlocal merodromy of $\widetilde{\cF}$ along $\gamma$ is the composition of parallel transport maps of the microstalks
    $$t_q \circ \phi_{s \circ {c}_q} \circ \phi_{s \circ {\gamma}} \circ \phi_{s \circ {c}_p}^{-1} \circ t_p^{-1}: \, \Bbbk^r \xrightarrow{\sim} m_{L}(\widetilde{\cF})_{p_{ij}} \xrightarrow{\sim} m_{L}(\widetilde{\cF})_p \xrightarrow{\sim} m_{L}(\widetilde{\cF})_q \xrightarrow{\sim} m_{L}(\widetilde{\cF})_{q_{ij'}} \xrightarrow{\sim} \Bbbk^r.$$
\end{definition}

\noindent The values of microlocal merodromies depend on the choice of the spin structure $s$, and the lifting of the paths in the relative oriented frame bundle. The reason for the spin structure on $\Lambda$ to be null cobordant is because it should coincide with the restriction of the spin structure on the surface $\widetilde{L}$.

\begin{remark}
    Consider the case $r = 1$. When $\widetilde{L}$ is the Legendrian lift of an embedded Lagrangian $L$, then sheaf quantization of the Lagrangian filling $L$ determines a torus chart in the framed moduli of sheaves by microlocal merodromies
    $$H^1(L \backslash T, \Lambda \backslash T; \Bbbk^\times) \xrightarrow{\sim} \cM_1^{\mu,\textit{fr}}(\Sigma \times \bR, \widetilde{L})_0 \hookrightarrow \cM_1^{\mu,\textit{fr}}(\Sigma, \Lambda)_0.$$
    The microlocal merodromy map is a regular function on a Zariski open set of $\cM_1^{\mu,\textit{fr}}(\Sigma, \Lambda)_0$. In Section \ref{sec:clust-A} we explain that they can actually be extended to regular functions on the whole space $\cM_1^{\mu,\textit{fr}}(\Sigma, \Lambda)_0$.
\end{remark}
\begin{remark}
    If $\Lambda$ is a Legendrian link with $n$ base points, at least one in each connected component, and $L$ a connected Lagrangian filling, then there is an exact sequence
    $$0 \rightarrow H_1(L, \Lambda; \bZ) \rightarrow H_1(L \backslash T, \Lambda \backslash T; \bZ) \rightarrow \widetilde{H}_0(T; \bZ) \rightarrow 0.$$
    The exact sequence splits and hence we obtain the isomorphism
    $$H_1(L \backslash T, \Lambda \backslash T; \Bbbk^\times) \cong H_1(L, \Lambda; \Bbbk^\times) \times (\Bbbk^\times)^{n-1}.$$
\end{remark}

\noindent Definition \ref{def:micromero} can be a priori difficult to work with in practice. However, since the choice of spin structures can be reduced to the choice of sign curves, microlocal merodromies can actually be computed in examples, as shown in Section \ref{sec:quan-weave}.
\newpage

\bibliography{ConjugateWeave}
\bibliographystyle{amsplain}

\end{document}